\documentclass[11pt]{amsart}

\usepackage{float}
\usepackage{placeins}
\usepackage{caption}

\usepackage{amsmath,amssymb,amscd,amsthm,graphicx,enumerate, mathrsfs, amsfonts}
\usepackage{graphicx}  
\usepackage{float}
\usepackage{hyperref}
\usepackage{xypic}
\usepackage{xcolor}
\usepackage{comment}
\hypersetup{breaklinks=true}
\hypersetup{
		colorlinks   = true, 
		urlcolor     = blue, 
		linkcolor    = blue,
		citecolor    = blue 
}

\usepackage[left=2.5cm, right=2.5cm, bottom=2.7cm]{geometry}

\numberwithin{equation}{section}
\usepackage{mathtools}
\mathtoolsset{showonlyrefs}

\theoremstyle{plain}

\newtheorem{theorem}{Theorem}[section]
\newtheorem{proposition}[theorem]{Proposition}
\newtheorem{lemma}[theorem]{Lemma}
\newtheorem{corollary}[theorem]{Corollary}

\newtheorem{claim}[theorem]{Claim}

\theoremstyle{definition}
\newtheorem{definition}[theorem]{Definition}
\newtheorem{remark}[theorem]{Remark}

\newcommand{\fn}{function}
\newcommand{\Rm}{Riemannian}
\newcommand{\cn}{constant}

\newcommand{\mt}{metric}
\newcommand{\seq}{sequence}
\newcommand{\w}{with}
\newcommand{\cts}{continuous}
 
\newcommand{\Thm}{Theorem}

\newcommand{\st}{such that}

\newcommand{\te}{there exist}
\newcommand{\tes}{there exists}
\newcommand{\Te}{There exist}
\newcommand{\Tes}{There exists}
\newcommand{\fa}{for all}
\newcommand{\tf}{Therefore}
\newcommand{\sps}{Suppose} 
\newcommand{\hn}{Hence}

\newcommand{\wrt}{with respect to}

\newcommand{\bbr}{\mathbb{R}}

\newcommand{\bbb}{\mathbb{B}}

\newcommand{\ra}{\rightarrow}

\newcommand{\del}{\partial}

\newcommand{\cA}{\mathcal{A}}
\newcommand{\rg}{\mathrm{g}}
\newcommand{\bn}{\mathbf{n}}

\newcommand{\A}{\mathcal{A}}
\newcommand{\Ric}{\operatorname{Ric}}

\newcommand{\al}{\alpha}
\newcommand{\be}{\beta}
\newcommand{\ga}{\gamma}
\newcommand{\de}{\delta}
\newcommand{\ve}{\varepsilon}
\newcommand{\et}{\eta}
\newcommand{\ph}{\phi}
\newcommand{\vp}{\varphi}
\newcommand{\ps}{\psi}

\newcommand{\la}{\lambda}
\newcommand{\om}{\omega}
\newcommand{\rh}{\rho}
\newcommand{\si}{\sigma}
\newcommand{\tht}{\theta}
\newcommand{\ta}{\tau}
\newcommand{\ch}{\chi}

\newcommand{\Ga}{\Gamma}
\newcommand{\De}{\Delta}
\newcommand{\Ph}{\Phi}

\newcommand{\La}{\Lambda}
\newcommand{\Om}{\Omega}
\newcommand{\Si}{\Sigma}
\newcommand{\Tht}{\Theta}
\newcommand{\na}{\nabla}
\newcommand{\eps}{\varepsilon}
\newcommand{\nm}[1]{\left\|#1\right\|}
\newcommand{\md}[1]{\left|#1\right|}

\newcommand{\mD}[1]{|#1|}

\newcommand{\cx}{\mathcal{X}}
\newcommand{\cy}{\mathcal{Y}}
\newcommand{\vth}{\vartheta}
\newcommand{\va}{v_{\alpha}}
\newcommand{\ffa}{\mathfrak{f}_{\al}}
\newcommand{\fwa}{\mathfrak{w}_{\alpha}}
\newcommand{\faa}{\mathfrak{a}_{\alpha}}

\newcommand{\sfu}{\mathsf{u}}
\newcommand{\snf}{\mathsf{f}}
\newcommand{\snG}{\mathsf{G}}

\newcommand{\tui}{\tilde{u}_i} 
\newcommand{\tuin}{\tilde{u}_{\infty}}

\newcommand{\sff}{\mathrm{I\!I}}
\newcommand{\Sia}{\Sigma_{\alpha}}
\newcommand{\Sib}{\Sigma_{\beta}}
\newcommand{\bha}{\mathbb{H}_{\alpha}}
\newcommand{\bhb}{\mathbb{H}_{\beta}}
\newcommand{\bbh}{\mathbb{H}}
\newcommand{\Hb}{\mathbb{H}}
\newcommand{\bH}{\overline{\mathbb{H}}}

\newcommand{\Gb}{\mathbb{G}}
\newcommand{\bg}{\mathbb{G}}
\newcommand{\bga}{\mathbb{G}_{\alpha}}
\newcommand{\bgah}{\hat{\bg}_{\al}}
\newcommand{\bgab}{\overline{\bg}_{\al}}

\newcommand{\vthai}{\vartheta_{\alpha}^{-1}}
\newcommand{\teta}{\tilde{\eta}_{\alpha}}

\DeclareMathOperator{\spt}{spt}

\newcommand{\apeqref}[2]{%
 \textup{(\hyperlink{#2}{\ref*{#1}})}%
}

\newcommand{\restatetag}[2]{%
  \tag{\hypertarget{#2}{}\hyperref[#1]{\ref*{#1}}}%
}

\begin{document}

\title[Second order estimates on transition layers: Part I]{Second order estimates on transition layers for the Allen--Cahn equation with homogeneous Neumann boundary conditions: Part I -- Boundary-orthogonal transition layers}

\author[Akashdeep Dey]{Akashdeep Dey}
\address{Chennai Mathematical Institute, H1, SIPCOT IT Park, Chennai, Tamil Nadu 603103, India}
\email{akashdeep@cmi.ac.in}

\author[Wenkui Du]{Wenkui Du}
\address{School of Mathematics, Hunan University, Changsha, Hunan Province,  410082, China}
\email{duwenkui@hnu.edu.cn}

\author[Davide Parise]{Davide Parise}
\address{Department of Mathematics, Zeeman Building, University of Warwick, Gibbet Hill Road, Coventry CV4 7AL, UK}
\email{Davide.Parise@warwick.ac.uk}

\author[Lorenzo Sarnataro]{Lorenzo Sarnataro}
\address{Department of Mathematics, University of Toronto, 40 St. George Street, Toronto, ON M5S 2E4 Canada}
\email{lorenzo.sarnataro@utoronto.ca}
\begin{abstract}
We study stable solutions of the Allen--Cahn equation with homogeneous Neumann boundary condition on Riemannian manifolds with boundary, in the regime where the transition layers meet the boundary orthogonally. Under an a priori bound for the enhanced second fundamental form in the transition region, we establish uniform second-order H\"older estimates and quantitative mean curvature decay for the nodal hypersurfaces up to the boundary in ambient dimensions at most 10. The proof adapts the technique pioneered by Wang--Wei \cite{wang-wei,wang-weiadv}, adding a further correction determined by the geometry of the ambient boundary. Its leading contribution to mean curvature cancels on the nodal set but survives on non-zero level sets in the transition region, whose mean curvature need not decay uniformly. This phenomenon has no interior analogue and is exhibited by stable solutions constructed by Kowalczyk \cite{kowalczyk05}.
\end{abstract}

\maketitle
\vspace{-5mm}
\tableofcontents
 
\section{Introduction}
In this paper we establish second-order estimates for stable solutions of the Allen--Cahn equation with homogeneous Neumann boundary conditions on compact Riemannian manifolds with boundary, in the situation where the transition layers of the solution meet the boundary orthogonally. 
Together with the estimates in the forthcoming paper \cite{FBWW2}, which deal with the situation where the transition layers of the solution approach the boundary parallelly from one side, they fully extend the regularity theory developed by Wang--Wei \cite{wang-wei,wang-weiadv} to the case of Neumann solutions on manifolds with boundary. In \cite{FBWW3}, we shall use the estimates in Theorem \ref{t.main.mfd} and those in \cite{FBWW2} to obtain a min-max construction of free boundary geodesics and geodesic billiards in compact Riemannian surfaces with boundary, and free boundary minimal hypersurfaces in compact Riemannian manifolds with boundary of dimensions 3 and 4, entirely based on solutions of the Allen--Cahn equation. Moreover, we shall obtain a characterisation of curves realising the (phase-transition) $p$--widths of the length functional on a Riemannian surface with boundary similar to the one obtained by Chodosh--Mantoulidis in the closed case \cite{chodosh-mantoulidis-2d}.

The estimates in this paper differ from the interior ones of \cite{wang-weiadv} (and the corresponding interior estimates in Riemannian 3-manifolds due to Chodosh--Mantoulidis \cite{chodosh-mantoulidis}) in two key aspects caused by the geometry of the boundary: a logarithmic loss in the  mean curvature decay and the quantitative separation estimates, and a distinction between the zero level set and the other level sets in the transition layer.
While it is unclear whether the first difference is a feature of our proof and whether it could potentially be overcome, the second one reflects a genuine boundary phenomenon, as the stable solutions constructed by Kowalczyk in \cite{kowalczyk05} show (see Section \ref{sec:comparison-estimates} for a careful discussion of these differences). This shows the subtleness of the boundary problem, and highlights a phenomenon that has no interior counterpart.

\subsection{Setup and main result} Let $(M^{n+1},\partial M,\rg)$ be a smooth compact Riemannian manifold with boundary. We consider solutions $u:M\to(-1,1)$ of Allen--Cahn equation with vanishing Neumann boundary condition
\begin{equation}\label{e.AC}
\begin{cases}
\varepsilon^2\Delta_\rg u=W'(u)&\text{in }M,\\
\partial_\nu u=0&\text{on }\partial M,
\end{cases}
\end{equation}
where $\nu$ is the inward-pointing unit normal to $\partial M$. Throughout the paper, $W$ is a smooth even double-well potential, positive in $(-1,1)$, with
\[
W(\pm1)=W'(\pm1)=0,\qquad W''(\pm1)=1,
\]
and with $0$ as its unique critical point in $(-1,1)$, satisfying $W''(0)<0$. Model examples include $W(t)=\frac{(1-t^2)^2}{8}$ and $W(t)=\frac{1+\cos (\pi t)}{\pi^2}$.

Equation \eqref{e.AC} is the Euler--Lagrange equation for the Allen--Cahn energy
\[
E_\varepsilon(u)=\int_M\left(\frac{\varepsilon}{2}|\nabla u|^2
+\frac{W(u)}{\varepsilon}\right)\,d\mathrm{vol}_\rg.
\]
If $S\subset\partial M$ is relatively open, stability in $M\cup S$ means that
\[
\int_M\left(\varepsilon^2|\nabla\psi|^2+W''(u)\psi^2\right)
\,d\mathrm{vol}_\rg\geq0
\quad\text{for every }\psi\in C_c^1(M\cup S).
\]
In particular, these test functions need not vanish on $S$.

At a point where $\nabla u\neq0$, the enhanced second fundamental form is defined to be
\[
\mathcal A(u)=\nabla\!\left(\frac{\nabla u}{|\nabla u|}\right).
\]
We use the conventions and notation of Section \ref{s.basic.definitions}.

\begin{theorem}[Boundary second-order estimates for stable transition layers]\label{t.main.mfd}
Let $2\leq n+1\leq10$ and let $(M^{n+1},\partial M,\rg)$ be a smooth Riemannian manifold with boundary. Let $S\subset\partial M$ be relatively open, and let $M'\Subset M\cup S$. Fix $\theta\in(0,1)$ and $0<b<1$. For every $\Lambda>0$, there are constants $\varepsilon_0>0$ and $C<\infty$ with the following property.

Suppose that $0<\varepsilon\leq\varepsilon_0$ and that $u_\varepsilon\in C^\infty(M\cup S,(-1,1))$ is stable in $M\cup S$ and satisfies
\[
\varepsilon^2\Delta_\rg u_\varepsilon=W'(u_\varepsilon)\quad\text{in }M,
\qquad \partial_\nu u_\varepsilon=0\quad\text{on }S,
\]
together with
\begin{equation}\label{eq:C11-estimate-main}
|\mathcal A(u_\varepsilon)|\leq\Lambda
\quad\text{on }\{|u_\varepsilon|\leq1-b\}\cap\{|\nabla u_\varepsilon|\neq0\}.
\end{equation}
Then every component $\Gamma$ of $\{u_\varepsilon=0\}$ meeting $S$ satisfies 
\begin{equation}\label{e.mc.rates}
\|H_\Gamma\|_{L^\infty(\Gamma\cap M')}
\leq C\varepsilon|\log\varepsilon|^2,
\qquad
[H_\Gamma]_{C^\theta(\Gamma\cap M')}
\leq C\varepsilon^{1-\theta}|\log\varepsilon|^2,
\end{equation}
and
\begin{equation}\label{eq:C2,theta-main}
\|\mathrm{I\!I}_\Gamma\|_{C^\theta(\Gamma\cap M')}\leq C,
\end{equation}
where $\mathrm{I\!I}_\Gamma$ denotes the second fundamental form of the hypersurface $\Gamma$, and $H_\Gamma$ its mean curvature.
The constants depend only on $M'$, the fixed geometry near $M'$, $W$, $b$, $\theta$, and $\Lambda$, and are independent of the number of transition layers.
\end{theorem}

\begin{remark}\label{e.dim.2}
Notice that, if $n+1=2$, then the level set $\{u_\varepsilon=0\}$ is a union of embedded curves, and the norm of its second fundamental form is simply the absolute value of its geodesic curvature $k_{\{u_\varepsilon=0\}}$, which is clearly the same as the mean curvature of $\{u_\varepsilon=0\}$. Therefore, in ambient dimension 2, \eqref{e.mc.rates} gives
\begin{equation}
    \|k_\Gamma\|_{L^\infty(\Gamma\cap M')}
\leq C\varepsilon|\log\varepsilon|^2,
\qquad
[k_\Gamma]_{C^\theta(\Gamma\cap M')}
\leq C\varepsilon^{1-\theta}|\log\varepsilon|^2,
\end{equation}
for every component $\Gamma$ of $\{u_\varepsilon=0\}$ meeting $S$. On the other hand, unlike in the interior case of \cite[Theorem 3.6]{wang-wei} and \cite[Proposition C.4]{chodosh-mantoulidis-2d}, it is not possible to obtain decay estimates (in $\varepsilon$) for the enhanced second fundamental form $\mathcal{A}(u_\eps)$. Indeed, at a point $p\in \{u_\eps=0\}\cap S$, by \eqref{|Anormsquare|2},
\[
|\mathcal A(u_\eps)|^2=|k_{\{u_\eps = 0\}}|^2+|\partial_\nu \log|\nabla u_\eps||^2,
\]
and an easy computation (see Remark \ref{r.s11.enhanced.boundary}) gives $|\partial_\nu\log|\nabla u_\eps||^2
 =|k_S|^2$, where $k_S$ denotes the geodesic curvature of the boundary at $p\in \{u_\eps=0\}\cap S$. In particular, $k_S$ is independent of $\varepsilon$. See also Remark \ref{r.s11.boundary.dependence} for a more detailed discussion of this issue.
\end{remark}

\begin{remark}\label{r.intro.totally.geodesic}
If $S$ is totally geodesic on a neighbourhood of $M'\cap S$, the estimates improve and extend to all transition levels. More precisely, if $0<b<b'<1$, then uniformly for $t\in[-1+b',1-b']$, every component $\Gamma_t$ of $\{u_\varepsilon=t\}$ meeting $S$ satisfies, on $M'$,
\[
\|\mathrm{I\!I}_{\Gamma_t}\|_{C^\theta}\leq C,
\qquad
\|H_{\Gamma_t}\|_{L^\infty}\leq C\varepsilon|\log\varepsilon|,
\qquad
[H_{\Gamma_t}]_{C^\theta}\leq C\varepsilon^{1-\theta}|\log\varepsilon|,
\]
where the constant $C$ (in addition to the geometry of $M'$, $W$, $b$ , $\theta$, $\Lambda$) also depends on $b'$.
If, in addition, $n+1=2$, then  
\[
|\mathcal A(u_\varepsilon)|\leq C\varepsilon|\log\varepsilon|
\quad\text{on }M'\cap\{|u_\varepsilon|\leq1-b'\}.
\]
\end{remark}

At regular boundary points, the Neumann condition implies that the nodal set meet $S$ orthogonally. The nodal estimates therefore give local $C^{2,\theta}$ control of the second fundamental form of the nodal set up to the boundary. In ambient dimension two, the norm of the second fundamental form of a curve equals the absolute value of its geodesic curvature, so \eqref{e.mc.rates} gives decay (and not simply boundedness) of curvature.

Like the main result in \cite{wang-weiadv}, Theorem \ref{t.main.mfd} is a conditional result, in the sense that it assumes the initial $C^{2}$ bound \eqref{eq:C11-estimate-main} and improves it to the $C^{2,\theta}$ estimates \eqref{e.mc.rates} and \eqref{eq:C2,theta-main}. In order to upgrade this conditional result to a full unconditional curvature estimate on a compact Riemannian manifold with boundary as in \cite[Theorem 3.4]{chodosh-mantoulidis}, two further ingredients are needed: the boundary-parallel estimates of \cite{FBWW2}, and the classification result for entire stable solutions in $\mathbb{R}^{n+1}$ with uniform Euclidean energy growth, which is now available in dimensions $n+1=2,3,4$ by \cite  {GhoussoubGui1998, AmbrosioCabre2000, florit2025stable}. Notice that in dimensions $n+1=2,3$, by \cite{GhoussoubGui1998,
Liu2026StableDeGiorgi, ChanFernandezRealFigalliFloritSimonSerra2026}, the same classification result holds without the additional assumption of uniform Euclidean energy growth.

On the other hand, similarly to \cite{wang-weiadv}, no bound on the number of layers is a priori imposed and the dimensional restriction comes from the stability estimates for the scalar equation (i.e. the two component Toda system) to which stability is reduced in Section \ref{s.separation.emden}. Nonetheless, as far as geometric applications relying on an  unconditional curvature estimate are concerned, one cannot hope to go beyond ambient dimension 7, since the rigidity result for entire stable solutions in $\mathbb{R}^{n+1}$ is bound to fail there, as shown in \cite{PW13}. Notice here the analogy with rigidity results for stable minimal surfaces in $\mathbb{R}^{n+1}$ with Euclidean area growth \cite{SSY,Bellettini}.

\subsection{Comparison with the interior estimates.}\label{sec:comparison-estimates}
Our analysis builds on the reduction of stable Allen--Cahn transition layers to an approximately stable Toda system developed by Wang and Wei in \cite{wang-weiadv}. However the estimates proven in Theorem \ref{t.main.mfd} differ from those of \cite{wang-weiadv} in two important respects:
\begin{enumerate}[(i)]
\item our estimates yield an $O(\varepsilon|\log\varepsilon|^2)$ bound for the zero level mean curvature, instead of the Euclidean $O(\varepsilon)$ rate, or the interior Riemannian $o(\varepsilon|\log\varepsilon|)$ rate (see \cite[Remark 10.9]{wang-weiadv});
\item our estimates only show decay of the mean curvature for the zero level set, and not for all level sets $\{u_\varepsilon=t\}$ for $t\in[-1+b, 1-b]$.
\end{enumerate}
Let us briefly discuss where these differences come from, and where the theorem could conceivably be strengthened and where it is essentially sharp.

Let us begin with (i). As noticed in \cite[Remark 10.9]{wang-weiadv} (see also \cite[Corollary 3.3]{chodosh-mantoulidis}), it is not surprising that introducing metric errors into the Wang--Wei reduction machinery would lead to a loss in the estimates, and, as the mentioned references show, this is already present in the Riemannian interior theory. It is not known at the moment if the logarithmic loss is a necessary feature of the general Riemannian (interior or boundary) theory, or simply an artefact of the proof strategy. On the other hand, it is worth noticing that the decay rates in \eqref{e.mc.rates} are strictly weaker than those in Riemannian interior theory of \cite{wang-weiadv} or \cite{chodosh-mantoulidis}. Even under the (significantly stronger) assumption that the boundary is totally geodesic, we are not able to recapture the same rates of decay.

This discrepancy comes from the estimates on quantitative separation of neighbouring sheets, rather than from the final elliptic estimates. The loss first occurs in the distance-gap comparison of Claim \ref{c.s10.distance.comparison}, which is used to prove Proposition \ref{p.s10.decay}. Curvature of the boundary introduces first-order metric terms when the signed Riemannian distance between sheets is compared with the difference of their coordinate heights. 
The decay estimate of Proposition \ref{p.s10.decay} can therefore only be established above the threshold $\varepsilon^2|\log\varepsilon|^2$. Its iteration in Corollary \ref{c.s10.global.decay}, followed by the improved Toda equation in Proposition \ref{p.s11.zero.level}, gives an $O(\varepsilon|\log\varepsilon|^2)$ bound, rather than the Euclidean interior rate $O(\varepsilon)$ of \cite[Theorem 1.1]{wang-weiadv}. Nonnegative ambient Ricci curvature does not remove the boundary condition for the intrinsic separation function in Remark \ref{r.s11.intrinsic.boundary}. 
On the other hand, if the boundary is totally geodesic, the first-order metric contribution in the distance comparison disappears, as explained in Remark \ref{r.s10.threshold} and the separation estimate then improves by one logarithm. Neither this bound nor the general-boundary squared-logarithm bound is claimed to be optimal.

Let us now discuss (ii). The failure to extend the zero-level decay estimate to all transition levels has a different origin. Even though the zero sheets meet the boundary orthogonally, their nearby parallel leaves need not do so. Consequently, a heteroclinic solution depending on signed distance does not by itself satisfy the Neumann condition, and the approximation constructed in Section \ref{s.approx.toda} includes the correction $\Gb_\alpha$. The evenness of the potential makes this correction $\Gb_\alpha$ odd in its normal variable, so its leading contribution vanishes at the zero level but need not vanish at $s(t)=\Hb^{-1}(t)$ when $t\neq0$. Proposition \ref{p.s11.level.H} shows that the horizontal Laplacian of the correction $\Gb_\alpha$ enters the mean curvature of the actual level graph. This contribution is only controlled at order $\varepsilon$ in the rescaled metric, corresponding to order one in the original metric. When the boundary is totally geodesic, $\Gb_\alpha$ vanishes identically, and uniform estimates hold for all levels in a fixed compact subinterval of $(-1,1)$.

Unlike (i), the failure of mean-curvature decay on nonzero levels is demonstrably a feature of the problem in question, as the examples in \cite{kowalczyk05} show. Indeed, Kowalczyk \cite{kowalczyk05} constructs solutions of \eqref{e.AC} in bounded domains of $\mathbb R^2$ with smooth boundary concentrating near nondegenerate straight line segments meeting the  boundary orthogonally and computes their Morse indices explicitly. In the regime where Kowalczyk's solutions are stable, they illustrate why our estimates must distinguish the nodal set from the other transition levels. The level-set geometry obtained from his boundary-layer expansion \cite[Section 2.4]{kowalczyk05} shows that, at an endpoint with nonzero
boundary curvature, a fixed nonzero level develops an order-one curvature contribution over distances of order $\varepsilon$.  Thus its mean curvature need not tend to zero, and its $C^\theta$ curvature seminorm can grow like $\varepsilon^{-\theta}$, despite a bounded $L^\infty$ curvature norm.  On the nodal set, on the other hand, the leading boundary correction
vanishes by oddness, and the expansion instead gives an $O(\varepsilon)$ bound for the mean curvature. This contrasts with the interior Euclidean estimates of
Wang--Wei \cite[Theorem 1.1]{wang-weiadv} and the interior Riemannian estimates of Chodosh--Mantouldis \cite[Corollary 3.3]{chodosh-mantoulidis}, which give uniform bounds for $\|H\|_{L^\infty}$ and $\|\mathrm{I\!I}\|_{C^\theta}$ on all transition levels.

\subsection{Outline of the proof.}
We follow the overall strategy of Wang--Wei \cite{wang-weiadv} with a few important differences.

By using the assumptions in Theorem \ref{t.main.mfd}, in Section \ref{s.preliminary.setup} we reduce the theorem to the local graphical statement in Theorem~\ref{t.main.ball}.
After dilation by $\varepsilon^{-1}$, the curvature of each sheet of the nodal set becomes $O(\varepsilon)$, and neighbouring sheets become uniformly separated as $\varepsilon\to0$. 
Notice that ordinary Fermi coordinates with respect to a sheet of the nodal set do not in general parametrise a collar neighbourhood of the ambient boundary. To obtain coordinates adapted to \emph{both} the sheets and the boundary, we use a \emph{twisted Fermi coordinates} construction inspired by Pacard--Ritor\'e \cite[Section~10]{Pacard-Ritore}; a closely related construction appears in the work of Kapouleas--Li on free boundary minimal surfaces \cite[Definition~2.10 and Lemma~2.12]{KapouleasLi21}. The key ingredient in this step is the construction of a vector field $Z$ which is transverse to the sheets (but not necessarily normal to them), tangential to the ambient boundary and satisfies the estimates in Lemma \ref{l.Z.Z-tld}, which are crucial to control the mixed derivative and commutator terms introduced without destroying the Toda structure of the leading interaction. 

In Section \ref{s.approx.toda}, we approximate $u$ by alternating translated heteroclinics together with a boundary correction, which is necessary because the one-dimensional heteroclinic solution has a leading order Neumann derivative of order $\varepsilon$. For each sheet $\Sigma_\alpha$ of the nodal set, we cancel this $O(\varepsilon)$ error by solving the linear problem \eqref{G-equationwithout cutoff} on $\Sigma_\alpha\times\mathbb R$ and then using a cutoff to localise around the sheet. Evenness of the potential $W$ guarantees that this correction $\mathbb{G}$ is odd in the twisted Fermi signed distance variable. Since the translation mode $\mathbb H'$ is even, the correction $\mathbb{G}$ is orthogonal to it, and vanishes on the central slice $\{z=0\}$. The construction of our correction $\mathbb{G}$ is analogous to the improvement of the approximate solution in del Pino--Kowalczyk--Wei \cite[Section~5]{del2008toda}: in both cases, the  $O(\varepsilon)$ leading order in the Neumann derivative is cancelled by adding an odd correction solving a linearized boundary value problem.

For a family of shifts $h=(h_\alpha)$, let $\mathbb H(\cdot;h)$ denote the alternating superposition of shifted, cut-off heteroclinic profiles, and let $\mathbb G(\cdot;h)$ denote the corresponding sum of shifted boundary corrections. Write
\[
\mathbb H_*=\mathbb H(\cdot;h),\quad
\mathbb G_*=\mathbb G(\cdot;h),\quad
g_*=\mathbb H_*+\mathbb G_*,
\qquad
\phi=u-g_*,
\]
so the stars denote the full alternating superposition, evaluated at the selected shifts. 
As in \cite{wang-weiadv}, we choose optimal shifts $h_\alpha$ so that, in twisted Fermi coordinates with respect to each sheet,
\[
\int_{\mathbb R}\phi(y,z)\mathbb H'_\alpha(y,z)\,dz=0,
\quad
\mathbb H'_\alpha(y,z)
=\overline{\mathbb H}'\!\left((-1)^\alpha(z-h_\alpha(y))\right).
\]
This orthogonality condition separates the equations for the layer positions from the coercive equation for the remainder. 

Similarly to \cite{wang-weiadv}, in Section~\ref{Equation of error and mean curvature approximation by Toda system of layers}, projection of the Allen--Cahn equation onto the translational mode $\mathbb H'_\alpha$ yields a Toda-type relation between the mean curvature, the Laplacian of $h_\alpha$, and the exponential interactions with the two adjacent sheets. 
\[
H_\alpha(y,0)+\Delta_{\alpha,0}h_\alpha(y) = \frac{2A_1^2}{\sigma_0}\left(e^{-|t_{\alpha-1}(y,0)|}-e^{-|t_{\alpha+1}(y,0)|}\right)+E_\alpha^0(y),
\]
where, at this stage, the error $E_\alpha^0$ has not yet been shown to be . 

Section~\ref{Inner-outer gluing regularity estimates of error} closes these estimates by an inner--outer gluing estimates argument, as in \cite[Section~6]{wang-weiadv}. Near a layer, orthogonality removes the translation kernel and gives coercive estimates for the model linearised operator. Away from adjacent layers, the solution is close to $\pm 1$ and $W''(u)$ has a uniform positive lower bound. Combining the two estimates, together with the Neumann derivative of $\phi$, gives
\[
\|\phi\|_{C^{2,\theta}}
+\max_\alpha\|h_\alpha\|_{C^{2,\theta}}
+\max_\alpha\|H_\alpha+\Delta_{\alpha,0}h_\alpha\|_{C^\theta}
\lesssim\varepsilon^2+A,
\]
where $A$ measures the largest exponential interaction between neighbouring sheets of the nodal set.

This estimate is still insufficient for the final separation argument of Section~\ref{s.separation.emden}, as $\Delta_{\alpha,0}h_\alpha$ could still be of the same order as the leading interaction. Section~\ref{improved estimates section} therefore repeats the inner--outer analysis for the horizontal derivatives of $\phi$, following \cite[Section~7]{wang-weiadv}. Unlike \cite{wang-weiadv}, we need to deal with the boundary conditions, which must be treated separately. Derivatives tangent to the boundary have controlled Neumann data, while the inward horizontal derivative has controlled Dirichlet data. As a result the mean curvature is shown to be equal to the difference of the two adjacent interactions, up to an error $O(\varepsilon^2)+o(A)$, which is small relative to the interaction scale used in the decay argument of Proposition~\ref{p.s10.decay}.

The Toda equations alone do not force the sheets to separate. In Section~\ref{s.stability.toda}, we use Allen--Cahn stability with test functions
\[
\varphi=\sum_\alpha\eta_\alpha
\bigl(\mathbb H'_\alpha+\mathbb G'_\alpha\bigr),
\]
where $\eta_\alpha$ is a function on the sheet $\Sigma_\alpha$. Unlike \cite[Section 8]{wang-weiadv}, the boundary correction also enters the linearised equation and we must combine the horizontal and transverse contributions to the diagonal energy by integration by parts before estimating them, and exploit the cancellations supplied by the equation defining the correction $\mathbb{G}$.  A Robin-type boundary term of order $\varepsilon$ still remains, which is controlled by trace estimates with a small loss in the constant in front of the Dirichlet energy and controlled $L^2$ errors.

Section~\ref{s.separation.emden} combines the Toda equation with the stability inequality to obtain a quantitative lower bound for the distances between neighbouring sheets. The height difference (with respect to the base $\{z=0\}$) satisfies a homogeneous Neumann condition, which allows to perform integration by parts without producing additional boundary terms.
On the other hand, using height difference rather than sheet separation causes an additional logarithmic loss compared to Wang--Wei, as the first-order metric error in boundary Fermi coordinates limits the iteration to $A\lesssim\varepsilon^2|\log\varepsilon|^2$. Wang and Wei manage to improve this bound to $A\lesssim\varepsilon^2$ in Euclidean space by deriving an equation directly for the separation function on a sheet; however, extending this improvement to our setting would require further work, as the intrinsic separation function does not generally satisfy a homogeneous Neumann boundary condition. Indeed, the normal derivative of this intrinsic separation function contains a term involving the second fundamental form of the ambient boundary whose contribution need not be small after rescaling. Therefore, implementing an intrinsic improvement argument in our setting would require further estimates to control this boundary contribution; these are not established in the present paper.

The proof of the pointwise estimates also differs from \cite[Section 9]{wang-weiadv}.  Following Farina \cite{farina2007stable}, we first combine the scalar inequality with stability to obtain $L^p$ estimates for the normalized exponential interaction, with $p>n/2$ and then obtain local boundedness by a blowup argument.  This argument replaces the Green-function and Morrey estimates used in \cite[Lemma 9.7-Lemma 9.8]{wang-weiadv} and applies both to interior balls and to boundary half-balls.
As stability gives $L^p$ control of the interaction for exponents $p<5$ and local boundedness requires $p>n/2$, this forces $n\leq9$, explaining the ambient dimension bound $n+1\leq10$.

Finally, in Section \ref{s.final.estimates}, we insert this bound into the improved Toda equation and use the oddness of the boundary correction to obtain the nodal mean-curvature estimates. 
On the other hand, as explained in Section \ref{sec:comparison-estimates}, for a nonzero transition level, we represent the portion of $\{u=t\}$ near $\Sigma_\alpha$ as a graph $z=\ell_\alpha(y,t)$ by solving $u(y,\ell_\alpha(y,t))=t$.
The boundary correction changes the height of this graph by a quantity proportional to $\mathbb G_\alpha(y,\mathbb H^{-1}(t))$. Differentiating twice along the sheet therefore introduces second horizontal derivatives of this correction. These derivatives vanish on the central profile slice because $\mathbb G_\alpha(\cdot,0)\equiv0$, but generally contribute to the mean curvature of nonzero levels, see Lemma~\ref{l.s11.level.graph} and Proposition~\ref{p.s11.level.H}, giving rise to the difference in behaviour between the nodal set and the other level sets in the transition region, which is the most relevant difference between the boundary problem and the interior one.
When the boundary is totally geodesic, the correction vanishes altogether, and the improved separation estimate gives the expected bounds on every transition level. 

\subsection*{Acknowledgements} The first-named author was partially supported by the NSERC Discovery Grants of Robert Haslhofer and Yevgeny Liokumovich during his postdoctoral fellowship at the University of Toronto, and is currently partially supported by an Infosys Foundation grant. The second-named author was partially supported by the NSERC Discovery Grant of Yevgeny Liokumovich during his postdoctoral fellowship at the University of Toronto, an AMS-Simons Travel Grant while at the Massachusetts Institute of Technology, and is currently supported by Hunan University. The third-named author was partially supported by the Cecilia Tanner Research Impulse Fund of Imperial College London. The fourth-named author is partially supported by the NSERC Discovery Grants of Robert Haslhofer and Yevgeny Liokumovich.  

\noindent We also thank Costante Bellettini, Otis Chodosh, Robert Haslhofer, Yevgeny Liokumovich, Christos Mantoulidis, Luca Spolaor, Kelei Wang, and Juncheng Wei for helpful discussions. 

\subsection*{AI disclosure} 
The writing and mathematical content of this paper is due to the authors. AI tools were used only to assist with literature searches and language editing at the very end of this project.

\section{Basic definitions and notation}\label{s.basic.definitions}
In this section, we present some basic definitions and introduce the notation which we shall use throughout the paper.

Let $W\in C^\infty(\mathbb R)$ is an even double-well potential, namely $W:\mathbb{R}\to [0, +\infty)$ has the following properties
\begin{enumerate}
	\item $W$ vanishes precisely at $\pm 1$, and $W'(\pm 1)=0$, $W''(\pm 1) = 1$;
	\item $0$ is the unique critical point of $W$ in $(-1,1)$, $W'(0) = 0$, and $W''(0) < 0$;
	\item $W(t) = W(-t)$.
\end{enumerate}
The normalisation $W''(\pm1)=1$ is responsible for the decay rate $e^{-|t|}$ in
Appendix \ref{s.1d.sol}.  Evenness is used in Section \ref{s.approx.toda} to make the boundary
correction odd in the normal variable.

Let $(M^{n+1},\partial M,\rg)$ be a smooth Riemannian manifold with boundary. We shall often use $\langle \cdot, \cdot, \rangle$ to denote the inner product with respect to the Riemannian metric $\mathrm{g}$.
Throughout this paper we shall consider stable critical points $u:M\to(-1, 1)$ of the Allen--Cahn energy
\[
E_\varepsilon(u)=\int_M\left(\frac{\varepsilon}{2}|\nabla u|^2
+\frac{W(u)}{\varepsilon}\right)\,d\mathrm{vol}_\rg.
\]
Any such critical point is a solution of \eqref{e.AC}. If $S\subset \partial M$, then  a solution $u$ of \eqref{e.AC} is \emph{stable in $M\cup S$} if
\begin{equation}\label{e.stability.definition}
\frac{d^2}{dt^2}\Bigl|_{t=0}E_\varepsilon(u+t\psi)=\int_M\left(\varepsilon^2|\nabla\psi|^2+W''(u)\psi^2\right)
 \,d\mathrm{vol}_\rg\geq0
\end{equation}
for every $\psi\in C_c^\infty(M\cup S)$.  Notice that test functions $\psi$ need not vanish on $S$.

For a solution $u_\varepsilon$ of \eqref{e.AC} and $x\in M$ such that $|\nabla u(x)|\neq 0$, we define the \emph{enhanced second fundamental form} $\mathcal A(u)$ of the level set $\{u= u(x)\}$ at $x$ by
\[
 \mathcal A(u)=\nabla\!\left(\frac{\nabla u}{|\nabla u|}\right).
\]
Notice that $\nabla u(x)/|\nabla u(x)|$ is the unit normal of the hypersurface  $\{u= u(x)\}$ at $x$.
We extend $\mathcal{A}(u)$ to the whole of $M$ by setting $|\mathcal A(u)|=0$ on the critical set.  
It is easy to check that if $|\nabla u(x)|\neq 0$, then at $x$, we have
\begin{equation}\label{|Anormsquare|2}
 |\mathcal A(u)|^2
 =\frac{|\nabla^2u|^2-|\nabla|\nabla u||^2}{|\nabla u|^2}
 =|\mathrm{I\!I}(u)|^2+|\nabla^T\log|\nabla u||^2,
\end{equation}
where $\sff(u)$ denotes the second fundamental form of the regular hypersurface $\{u=u(x)\}$ at $x$ and $\nabla^T$ denotes tangential differentiation along $\{u=u(x)\}$.
Thus $|\mathcal{A}(u)|$ dominates the second fundamental form of the level sets.

Our sign convention for the second fundamental form $\sff$ and the mean curvature $H$ of a hypersurface is
\[
 \mathrm{I\!I}(X,Y)=-\langle\nabla_X\bn,Y\rangle,
 \qquad H=\operatorname{tr}\mathrm{I\!I},
\]
where $\bn$ is the chosen unit normal.  With this convention, the linearization of mean
curvature in the normal direction $\psi\bn$ is
$\Delta\psi+(|\mathrm{I\!I}|^2+\operatorname{Ric}(\bn,\bn))\psi$.

\subsection{Balls in Euclidean space}
Consider Euclidean coordinates $(x_1, \dots, x_n, x_{n+1})$ on $\mathbb{R}^{n+1}$. We shall adopt the following notation for various types of balls in $\mathbb{R}^{n+1}$.
\begin{itemize}
	\item $\bbb_r(p)$ (resp. $\bar{\bbb}_r(p)$) : the open (resp. closed) ball in $\bbr^{n}$ centered at $p$ with radius $r$.
	\item $\bbb^+_r(p)$ (resp. $\bar{\bbb}^+_r(p)$) : the open (resp. closed) half-ball $\bbb_r(p)\cap \{x_1\geq 0\}$ (resp. $\bar{\bbb}_r(p)\cap \{x_1\geq 0\}$).
	\item $\del_0\bbb^+_r(p)$= $\del \bbb^+_r(p)\cap \{x_1= 0\}$, $\del_+\bbb^+_r(p)$= $\del \bbb^+_r(p)\cap \{x_1> 0\}$.
	\item $B_r(p)$ (resp. $\bar{B}_r(p)$) : the open (resp. closed) ball in $\bbr^{n+1}$ centered at $p$ with radius $r$. 
	\item $B^+_r(p)$ (resp. $\bar{B}^+_r(p)$) : the open (resp. closed) half-ball $B_r(p)\cap \{x_1\geq 0\}$ (resp. $\bar{B}_r(p)\cap \{x_1\geq 0\}$)
	\item $\del_0B^+_r(p)$ = $\del B^+_r(p)\cap \{x_1= 0\}$, $\del_+B^+_r(p)$= $\del B^+_r(p)\cap \{x_1= 0\}$.
\end{itemize}

\subsection{Constants dependence} 
If $a \leq C b$, for some universal constant $C$, we will frequently use the notation $a \lesssim b$, or $a = O(b)$. If we wish to explicit the dependence of the constant $C$ on certain parameters $\ell_1, \ell_2, \ldots, \ell_j$, we will use the symbol $\lesssim_{\ell_1, \ell_2, \ldots, \ell_j}$, or simply $O_{\ell_1, \ell_2, \ldots, \ell_j}$. When we have the two-sided bound $C^{-1} b \leq a \leq C b$, we will write $a \sim b$, or $a \sim_{\ell_1, \ell_2, \ldots, \ell_j}$, if the dependence depends on the parameters ${\ell_1, \ell_2, \ldots, \ell_j}$. 

Unless explicitly indicated otherwise, all constants are independent of $\varepsilon$, of the number of layers, and of the layer index.
An estimate containing an unspecified derivative order is understood for each fixed order, with
the constant allowed to depend on that order and on the corresponding smooth bounds for the
fixed geometric data.  Finally, every little-$o$ estimate in the paper is uniform in all layer and
level parameters occurring in its statement.

\part{Derivation of the Toda system}

\section{Preliminary setup}\label{s.preliminary.setup}
In this section, we shall prove some preliminary results about stable solutions of \eqref{e.AC} satisfying the vanishing Neumann boundary condition. This preliminary analysis will allow us to reduce Theorem \ref{t.main.mfd} to the local statement of Theorem \ref{t.main.ball}.

Throughout this section, let $(M^{n+1}, \partial M, \rg)$ be a smooth compact $(n+1)$-dimensional Riemannian manifold with non-empty smooth boundary and let $\ve>0$.
Let $S\subset\partial M$ be a relatively open domain in $\partial M=\overline{M}\setminus M$, and let $u_{\varepsilon}\in C^{\infty}(M\cup S, (-1,1))$ be a stable solution of \eqref{e.AC} satisfying the vanishing Neumann boundary condition along $S$. Suppose $u_{\ve}$ satisfies \eqref{e.esff.bd} and $M'\Subset M \cup S$. Fix $b'\in(b,1)$, $\Lambda>0$ and $M'\Subset M\cup S$, and assume
\begin{equation}\label{e.esff.bd}
	|\mathcal{A}(u_\varepsilon)|\leq \Lambda_0 \quad \text{on }\{|u_\varepsilon|\leq 1-b\}\cap \{|\nabla u_\varepsilon|\neq 0\}.
\end{equation}

\begin{lemma}\label{l.nab.u}
Under the assumptions of Theorem \ref{t.main.mfd}, there are constants
$\varepsilon_\bullet>0$ and $c>0$, depending only on $b$, $\Lambda$, $W$, and the fixed
geometry near $M'$, such that, whenever $0<\varepsilon\leq\varepsilon_\bullet$,
\begin{equation}\label{e.nab.u}
 \varepsilon|\nabla u_\varepsilon|\geq c
 \qquad\text{in }\{|u_\varepsilon|\leq1-b\}\cap M'.
\end{equation}\end{lemma}

\begin{proof}
The argument is similar to \cite[Proof of Lemma 2.1]{wang-weiadv}. Let us assume by contradiction that \te\ sequences $(u_{\ve_i},\ve_i)$ \w\ $\ve_i\ra 0$, $p_i\in \{\mD{u_{\ve_i}}\leq 1-b\}\cap M'$ so that $u_{\ve_i}\in C^{\infty}(M\cup S, (-1,1))$ is a stable solution of \eqref{e.AC} \w\ $\ve=\ve_i$ satisfying the vanishing Neumann boundary condition along $S$ and $\ve_i\mD{\na u_{\ve_i}(p_i)}\ra 0$. Let $d_i=d_\rg(p_i,S)=d_\rg(p_i, q_i)$, where $q_i\in S$ is the nearest point projection of $p_i$ to $S$; we need to consider two cases:
\begin{enumerate}[(i)]
\item either $\liminf_{i\ra \infty}\ve_i^{-1}d_i=\infty$, 
\item or $\liminf_{i\ra \infty}\ve_i^{-1}d_i<\infty$.
\end{enumerate}

Assume first that $d_i/\varepsilon_i\to\infty$. In geodesic normal coordinates centreed at
$p_i$, define
\[
 E_{p_i,\varepsilon_i}(v)=\exp_{p_i}(\varepsilon_i v),
 \qquad
 \rg_{p_i,\varepsilon_i}=\varepsilon_i^{-2}E_{p_i,\varepsilon_i}^*\rg,
 \qquad
 \widetilde u_i=u_{\varepsilon_i}\circ E_{p_i,\varepsilon_i}.
\]
After identifying $T_{p_i}M'$ with $\bbr^{n+1}$  by choosing an orthonormal basis, the domain of $E_{p,\ve}$ can be identified with $B_{\ve^{-1}r}(0)$, where $r=\operatorname{inj}_p$. Then $\tilde{u}_i$ is a stable solution of \eqref{e.AC} \w\ $\ve=1$ on $(B_{r_i}(0),\mathrm{g}_{p_i,\ve_i})$, where $r_i=(2\ve_i)^{-1}\min\{d_i, \operatorname{inj}_{p_i}\}$, $\mD{\tui(0)}\leq 1-b$, $\mD{\na \tui(0)}\ra 0$ and $\mathrm{g}_{p_i,\ve_i}$ converges to the Euclidean \mt\ in $C^{\infty}_{\text{loc}}(\bbr^{n+1})$. \hn, by elliptic estimates, $\tui$ is uniformly bounded in $C^{3}_{\text{loc}}(\bbr^{n+1})$. \tf, \tes\ $\tuin$, a stable solution of \eqref{e.AC} \w\ $\ve=1$ on $\bbr^{n+1}$, \st\ up to a subsequence, $\tui\ra\tuin$ in $C^{2}_{\text{loc}}(\bbr^{n+1})$. Since $\mD{\tuin(0)}\leq 1-b$, $\mD{\na \tuin(0)}= 0$, $\tuin$ is non-constant. As in \cite[Proof of Lemma 2.1]{wang-weiadv}, by \eqref{e.esff.bd}, we have $\mD{\cA(\tuin)}=0$ on $\{\mD{\tuin}<1-b\}\cap \{\mD{\na \tuin}\neq 0\}$. By the unique continuation principle (for example see \cite{jerison1985unique}) for the linear equation $\Delta \partial_{i}\tilde{u}_{\infty}=W'(\tilde{u}_{\infty})\partial_{i}\tilde{u}_{\infty}$ ($i=1, \dots, n+1$) and Lebesgue's density theorem, the critical set $\{\nabla \tilde{u}_{\infty}=0\}$ has zero Lebesgue measure. By \eqref{|Anormsquare|2} and $|\mathcal{A}(\tilde{u}_{\varepsilon})|\leq \Lambda \varepsilon$ (which follows from \eqref{e.esff.bd}), we conclude that $\{\tilde{u}_{\infty}=t\}$ is a flat hyperplane for almost every $t\in[-1+b, 1-b]$.  This implies that $\tuin=U(x\cdot e)$, where $U: \mathbb{R}\to [-1, 1]$ is a nonconstant stable entire solution of the 1-dimensional Allen--Cahn equation and $e$ is a  unit vector. In particular, $U$ must be the 1-dimensional heteroclinic solution $\Hb$, which implies that $\mD{\na \tuin(0)}\neq 0$. Thus we arrive at a contradiction.

Assume now $\liminf_{i\ra \infty}\ve_i^{-1}d_i<\infty$, \tes\ $d\geq 0$ \st,\ possibly after passing to a subsequence, $\lim_{i\ra \infty}\ve_i^{-1}d_i=d$. In boundary Fermi coordinates centreed at $q_i$, define
\[
 E_{q_i,\varepsilon_i}(x_1,x')
 =\exp^M_{\exp^{\partial M}_{q_i}(\varepsilon_i x')}
 \!\left(\varepsilon_i x_1\nu_{\partial M}
 \bigl(\exp^{\partial M}_{q_i}(\varepsilon_i x')\bigr)\right),
 \qquad
 \rg_{q_i,\varepsilon_i}=\varepsilon_i^{-2}E_{q_i,\varepsilon_i}^*\rg.
\]
These coordinates map $\{x_1=0\}$ exactly to $\partial M$.  Thus
$\widetilde u_i=u_{\varepsilon_i}\circ E_{q_i,\varepsilon_i}$ satisfies the homogeneous
Neumann condition on the flat boundary, the metrics converge smoothly to the Euclidean metric
on $\mathbb R^{n+1}_+$, and the preimages of $p_i$ remain in a fixed compact set. Then there is a fixed $r>0$ such that for all sufficiently large $i$ the normal exponential map of $\partial M$ is a diffeomorphism of $B^{\partial M}_r(q_i)\times[0,r)$.
In particular, $\tilde{u}_i$ is a stable solution of \eqref{e.AC} \w\ $\ve=1$ on $(B^+_{r_i}(0),\mathrm{g}_{q_i,\ve_i})$, where $r_i=(2\ve_i)^{-1}\min\{d(q_i,\del M\setminus S), r\}$, and it satisfies the vanishing Neumann boundary condition along $\del_0B^+_{r_i}(0)$. Moreover, if $\tilde{p}_i=E_{q_i,\ve_i}^{-1}(p_i)$, then $\mD{\tui(\tilde{p}_i)}\leq 1-b$ and $\mD{\na \tui(\tilde{p}_i)}\ra 0$. Since $\mathrm{g}_{p_i,\ve_i}$ converges to the Euclidean \mt\ in $C^{\infty}_{\text{loc}}(\bbr^{n+1}_+)$, $\tui$ is uniformly bounded in $C^{3}_{\text{loc}}(\bbr^{n+1}_+)$. \tf,\ \tes\ $\tuin^*$, a stable solution of \eqref{e.AC} \w\ $\ve=1$ on $\bbr^{n+1}_+$, \st\ up to a subsequence, $\tui\ra\tuin^*$ in $C^{2}_{\text{loc}}(\bbr^{n+1}_+)$. Moreover, \tes\ $\tilde{p}_{\infty}$ \st\ up to a subsequence, $\tilde{p}_i \ra \tilde{p}_{\infty}$; so $\mD{\tilde{p}_{\infty}}=d$, $\mD{\tuin^*(\tilde{p}_{\infty})}\leq 1-b$, $\mD{\na \tuin^*(\tilde{p}_{\infty})}= 0$. Hence $\tuin^*$ is non-constant. We define  $\tuin\in C^2(\bbr^{n+1})$ by
\begin{equation}\label{e.even.refl}
\tuin(x_1,x_2,\dots,x_{n+1})=
\begin{cases}	
\tuin^*(x_1,x_2,\dots,x_{n+1}) &\text{ when } x_1\geq 0;\\
\tuin^*(-x_1,x_2,\dots,x_{n+1}) &\text{ when } x_1\leq 0.
\end{cases}
\end{equation}
Then $\tuin$ is a stable solution of \eqref{e.AC} \w\ $\ve=1$ on $\bbr^{n+1}$. Since $u_{\ve_i}$ satisfies \eqref{e.esff.bd}, arguing as in \cite[Proof of Lemma 2.1]{wang-weiadv}, we conclude that $\mD{\cA(\tuin)}=0$ on $\{\mD{\tuin}<1-b\}\cap \{\mD{\na \tuin}\neq 0\}$.  Similarly to the previous case, by unique continuation and the Lebesgue differentiation theorem, this implies $\tuin$ is again of the form $\tuin(x)=\Hb(x\cdot e)$ for some unit vector $e$. Hence $\mD{\na \tuin(\tilde{p}_{\infty})}\neq 0$. Thus we arrive again at a contradiction.
\end{proof}

\subsection{Fermi coordinates with respect to the boundary and graphical representation of layers}\label{s.graph}
We shall now discuss the definition and the properties of Fermi coordinates with respect to $S$ following \cite[Appendix A]{LZ}. Let $S' \Subset S$ and \(p\in S'\). Let \(x_1=\text{dist}_\rg(-,\del M)\), which is smooth in a neighbourhood of \(p\) contained in $M'$. Suppose  $(x_2,\dots, x_{n+1})$ is a normal coordinate system on $S$ centred at $p$, which is obtained by fixing an orthonormal basis of $T_pS$. Then $(x_1,x_2,\cdots,x_{n+1})$ is a Fermi coordinate system for $(M,\del M)$ near $p$. The Fermi distance function from \(p\) is defined by $\bar{r}_p(x)= \sqrt{x_1^2+x_2^2+\dots+ x_{n+1}^2}$. The Fermi coordinate ball $\mathcal{B}_r(p)$ defined by $\mathcal{B}_r(p)\coloneqq \{\bar{r}_p<r\}$ can be identified with $B^+_r(0)\subset\mathbb{R}^{n+1}$. 

\begin{lemma}\label{l.Fermi.coord}\cite[Lemma A.2]{LZ}
\Te\ constants $r_1$ and $\{\mathsf{C}_m\}_{m=0}^\infty$ depending only on the \mt\ \(\rg\) and $S'$ \st\ the following holds. Let $\rg_{ij}$, $\bar{\Ga}_{ij}^k$ denote the components of $\rg$ and the Christoffel symbols in a Fermi coordinate system centreed at $p$. Then
\begin{equation}\label{e.g.Ga.0}
\rg_{11}=1,\quad \rg_{1i}=0,\quad\bar{\Ga}_{i1}^1=\bar{\Ga}_{11}^i=\bar{\Ga}_{11}^1=0\quad \forall i\geq 2
\end{equation}
and for all $1 \le i, j \le n+1$,
\begin{equation}\label{e.g.Ga.est}
\md{\rg_{ij}(x)-\de_{ij}}\leq C_0\bar{r}_p(x)\text{ on }\mathcal{B}_{r_1}(p),\quad \|\rg_{ij}-\de_{ij}\|_{C^m(\mathcal{B}_{r_1}(p))}\leq \mathsf{C}_m\; \forall\; m \ge 0.
\end{equation}
\end{lemma}

For the current purpose, let $S'=\{|u_{\ve}|<1-b\}\cap S\cap M'$, and let $p\in S'$. By Lemma \ref{l.nab.u}, $S'\Subset \{|\nabla u_\ve|\neq 0\}$. Since $u_{\ve}$ satisfies the vanishing Neumann boundary condition on $S$, $\na u_{\ve}/\md{\na u_{\ve}}\in T_pS$. Choose an orthonormal basis $\{e_2,\dots,e_{n+1}\}$ of $T_pS$ such that $e_{n+1}=\na u_{\ve}/\md{\na u_{\ve}}$. Let $(x_1,x_2,\cdots,x_{n+1})$ be the Fermi coordinate system for $(M,\del M)$ centred at $p$ so that for $i\geq 2$, $\del_{x_i}|_p=e_i$.  

Then, for the function 
$F(x)=u_{\varepsilon}(x_1,\dots, x_n, x_{n+1})-t$, $\partial_{x_{n+1}}F(x)=|\nabla u_{\varepsilon}|G(x)$, where the function $G$ is defined by $G(x)=\rg\left(\frac{\na u_{\ve}}{\md{\na u_{\ve}}},\del_{x_{n+1}}\right)$ on $\mathcal{B}_{r_1}(p)$ (where $r_1$ is as in Lemma \ref{l.Fermi.coord}). We note that $G(p)=1$ and its derivative on $\{\md{u_{\ve}}<1-b'\}\cap  \mathcal{B}_{r_1}(p)$ is bounded by a constant which depends only on $\La$ and $\rg$. 

By the mean value theorem and the implicit function theorem, there exists a constant $r_2=r_2(\La, \rg)$ \st\ in $\mathcal{B}_{r_2}(p)$, for all $t\in (-1+b',1-b')$, $\{u_{\ve}=t\}$ is a union of graphs $\cup_{\al}\{x_{n+1}=\bar{f}^t_{\al,\ve}(x_1,\dots,x_n)\}$. Moreover, $\del_1\bar{f}^t_{\al,\ve}=0$ along $S$ (since $\del_1u_{\ve}=0$ along $S$) and $|\na^2\bar{f}^t_{\al,\ve}|\leq |\mathcal{A}|$ is bounded by a \cn\ which depends only on $\La$ and $\rg$ by \eqref{e.esff.bd}. All of these estimates are
uniform in $t$ and in the graphical component.

\subsection{A reduction of curvature estimates}\label{s.reduction}
In view of the above discussion in Section \ref{s.graph}, to prove Theorem \ref{t.main.mfd}, it is enough to prove the following Theorem \ref{t.main.ball}. Indeed, \Thm\ \ref{t.main.ball} implies that \Thm\ \ref{t.main.mfd} holds on the collar neighbourhood $M'\cap\mathcal{N}_{r_2/3}(\del M)$ (where $r_2$ is as in Section \ref{s.graph}). On the other hand, on $M'\setminus \mathcal{N}_{r_2/4}(\del M)$, the proof of \Thm\ \ref{t.main.mfd} follows from \cite{wang-wei,wang-weiadv,mantoulidis,chodosh-mantoulidis}.

\begin{theorem}\label{t.main.ball}
Assume $2\leq n+1\leq10$. Let $B_3^+(0)$ be equipped with a smooth Riemannian metric $\hat{g}$ so that the following conditions are satisfied. 
\begin{itemize}
	\item $\del_{x_1}$ is the inward pointing unit normal along $\del_0B_3^+(0)$.
	\item If $\hat{g}_{ij}$, $\hat{\Ga}_{ij}^k$ denote the components of $\hat{g}$ and its Christoffel symbols, then, for every fixed $m$,
	\begin{align}
	&\hat{g}_{11}=1,\quad \hat{g}_{1i}=0,\quad\hat{\Ga}_{i1}^1=\hat{\Ga}_{11}^i=\hat{\Ga}_{11}^1=0\quad \forall i\geq 2,\label{e.g.Ga.0.ball}\\
	&\hat{g}_{ij}(0)=\de_{ij}\quad \nm{\hat{g}_{ij}-\de_{ij}}_{C^m(B_3^+(0))}\leq \et_0\label{e.g.Ga.est.ball}.
	\end{align}
\end{itemize}
Fix $\theta\in (0, 1)$ and $0<b<1$. There exist constants $\varepsilon_0>0$ and $C_0<\infty$, depending only on $n,\theta,b,W,c_0,\eta_0$, and $\Lambda_0$, such that the following holds. Suppose $0<\varepsilon\leq \varepsilon_0$, $u_{\varepsilon}:(B_3^+(0),\hat{g})\rightarrow (-1,1)$ is a stable solution of \eqref{e.AC} satisfying the following conditions.

\noindent(H1) $u_{\varepsilon}$ satisfies the vanishing Neumann boundary condition $\partial_{x_1}u_\varepsilon=0$ along $\del_0B_3^+(0)$ and
\begin{equation}\label{e.nab.u.ball}
\varepsilon|\nabla u_{\varepsilon}|\geq c_0 \text{ on }\{|u_{\varepsilon}|\leq 1-b\}.
\end{equation}

\noindent(H2) Let $\hat{\Xi}$ denote the union of all the connected components of $\{u_{\ve}=0\}\cap \left(\bbb^+_{2}(0)\times (-2,2)\right)$ which have non-empty intersections with $\bbb^+_{2}(0)\times (-1.5,1.5)$. Then $\hat{\Xi}$ can be written as a union of graphs
\[\hat{\Xi}=\bigcup_{\al}\{x_{n+1}=\hat{f}_{\al,\ve}(x'):x'=(x_1,\dots,x_n)\in \bbb^+_{2}(0)\}\]
where $\hat{f}_{\al,\ve}\in C^{\infty}(\bbb^+_{2}(0),(-1.6,1.6))$ \w\ $\del_{x_1} \hat{f}_{\al,\ve} = 0 \text{ along } \del_0\bbb^+_{2}(0)$,
\begin{equation}\label{e.na.ftl}
|\na\hat{f}_{\al,\ve}|\leq \et_0,\; |\na^2\hat{f}_{\al,\ve}|\leq \La_0 \text{ on }\bbb^+_{2}(0).
\end{equation}

\noindent(H3) We have \begin{equation}\label{e.esff.bd.ball}
|\mathcal{A}(u_{\varepsilon})|\leq \Lambda_0 \quad \text{in}\,\,\{|u_{\varepsilon}|\leq 1- b\}\cap \{|\nabla u_{\varepsilon}|\not=0\}.
\end{equation}

Then every component of $\hat\Xi$ satisfies, on $B_{1/2}^+(0)$,
\begin{equation}\label{e.main.ball.zero.II}
 \|\sff_{\{u_\varepsilon=0\}}\|_{C^\theta}\leq C_0,
\end{equation}
and
\begin{equation}\label{e.mn.curv.ball}
 \|H_{\{u_\varepsilon=0\}}\|_{L^\infty}
 \leq C_0\varepsilon|\log\varepsilon|^2,
 \qquad
 [H_{\{u_\varepsilon=0\}}]_{C^\theta}
 \leq C_0\varepsilon^{1-\theta}|\log\varepsilon|^2.
\end{equation}
\end{theorem}

\begin{remark}
The constants $\et_0$ and $\La_0$ in \eqref{e.g.Ga.est.ball}, \eqref{e.nab.u.ball}, \eqref{e.na.ftl} and \eqref{e.esff.bd.ball} are fixed before $\ve_0$ is chosen and will be assumed to be sufficiently small. 
\end{remark}

\begin{remark}
If, in addition, $\partial_0B_3^+(0)$ is totally geodesic and $0<b<b'<1$, then, uniformly for
$t\in[-1+b',1-b']$, every corresponding graphical component of
$\{u_\varepsilon=t\}$ satisfies
\begin{equation}\label{e.main.ball.all.II}
 \|\sff_{\{u_\varepsilon=t\}}\|_{C^\theta(B_{1/2}^+(0))}\leq C_0,
\end{equation}
and
\begin{equation}\label{e.main.ball.all.H}
 \|H_{\{u_\varepsilon=t\}}\|_{L^\infty(B_{1/2}^+(0))}
 \leq C_0\varepsilon|\log\varepsilon|,
 \qquad
 [H_{\{u_\varepsilon=t\}}]_{C^\theta(B_{1/2}^+(0))}
 \leq C_0\varepsilon^{1-\theta}|\log\varepsilon|.
\end{equation}
where $C_0$ also depends on $b'$ (in addition to $n,\theta,b,W,c_0,\eta_0$, and $\Lambda_0$).
\end{remark}

In the following sections, we shall first prove the mean curvature decay estimates for the zero level set $\{u_\ve=0\}$.

\section{Geometric setup and twisted Fermi coordinates} \label{section twisted Fermi}
In this section we define suitable coordinates (which we refer to as \textit{twisted Fermi coordinates})  near the boundary which are adapted to the geometry of the components of the zero level set of the solution $u_\ve$.

Before introducing this coordinate system, following \cite{wang-weiadv}, we rescale distances by $\varepsilon^{-1}$, so that we work with a stable solution of the Allen--Cahn equation with parameter 1 in ball of radius of order $\varepsilon^{-1}$.

\subsection{Basic geometric setup}\label{ss.twisted.Fermi.setup}
\subsubsection{Properties of $\varepsilon^{-1}$-rescaled geometry}
Let $B_3^+(0)$ be equipped with a \Rm\ \mt\ $\hat{g}$ satisfying the conditions of Theorem \ref{t.main.ball}.
From now on, as in \cite{wang-wei,wang-weiadv}, we work in the rescaled domain $(B^+_{3R}(0),g)$ where $R=\ve^{-1}$ and $g$ is the rescaled metric given by
\begin{equation}\label{e.mt.scale}
g_{ij}(x) = \hat{g}_{ij}(\ve x).
\end{equation}
Notice that, by \eqref{e.g.Ga.est.ball}, the metric $g$ satisfies 
\begin{equation}\label{e.rescaled.bound.gEucl}
    \sup_{x\in B^+_{3R}(0)}\sum_{\ell=0}^4\varepsilon^{-\ell}|\partial^{(\ell)}(g_{ij}-\delta_{ij})(x)|\leq C(n)\eta_0
\end{equation}
for all indices $i,j$.
Therefore, the Christoffel symbols $\Gamma_{ij}^k$ of the metric $g$ satisfy  
\begin{equation}\label{e.rescaled.bound.Christoffel}
    \sup_{x\in B^+_{3R}(0)}\sum_{\ell=0}^3\varepsilon^{-(\ell+1)}|\partial^{(\ell)}\Gamma_{ij}^k(x)|\leq C(n)\eta_0,
\end{equation}
and the components $R_{ijkl}$ of the Riemann tensor of $g$ are of order $\varepsilon^2$:
\begin{equation}\label{e.rescaled.bound.Riemann}
    \sup_{x\in B^+_{3R}(0)}\varepsilon^{-2}|R_{ijkl}(x)|\leq C(n)\eta_0.
\end{equation}
Finally, the second fundamental form $\sff_{\partial_0B^+_{3R}(0)}$ of the boundary and its derivatives also satisfy similar estimates:
\begin{equation}\label{e.rescaled.bound.sff}
    \varepsilon^{-(\ell+1)}|\nabla^\ell\sff_{\partial_0B^+_{3R}(0)}|\leq C(n)\eta_0
\end{equation}
for $\ell=0,1,2,3$.

If we define $u(x)=u_{\ve}(\ve x)$, then $u:(B^+_{3R}(0),g)\ra (-1,1)$ satisfies the Allen--Cahn equation with parameter $1$, i.e.
\begin{equation}\label{e.ac.1}
-\De u + W'(u)=0, 
\end{equation}
with Neumann boundary condition $\partial_{x_1} u = 0$ on $\partial_0 B_{3R}^+(0)$. In the rescaled coordinates, by \eqref{e.esff.bd.ball} we have
\begin{equation}\label{0A=O(e)}
    |\mathcal{A}|=O(\ve)
\end{equation}
on $B^+_{3R-1}(0)\cap \{|u|\leq 1-b\}$. In particular, the estimate holds in a fixed neighbourhood of every component of the zero level set.

Note that in above setup, the following derivative estimates holds (see also \cite[Lemma 8.1]{wang-wei}). 

\begin{proposition}[Derivative estimates for the enhanced second fundamental form]\label{l.grad.A}
For every fixed integer $k\geq0$,
\begin{equation}\label{A=O(e)}
    |\na^k \A|=O(\ve)
\end{equation}
on $B^+_{3R-1}(0)\cap \{|u|\leq 1-b\}$ \fa\ $k\geq 0$. As a consequence, if $x\in \{|u|\leq 1-b'\}$ and $\Si = \{u=u(x)\}$, then
\begin{equation}\label{e.grad.sff}
|\na_{\Si}^k\sff_{\Si}|=O(\ve) \qquad \text{on } \Si\cap B^+_{3R-1}(0) \ \  \forall  k\geq 0.
\end{equation}
\end{proposition}
\begin{proof}
    See Appendix \ref{Agradient} for the proof.
\end{proof}

Let $\Xi$ denote the union of all the connected components of $\{u=0\}\cap (\bbb^+_{2R}(0)\times (-2R,2R))$ which have non-empty intersections with $\bbb^+_{2R}(0)\times (-1.5R,1.5R)$. Then, using the notations of Section \ref{s.reduction}, $\Xi$ can be written as a union of graphs
\begin{equation}\label{e.Si.al}
\Xi=\bigcup_{\al}\Si_{\al},\quad \Si_{\al}=\{x_{n+1}=f_{\al}(x'):x'=(x_1,\dots,x_n)\in \bbb^+_{2R}(0)\},
\end{equation}
where $f_{\al}\in C^{\infty}(\bbb^+_{2R}(0),(-1.6R,1.6R))$, $f_{\al}(x')=\ve^{-1}\hat{f}_{\al,\ve}(\ve x')$, $\del_1 f_{\al} = 0 \text{ along } \del_0\bbb^+_{2R}(0)$,
\begin{equation}\label{e.na.f}
\md{\na f_{\al}}\leq \et_0,\quad |\na^2 f_{\al}|\leq \La_0\ve, \quad |\na^k f_{\al}|=O(\ve)\ \ \forall  k\geq 2 \text{ on }\bbb^+_{2R}(0).
\end{equation}

We fix an index $\al$ and set $f=f_{\al}$, $\Si = \{x_{n+1}=f(x'):x'=(x_1,\dots,x_n)\in \bbb^+_{2R}(0)\}$. Following \cite[Page 28]{Pacard-Ritore}, we discuss the definition and the properties of the \textit{twisted Fermi coordinates} \wrt\ $\Si$. 

Applying the standard extension operator in $C^m$ for every fixed $m$ (see \cite[Lemma 6.37]{gilbarg1977elliptic}), choose smooth extensions
$f^*$ and $g^*$ to the full balls, equal to $\hat f_{\alpha,\varepsilon}$ and $\hat g$ on the
corresponding half-balls, such that, for every fixed $m$ used below,
\begin{align}
 \|f^*\|_{C^m(\mathbb B_2(0))}
 &\leq C(n,m)\|\hat f_{\alpha,\varepsilon}\|_{C^m(\mathbb B_2^+(0))},
 \label{e.f*.C^4}\\
 \|\delta_{ij}-g^*_{ij}\|_{C^m(B_3(0))}
 &\leq C(n,m)\|\delta_{ij}-\hat g_{ij}\|_{C^m(B_3^+(0))}.
 \label{e.g*.C^3}
\end{align}
Define $\widetilde f(x')=\varepsilon^{-1}f^*(\varepsilon x')$ and
$\widetilde g_{ij}(x)=g^*_{ij}(\varepsilon x)$.  These extensions are smooth, agree with the
original data on the half-balls, and satisfy the rescaled estimates at every derivative order used
below. Then $\tilde{f}=f$ on $\bbb^+_{2R}(0)$ and $\tilde{g}=g$ on $B^+_{3R}(0)$. 
Notice that $\tilde{g}$ satisfies similar estimates to \eqref{e.rescaled.bound.gEucl}, \eqref{e.rescaled.bound.Christoffel}, and \eqref{e.rescaled.bound.Riemann} (possibly with  different dimensional constants) by \eqref{e.g*.C^3}.

Subsequently, we will assume that all the \mt\ dependent objects on $B_{3R}(0)$ (e.g. $\langle \cdot ,\cdot \rangle$, $\na$ etc.) are \wrt\ \(\tilde{g}\). Notice that $\tilde g=g$ on $B^+_{3R}(0)$. We define $\tilde{\Si} = \{x_{n+1}=\tilde{f}(x'):x'=(x_1,\dots,x_n)\in \bbb^+_{2R}(0)\}$.

\subsubsection{Properties of the normal exponential map and Fermi coordinates} Let $\Om= \bbb_{2R}(0)\times (-2R,2R)$ and $T\Om$ denote the tangent bundle of $\Om$ which can be canonically identified with $\Om\times \bbr^{n+1}$. Let $(x_1,\dots,x_{n+1},v_1,\dots,v_{n+1})$ be the standard coordinates on $T\Om$, $\pi_x:T\Om \ra \Om$ and $\pi_v:T\Om \ra \bbr^{n+1}$ the projection maps. Let $\Ph(s,-)$ be the flow of the vector field 
\[X(x,v)=\sum_k v_k\del_{x_k}\big\vert_{(x,v)}-\sum_{i,j,k}\Ga_{ij}^k(x)v_iv_j\del_{v_k}\big\vert_{(x,v)},\]
$\Ph_1=\pi_x\circ\Ph$, $\Ph_2=\pi_v\circ \Ph$. Then $\Ph_1(s,x,v)=\exp_x(sv)$ and $\Ph_2(s,x,v)=\frac{d}{ds}\exp_x(sv)$,
\begin{align}\label{e.Ph}
\del_s\Ph_1^k=\Ph_2^k, \quad \text{and} \quad \del_s \Ph_2^k=-\sum_{ij}\Ga^k_{ij}(\Ph_1)\Ph_2^i\Ph_2^j,  
\end{align}
for all $k$. 
Let $\mathcal S$ be the set of $(x,v)$ for which $|v|\leq\delta R$ and the geodesic
$s\mapsto\exp_x(sv)$ remains in $\Omega$ for $0\leq s\leq1$, where $\delta>0$ is fixed and
small, and define $\mathscr{E}:\mathcal{S}\to \Om$ by $\mathscr{E}(x,v)=\exp_x(v)$. Note that the coefficients $\Gamma_{ij}^k$ are computed in the $x_i$ coordinates. \newcommand{\sE}{\mathscr{E}}

\begin{lemma}\label{l.exp.map}
	For all $(x,v)\in \mathcal{S}$ and indices $i,j$, there holds
	\begin{equation}
		|\del_{x_i}\sE^j(x,v) - \de_{ij}|+ |\del_{v_i}\sE^j(x,v) - \de_{ij}|\lesssim \ve|v|,\quad |D^m \sE^i(x,v)|\lesssim \ve(1+|v|)\; \forall  m\geq 2.
	\end{equation}
\end{lemma}
 \begin{proof}
 The proof is similar to \cite[Proposition A.1]{KapouleasLi21}. Since $\sE(x,v)= \Ph_1(1,x,v)$, we are going to restrict $\Ph$ on $[0,1]\times \mathcal{S}$ in this proof. Differentiating \eqref{e.Ph} we obtain
 	\begin{align}\label{e.Ph1}
 	\del_s\del_{*}\Ph_1^k=\del_{*}\Ph_2^k, \qquad \del_s \del_{*}\Ph_2^k=-\sum_{ijl}\del_{x_l}\Ga^k_{ij}(\Ph_1)\del_{*}\Ph_1^l\Ph_2^i\Ph_2^j-2\sum_{ij}\Ga^k_{ij}(\Ph_1)\del_{*}\Ph_2^i\Ph_2^j,
 	\end{align}
where $\del_{*}$ is any element in $\{\del_{x_i},\del_{v_i}:i=1,\dots,n+1\}$. We note that $\Ph(0,x,v)=(x,v)$, $|\Ph_2(s,x,v)|=|v|$ and (since the \mt\ $\tilde{g}$ is uniformly close to the Euclidean \mt) $\ve|v|=O(1)$ \fa\ $(x,v)\in \mathcal{S}$. \tf, using Gronwall's inequality for vector valued functions \cite[Theorem D.2]{Lee}, we conclude that $|\del_{*}\Ph(s,x,v)|$ is uniformly bounded for $(s,x,v)\in [0,1]\times \mathcal{S}$. Using this, one can deduce from \eqref{e.Ph1} that $|\del_s\del_{*}\Ph_2^k|\lesssim \ve |v|$ and
$|\del_{*}\Ph_1^k(1,x,v)-\del_{*}\Ph_1^k(0,x,v)-\del_{*}\Ph_2^k(0,x,v)|\lesssim \ve|v|.$
This implies the $C^1$ estimate for $\sE$.

Differentiating \eqref{e.Ph1} we obtain
\begin{align}\label{e.Ph2}
\begin{split}
    &\del_s\del_{\circ,*}\Ph_1^k=\del_{\circ,*}\Ph_2^k,\\
&\del_s \del_{\circ,*}\Ph_2^k=-\sum_{ijl}\del_{x_l}\Ga^k_{ij}(\Ph_1)\del_{\circ,*}\Ph_1^l\Ph_2^i\Ph_2^j-2\sum_{ij}\Ga^k_{ij}(\Ph_1)\del_{\circ,*}\Ph_2^i\Ph_2^j+\mathcal{E}
\end{split}
\end{align}
where $\del_{\circ}$ is any element in $\{\del_{x_i},\del_{v_i}:i=1,\dots,n+1\}$ and $|\mathcal{E}|\lesssim \ve(1+|v|).$ Note that $\mathcal{E}$ does not contain any second order derivatives of $\Ph$. \tf, using Gronwall's inequality for vector valued functions \cite[Theorem D.2]{Lee}, we obtain that $|\del_{\circ,*}\Ph(s,x,v)|$ is uniformly bounded for $(s,x,v)\in [0,1]\times \mathcal{S}$. Using this, one can deduce from \eqref{e.Ph2} that $|\del_s\del_{\circ,*}\Ph_2^k|\lesssim \ve (1+|v|)$, which implies $|\del_{\circ,*}\Ph_1^k|\lesssim \ve (1+|v|)$. Thus we obtain the $C^2$ estimate for $\sE$ and the higher order estimates can be proved in a similar way. 
\end{proof}


Let $\tilde{t}: \bbb_{1.8R}(0)\times (-2R,2R) \to \bbr$ be the signed distance \fn\ from $\tilde{\Si}$ \st\ $\partial_{x_{n+1}}\tilde{t}>0$ along \(\tilde{\Si}\). By \eqref{e.na.f}, if $\et_0, \La_0$ are sufficiently small
\begin{enumerate}[(i)]
	\item \(\tilde{t}\) is smooth on $\bbb_{1.8R}(0)\times (-2R,2R)$;
	\item the nearest point projection map \(\tilde{\Pi}:\bbb_{1.8R}(0)\times (-2R,2R)\ra \tilde{\Si}\) is a well-defined smooth map, whose image is contained in \(\tilde{\Si} \cap (\bbb_{1.9R}(0)\times (-2R,2R))\).
\end{enumerate}

\begin{remark}
    In what follows, we will denote by $\nabla$ the gradient of a scalar map, while $D$ will be used for the gradient of vector valued maps, i.e. the gradient of each component. Furthermore, everything will be computed in the ambient coordinates $(x_1, x_2, \ldots, x_{n + 1})$.
\end{remark}

\begin{lemma}\label{l.tilde.t.Pi}
	There exists $\tilde\de>0$, independent of $\ve$, \st\ the following holds.
	\begin{equation}\label{e.tilde.t.Pi}
	|D\tilde{\Pi}|=O(1),\quad |\na^m\tilde{t}|+|D^m\tilde{\Pi}|\lesssim \ve(1+|\tilde{t}|) \quad \text{ on }\{|\tilde{t}|\leq \tilde\de R\}\;\forall m\geq 2.
	\end{equation}
\end{lemma}
\begin{proof}
Let $\tilde{\sE}:\tilde{\Si}\times (-\tilde\de R, \tilde\de R)\to \bbb_{1.8R}(0)\times (-2R,2R)$ be defined by $\tilde{\sE}(y,\tilde{z})=\exp_y(\tilde{z}\bn(y))=\sE(y,\tilde{z}\bn(y))$, where $\bn$ is the unit normal along $\tilde{\Si}$ satisfying $\langle  \bn , \del_{x_{n+1}}\rangle>0$, and $\tilde\delta>0$ will be determined later in the proof. We note that if $y=(y_1,\dots,y_n)\in \tilde{\Si}$, then $x_i(y)=y_i$ if $1\leq i\leq n$ and $x_{n+1}(y)=\tilde{f}(y_1,\dots, y_n)$. Hence
\begin{align}\label{e.tilde.E}
\begin{split}
\del_{y_j}\tilde{\sE}^i(y,\tilde{z}) &=\del_{x_j}\sE^i(y,\tilde{z}\bn(y))+\del_{x_{n+1}}\sE^i(y,\tilde{z}\bn(y))\del_{x_j}f(y_1,\dots,y_n)\\
& +\sum_k\tilde{z}\del_{v_k}\sE^i(y,\tilde{z}\bn(y))\del_{y_j}\bn^k(y),\\
\del_{\tilde{z}}\tilde{\sE}^i(y,\tilde{z}) &= \sum_k\del_{v_k}\sE^i(y,\tilde{z}\bn(y))\bn^k(y).
\end{split} 
\end{align}
\tf, by Lemma \ref{l.exp.map} and \eqref{e.na.f}, we have 
\begin{equation}\label{e.del.tilde.E}
|\del_{y_j}\tilde{\sE}^i(y,\tilde{z})-\de_{ij}|\lesssim \ve|\tilde{z}|+\et_0,\;1\leq i,j\leq n+1,
\end{equation}
where $\del_{y_{n+1}}$ stands for $\del_{\tilde{z}}$. This implies that if $\de$ and $\et_0$ are sufficiently small, the Jacobian of  $\tilde{\sE}$ is greater or equal than $1/2$. Differentiating \eqref{e.tilde.E} and using Lemma \ref{l.exp.map}, one can show that
\begin{equation}\label{e.D^k.tilde.E}
|D^m\tilde{\sE}(y,\tilde{z})|\lesssim \ve(1+|\tilde{z}|)\;\forall m\geq 2.
\end{equation}
Since $\tilde{t}(\tilde{\sE}(y,\tilde{z}))=\tilde{z}$ and $\tilde{\Pi}(\tilde{\sE}(y,\tilde{z}))=y$, using inverse function theorem, one can deduce from \eqref{e.del.tilde.E} and \eqref{e.D^k.tilde.E} that if $\tilde\de$ and $\et_0$ are sufficiently small,
\begin{equation}\label{e.tilde.t.Pi'}
|D\tilde{\Pi}|=O(1),\quad |\na^m\tilde{t}|+|D^m\tilde{\Pi}|\lesssim \ve(1+|\tilde{t}|) \text{ on }\{|\tilde{t}|\leq \tilde\de R\}\;\forall m\geq 2.
\end{equation}
\end{proof}

\subsection{Definition and properties of twisted Fermi coordinates}\label{s.tw.Fer.def} 
\subsubsection{The twisted Fermi vector field  and twisted Fermi coordinates}\label{sectwisted fermi def}

Let $\tilde{F}: \bbb_{2R}(0)\times (-2R,2R) \ra \bbr$ be defined by $\tilde{F}(x)=x_{n+1}-\tilde{f}(x')$. We define the following vector fields on $\bbb_{1.8R}(0)\times (-2R,2R)$:
\begin{align}
&\tilde{Z}=\na \tilde{t};\quad Z_1=\frac{\na \tilde{F}}{|\na \tilde{F}|};\quad Z_2'=\del_{x_1} - \left\langle \del_{x_1}, Z_1\right\rangle Z_1;\quad Z_2 = \frac{Z_2'}{\md{Z_2'}}; \quad Z=\frac{\tilde{Z}-\langle\tilde{Z},Z_2\rangle Z_2}{1-\langle\tilde{Z},Z_2\rangle^2}.
\end{align}
Then, $|Z_2'|\to 1$ and $|\langle\tilde{Z},Z_2\rangle|\to 0$ uniformly as $(\et_0+\La_0)\ra 0$. Along $\Si$, $\tilde{Z}=Z_1$ and is orthogonal to $\Si$. Furthermore, $Z_2=\del_{x_1}$ along $\del_0\bbb^+_{1.8R}(0)\times (-2R,2R)$ and is tangential to $\Si$. 

These facts imply the following lemma.

\begin{lemma}\label{l.vect.Z}
The vector field $Z$ has the following properties.
\begin{itemize}
	\item[(i)] $Z$ is tangential to $\del_0\bbb^+_{1.8R}(0)\times (-2R,2R)$.
	\item[(ii)] $Z=\tilde{Z}$ along $\Si$ and is perpendicular to $\Si$.
	\item[(iii)] $\langle Z, \tilde{Z}\rangle=1$.
	\item[(iv)] $(Z-\tilde{Z})$ is perpendicular to $\tilde{Z}$.
	\item[(v)] If $Z=\sum_iZ^i\del_{x_i}$ then $|Z^i-\de_{i,n+1}|\lesssim \et_0+\La_0$.
\end{itemize}
\end{lemma}
Combining \eqref{e.grad.sff}, \eqref{e.tilde.t.Pi}, \eqref{e.na.f}, \eqref{e.f*.C^4}, and \eqref{e.g*.C^3} we infer 
\begin{equation}\label{e.na.Z}
|\na^m \tilde{Z}|+|\na^m Z_1|+|\na^m Z_2|+|\na^m Z|\lesssim \ve(1+|\tilde{t}|) \quad \text{on }\{|\tilde{t}|\leq \tilde\de R\}\;\forall m\geq 1.
\end{equation}

Let $\Tht(\ta,-)=\Tht_{\ta}$ denote the flow of $Z$ on $\bbb_{1.8R}(0)\times (-2R,2R)$, so $\Tht(0,x)=x$ and
\begin{equation}\label{e.Tht}
\del_{\ta}\Tht = Z(\Tht).
\end{equation}

\begin{definition}[Twisted Fermi coordinates]\label{twisted fermi def}
Given a coordinate system $y=(y_1, \dots, y_n)$ on $\Si$,  \textit{twisted Fermi coordinates} on $\bbb^+_{1.8R}(0)\times (-2R,2R)$ with respect to $\Si$ are defined by $Y_{\Si}(y,z)=\Tht_{z}(y)$, for $y\in \Si$.
\end{definition}
By Lemma \ref{l.vect.Z} items (i) and (v), and \eqref{e.na.Z}, if $p\in \bbb^+_{1.6R}(0)\times (-2R,2R)$, then
\begin{enumerate}[(i)]
	\item \(\Tht_{\ta}(p)\in \bbb^+_{1.7R}(0)\times (-2R,2R)\) for all $\ta$ belonging to the domain of $\ta \mapsto \Tht_{\ta}(p)$,
	\item \tes\ a unique time $\ta_p\in \bbr$ \st\ $\Tht_{-\ta_p}(p)\in \Si$,	
\end{enumerate}  
provided $\et_0$ and $\La_0$ are sufficiently small. 

\begin{definition}[Twisted Fermi coordinate time function and projection]\label{d.twisted.time.projection}
The twisted Fermi coordinate time function $t:\bbb^+_{1.6R}(0)\times (-2R,2R) \ra \bbr$ and projection map $\Pi:\bbb^+_{1.6R}(0)\times (-2R,2R) \ra \Si\cap (\bbb^+_{1.7R}(0)\times (-2R,2R))$ are defined by $t(p)=\ta_p$ and $\Pi(p)=\Tht(-t(p),p)$. 
\end{definition}

From now on, $(y_1,\dots,y_n,z)$ will denote twisted Fermi coordinates \wrt\ this coordinate system on $\Si$.

\begin{remark}
	Similar coordinate systems near a boundary point appear in \cite[Section 10]{Pacard-Ritore} and \cite[Definition 2.10]{KapouleasLi21}.
\end{remark}

\begin{lemma}\label{l.Fc.vs.tFc.1}
Let $p\in \bbb^+_{1.6R}(0)\times (-2R,2R)$. Then $\tilde{t}(p)=t(p)$.
\end{lemma}
\begin{proof}
Let $q=\Pi(p)$. We consider the function $\ga(\ta)=\tilde{t}(\Tht(\ta,q))$. Then $\ga(0)=0$ and by Lemma \ref{l.vect.Z}(iv), $\ga'(\ta)=\left\langle \na \tilde{t}, Z\right\rangle(\Tht(\ta,q))=1$. Hence, $\tilde{t}(p)=\ga(\ta_p)=\ta_p=t(p)$.
\end{proof}

\begin{lemma}\label{l.Tht}
	Let us fix $K>0$. For all $(\ta,x)$ satisfying $|\ta|,|t(x)|\leq K\sqrt{R}$ and for all indices $i,j$, there holds
	\begin{align}
	&|\del_{x_i}\Tht^j(\ta,x) - \de_{ij}|\lesssim \ve(|\ta|+|\ta|^2+|t(x)|^2),\\ 
& |D^m_x \Tht^i(\ta,x)|\lesssim \ve(|\ta|+\ta^2+t(x)^2),\qquad |D^m \Tht^i(\ta,x)|\lesssim \ve(1+|\ta|^2+|t(x)|^2)\; \forall  m\geq 2.
	\end{align}
\end{lemma}
\begin{proof}
The proof is similar to the proof of Lemma \ref{l.exp.map}. Differentiating \eqref{e.Tht} we obtain
\begin{equation}\label{e.del.x.Tht}
\del_{\ta}\del_{x_i}\Tht^k=\sum_p\del_{x_p}Z^k(\Tht)\del_{x_i}\Tht^p.
\end{equation}
Furthermore, \eqref{e.na.Z} implies
$|\na^mZ(\Tht(\ta,x))|\lesssim \ve(1+|\ta|+|t(x)|)$ for all $m\geq 1$.
\tf, using Gronwall's inequality for vector valued functions \cite[Theorem D.2]{Lee}, we conclude that $|\del_{x_i}\Tht^k|=O(1)$. Using this estimate, we deduce from \eqref{e.del.x.Tht} that 
\begin{equation}\label{e.del2.ta.x.Tht}
|\del_{\ta}\del_{x_i}\Tht^k(\ta,x)|\lesssim \ve(1+|\ta|+|t(x)|),
\end{equation}
hence
\begin{equation}
|\del_{x_i}\Tht^k(\ta,x)-\del_{x_i}\Tht^k(0,x)|\lesssim \ve(|\ta|+|\ta|^2+|t(x)|^2).
\end{equation}
This implies the $C^1$ estimate for $\Tht$. 

Differentiating \eqref{e.del.x.Tht}, we obtain
\begin{equation}\label{e.del2.x.Tht}
\del_{\ta}\del_{x_ix_j}\Tht^k=\sum_{p,q}\del_{x_px_q}Z^k(\Tht)\del_{x_i}\Tht^p\del_{x_j}\Tht^q+\sum_p\del_{x_p}Z^k(\Tht)\del_{x_ix_j}\Tht^p.
\end{equation}
\tf, using Gronwall's inequality for vector valued functions \cite[Theorem D.2]{Lee}, we conclude that $|\del_{x_ix_j}\Tht^k|=O(1)$. Using this estimate, we deduce from \eqref{e.del2.x.Tht} that 
\begin{equation}
|\del_{\ta}\del_{x_ix_j}\Tht^k(\ta,x)|\lesssim \ve(1+|\ta|+|t(x)|);
\end{equation}
hence
\begin{equation}\label{e.est.del2.x.Tht}
|\del_{x_ix_j}\Tht^k(\ta,x)|\lesssim \ve(|\ta|+|\ta|^2+|t(x)|^2).
\end{equation}
Combining \eqref{e.del2.ta.x.Tht} and \eqref{e.est.del2.x.Tht}, we obtain the $C^2$ estimates for $\Tht$. Higher order estimates can be proved in a similar way.
\end{proof}

\begin{lemma}\label{l.t.Pi}
There exists $0<\de<\tilde{\delta}$, independent of $\ve$, \st\ the following holds.
	\begin{equation}\label{e.t.Pi}
	|D\Pi|=O(1),\quad |\na^m t|\lesssim \ve(1+|z|),\quad |D^m \Pi|\lesssim \ve(1+|z|^2) \text{ on }\{|z|\leq \de \sqrt{R}\}\;\forall m\geq 2.
	\end{equation}
\end{lemma}
\begin{proof}
The estimates for $|\na^m t|$ follow from Lemma \ref{l.tilde.t.Pi} and Lemma \ref{l.Fc.vs.tFc.1}. Let $\hat{\Tht}$ denote the restriction of $\Tht$ on $[-\de \sqrt{R}, \de \sqrt{R}]\times \Si$, where $\de\in(0,\tilde\delta)$ is to be determined. We note that if $y=(y_1,\dots,y_n)\in \Si$, then $x_i(y)=y_i$ if $1\leq i\leq n$ and $x_{n+1}(y)=f(y_1,\dots, y_n)$. Hence
\begin{align}\label{e.hat.Tht}
& \del_{y_i} \hat{\Tht}^j(z,y)=\del_{x_i}\Tht^j(z,y)+\del_{x_{n+1}}\Tht^j(z,y)\del_{x_i}f(y_1,\dots,y_n),\\
& \del_z \hat{\Tht}^j(z,y)=Z^j(\Tht(z,y)).
\end{align}
\tf, by Lemma \ref{l.Tht}, \eqref{e.na.f} and Lemma \ref{l.vect.Z}(v), as $(\de+\et_0+\La_0)\to 0$ we infer 
\begin{equation}\label{e.del.hat.Tht}
|\del_{y_i}\hat{\Tht}^j(z,y)-\de_{ij}|=o(1),\;1\leq i,j\leq n+1,
\end{equation}
where $\del_{y_{n+1}}$ stands for $\del_{z}$. This implies if $\de,\;\et_0,\;\La_0$ are sufficiently small, the Jacobian of  $\hat{\Tht}$ is greater or equal than $1/2$. Differentiating \eqref{e.hat.Tht} and using Lemma \ref{l.Tht}, one can show that
\begin{equation}\label{e.D^k.hat.Tht}
|D^m\hat{\Tht}(z,y)|\lesssim \ve(1+|z|^2)\;\forall m\geq 2.
\end{equation}
Since $t(\hat{\Tht}(z,y))=z$ and $\Pi(\hat{\Tht}(z,y))=y$, using inverse function theorem, one can deduce from \eqref{e.del.hat.Tht} and \eqref{e.D^k.hat.Tht} that if $\de,\;\et_0,\;\La_0$ are sufficiently small,
\begin{equation}\label{e.D.t.Pi}
|D\Pi|=O(1),\quad |D^m \Pi|\lesssim \ve(1+|z|^2) \text{ on }\{|z|\leq \de \sqrt{R}\}\;\forall m\geq 2.
\end{equation}
\end{proof}

Notice that by definition of twisted Fermi coordinates, if $x\in \bbb^+_{1.6R}(0)\times (-2R,2R)$, then
\begin{equation}\label{e.partial_yi.differential.flow}
    \del_{y_i}|_x=D\Tht_{t(x)}(\Pi(x))\big(\del_{y_i}|_{\Pi(x)}\big).
\end{equation}
Moreover, for $y=(y_1,\dots,y_n)\in \Si$, clearly
$\del_{y_i}|_y=\del_{x_i}|_y+\del_{x_i}f(y_1,\dots, y_n)\del_{x_{n+1}}|_y$, since $\Sigma$ is the graph of $f$ over $\mathbb{B}^{+}_{2R}(0)$.

Hence, we obtain
\begin{equation}\label{e.del_y_i}
    \del_{y_i}|_{(y,z)}=\sum_j\del_{x_i}\Tht^j(z,y)\del_{x_j}|_{(y,z)}+\sum_j\del_{x_i}f(y_1,\dots, y_n)\del_{x_{n+1}}\Tht^j(z,y)\del_{x_j}|_{(y,z)}.
\end{equation}
Equation \eqref{e.del_y_i} has several useful consequences, which we shall record below.

\begin{lemma}[Covariant derivative estimates]\label{l.na.y_i}
The following estimates hold  for $\partial_{y_i}$ and its covariant derivatives
\begin{align}
    \label{e.del.y_i.O(1)}
	|\partial_{y_i}| = O(1)  &\text{ on }\{|z|\leq \delta\sqrt{R}\},\\
 \label{e.na.y_i}
	|\na^m\del_{y_i}|\lesssim \ve(1+|z|^2) &\text{ on }\{|z|\leq \delta\sqrt{R}\}\;\forall m\geq 1. 
\end{align} 
\end{lemma}

\begin{proof}
The lemma follows from \eqref{e.del_y_i} and Lemma \ref{l.Tht}.
\end{proof}

\begin{lemma}[Neumann derivative estimate]\label{l.del.y_1.t_be}
The following estimate holds
\begin{equation}\label{e.del.x_1-del.y_1}
	|\del_{x_1}\varphi(y, z)-\del_{y_1}\varphi(y, z)|\lesssim \ve|z|(1+|z|)|\nabla \varphi(y, z)|\text{ on }\partial_0\bbb^+_{1.6R}(0)\times [-\delta\sqrt R,\delta\sqrt R]
\end{equation}
for any smooth function $\varphi$ on $\bbb^+_{1.6R}(0)\times (-2R,2R)$.
\end{lemma}

\begin{proof}
Notice that along $\del_0\Si$, $\partial_{x_1}f=0$ by the setup in \ref{ss.twisted.Fermi.setup}, and $\del_{x_1}=\del_{y_1}$ by the definition of twisted Fermi coordinates. Hence, by \eqref{e.del_y_i}, on $\del_0\{|z|\leq K\sqrt R |\}$
\begin{equation}
   \del_{y_1}|_{(y,z)}=\sum_j\del_{x_1}\Tht^j(z,y)\del_{x_j}|_{(y,z)}.            
\end{equation}
Then \eqref{e.del.x_1-del.y_1} follows directly from Lemma \ref{l.Tht}.
\end{proof}

\begin{lemma}\label{l.Fc.vs.tFc.2}
	\Tes\ $\de >0$ \st\ the following holds: 
	$$d_{\Si}\left(\tilde{\Pi}(p),\Pi(p)\right)\lesssim \ve (|t(p)|^2+|t(p)|^3) \text{ on } \{|t|\leq \de R\}.$$
\end{lemma}

\begin{proof}
	Let $q=\Pi(p)$. We consider the function $G(\ta)=(\tilde{\Pi}_1(\Tht(\ta,q)),\dots, \tilde{\Pi}_n(\Tht(\ta,q)))$, where $\tilde{\Pi}_i=y_i\circ \tilde{\Pi}$. Then $G(0)=(y_1(q),\dots,y_n(q))$,
	\[\frac{d}{d\ta}\tilde{\Pi}(\Tht(\ta,q))=D\tilde{\Pi}(\Tht(\ta,q))(Z(\Tht(\ta,q))).\]
	Note that $\tilde{\Pi}$ is constant along the integral curves of $\tilde{Z}$. \tf, by Lemma \ref{l.vect.Z}(ii), $G'(0)=(0,\dots,0)$. By Lemma \ref{l.tilde.t.Pi} and \eqref{e.na.Z}, $\md{G''(\ta)}\lesssim \ve(1+|\ta|)$. \tf,\ by Taylor's theorem,  $$d_{\Si}\left(\tilde{\Pi}(p),\Pi(p)\right)\lesssim |G(t(p))-G(0)|\lesssim\ve(|t(p)|^{2}+|t(p)|^3).$$
\end{proof}

Since $Z=\tilde{Z}$ along $\Si$, the following lemma follows from \eqref{e.na.Z}.
\begin{lemma}\label{l.Z.Z-tld}
	There holds
	$$|Z-\tilde{Z}|\lesssim \ve(|z|+z^2),\quad  |\na^m (Z-\tilde{Z})|\lesssim \ve(1+|z|) \text{ on }\{|z|\leq \de R\}\;\forall m\geq 1.$$
\end{lemma}

\begin{remark}\label{r.tZ-Z}
    By Lemma \ref{l.vect.Z} (iv),
\begin{equation}\label{e.decomp.tZ}
	\del_{\tilde{z}}=\del_z-T=\partial_z-\sum_{i=1}^{n}T^i\del_{y_i}.
\end{equation}
By Lemma \ref{l.Z.Z-tld}, for each index $i$, we have estimates
\begin{equation}\label{e.difference.z.tz}
	|T^i|\lesssim \ve(|z|+z^2),\quad |\na^m T^i|\lesssim \ve(1+|z|) \text{ on }\{|z|\leq \de R\}\;\forall m\geq 1.
\end{equation}
Notice that $T^i(y,z)=-g^{iz}(y,z)$.
\end{remark}

\subsubsection{Boundary behaviour of twisted Fermi coordinates} Finally, we shall carefully look at how these twisted Fermi coordinates behave at the boundary. The decomposition in the following lemma will be crucial in the Inverse Function Theorem argument of Section \ref{s.approx.toda}.

\begin{lemma}\label{l.dz(nu)}
	Let $J:\del_0\bbb^+_{1.6R}(0)\times (-2R,2R)\ra \bbr$ be defined by $J=\del_{x_1}t=dz(\nu)$, $w(y)=\del_z J(y,0)$ and $\tilde{J}(y,z)=J(y,z)-w(y)z$. Then
	\begin{equation}\label{e.dz(nu).1}
		|J|=|\partial_{\nu} z|\lesssim \ve|z|(1+|z|+|z|^2) \text{ on }\{|z|\leq \de \sqrt R\}.
	\end{equation} 
	Moreover,
	\begin{equation}\label{e.w.1}
		|\na^k w|=O(\ve)\text{ on } \del_0\bbb^+_{1.6R}(0) \;\forall k\geq 0,
	\end{equation}
    and
    \begin{equation}\label{e.w.2}
		|\na^k_y \tilde{J}|\lesssim \ve^2(1+|z|^3)z^2 \text{ on }\{|z|\leq \delta\sqrt R\}\;\forall k\geq 0.
	\end{equation}
\end{lemma}

\begin{proof}
By Lemma \ref{l.Fc.vs.tFc.1}, $\nabla t=\widetilde Z$, and hence
$J=\langle\widetilde Z,\nu\rangle$.  The sheet meets the boundary orthogonally, so
$J(y,0)=0$.  Since $\partial_z=Z$ and the integral curves of $\widetilde Z$ are unit-speed geodesics,
\begin{equation}\label{e.partial.z.J}
 \partial_zJ
 =\langle\nabla_{Z-\widetilde Z}\widetilde Z,\nu\rangle
  +\langle\widetilde Z,\nabla_Z\nu\rangle.
\end{equation}
With the sign convention of Section \ref{s.basic.definitions},
$\langle X,\nabla_Y\nu\rangle=-\mathrm{I\!I}_{\partial_0}(Y,X)$ for vectors tangent to the boundary.
At $z=0$, where $Z=\widetilde Z$, this gives
\begin{equation}\label{e.w.boundary.geometry}
 w(y)=\partial_zJ(y,0)=-\mathrm{I\!I}_{\partial_0}(\widetilde Z,\widetilde Z)(y,0).
\end{equation}
In particular, \eqref{e.rescaled.bound.sff} and the tangential derivative bounds for $\widetilde Z$
imply \eqref{e.w.1}.

Differentiating \eqref{e.partial.z.J} once more in the $Z$-direction, and using
$\nabla_{\widetilde Z}\widetilde Z=0$, expresses $\partial_{zz}J$ as a sum of contractions of
$\nabla^3t$, $\nabla^2t$, $Z-\widetilde Z$, $\nabla(Z-\widetilde Z)$,
$\mathrm{I\!I}_{\partial_0}$, and $\nabla\mathrm{I\!I}_{\partial_0}$.  Lemmas \ref{l.t.Pi} and
\ref{l.Z.Z-tld}, together with \eqref{e.rescaled.bound.sff}, therefore give
\begin{equation}\label{e.est.partial.zz.J}
 |\nabla_y^k\partial_{zz}J(y,z)|
 \lesssim_k \varepsilon^2(1+|z|^3)
 \qquad (|z|\leq\delta\sqrt R)
\end{equation}
for every fixed $k$.  The same differentiated identity, with one fewer $z$-derivative, gives
$|\nabla_y^k\partial_zJ|\lesssim_k\varepsilon(1+|z|+|z|^2)$.
Taylor's theorem at $z=0$, using $J(y,0)=0$ and the definition of $w$, now yields
\eqref{e.dz(nu).1} and
\[
 |\nabla_y^k(J(y,z)-w(y)z)|
 \leq |z|^2\sup_{|s|\leq|z|}|\nabla_y^k\partial_{zz}J(y,s)|,
\]
which is exactly \eqref{e.w.2}.
\end{proof}

\subsection{Geometric approximation}
Let $\Si_{\al}$ be as in \eqref{e.Si.al}. $\Si_{\al}$ is parametrized by $y\mapsto (y,f_{\al}(y))$, and twisted Fermi coordinates $(y,z)$ with respect to $\Si_{\al}$ are defined as in Section \ref{s.tw.Fer.def}. 
As in Section \ref{s.tw.Fer.def}, let $t_{\al}$ and $\Pi_{\al}$ denote the twisted Fermi coordinate time function and projection map, and let $\Si_{\al, z}=\{t_{\al}=z\}$ be parametrized by $y\mapsto Y_{\Sia}(y,z)$. 

We shall identify a point $p=Y_{\Sia}(y,z)$ with its twisted Fermi coordinates $(y,z)$, and represent the metric $g$ at $p$ as a matrix
\begin{equation}\label{e.metric.matrix}
\renewcommand\arraystretch{1.8}
g(y,z)=\left(\begin{array}{@{}c|c@{}}
\big(g_{ij}(y,z)\big)_{i,j=1}^n & \big(g_{iz}(y,z)\big)_{i=1}^n \\
\hline
\big(g_{zj}(y,z)\big)_{j=1}^n & g_{zz}(y,z)
\end{array}\right)
\end{equation}
where $g_{ij}(y,z)\coloneqq \langle\partial_{y_i}|_{(y,z)}, \partial_{y_j}|_{(y,z)}\rangle$, $g_{iz}(y,z)\coloneqq \langle\partial_{y_i}|_{(y,z)}, \partial_{z}|_{(y,z)}\rangle$ etc.

Let $g_{\al}(y,z),\;\Gamma_{\alpha}(y,z),\;\sff_{\al}(y,z),\; H_{\al}(y,z)$ respectively denote the induced \mt\ on $\Si_{\al, z}$, its Christoffel symbols, the second fundamental form of $\Si_{\al, z}$ and the mean curvature of $\Si_{\al, z}$ at $(y,z)$.
Since, by definition $\{\partial_{y_1}|_{(y,z)}, \dots, \partial_{y_n}|_{(y,z)}\}$ are tangent to $\Sigma_{\alpha, z}$ at $(y,z)$, we have $g_{\alpha,ij}(y,z)=g_{ij}(y,z)$. Similarly, by inverting \eqref{e.metric.matrix}, we see that
\begin{equation}\label{e.g-1.alpha}
    g_\alpha^{ij}(y,z)=g^{ij}(y,z)-g^{zz}(y,z)^{-1}g^{iz}(y,z)g^{zj}(y,z)
\end{equation}
for all $1\leq i,j\leq n$.

Sometimes we shall use the abbreviated notation $g_{\al,z},\;\Gamma_{\al,z},\;\sff_{\al,z},\; H_{\al,z}$ to denote the same quantities, or even $g_{\al},\;\Gamma_\alpha,\;\sff_{\al},\; H_{\al}$, when there is no danger of ambiguity. The covariant derivative and the Laplacian on $\Si_{\al, z}$ \wrt\ the induced \mt\ are denoted by $\na_{\al,z}$ and $\De_{\al,z}$. Finally, we denote the volume form induced on $\Sigma_{\alpha, z}$ by $dA_{\alpha, z} = \lambda_{\alpha}(y, z)dy$, where $\lambda_{\alpha} = \sqrt{\det g_{\alpha}(y, z)}$. 

\begin{lemma}\label{l.g^zz.twisted}
	In twisted Fermi coordinates $(y,z)$ with respect to $\Sigma_\alpha$, $g^{zz}(y,z)=1$
\end{lemma}
\begin{proof}
	This follows from Lemma \ref{l.Fc.vs.tFc.1}. Indeed, $1=|\nabla\tilde{t}|^2=|\nabla t|^2=g^{zz}$. 
\end{proof}

\begin{proposition}[Geometric Expansion in Twisted Fermi Coordinates]\label{geometric approximation}
	Let $g_{\al}(y,z),\;\Gamma_{\alpha}(y,z),$ $\sff_{\al}(y,z),\; H_{\al}(y,z)$ respectively denote the induced \mt\ on $\Si_{\al, z}$, the second fundamental form of $\Si_{\al, z}$ and the mean curvature of $\Si_{\al, z}$ at $(y,z)$. For each  fixed $K>0$, if $z\in(-K |\log \varepsilon|, K|\log \varepsilon|)$, we have for all $1\leq i,j,k,l\leq n$
	\begin{gather}\label{e.g_ij.approx}
		|g_{\al, ij}(y, z)-g_{\al,ij}(y, 0)|+|g^{ij}_{\al}(y, z)-g^{ij}_{\al}(y, 0)|\lesssim \varepsilon |z|(1+|z|),\\
    \label{e.na.g.approx}
		|\partial_{y_l} g_{\al,ij}(y, z)|+ |\partial_{y_k}\partial_{y_l} g_{\al,ij}(y, z)|+ |\partial_{y_l} g^{ij}_{\al}(y, z)|+ |\partial_{y_k}\partial_{y_l} g^{ij}_{\al}(y,z)| \lesssim\varepsilon(1+|z|^2),\\
    \label{e.gamma.alpha.approx}
        |\Gamma_{\al, ij}^k(y,z)-\Gamma_{\al, ij}^k(y,0)|\lesssim \varepsilon |z|(1+|z|),\\
    \label{e.partial.gamma.alpha.approx}
        |\partial_{y_l}\Gamma_{\al, ij}^k(y,z)|\lesssim \varepsilon(1+|z|^2),\\
    \label{e.sff.approx}
		|\mathrm{I\!I}_{\al}(y, z)-\mathrm{I\!I}_{\al}(y, 0)|\lesssim\varepsilon^2|z|(1+|z|^3),\\
    \label{e.mn.curv.approx}
 		|H_{\al}(y, z)-H_{\al}(y, 0)|\lesssim\varepsilon^2|z|(1+|z|^3),\\
    \label{e.nablasff}
        \vert\nabla_{\al,z}\sff_\alpha(y, z)-\nabla_{\al, 0}\sff_\alpha(y, 0)\vert\lesssim \varepsilon^2|z|(1+|z|^3)\\
    \label{e.vol.form}
		\vert \lambda_\alpha(y, z) - \lambda_\alpha(y, 0) \vert \lesssim \varepsilon |z|(1 + |z|). 
    \end{gather}    
	with constants depending on $n, K, \delta, \eta_0$ and $\Lambda_0$.
\end{proposition}
\noeqref{e.gamma.alpha.approx}\noeqref{e.nablasff}

\begin{proof}
    See Appendix \ref{ap:geometric.approx} for the proof.
\end{proof}

Recall that in Fermi coordinates $(\tilde{y}, \tilde{z})$ with respect to $\Sigma_\alpha$, we have the following well-known formula which relates the ambient Laplacian $\Delta$ to the induced Laplacian $\Delta_{\alpha, \tilde z}$ on $\Sigma_{\alpha, \tilde z}=\{\tilde{t}_\alpha =\tilde z\}$ and the mean curvature $H_{\alpha, \tilde z}$ of $\Sigma_{\alpha, \tilde z}$:
\[\De\vp (\tilde{y},\tilde{z}) = \De_{\alpha,\tilde{z}}\vp(\tilde{y},\tilde{z}) + \del_{\tilde{z}\tilde{z}}\vp(\tilde{y},\tilde{z}) - H_{\alpha, \tilde z}(\tilde{y},\tilde{z})\del_{\tilde{z}}\vp(\tilde{y},\tilde{z}). \]
Since $\tilde{z}=z$, we have the following lemma, which establishes an analogous formula for twisted Fermi coordinates

\begin{lemma}\label{l.Lapl.tw.Fer}
Let $T=(T^i)$ be as in Remark \ref{r.tZ-Z}. Then,  if $(y,z)$ denote twisted Fermi coordinates with respect to $\Sigma_\alpha$, we have 
\begin{align}\label{e.Lapl.tw.Fer}
	\De \vp = & \De_{\alpha, z}\vp + \del_{zz}\vp -H_\alpha(y,z)\del_z\vp +\mathscr{E}_{\alpha}(y, z)\varphi
	\end{align}
where 
\begin{align}\label{diff operator E=E1+E2}
\mathscr{E}_{\alpha}(y, z)\varphi
= H_\alpha(y,z) T \ast \del_y \vp + T\ast T \ast \del_{yy}^2\vp + T\ast \na T \ast \del_y \vp + \na T \ast \del_y \vp + T \ast \del^2_{y,z} \vp.
\end{align}
In particular, every
coefficient in $\mathscr E_\alpha$ is controlled by \eqref{e.difference.z.tz}.
\end{lemma}

Let $\varphi_0$ be a $C^2$ function on $\Sigma_\alpha$. we can extend $\vp_0$ to a $C^2$ function on $\bbb^+_{1.6R}(0)\times (-2R,2R)$ (which we will denote by $\vp$) by setting it to be constant along the flowlines of the vector field $Z$. Namely, we let
\begin{equation}\label{e.const.extension}
    \vp(y,z) \coloneqq \vp_0(y)
\end{equation}
in twisted Fermi coordinates with respect to $\Sigma_{\alpha}$.
Notice that by definition $\partial_z\vp=0$, and, since $[\partial_z, \partial_{y_i}]=0$, we have $\partial_z\partial_{y_i}\vp=\partial_z\partial_{y_i}\partial_{y_j}\vp=0$. Hence, 
\begin{equation}\label{e.partial.constant.extension}
    \partial_{y_i}\vp(y,z)=\partial_{y_i}\vp_0(y) \;\text{ and }\; \partial_{y_i}\partial_{y_j}\vp(y,z)=\partial_{y_i}\partial_{y_j}\vp_0(y).
\end{equation}

\begin{proposition}\label{prope.comparing laplacian}
	Let $K>0$ be fixed. For $\vp_0\in C^2(\Sia)$, let $\vp\in C^2(\bbb^+_{1.6R}(0)\times (-2R,2R))$ be defined as in \eqref{e.const.extension}. Then, there holds
	\begin{equation} \label{e.comparing laplacian}
		|\Delta_{\al, z}\varphi-\Delta_{\al, 0} \varphi|\lesssim\varepsilon|z|(1+|z|)(|\nabla^2_{\al, 0}\varphi_0|+|\nabla_{\al, 0}\varphi_0|)
	\end{equation}
	for $z\in (-K|\log \varepsilon|, K|\log \varepsilon|)$, where the constant depends on $n, K, \delta, \eta_0$ and $\Lambda_0$.
\end{proposition}

\begin{proof}
For a general function $\vp\in C^2(\bbb^+_{1.6R}(0)\times (-2R,2R))$,
\begin{equation}\label{e.lapl.gamma}
\Delta_{\alpha,z}\varphi(y,z)
=\sum_{i,j=1}^ng_\alpha^{ij}(y,z)\partial_{y_iy_j}\varphi(y,z)
-\sum_{i,j,k=1}^ng_\alpha^{ij}(y,z)\Gamma_{\alpha,ij}^k(y,z)\partial_{y_k}\varphi(y,z).
\end{equation}
Thus,
\begin{equation}\label{e.lapl.difference}
\begin{split}
	\Delta_{\alpha, z}\varphi (y,z)-\Delta_{\alpha, 0}\varphi (y,0) &= \sum_{i,j=1}^n(g_{\alpha}^{ij}(y,z)-g_{\alpha}^{ij}(y,0))\partial_{y_i}\partial_{y_j}\varphi(y,z)\\
	&\quad-\sum_{i,j,k=1}^n(g_{\alpha}^{ij}(y,z)\Gamma_{\alpha,ij}^k(y,z)-g_{\alpha}^{ij}(y,0)\Gamma_{\alpha,ij}^k(y,0))\partial_{y_k}\varphi(y,z)\\
	&\quad +\sum_{i,j=1}^ng_{\alpha}^{ij}(y,0)(\partial_{y_i}\partial_{y_j}\varphi(y,z)-\partial_{y_i}\partial_{y_j}\varphi(y,0))\\
	&\quad -\sum_{i,j,k=1}^n g_{\alpha}^{ij}(y,0)\Gamma_{\alpha,ij}^k(y,0)(\partial_{y_k}\varphi(y,z)-\partial_{y_k}\varphi(y,0)).
\end{split}
\end{equation}
In the case where $\varphi$ is the $z$-constant extension of $\vp_0\in C^2(\Sia)$ as in \eqref{e.const.extension}, then \eqref{e.partial.constant.extension} holds, and \eqref{e.partial.constant.extension} follows directly from \eqref{e.g_ij.approx} and \eqref{e.gamma.alpha.approx}.
\end{proof}

Finally, we shall record for later use the following expansion for the mean curvature of the graph of a function $\vp\in C^2(\Sigma_\alpha)$. This is analogous to \cite[Lemma 2.9]{chodosh-mantoulidis}.

\begin{lemma}\label{mean curvature expansion tf}
Fix $K>0$.  Let $\varphi\in C^2(\Sigma_\alpha)$ satisfy
\begin{equation*}
 \|\varphi\|_{L^\infty(\Sigma_\alpha)}\leq K|\log\varepsilon|,
 \qquad |\nabla_{\alpha,0}\varphi|\leq\eta_0,
 \qquad |\nabla^2_{\alpha,0}\varphi|\leq\Lambda_0\varepsilon.
\end{equation*}
Let $\Sigma_\alpha[\vp]$ denote the graph of $\vp$ in twisted Fermi coordinates, parametrized by
$y\mapsto Y_{\Sigma_\alpha}(y,\varphi(y))$, and let $H_{\Sigma_\alpha[\vp]}$ denote its mean curvature. 
For each fixed $\varphi$, define an operator on $\psi\in C^2(\Sigma_\alpha)$ by
\[
 \mathcal L_{\Sigma_\alpha[\varphi]}\psi(y)
 :=\left[
 \frac{\Delta_{\Sigma_\alpha[\varphi]}(\psi\circ\Pi_\alpha)}
 {\langle\bn_{\Sigma_\alpha[\varphi]},\nabla t_\alpha\rangle}
 \right]\!\left(Y_{\Sigma_\alpha}(y,\varphi(y))\right),
\]
where $\psi\circ\Pi_\alpha$ is restricted to the graph before applying its
intrinsic Laplacian.  For all sufficiently small $\varepsilon$, this is a
linear uniformly elliptic operator in weighted divergence form, and
\[
 \mathcal L_{\Sigma_\alpha[\varphi]}1=0,
 \qquad \mathcal L_{\Sigma_\alpha[0]}=\Delta_{\alpha,0}.
\]
Moreover,
\begin{equation*}
\begin{split}
 H_{\Sigma_\alpha[\varphi]}-H_{\Sigma_\alpha}
 ={}&\mathcal L_{\Sigma_\alpha[\varphi]}\varphi
 +\bigl(|\sff_{\Sigma_\alpha}|^2
       +\Ric_g(\bn_{\Sigma_\alpha},\bn_{\Sigma_\alpha})\bigr)\varphi
 +\mathcal Q(y,\varphi,d\varphi),
\end{split}
\end{equation*}
where $\mathcal Q(y,z,p)$ is a smooth function of the base point, height,
and covector $p\in T_y^*\Sigma_\alpha$, and satisfies
\begin{equation*}
\begin{split}
 |\mathcal Q(y,\varphi,d\varphi)|
 \leq{}& C\varepsilon|d\varphi|^2
 +C\varepsilon|\varphi|^2
          |\nabla_{\alpha,0}H_\alpha(y,0)|+C\varepsilon^3|\varphi|^2(1+|\varphi|^2).
\end{split}
\end{equation*}
The norms on the right are taken with respect to $g_{\alpha,0}$.
The ellipticity constants and $C$ are independent of $\varepsilon$ and
$\alpha$.
\end{lemma}

\begin{proof}
    See Appendix \ref{ap:geometric.approx}.
\end{proof}

\subsection{Commutator estimates}
We shall now record some easy commutator estimates, which will be used throughout the rest of the paper.

\begin{proposition}\label{p.commutator.laplacian}
Let $K>0$ be fixed. For all $\vp\in C^3(\bbb^+_{1.6R}(0)\times (-2R,2R))$, we have the following commutator estimate 
\begin{equation}\label{e.exchange.derivative}
	|\Delta_{\al, z} \partial_{y_i}\varphi-\partial_{y_i}\Delta_{\al, z}\varphi|\lesssim \varepsilon(1+|z|^2)(|\nabla^2_{\al, z}\varphi|+|\nabla_{\al, z}\varphi|)
\end{equation}
for $z\in (-K|\log \varepsilon|, K|\log \varepsilon|)$.
The constant depends on $n, K, \delta, \eta_0$ and $\Lambda_0$.

\end{proposition} 

\begin{proof}
Using the formula in \eqref{e.lapl.gamma}, we see that
\begin{equation}
\begin{split}
    \Delta_{\al, z} \partial_{y_i}\varphi-\partial_{y_i}\Delta_{\al, z}\varphi &= -\sum_{j,k=1}^n(\partial_{y_i}g_{\alpha}^{jk})\partial_{y_j}\partial_{y_k}\varphi+\sum_{j,k,l=1}^n\left((\partial_{y_i}g_{\alpha}^{jk})\Gamma_{\alpha,jk}^l+g_{\alpha}^{jk}(\partial_{y_i}\Gamma_{\alpha,jk}^l)\right)\partial_{y_l}\varphi
\end{split}
\end{equation}
where all terms are evaluated at $(y,z)$. Then \eqref{e.exchange.derivative} follows from \eqref{e.g_ij.approx},
\eqref{e.na.g.approx},
\eqref{e.gamma.alpha.approx}, and \eqref{e.partial.gamma.alpha.approx}.
\end{proof}

We shall also need some commutator estimates for derivatives along the boundary of the domain.

\begin{proposition}
Let $K>0$ be fixed. For all $\vp\in C^3(\bbb^+_{1.6R}(0)\times (-2R,2R))$, we have the following commutator estimates
\begin{align}
    \label{e.exchange.zneumannderivative0}
        |\partial_{z}\partial_{\nu}\varphi-\partial_{\nu}\partial_{z}\varphi|&\lesssim \varepsilon{(1+|z|)}|\nabla \varphi|,\\
    \label{e.exchange.1neumannderivative0}
        |\partial_{\nu}\partial_{y_1}\varphi-\partial_{y_1}\partial_{\nu}\varphi|&\lesssim \varepsilon(1+|z|^2)|\nabla \varphi|,\\
    \label{e.exchange.neumannderivative0}
        |\partial_{\nu}\partial_{y_i}\varphi-\partial_{y_i}\partial_{\nu}\varphi|&\lesssim \varepsilon |z|  (1+|z|)|\nabla \varphi|\;\text{ for all $2\leq i\leq n$}
\end{align}
on $\partial_0\{|z|\leq K|\log\varepsilon|\}$.
The constants depend on $n, K, \delta, \eta_0$ and $\Lambda_0$.
\end{proposition}

\begin{proof}
Recall that $\nu=\partial_{x_1}$ on $\partial_0\{|z|\leq K|\log\varepsilon|\}$. Thus, $\partial_{z}\partial_{\nu}\varphi-\partial_{\nu}\partial_{z}\varphi=[Z, \partial_{x_1}]\varphi$. If we write $Z=\sum_{k=1}^{n+1}Z^k\partial_{x_k}$, then $[Z, \partial_{x_1}]=-\sum_{k=1}^{n+1}(\partial_{x_1}Z^k)\partial_{x_k}$, and \eqref{e.exchange.zneumannderivative0} follows directly from \eqref{e.na.Z}.

Now, let $L_i\coloneqq [\partial_{x_1}, \partial_{y_i}]$, so that \eqref{e.exchange.1neumannderivative0} and \eqref{e.exchange.neumannderivative0} amount to estimating $L_i$ for $i=1, 2, \dots, n$.\\ Let $L_i|_{(y,z)}=\sum_{k=1}^{n+1}L_i^k(y,z)\partial_{x_k}|_{(y,z)}$.
On $\partial_0\Sigma_\alpha$, $\partial_{y_i}|_{(y,0)}=\partial_{x_i}|_{(y,0)}+\partial_{x_i}f(y)\partial_{x_{n+1}}|_{(y,0)}$. Therefore,
\begin{equation}
    [\partial_{x_1}, \partial_{y_i}]|_{(y,0)}=\partial_{x_1}\partial_{x_i}f(y)\partial_{x_{n+1}}|_{(y,0)} = \partial_{x_i}\partial_{x_1}f(y)\partial_{x_{n+1}}|_{(y,0)}.
\end{equation}
If $i=1$, for $(y,0)\in \partial_0\Sigma_\alpha$ we have
\begin{equation}\label{e.Y_1(0)}
|L^k_1(y,0)|\lesssim\Lambda_0\varepsilon
\end{equation}
 for all $k=1, \dots, n+1$ by \eqref{e.na.f}.
On the other hand, if $2\leq i\leq n$, then by the Neumann condition for $f$ on $\del_0\bbb^+_{2R}(0)$, if $(y,0)\in \partial_0\Sigma_\alpha$. 
\begin{equation}\label{e.Y_i(0)}
    L^k_i(y,0)=0
\end{equation}
 for all $k=1, \dots, n+1$.

We now want to estimate $L_i^k(y,z)$ by deriving the ODE it satisfies.

Since $\partial_{y_i}, \partial_z$ are coordinate vectors $[Z, \partial_{y_i}]=0$ everywhere. Moreover, by direct computation $[Z,L_i]=[[Z,\partial_{x_1}], \partial_{y_i}]$, i.e.
\begin{equation}
\begin{split}
    \partial_zL_i^k = [Z, L_i]^k+\sum_{j=1}^{n+1}(\partial_{x_j}Z^k)L_i^j
    =[[Z,\partial_{x_1}], \partial_{y_i}]^k+\sum_{j=1}^{n+1}(\partial_{x_j}Z^k)L_i^j.
\end{split}
\end{equation}
Let us write $\partial_{y_i}|_{(y,z)}=\sum_{m=1}^{n+1}b_i^m(y,z)\partial_{x_m}|_{(y,z)}$, where $b_i^m(y,z)=\partial_{x_i}\Theta_m(z,y)+\partial_{x_i}f(y)\partial_{x_{n+1}}\Theta_m(z,y)$ by \eqref{e.del_y_i}.
We have
\begin{equation}
\begin{split}
    [[Z,\partial_{x_1}], \partial_{y_i}]^k=\sum_{j=1}^{n+1}b_i^j(\partial_{x_j}\partial_{x_1}Z^k)-(\partial_{x_j}b_i^k)(\partial_{x_1}Z^j).
\end{split}
\end{equation}
Hence, the functions $L_i^k$ satisfy an ODE system of the form 
\[\partial_zL_i^k(y,z) = \sum_{j=1}^{n+1}c_j(y,z)L_i^j(y,z)+d_i^j(y,z),\]
where the coefficients satisfy $|c_j(y,z)|\lesssim\varepsilon(1+|z|)$ and $|d_i^j(y,z)|\lesssim \varepsilon(1+|z|)$ by \eqref{e.na.Z}.
By Gronwall's inequality for vector valued functions \cite[Theorem D.2]{Lee}
\begin{equation}
\begin{split}
    |L_i(y,z)|\lesssim \exp(\varepsilon|z|(1+|z|))(|L_i(y,0)|+\varepsilon|z|(1+|z|)).
\end{split}
\end{equation}
Therefore, for $|z|<K|\log\varepsilon|$, $|L_i(y,z)|\lesssim |L_i(y,0)|+\varepsilon|z|(1+|z|)$,
which implies \eqref{e.exchange.1neumannderivative0} and \eqref{e.exchange.neumannderivative0}, by \eqref{e.Y_1(0)} and \eqref{e.Y_i(0)} respectively.
\end{proof}

\subsection{Comparison of twisted Fermi distance functions}
\label{ss.comparison}

As in \cite{wang-wei, wang-weiadv}, we shall need to compare distances with respect to different sheets. The following proposition in the analogue of \cite[Lemma 3.4]{wang-weiadv}, which is itself \cite[Lemma 8.3]{wang-wei}. The only difference in our situation is that we shall use twisted Fermi coordinates as defined in Section \ref{s.tw.Fer.def}. In particular, given distinct connected components $\Sigma_\alpha$, $\Sigma_\beta$ of $\Xi$, $t_{\al}$ and $\Pi_{\al}$ denote the twisted Fermi coordinate time function and projection map with respect to $\Sigma_\alpha$, and let $t_{\beta}$ and $\Pi_{\beta}$ denote the twisted Fermi coordinate time function and projection map with respect to $\Sigma_\beta$. Let $d_{\Sia}$, $d_{\Sigma_\beta}$ denote the intrinsic distances on the hypersurfaces $\Sia$ and $\Sigma_\beta$ respectively, and let $d$ denote the ambient distance, so that $d|_{\Sia}, d|_{\Sigma_\beta}$ are the restrictions of the ambient distance to $\Sia$ and $\Sigma_\beta$ respectively.

\begin{proposition}\label{l.projection.error}
	For any $K>0$, there is a constant $C(K)$ so that for $\alpha\not= \beta$, $x\in \bbb_{2R}^+(0)\times (-2R,2R)$ with $|t_{\alpha}(x)|, |t_{\beta}(x)|\leq K|\log \varepsilon|$, 
	\begin{align}
		\label{e.proj.1}
		d_{\Sigma_ {\beta}}(\Pi_{\beta}\circ\Pi_{\alpha}(x), \Pi_{\beta}(x)) &\leq C(K)\varepsilon^{1/2}|\log \varepsilon|^{3/2}\\
		\label{e.proj.2}
		|t_{\beta}(\Pi_{\alpha}(x))+t_{\alpha}(\Pi_{\beta}(x))|&\leq C(K)\varepsilon^{1/2}|\log \varepsilon|^{3/2}\\
		\label{e.proj.3}
		|t_{\alpha}(x)-t_{\beta}(x)+t_{\beta}(\Pi_{\alpha}(x))|&\leq C(K)\varepsilon^{1/2}|\log \varepsilon|^{3/2}\\
		\label{e.proj.4}
		|t_{\alpha}(x)-t_{\beta}(x)-t_{\alpha}(\Pi_{\beta}(x))|&\leq C(K)\varepsilon^{1/2}|\log \varepsilon|^{3/2}\\
		\label{e.proj.5}
		\left|1-\langle\nabla t_{\alpha}(x),\nabla  t_{\beta}(x)\rangle\right|&\leq C(K)\varepsilon^{1/2}|\log \varepsilon|^{3/2}
	\end{align}
\end{proposition}
\noeqref{e.proj.5}\noeqref{e.proj.4}\noeqref{e.proj.3}\noeqref{e.proj.2}\noeqref{e.proj.1}

\begin{proof}
	The proof is a simple modification of \cite[Lemma 8.3]{wang-wei} (or, more precisely, its Riemannian version in \cite[Lemma B.1]{chodosh-mantoulidis}), thanks to Lemma \ref{l.Fc.vs.tFc.1} and Lemma \ref{l.Fc.vs.tFc.2}, which show that, under the current assumptions, the twisted Fermi coordinate time functions $t_\alpha, t_\beta$ coincide with the corresponding Fermi distance functions $\tilde{t}_\alpha, \tilde{t}_\beta$, and that the twisted Fermi projection maps $\Pi_\alpha, \Pi_\beta$ are comparable to the corresponding Fermi projection maps $\tilde{\Pi}_\alpha, \tilde{\Pi}_\beta$, up to a small error of order $\varepsilon|\log\varepsilon|^3$. 
\end{proof}
Let us now state a few simple but rather useful consequences of Proposition \ref{l.projection.error}.\\
Let $\alpha, \beta, K$ be as in Proposition \ref{l.projection.error}, and let $(y,z)$ denote twisted Fermi coordinates with respect to $\Sigma_\alpha$. Assume $x\in \bbb_{1.5R}^+(0)\times (-2R,2R)$ is such that $t_\alpha, \Pi_\alpha$ and $t_\alpha, \Pi_\alpha$ are defined at $x$, and that $|t_\alpha(x)|, |t_\beta(x)|\leq K|\log\varepsilon|$, so that the conclusions of Proposition \ref{l.projection.error} apply.

\begin{proposition}\label{p.conseq.almost.parallel}
Let $x$ be as above and assume it has twisted Fermi coordinates $(y,z)$ with respect to $\Sigma_\alpha$.
Then the following estimates hold:
\begin{align}\label{eq:tb-tb0}
	&t_\beta(y,z) = t_\beta(y,0)+z+o(\varepsilon^{1/3})\\
\label{e.partial_zt_beta}
    &|\partial_z t_\beta-1|\lesssim\varepsilon^{1/3}\\
\label{0Dalphaztbeta}
	&|\nabla_{\alpha, z} t_\beta|\lesssim \varepsilon^{\frac{1}{5}}\\
\label{nabla2alphaztbeta}
    &|\nabla^2_{\alpha,z}t_\beta|\lesssim \varepsilon
\end{align}
\end{proposition}

Note that the point of this proposition is that in \eqref{eq:tb-tb0}, \eqref{e.partial_zt_beta}, and  \eqref{0Dalphaztbeta}, $z$ and $\partial_z$ represent the twisted Fermi coordinate function and the twisted Fermi coordinate vector field \emph{with respect to} $\Sigma_\alpha$. Note that these are different from the twisted Fermi coordinate function and the twisted Fermi coordinate vector field \emph{with respect to} $\Sigma_\beta$. 

\begin{remark}
The exponents appearing in \eqref{eq:tb-tb0}, \eqref{e.partial_zt_beta}, and \eqref{0Dalphaztbeta} are not sharp. Indeed, as it is apparent from the proof, from Proposition \ref{l.projection.error}, we can obtain
\begin{align*}
    &t_\beta(y,z) = t_\beta(y,0)+z+O(\varepsilon^{1/2}|\log \varepsilon|^{3/2})\\
    &|\partial_z t_\beta-1| \lesssim \varepsilon^{1/2}|\log \varepsilon|^{3/2}\\
    &|\nabla_{\alpha, z} t_\beta| \lesssim \varepsilon^{1/4}|\log \varepsilon|^{3/4}\\
\end{align*}
In particular, some of the exponents appearing e.g. in \eqref{horizontalphicontrolh} and in some other estimates of the later section could be slightly improved. This however does not seem to significantly affect the final results.
\end{remark}

\begin{proof}
First, notice that by \eqref{e.proj.4}, after swapping $\alpha$ and $\beta$, if $x$ has twisted Fermi coordinates $(y,z)$ with respect to $\Sigma_\alpha$, then by a slight abuse of notation, we can write $t_\beta(y,z)=t_\beta(y,0)+z+O(\varepsilon^{1/2}|\log \varepsilon|^{3/2})$.
In particular, for all $\theta\in(0,1/2)$ and all sufficiently small $\varepsilon$ (depending on $K$ and $\theta$)
\begin{equation}
	t_\beta(y,z)=t_\beta(y,0)+z+o(\varepsilon^{1/2-\theta})
\end{equation}
provided $|t_\beta(y,z)|, |z|\leq K |\log \varepsilon|$. This proves \eqref{eq:tb-tb0}.

Next, notice that at $x$, since $|\nabla t_\beta|=1$,
	\begin{equation}
		\begin{split}
			|\partial_z t_\beta-1| &= |\langle\nabla t_\alpha+(\partial_z-\nabla t_\alpha), \nabla t_\beta\rangle-1|=|\langle \nabla t_\alpha, \nabla t_\beta\rangle-1+\langle \partial_z-\nabla t_\alpha, \nabla t_\beta\rangle|\\
			&\leq |\langle \nabla t_\alpha, \nabla t_\beta\rangle-1|+|\langle \partial_z-\nabla t_\alpha, \nabla t_\beta\rangle|\\
			&\lesssim \varepsilon^{1/2}|\log \varepsilon|^{3/2} + \varepsilon(|t_\alpha(x)|+|t_\alpha(x)|^2)\\
			&\lesssim \varepsilon^{1/2}|\log \varepsilon|^{3/2} + \varepsilon (|\log \varepsilon|+|\log \varepsilon|^2)\\
			&\lesssim \varepsilon^{1/2-\theta}
		\end{split}
	\end{equation}
	for any $\theta\in(0,1/2)$, by Lemma \ref{l.Z.Z-tld} and Proposition \ref{l.projection.error}, which implies \eqref{e.partial_zt_beta}.

    Similarly,
	\begin{equation}
		|\nabla_{\alpha, z}t_\beta|\lesssim|\nabla t_\beta-\langle\nabla t_\beta, \nabla t_\alpha\rangle \nabla t_\alpha| \!=\!\sqrt{1-\langle \nabla t_\alpha, \nabla t_\beta \rangle^2}\leq \sqrt{2(1-\langle \nabla t_\beta, \nabla t_\alpha\rangle)}\lesssim \varepsilon^{1/4}|\log \varepsilon|^{3/4}
	\end{equation}
	by \eqref{e.proj.5}. In particular, $|\nabla_{\alpha, z} t_\beta|\lesssim \varepsilon^{1/4-\theta}$ for any $\theta\in(0,1/4)$, thus proving \eqref{0Dalphaztbeta}.

    Next, notice that $\nabla_{\alpha, z}^2t_\beta=\nabla^2 t_\beta+\langle \nabla t_\beta, \nabla t_\alpha\rangle \sff_{\al, z}$. Since $t_\beta=\tilde t_\beta$, then $\nabla^2 t_\beta$ is precisely the second fundamental form of $\Sigma_{\beta, t_\beta}$. Recall that by \eqref{e.na.f}, $|\sff_{\alpha, 0}|=O(\varepsilon)$ and $|\sff_{\beta, 0}|=O(\varepsilon)|$. Since we are assuming $|t_\alpha(x)|, |t_\beta(x)|\leq K|\log \varepsilon|$, by \eqref{e.sff.approx}, 
	\[|\sff_{\alpha,z}-\sff_{\alpha, 0}| = |\sff_{\alpha, t_{\alpha}(x)}-\sff_{\alpha,0}|\lesssim \varepsilon^2|\log\varepsilon|^4,\]
	so that $|\sff_{\alpha, z}|=O(\varepsilon)$, and, similarly, $|\nabla^2 t_\beta|=|\sff_{\beta,t_\beta(x)}|=O(\varepsilon)$. Hence, by \eqref{e.proj.5}, \eqref{nabla2alphaztbeta} holds.
\end{proof}

As mentioned in the proof of Proposition \ref{l.projection.error}, by \cite[Lemma 2.2]{wang-weiadv}, $d_{\Sigma_\alpha}, d_{\Sigma_\beta}$ are comparable to $d|_{\Sigma_\alpha}, d|_{\Sigma_\beta}$ respectively, up to a universal constant depending on $\eta_0$ and $\Lambda_0$. By Lemma \ref{l.t.Pi}, $\Pi_\alpha|_{\Sigma_\beta}$ and $\Pi_\beta|_{\Sigma_\alpha}$ are uniformly bi-Lipschitz in the regions where $|t_\alpha|, |t_\beta|\leq \delta\sqrt{R}$.
    
The following lemma is an easy consequence of the results of this section.
\begin{lemma}\label{l.DPi.beta}
Let $K>0$ be fixed. Let $(y, z)$ be twisted Fermi coordinates with respect to $\Sia$, and let $\Pi_\beta$ denote the twisted Fermi coordinate projection map with respect to $\Sigma_\beta$. 
Then,
\begin{equation}\label{e.Dpi.bound}
    |D\Pi_\beta(\partial_{y_i})|=O(1)
\end{equation}
for $i=1, \dots, n$, and
\begin{equation}\label{e.Dzpi.bound}
    |D\Pi_\beta(\partial_{z})|\lesssim \varepsilon^{\frac{1}{5}}
\end{equation}
uniformly in the region where $|t_\alpha|, |t_\beta|\leq K|\log \varepsilon|$.
\end{lemma}
\begin{proof}
Equation \eqref{e.Dpi.bound} follows from Lemma \ref{l.t.Pi}.  For
\eqref{e.Dzpi.bound}, write $Z_\gamma$ and $\widetilde Z_\gamma=\nabla t_\gamma$ for the twisted
and ordinary Fermi fields associated with $\Sigma_\gamma$.  The projection $\Pi_\beta$ is constant
along the flow of $Z_\beta$, and hence $D\Pi_\beta(Z_\beta)=0$.  Since
$\partial_z=Z_\alpha$ in the $\alpha$-coordinates,
\[
 D\Pi_\beta(\partial_z)=D\Pi_\beta(Z_\alpha-Z_\beta).
\]
Moreover,
\[
 Z_\alpha-Z_\beta
 =(\widetilde Z_\alpha-\widetilde Z_\beta)
 +(Z_\alpha-\widetilde Z_\alpha)-(Z_\beta-\widetilde Z_\beta).
\]
The first term is $O(\varepsilon^{1/4}|\log\varepsilon|^{3/4})$ by
\eqref{e.proj.5}, while the last two are $O(\varepsilon(|\log\varepsilon|+|\log\varepsilon|^2))$
by Lemma \ref{l.Z.Z-tld}.  Combining these estimates with the uniform bound for $D\Pi_\beta$
gives \eqref{e.Dzpi.bound}.
\end{proof}

\subsection{Comparison of boundary derivatives}\label{ss.comparison.boundary.derivatives}
Finally, let us consider how twisted Fermi coordinates with respect to different sheets $\Sigma_{\alpha}$, $\Sigma_\beta$ behave along the boundary.

Let $K>0$ be fixed, and let $x\in \del_0\bbb^+_{1.6R}(0)\times (-2R,2R)$ be such that
$|t_{\alpha}(x)|, |t_{\beta}(x)|\leq K|\log \varepsilon|$. Let $(y,z)$ denote twisted Fermi coordinates with respect to $\Sigma_\alpha$.

Recall from Lemma \ref{l.dz(nu)}, $J_\alpha:\partial_0\bbb^+_{1.6R}\times(-2R, 2R)\to \mathbb{R}$ is defined by $J_\alpha=\partial_{x_1}t_\alpha$. Similarly, one can define $J_\beta$ by $J_\beta=\partial_{x_1}t_\beta$. Now, for $\be\neq \al$, let \(\mathsf{J}_{\al \be}\) denote the restriction of \(\del_{y_1}t_{\be}\) to \(\del_0\bbb^+_{1.6R}(0)\times (-2R,2R)\). We have the following estimate.

\begin{lemma}\label{l.del.y_1.t_be 2} There holds
\begin{equation}
	|\na^m \mathsf{J}_{\al \be}|\lesssim \ve(1+|z|^3+ |t_{\be}| + t_{\be}^2) \; \text{ on } \del_0\{x:|t_{\al}(x)|, |t_{\be}(x)| \leq K |\log \ve|\}\quad \forall  m \ge 0.
\end{equation}	
\end{lemma}
\begin{proof}
	We note that
	\begin{equation}
		\mathsf{J}_{\al \be} = \langle \na t_{\be} , \del_{y_1} - \del_{x_1} \rangle + J_{\be}
	\end{equation}
	Since $\del_{x_1}=\del_{y_1}$ along $\del_0\Sia$, it follows from Lemma \ref{l.na.y_i} that 
	\begin{equation}
		|\na^m (\del_{x_1}-\del_{y_1})|\lesssim \ve(1 +|z|^3)\;\text{ on }\del_0\{|z|\leq K|\log \ve |\}\quad \forall m\ge 0.
	\end{equation}
Using Lemma \ref{l.dz(nu)} in the form $J_\beta=w_\beta t_\beta+\widetilde J_\beta$, the apparently cubic part of the crude bound is absorbed by
$\varepsilon^2|t_\beta|^2(1+|t_\beta|^3)\leq C_K\varepsilon$ in the logarithmic strip.  Together with the preceding estimate and Lemma \ref{l.t.Pi}, this gives the stated bound for $\mathsf J_{\alpha\beta}$.
\end{proof}

\subsection{Separation of layers}
Analogously to \cite[Section 3.2]{wang-weiadv}, given a point $y\in \Sigma_\alpha\cap(\bbb^+_{1.6R}(0)\times (-2R,2R))$, we let 
\begin{equation} \label{definition D alpha}
	D_\alpha(y)\coloneqq \min_{\beta \neq \alpha}|t_\beta(y)|.
\end{equation}
Then, for $\lambda\in \mathbb{R}$, define the set
\begin{equation}\label{Mlambdaalpha}
	\mathcal{M}^\lambda_\alpha\coloneqq \left\{x\in \bbb^+_{1.6R}(0)\times (-2R,2R): |t_\alpha(x)|<\min\{|t_{\alpha-1}(x)|+\lambda, |t_{\alpha+1}(x)|+\lambda\}\right\}.
\end{equation}
The following simple lemma will be used throughout the rest of the paper.

\begin{lemma}[c.f. {\cite[Lem 3.6]{wang-weiadv}}] \label{l.sum.of.exp}
	For any $y\in \Sigma_{\alpha}$ and $\sigma > 0$, we have 
	\begin{equation}\label{exponential sum1}	
		\sum_{\beta\not=\alpha}e^{- \sigma|t_{\beta}(y)|}\lesssim_{K, \sigma} e^{- \sigma D_{\alpha}(y)},  
	\end{equation}
	uniformly with respect to $y$ and $\alpha$. For  $x\in \mathcal{M}^{4}_{\alpha}$ with $|t_\alpha(x)|\leq K|\log\varepsilon|$, we have 
	\begin{equation}\label{exponential sum2}
		\sum_{\beta\not=\alpha}e^{-\sigma|t_{\beta}(x)|}\lesssim_{K, \sigma} e^{-\sigma|t_{\alpha+1}(x)|}+e^{-\sigma|t_{\alpha-1}(x)|}
	\end{equation}
\end{lemma}

In order to prove Lemma \ref{l.sum.of.exp}, we need the following two preliminary results.

\begin{lemma}\label{l.1D.soln}
	Let $\ve_i>0$ be a \seq\ which converges to $0$, $R_i=\ve_i^{-1}$, $u_i(x):=u_{\ve_i}(\ve_ix)$, $x_i\in \{u_i=0\}\cap (\bbb^+_{R_i}(0)\times (-R_i, R_i))$, $\Si_{x_i}$ be the connected component of $\{u_i=0\}$ in $\bbb^+_{R_i}(0)\times (-R_i, R_i)$, which contains $x_i$. \sps\ in  twisted Fermi coordinates \wrt\ $\Si_{x_i}$, $x_i=(y_i,0)$, $\tilde{u}_i(y,z):= u_i(y_i+y,z)$ and $u_i(y,z)$ has the same sign as $z$. Then, up to a subsequence, $\tilde{u}_i(y,z)$ converges to $\bbh(z)$ in $C^3_{\text{loc}}$.
\end{lemma}
\begin{proof}	
The proof is similar to \cite[Proof of Lemma 2.1]{wang-weiadv}. \sps\ in  twisted Fermi coordinates \wrt\ $\Si_{x_i}$, $x_i=(y_{i,1},\dots, y_{i,n},0)$. When $\liminf_{i\ra \infty}y_{i,1}=\infty$, $\tilde{u}_i$ is a stable solution of \eqref{e.ac.1} on $B_{y_{i,1}}(0)$, $\tui(y_1,\dots,y_n,0)=0$ and $\del_z\tui(y_1,\dots,y_n,0)\geq 0$. By the elliptic estimates, $\tui$ is uniformly bounded in $C^{4}_{\text{loc}}(\bbr^{n+1})$. \tf, \tes\ $\tuin$, a stable solution of \eqref{e.ac.1} on $\bbr^{n+1}$, \st\ up to a subsequence, $\tui\ra\tuin$ in $C^{3}_{\text{loc}}(\bbr^{n+1})$, $\tuin(y_1,\dots,y_n,0)=0$ and $\del_z\tuin(y_1,\dots,y_n,0)\geq 0$. As explained in \cite[Proof of Lemma 2.1]{wang-weiadv}, since $u_{\ve}$ satisfies \eqref{e.esff.bd}, there holds $\mD{\cA(\tuin)}=0$ on $\{\mD{\tuin}<1-b'\}\cap \{\mD{\na \tuin}\neq 0\}$. By the unique continuation principle, this implies $\tuin$ is a lift of $\Hb$; hence $\tuin(y_1,\dots,y_n,z)=\bbh(z)$.
	
	If $\liminf_{i\ra \infty}y_{i,1}<\infty$, then, possibly after passing to a subsequence, we can assume that $y_{i,1}$ is uniformly bounded. Let $x'_i=(y'_i,0)=(0,y_{i,2},\dots, y_{i,n},0)$ (which is a point in $\del_0\bbb^+_{R_i}(0)\times (-R_i, R_i)$). We define $\hat{u}_i(y,z)=u_i(y'_i+y,z)$. Then $\hat{u}_i$ is a stable solution of \eqref{e.ac.1} on $B^+_{R_i}(0)$ satisfying the vanishing Neumann boundary condition along $\del_0B^+_{R_i}(0)$, $\hat{u}_i(y_1,\dots,y_n,0)=0$ and $\del_z\hat{u}_i(y_1,\dots,y_n,0)\geq 0$. $\hat{u}_i$ is uniformly bounded in $C^{4}_{\text{loc}}(\bbr^{n+1}_+)$. Hence \tes\ $\tuin$, a stable solution of \eqref{e.ac.1} on $\bbr^{n+1}_+$ satisfying the vanishing Neumann boundary condition along $\del \bbr^{n+1}_+$, \st\ up to a subsequence, $\hat{u}_i\ra\tuin$ in $C^{3}_{\text{loc}}(\bbr^{n+1}_+)$, $\tuin(y_1,\dots,y_n,0)=0$ and $\del_z\tuin(y_1,\dots,y_n,0)\geq 0$. We extend $\tuin$ on $\bbr^{n+1}$ by even reflection (as in \eqref{e.even.refl}). Then $\tuin$ is a stable solution of \eqref{e.ac.1}. Since $u_{\ve}$ satisfies \eqref{e.esff.bd}, arguing as in \cite[Proof of Lemma 2.1]{wang-weiadv}, we conclude that $\mD{\cA(\tuin)}=0$ on $\{\mD{\tuin}<1-b'\}\cap \{\mD{\na \tuin}\neq 0\}$. By the unique continuation principle, this implies $\tuin$ is a lift of $\Hb$; hence $\tuin(y_1,\dots,y_n,z)=\bbh(z)$. Thus $\hat{u}_i(y,z)\ra\Hb(z)$ in $C^{3}_{\text{loc}}(\bbr^{n+1}_+)$. Since $y_{i,1}$ is uniformly bounded, this implies $\tui(y,z)\ra\Hb(z)$ in $C^{3}_{\text{loc}}$.	
\end{proof}
The next lemma then follows directly. The proof is the same as \cite[proof of eq. (2.1)]{wang-weiadv}.
\begin{lemma}\label{l.D_al.infty}
	As $\ve\ra 0$, $\inf\limits_{\al}\inf\{D_{\al}(y):y\in\Sia\cap (\bbb^+_R(0)\times (-R,R)) \} \ra +\infty$.
\end{lemma}
We can finally give the proof of Lemma \ref{l.sum.of.exp}.

\begin{proof}[Proof of Lemma \ref{l.sum.of.exp}]
	The proof of \eqref{exponential sum1} is identical to the proof of \cite[Lemma 3.6]{wang-weiadv}, with Lemma \ref{l.D_al.infty} replacing \cite[Lemma 2.2]{wang-weiadv}, and Proposition \ref{l.projection.error} replacing \cite[Lemma 3.4]{wang-weiadv}. Then, \eqref{exponential sum2} follows directly from \eqref{exponential sum1}.
\end{proof}

\section{An optimal approximate solution}\label{s.approx.toda}
In this section, we aim to find the optimal approximate solution of $u$ by modifying the sign alternating heteroclinic approximation $\mathbb{H}(y, z; h)$ in \eqref{Hb(y, z; h)} to vanishing Neumann boundary adapted sign alternating approximation  $g(y, z; h)$ defined in \eqref{def_g=H+G},  and proving the existence of optimal layer shift $h$ in Proposition \ref{p.optimal_h_shift}.

\subsection{Optimal approximate solution ansatz}
We fix a smooth even cut-off function $\zeta\in C^{\infty}_{c}(\mathbb{R})$ with $\zeta=1$ on $(-1, 1)$, $0\leq \zeta\leq 1$, $\spt \zeta\subset(-2,2)$, and $|\zeta'|+|\zeta''|\leq 16$. Define 
\begin{equation}\label{e.zeta_ep}
    \zeta_\varepsilon(t)\coloneqq\zeta\left(\frac{t}{4|\log\varepsilon|}\right)
\end{equation}
and, similarly, $\zeta'_\varepsilon(t)\coloneqq\zeta'((4|\log\varepsilon|)^{-1}t)$, $\zeta''_\varepsilon(t)\coloneqq\zeta''((4|\log\varepsilon|)^{-1}t)$.
Then, let 
\begin{equation}\label{cut-H-def}
\overline{\Hb}(t)=\zeta_\varepsilon(t)\Hb(t)+[1-\zeta_\varepsilon(t)]\textrm{sgn}(t),\quad t\in\bbr.
\end{equation}
Then $\overline{\Hb}=1$ on $(8|\log \varepsilon|, \infty)$ and $\overline{\Hb}=-1$ on $(-\infty, -8|\log \varepsilon|)$. Notice that 
\begin{equation}\label{barHb' equation}
\overline{\Hb}''=W'(\overline{\Hb})+\bar{\xi}_1,\quad
\overline{\Hb}'''=W''(\overline{\Hb})\overline{\Hb}'+\bar{\xi}'_1,
\end{equation}
where $\textrm{spt}(\bar{\xi}_1)\subset (4|\log\varepsilon|, 8|\log\varepsilon|) \cup (-8|\log\varepsilon|, -4|\log\varepsilon|) $ and 
\begin{equation}\label{xiestimates}
    \sum_{k=0}^m|\bar{\xi}^{(k)}_1(t)|\lesssim_m e^{-\frac{1}{4}|t|}\varepsilon^3
\end{equation}
Recall the definition of $\sigma_0$ in Appendix \ref{s.1d.sol}, and notice that by \eqref{xiestimates} we have
\begin{equation}\label{H'2integral}
\int_{-\infty}^{\infty} \overline{\Hb}'(t)^2dt=\si_0+O(\ve^3).
\end{equation}
We define the approximate solution $\Hb(y, z; h)\in C^{3}(M)$ in the following way. 
First, recall from Section \ref{ss.twisted.Fermi.setup} the union $\Xi$ of all the connected components of $\{u=0\}\cap (\bbb^+_{2R}(0)\times (-2R,2R))$ which have non-empty intersections with $\bbb^+_{2R}(0)\times (-1.5R,1.5R)$
can be written as a union graphical sheets
\begin{equation}
\Xi=\bigcup_{\al}\Si_{\al},
\end{equation} 
as in \eqref{e.Si.al}.
Let $t_\al$ denote the twisted Fermi coordinate time function with respect to $\Sia$ as in Definition \ref{d.twisted.time.projection}, and consider the sets $\mathcal{M}_{\alpha}^{\lambda}$ as defined in \eqref{Mlambdaalpha}. 
Without loss of generality, we can assume that near $\Sia$, $u$ has the same sign as $(-1)^{\alpha}t_{\alpha}$. Given a vector of functions $h=(h_{\al})$, where $h_{\al}: \Si_{\al} \to \bbr$, in each $\mathcal{M}^{0}_{\alpha}$ (which is a neighbourhood of $\Sigma_{\alpha}$), let $\Hb(y, z; h)$ be defined by
\begin{equation}\label{Hb(y, z; h)}
\Hb(y, z; h)=\sum_{\beta}\left(\Hb_{\beta}(y, z; h_{\be})+\textrm{sgn}(\beta - \alpha)(-1)^{\beta}\right),
\end{equation}
where $(y, z)\in \mathcal{M}^{0}_{\alpha}$ denotes the twisted Fermi coordinates with respect to $\Sia$ and 
\begin{equation}\label{def_Hbeta}
\Hb_{\beta}(y, z; h_{\be})=\overline{\Hb}\left((-1)^{\beta}(t_{\beta}(y, z)-h_{\beta}(\Pi_{\beta}(y, z)))\right).
\end{equation}
In particular $\Hb_{\alpha}(y, z; h_{\al})=\overline{\Hb}((-1)^{\alpha}(z-h_{\alpha}(y)))$. 

We shall use the following notation about for derivatives of $\Hb_{\beta}$:
\begin{equation}\label{def_Hkbeta}
\Hb_{\beta}^{(k)}(y, z; h_\beta)=\overline{\Hb}^{(k)}\left((-1)^{\beta}(t_{\beta}(y, z)-h_{\beta}(\Pi_{\beta}(y, z)))\right),\ \;  k\geq 1.
\end{equation}

To derive the desired curvature estimates, we need to modify the approximate solution $\mathbb{H}(y, z; h)$ so that the corrected approximate solution satisfies good estimates on its Neumann derivative along the boundary of the domain.

By Proposition \ref{p.lin.ac}, we can define a function $\bgah \in C^{\infty}(\Si_{\al}\times \bbr)$ by solving the following mixed boundary value problem (as the boundary datum is odd in $\tau$):
\begin{align}\label{G-equationwithout cutoff}
\begin{cases}\Delta_{\al, 0} \bgah(y,\ta)+\del_{\ta \ta}\bgah(y,\ta)-W''(\Hb(\ta))\bgah(y,\ta)=0 & \text{ on }\Sigma_{\alpha}\times \mathbb{R};\\
\del_{y_1}\bgah(y,\ta)+w_{\al}(y)\bH'(\ta)\ta=0 & \text{ along }\del_0\Sigma_{\alpha}\times \mathbb{R};\\
\bgah=0 & \text{ along }\del_{+}\Sigma_{\alpha}\times \mathbb{R};\\
\int_{-\infty}^{\infty} \bgah(y,\ta)\Hb'(\ta)d\ta=0, 
\end{cases}
\end{align}
where $w_{\al}$ is as in Lemma \ref{l.dz(nu)}. 

Since, by Lemma \ref{p.lin.ac}, the solution to this problem is unique, it must satisfy the following condition
\begin{equation}\label{e.G_alpha.odd}
-\bgah(y,-z)=\bgah(y,z),
\end{equation}
i.e. $\bgah$ is odd in the $z$ variable.
Also, by \eqref{e.lin.ac.est} in Lemma \ref{p.lin.ac}, and, for every $\sigma\in(0,1)$ and every fixed non-negative
integer $m$,
\begin{equation}\label{e.G.est}
 \sum_{\ell=0}^{m}|\nabla^\ell\hat{\mathbb G}_\alpha(y,\tau)|
 \leq C_{m,\sigma}\varepsilon e^{-\sigma|\tau|}.
\end{equation}
The constants are uniform in $\alpha$ and $\varepsilon$.

As we did in the case of the one-dimensional heteroclinic solution, we need to evenly cut off the correction function $\bgah$ in order for it to have compact support. Therefore, let
\begin{equation}\label{eq:cutoffG}
	\bgab(y,z)=\zeta_\varepsilon(z)\bgah(y,z)
\end{equation}
for $z\in \bbr$.
Since the cut-off function $\zeta$ was assumed to be even, it is easy to see that $\overline{\Gb}_{\alpha}$ is an odd function in the $z$-variable. 

Notice that $\bgah$ satisfies 
\begin{align}\label{G-equation}
\begin{cases}\Delta_{\al, 0} \bgab(y,z)+\del_{z z}\bgab(y,z)-W''(\bH(z))\bgab(y,z) = \bar{\xi}_{2, \alpha}(y,z), & \text{ on }\Sigma_{\alpha}\times \mathbb{R};\\
\del_{y_1}\overline{\Gb}_{\alpha}(y, z)+w_{\al}(y)\bH'(z)z=\delta_\alpha(y,z) & \text{ along }\del_0\Sigma_{\alpha}\times \mathbb{R};\\
\overline{\Gb}_{\alpha}=0 & \text{ along }\del_{+}\Sigma_{\alpha}\times \mathbb{R};\\
\int_{-\infty}^{\infty} \overline{\Gb}_{\alpha}(y,z)\overline{\Hb}'(z)dz=0, 
\end{cases}
\end{align}
where
\begin{align}
\bar\xi_{2,\alpha}(y,z)={}&
 \frac{1}{2|\log\varepsilon|}\zeta_\varepsilon'(z)
   \partial_z\bgah(y,z)
 +\frac{1}{16|\log\varepsilon|^2}\zeta_\varepsilon''(z)\bgah(y,z)\notag\\
&+\bigl(W''(\mathbb H(z))-W''(\bH(z))\bigr)\bgab(y,z),\label{e.cutoff.residual.correct}
\end{align}
and
\begin{equation}\label{e.delta.error}
 \delta_\alpha(y,z)=(1-\zeta_\varepsilon(z))w_\alpha(y)\bH'(z)z.
\end{equation}
Notice that the last property in \eqref{G-equation} is a consequence of oddness in the $z$-variable.
By definition of the cutoff function, $\bar{\xi}_2$ is supported in $\Sigma_\alpha\times\left((4|\log\varepsilon|, 8|\log\varepsilon|) \cup (-8|\log\varepsilon|, -4|\log\varepsilon|)\right)$. Similarly, $\delta_\alpha$
is supported in $\partial_0\Sigma_\alpha\times\left((4|\log\varepsilon|, 8|\log\varepsilon|) \cup (-8|\log\varepsilon|, -4|\log\varepsilon|)\right)$. Moreover, by \eqref{e.G.est}

\begin{equation}\label{e.xi.delta.est}
    \sum_{k=0}^m|\na^k \bar{\xi}_{2, \alpha}(y,z)|+|\na^k \delta_{\alpha}(y,z)| \lesssim_m e^{-\frac{1}{4}|z|}\varepsilon^{3}.
\end{equation}

Given $h_{\al} : \Sia \to \bbr$, we define 
\begin{equation}
    \bga(y,z; h_{\al})\coloneqq\bar{\bg}_{\al}(y, (-1)^{\al}(z-h_{\al}(y))),
\end{equation}
and, given a vector of functions $h=(h_{\al})$, where $h_{\al}: \Sia \to \bbr$, we define
\begin{equation}
\bg(-;h):=\sum\limits_{\beta}\bg_{\beta}(-;h_{\beta}),    
\end{equation}
where $\bg_\beta(-;h_\beta)$ is defined analogously to \eqref{def_Hbeta}.
For simplicity, we will use the following notation
\begin{align}\label{sign}
    \bg_{\al}^{(k)}(y,z; h_{\al}) & \coloneqq(\del^{(k)}_{z}\bar{\bg}_{\al})(y, (-1)^{\al}(z-h_{\al}(y))), \\  
    \bg_{\al,i}(y,z; h_{\al}) & \coloneqq(\del_{y_i}\bar{\bg}_{\al})(y, (-1)^{\al}(z-h_{\al}(y))), \\ 
    \bg_{\al,ij}(y,z; h_{\al}) & \coloneqq (\del_{y_iy_j}\bar{\bg}_{\al})(y, (-1)^{\al}(z-h_{\al}(y))),
\end{align}
for $1\leq i,j\leq n$. When $k=1$, we shall write $\bga'$ instead of $\bga^{(1)}$.

These can also be combined, as in $\bg_{\al,ij}'(y,z; h_{\al}) = (\del_{y_iy_jz}\bar{\bg}_{\al})(y,(-1)^{\al}(z-h_{\al}(y)))$ etc.
Similarly, let $\xi_{2,\alpha}(y,z;h_\alpha)\coloneqq
\bar\xi_{2,\alpha}(y, (-1)^\alpha(z-h_\al(y)))$, $\xi'_{2,\alpha}(y,z;h_\al)\coloneqq(\partial_z\bar\xi_{2, \alpha})(y, (-1)^\al(z-h_\al(y)))$ etc.
Finally, let $\delta'_\alpha(y,z)\coloneqq (\partial_z\delta_\al)(y,z)$, $\delta_{\alpha,i}(y,z)\coloneqq (\partial_{y_i}\delta_\al)(y,z)$, etc.

For $k\geq 0$, we define
\begin{equation}\label{def_gk=H+G}
g_{\be}^{(k)}(y, z;h_{\be}) =\Hb_{\be}^{(k)}(y, z;h_{\be})+ \bg_{\be}^{(k)}(y, z;h_{\be})    
\end{equation}
and
\begin{equation}\label{def_g=H+G}
    g(y, z; h)=\Hb(y, z; h)+\bg(y, z; h).
\end{equation}

Clearly, an easy computation gives the following.
\begin{lemma}\label{lem:maing}We have
\begin{align}\label{eq:maing}
(\Delta_{\alpha,0}+\partial_z^2-W''(\Hb_\alpha))\Gb_\alpha
 ={}&|\nabla_{\alpha,0}h_\alpha|^2\Gb_\alpha^{(2)}
 -(-1)^\alpha(\Delta_{\alpha,0}h_\alpha)\Gb_\alpha^{(1)}\notag\\
 &-2(-1)^\alpha\sum_{i,j}g_\alpha^{ij}(y,0)
       \Gb'_{\alpha,i}\,\partial_{y_j}h_\alpha+\xi_{2,\alpha},
\end{align}
\begin{align}\label{eq:main}
 (\Delta_{\alpha,0}+\partial_z^2-W''(\Hb_\alpha))\Gb_\alpha^{(1)}
 ={}&|\nabla_{\alpha,0}h_\alpha|^2\Gb_\alpha^{(3)}
 -(-1)^\alpha(\Delta_{\alpha,0}h_\alpha)\Gb_\alpha^{(2)}\notag\\
 &-2(-1)^\alpha\sum_{i,j}g_\alpha^{ij}(y,0)
       \Gb''_{\alpha,i}\,\partial_{y_j}h_\alpha+W'''(\Hb_\alpha)\Hb_\alpha^{(1)}\Gb_\alpha+\xi'_{2,\alpha}.
\end{align}
\begin{align}
    \partial_{y_1}\Gb_{\alpha}(y, z; 0)+(-1)^{\alpha}w_{\alpha}(y)\Hb'_{\alpha}(y, z; 0)z &= \delta_{\alpha}\label{Neumanny1 G0}\\
\partial_{y_1 y_i}\Gb_{\alpha}(y, z; 0)+(-1)^{\alpha} \partial_{y_i}w_{\alpha}(y)\Hb'_{\alpha}(y, z; 0)z &= \delta_{\alpha, i} \quad (2\leq i\leq n). \label{Neumanny1yi G0}
\end{align}
\end{lemma}

Finally, we shall need some elementary estimates, which we record below in Lemma \ref{l.interaction.int} and Lemma \ref{l.interaction.int.2}. The proofs, which follow more or less directly from the results of Section \ref{ss.comparison} and the properties of the heteroclinic solution $\mathbb{H}$, are in Appendix \ref{s.proofs.integrals}.

\begin{lemma}\label{l.interaction.int}
Let $K>0$ be fixed, let $\beta\neq\alpha$ satisfy
$|t_\beta(y,0)|\leq K|\log\varepsilon|$, and let
$\sigma,\varsigma\geq0$ with $\sigma+\varsigma>0$.  For every $k\in\mathbb N$,
\begin{align}\label{equation 3 in l.interaction.int}
 \int_{-K|\log\varepsilon|}^{K|\log\varepsilon|}
 |z|^ke^{-\sigma|t_\beta(y,z)|-\varsigma|z|}\,dz
 &\lesssim (1+|t_\beta(y,0)|)^{k+1}
 e^{-\min\{\sigma,\varsigma\}|t_\beta(y,0)|},\\
\label{equation 1 in l.interaction.int}
 \int_{-K|\log\varepsilon|}^{K|\log\varepsilon|}
 |t_\beta(y,z)|^ke^{-\sigma|t_\beta(y,z)|-\varsigma|z|}\,dz
 &\lesssim (1+|t_\beta(y,0)|)^{k+1}
 e^{-\min\{\sigma,\varsigma\}|t_\beta(y,0)|}.
\end{align}
\end{lemma}

\begin{lemma}\label{l.interaction.int.2}
Let $K>0$ be fixed.  If $\beta\neq\alpha$ and
$|t_\beta(y,0)|\leq K|\log\varepsilon|$, then
\begin{equation}\label{equation 2 in l.interaction.int}
 \int_{\mathbb R}
 \left|\bH(t_\beta(y,z))+\operatorname{sgn}(\beta-\alpha)\right|
 |\bH'(z)|\,dz
 \lesssim (1+|t_\beta(y,0)|)e^{-|t_\beta(y,0)|}.
\end{equation}
Consequently, for every $\sigma\in(0,1)$,
\begin{equation}\label{e.interaction.sigma.correct}
 \int_{\mathbb R}
 \left|\bH(t_\beta(y,z))+\operatorname{sgn}(\beta-\alpha)\right|
 |\bH'(z)|\,dz
 \lesssim_\sigma e^{-\sigma|t_\beta(y,0)|}.
\end{equation}
\end{lemma}

\newcommand{\tty}{\mathtt{y}}
\subsection{Notation for Proposition \ref{p.optimal_h_shift}}
We are now going to introduce some notation which will be used in the statement and the proof of Proposition \ref{p.optimal_h_shift}. 

\subsubsection{Coordinate balls} Let us fix an index $\alpha$. Recall that $\Sia$ is the graph of a smooth function $f_\al:\bbb^+_{2R}(0)\to(-1.6R, 1.6R)$ which satisfies \eqref{e.na.f}. Therefore,
\begin{equation}\label{e.Sia.param}
    \Sia=\{(y, f_\al(y)): y\in\bbb^+_{2R}(0)\}\ \text{ and }\ \partial_0\Sia=\{(y, f_\al(y)): y\in\partial_0\bbb^+_{2R}(0)\}.
\end{equation}
Let us now define the following \textit{coordinate balls}.
\begin{definition}\label{def.notation.balls}
For $\bar x=(\bar y,f_\alpha(\bar y))\in\Sigma_\alpha$ set
\[
 B_r^\alpha(\bar x)=\{(y,f_\alpha(y)):y\in\mathbb B_{2R}^+(0),\ |y-\bar y|<r\}.
\]
For $\bar x\in\partial_0\Sigma_\alpha$ set
\[
 B_r^{0,\alpha}(\bar x)=\{(y,f_\alpha(y)):y\in\partial_0\mathbb B_{2R}^+(0),\ |y-\bar y|<r\}.
\]
Their relative closures are $\bar B_r^\alpha(\bar x)$ and $\bar B_r^{0,\alpha}(\bar x)$.
The notation with centre $\bar y$ has the same meaning under the graph parametrization.\end{definition}

\begin{remark}\label{r.ball.comparison}
    Notice that, by \eqref{e.rescaled.bound.gEucl} and \eqref{e.na.f}, it is easy to see that for all small enough $\eta_0$ (independent of $\varepsilon$), the coordinate balls $B^\alpha_r(\cdot)$,  $B^{0,\alpha}_r(\cdot)$ are uniformly comparable to the (intrinsic) geodesic balls in $\Sia$ and $\partial_0\Sia$ respectively. In particular, it is easy to check that there are constants $c_0(n), c_1(n)$ such that for all sufficiently small $\eta_0$, 
    \begin{equation}
        B^\alpha_{(1-c_1(n)\eta_0)r}(x)\subset B^{\Sia}_r(x)\subset B^\alpha_{(1+c_1(n)\eta_0)r}(x) \ \text{ for all } x\in \Sia,
    \end{equation}
    and
    \begin{equation}
        B^{0,\alpha}_{(1-c_0(n)\eta_0)r}(x)\subset B^{\partial_0\Sia}_r(x)\subset B^{0,\alpha}_{(1+c_0(n)\eta_0)r}(x) \ \text{ for all } x\in \partial_0\Sia,
    \end{equation}
    where $B^{\Sia}_r(x)$ (resp. $B^{\partial_0\Sia}_r(x)$) denotes the open geodesic ball of radius $r$ centred at $x$ in $\Sia$ (resp. $\partial_0\Sia$).
\end{remark}

\subsubsection{H{\"o}lder seminorms on coordinate balls}
Let $k\geq 0$, $\theta\in(0,1]$, and let $\vp\in C^{k, \theta}(\Sia)$ (resp. $C^{k, \theta}(\partial_0\Sia)$). Let $\bar x=(\bar y, f(\bar y))\in \Sia$ (resp. $\partial_0\Sia$), and let $x=(y, f(y))$ denote another point in $\Sia$ (resp. $\partial_0\Sia$). 

Let $j\leq k$, and a multi-index \(I=(i_1,\dots,i_j)\) (where each \(i_l\in \{1,\dots,n\}\), resp. \(i_l\in \{2,\dots,n\}\)) with  \(|I|=j\). Define \(\vp_I = \del_{y_{i_1} \dots y_{i_j}} \vp\), and the following quantities:
\begin{align}\label{j-theta norm}
\begin{split}
 &|\vp|_0(\bar x)=|\vp(\bar x)|,\\
 &|\vp|_{0,\theta}(\bar x)=
 \sup\left\{\frac{|\vp(x)-\vp(\bar x)|}{|y-\bar y|^\theta}:
 x\in\bar B_1^\alpha(\bar x),\ x\neq\bar x\right\}
 \quad\text{(respectively, $\bar B_1^{0,\alpha}(\bar x)$)},\\
 &|\vp|_j(\bar x)=\max_{|I|=j}|\vp_I(\bar x)|,
 \qquad |\vp|_{j,\theta}(\bar x)=
 \max_{|I|=j}|\vp_I|_{0,\theta}(\bar x),\\
 &|\vp|'_j(\bar x)=\sup\left\{|\vp|_j(x):x\in\bar B_1^\alpha(\bar x)\right\}
 \quad\text{(respectively, $\bar B_1^{0,\alpha}(\bar x)$)}.
\end{split}
\end{align}

We note the following elementary inequalities 
\begin{equation}\label{e.|phi.psi|_tht}
	|\ps^{-1}|_{0,\tht} \leq |\ps|_{0,\tht}|\ps^{-2}|'_0, \text{ if } \ps >0;\quad |\ps|_{0,\tht} \leq |\ps|'_1;\quad	|\vp \ps|_{0,\tht}\leq |\vp|_0|\ps|_{0,\tht}+|\vp|_{0,\tht}|\ps|'_0.
\end{equation}
and analogous higher order versions.

\begin{remark}
    Much like in Remark \ref{r.ball.comparison}, notice that, again by by \eqref{e.rescaled.bound.gEucl} and \eqref{e.na.f}, up to a harmless dimensional constant, for all small enough $\eta_0$ (independent of $\varepsilon$) and all $S\subset \Sia$ (resp. $S\subset \partial_0\Sia$), $\sup_{x\in S}|\vp|_{0,\tht}(x)$ is uniformly comparable to the H{\"o}lder seminorm $[\vp]_{\theta, S}$ defined using the intrinsic distance on $\Sia$. Similarly for $\sup_{x\in S}|\vp|_{j,\tht}(x)$ and $[\nabla^j\vp]_{\theta, S}$. 
    
    Therefore, after taking $\eta_0$ sufficiently small (independent of $\varepsilon$), we can use appropriate suprema of the quantities in \eqref{j-theta norm} to define the H{\"o}lder seminorms.
\end{remark}

\subsubsection{The restriction map} For $R/2\leq r\leq R$, $C^{+}_r(0):=\bar{\bbb}^{+}_r(0)\times [-r,r]$. Let 
\begin{equation}\label{restriction map}
    \mathscr{R}:C^{k,\tht}(\Sia\cap C^+_{R}(0))\to C^{k,\tht}(\Sia\cap C^+_{9R/10}(0))
\end{equation}
denote the restriction map. By \cite[Lemma 6.37]{gilbarg1977elliptic}, \tes\ $\mathscr{R}^{-1}$, which is a \cts\ right inverse of $\mathscr{R}$. Moreover if $\mathscr{R}^{-1}(\vp)=\tilde{\vp}$, then $\tilde{\vp}$ can be described more concretely as follows. Since $\Sia$ is the graph of $f_{\al}$ over $\bbb_{2R}^+(0)$, $\Sia\cap C^+_{r}(0)$ (with $r\leq R$) can be identified with $\bbb_{r}^+(0)$ using the map $(y_1,\dots,y_n)\mapsto (y_1,\dots,y_n,f_{\al}(y_1,\dots,y_n))$. If we make this identification and use the polar coordinate $(\mathfrak{r},\upsilon)$ on $\bar{\bbb}_{R}^+(0)$ (where $\mathfrak{r}$ denotes the radial coordinate), then $\tilde{\vp}(\mathfrak{r},\upsilon)=\vp(\mathfrak{r},\upsilon)$, if $\mathfrak{r}\leq r_0=9R/10$ and
\begin{equation}\label{e.R^-1.ph}
	\tilde{\vp}(\mathfrak{r},\upsilon)=\sum_{i=1}^{k+1}c_i\vp\left(r_0-\frac{\mathfrak{r}-r_0}{i+1},\upsilon\right), \text{ if }\mathfrak{r}\geq r_0,
\end{equation}
where the constants $c_i$ are chosen in such a way that $\tilde{\vp}$ is a $C^{k,\tht}$ function. Further, \(\bar{\mathscr{R}} : C^0(\Sia\cap C^+_{R}(0)) \to C^0(\Sia\cap C^+_{R}(0))\) is defined by $\bar{\mathscr{R}}(\vp) = \bar{\vp}$, where $\bar{\vp}(\mathfrak{r},\upsilon)=\vp(\mathfrak{r},\upsilon)$, if $\mathfrak{r}\leq r_0=9R/10$ and
\begin{equation}\label{e.R.bar.ph}
	\bar{\vp}(\mathfrak{r},\upsilon)=\frac{1}{2}\vp(\mathfrak{r},\upsilon) + \frac{1}{2(k+1)}\sum_{i=1}^{k+1}\vp\left(r_0-\frac{\mathfrak{r}-r_0}{i+1},\upsilon\right), \text{ if }\mathfrak{r}\geq r_0.
\end{equation} 

\subsection{Existence and estimates of optimal layer shift with orthogonality condition}
In order to state  Proposition \ref{p.optimal_h_shift}, let us define
\begin{align}\label{notation wwaa}
\begin{split} 
&\mathtt{I}_{\al,y} =\{\be: \be \neq \al,\; |t_{\be}(y,0)|< 17 |\log \ve|\},\\
     &v_\alpha(y,z)=u(y,z)-\Hb_\alpha(y,z;0),\\   
	 & \om_{\al}^{(j)}(y)=\int_{-\infty}^{\infty}\Big(\sum_{i=0}^j|\va|'_i + |\va|_{i,\tht} \Big)(y,z)(1+|z|^3)e^{-|z|} dz, \\
 &\fwa^{(j)}(y)=\om^{(j)}_{\al}(y)+\max_{\be\in \mathtt{I}_{\al,y}}\sup\{\om^{(j)}_{\be}(\Pi_{\be}(y',z)):  y' \in \bar{B}^{\al}_1(y), |z|< 9 |\log \varepsilon|\}, \\
 & \mathrm{a}_{\al}(y) = |e^{-D_{\al}}|'_0(y),\\
	 & \faa(y)=\mathrm{a}_{\al}(y)+\max_{\be\in \mathtt{I}_{\al,y}}\sup\{\mathrm{a}_{\be}(\Pi_{\be}(y',z)):  y' \in \bar{B}^{\al}_1(y), |z|< 9 |\log \varepsilon|\}.
\end{split}
\end{align}

\begin{proposition}[Optimal layer shift]\label{p.optimal_h_shift}
	There exists $h_{\alpha}\in C^{k,\tht}(\Sigma_{\alpha}\cap C^{+}_{R}(0))$ such that for each index $\alpha$ and $y\in \Sigma_{\alpha}\cap C^{+}_{9R/10}(0)$, for $\phi(y, z; h)=u(y, z)-\Hb(y, z; h)-\bg(y,z;h)$ there holds
	\begin{equation}\label{e.ortho.H'_al}
	\int^{\infty}_{-\infty}\phi(y, z; h)\Hb_{\alpha}'(y, z; h_\alpha)dz=\int^{\infty}_{-\infty}[u(y, z)-\Hb(y, z; h)-\bg(y,z;h)]\Hb_{\alpha}'(y, z; h_\alpha)dz=0,
	\end{equation}
where $(y,z)$ are twisted Fermi coordinates with respect to $\Sia$. Moreover, the following estimates hold uniformly \wrt\ the index $\al$. As $\varepsilon\to 0$,
\begin{equation}\label{e.h=o1 stronger} 
\begin{aligned}
		& |h_{\al}|_{j} + |h_{\al}|_{j, \tht} \lesssim \ve + \om^{(j)}_{\al}+\mathrm{a}_{\al}^{0.9},\; 0\leq j \leq k, \text{ on } \Sigma_{\alpha}\cap C^{+}_{9R/10}(0),\\ 
	&|\partial_{y_1}h_{\al}|_{(j-1)} + |\partial_{y_1}h_{\al}|_{(j-1),\tht} \lesssim (\ve+\om^{(j)}_{\al} + \mathrm{a}_{\al}^{0.9})(\ve + \fwa^{(j)} + \faa^{0.9}),\; 1\leq j \leq k,  \text{ on }\del_0\Sigma_{\alpha}\cap C^{+}_{9R/10}(0),
\end{aligned}
\end{equation}
\end{proposition}
\begin{remark}\label{r.optimal_h_shift}
    In later application, we will fix $k=3$ and $j=2$ in Proposition \ref{p.optimal_h_shift}, so that $h_{\alpha}=o(1)\in C^{3,\tht}(\Sigma_{\alpha}\cap C^{+}_{R}(0))$. In this case, for ease of notation, we write (compare with \eqref{notation wwaa})
    \begin{equation}\label{m=m2, w=w2}
        \fwa=\fwa^{(2)}\qquad \text{and} \qquad \om_{\al}=\om_{\al}^{(2)}.
    \end{equation}
    In particular, we can simplify the estimate \eqref{e.h=o1 stronger} satisfied by $h_{\alpha}$ in Proposition \ref{p.optimal_h_shift} to  
 \begin{equation}\label{e.h=o1}
\begin{aligned}
	& \sum_{i=0}^2 |h_{\al}|_{i} + |h_{\al}|_{2,\tht}\lesssim \ve + \fwa+\faa^{0.9}\text{ on } \Sigma_{\alpha}\cap C^{+}_{9R/10}(0),\\ 
	&\sum_{i=0}^1|\partial_{y_1}h_{\al}|_{i}+|\partial_{y_1}h_{\al}|_{1,\tht}\lesssim \ve^2+\fwa^2+\faa^{1.8}    \text{ on }\del_0\Sigma_{\alpha}\cap C^{+}_{9R/10}(0).
\end{aligned}
\end{equation}
\end{remark}

\begin{proof}[Proof of Proposition \ref{p.optimal_h_shift}]
To this end, we shall modify the proof in \cite[Proof of Proposition 4.1]{wang-weiadv} as follows. For each $k$ and $\theta\in (0, 1)$ and $\alpha$, we define new norms $\|\cdot\|_{*}$ and $\|\cdot\|_{\star}$ on $C^{k, \tht}(\Sia\cap C^+_{9R/10}(0))$ and $C^{k, \tht}(\Sia\cap C^+_{R}(0))$. We introduce $\|\mathfrak{h}\|_{*}:=\sum_{j=0}^k [\mathfrak{h}]_{*;j} + \sum_{j=1}^k [\mathfrak{h}]'_{*;j}$, where
\begin{align}
	& [\mathfrak{h}]_{*;j}=\La^{-1} \sup_{\Sia\cap C^{+}_{9R/10}(0)}  \frac{|\mathfrak{h}|_{j} + |\mathfrak{h}|_{j,\tht} }{\ve+\om^{(j)}_{\al}+\mathrm{a}_{\al}^{0.9}},\\
	& [\mathfrak{h}]'_{*;j}=\La^{-2}\sup_{\del_0\Sia\cap C^{+}_{9R/10}(0)} \frac{|\mathfrak{h}_{\nu}|_{(j-1)}+|\mathfrak{h}_{\nu}|_{(j-1),\tht}}{(\ve +\om^{(j)}_{\al}+\mathrm{a}_{\al}^{0.9}) (\ve+\fwa^{(j)}+\faa^{0.9})},
\end{align}
$\mathfrak{h}_{\nu}$ denotes the restriction of $\del_{y_1}\mathfrak{h}$ on $\del_0\Sia$, $\La > 1$ is a constant, which will be specified later. Let us denote \( \bar{\mathscr{R}}(\om^{(j)}_{\al})\), \(\bar{\mathscr{R}}(\mathrm{a}_{\al})\), \(\bar{\mathscr{R}}(\mathfrak{w}_{\al})\), and \(\bar{\mathscr{R}}(\mathfrak{a}_{\al})\) by \(\bar{\om}^{(j)}_{\al} \), \(\bar{\mathrm{a}}_{\al} \), \(\bar{\mathfrak{w}}_{\al} \), and \(\bar{\mathfrak{a}}_{\al} \), where restriction map $\bar{\mathscr{R}}$ is defined in \eqref{restriction map}.
$\|\mathfrak{h}\|_{\star}:=\sum_{j=0}^k[\mathfrak{h}]_{\star;j} + \sum_{j=1}^k [\mathfrak{h}]'_{\star; j}$, where
\begin{align}
	& [\mathfrak{h}]_{\star; j}=\La^{-1} \sup_{\Sia\cap C^{+}_{R}(0)}  \frac{|\mathfrak{h}|_{j} + |\mathfrak{h}|_{j,\tht} }{\ve+\bar{\om}^{(j)}_{\al}+\bar{\mathrm{a}}_{\al}^{0.9}}, \\
	& [\mathfrak{h}]'_{\star; j}=\La^{-2} \sup_{\del_0\Sia\cap C^{+}_{R}(0)} \frac{|\mathfrak{h}_{\nu}|_{(j-1)}+|\mathfrak{h}_{\nu}|_{(j-1),\tht}}{(\ve +\bar{\om}^{(j)}_{\al}+\bar{\mathrm{a}}_{\al}^{0.9}) (\ve+\bar{\mathfrak{w}}_{\al}^{(j)} + \bar{\mathfrak{a}}_{\al}^{0.9})}.
\end{align}

Following \cite[Proof of Proposition 4.1]{wang-weiadv}, we define 
$$\cx_{\al}=(C^{k,\tht}(\Sia\cap C^+_{R}(0)),\|\cdot\|_{\star}),\quad \cy_{\al}=(C^{k,\tht}(\Sia\cap C^+_{9R/10}(0)), \|\cdot\|_{*}),$$
$\cx=\bigoplus_{\al}\cx_{\al}$, $\cy=\bigoplus_{\al}\cy_{\al}$. The norm on $\cx$ is defined by $\|\mathbf{x}\|=\sup_{\al}\|\mathbf{x_{\al}}\|_{\star},$ where $\mathbf{x_{\al}}$ is the $\cx_{\al}$ component of $\mathbf{x}$. The norm on $\cy$ is defined in a similar way.

Recall the definition of $g_{\beta}(y, z;h_{\beta})$ and $g(y, z; h)$  from \eqref{def_gk=H+G} and \eqref{def_g=H+G}. Then we define $F:\cx\to\cy$ by
\begin{align}
    F_{\al}(h)(y)&=\int^{\infty}_{-\infty}[u(y, z)-g(y,z;h)]\Hb_{\alpha}'(y, z; h_{\al})dz\\
    &=\int^{\infty}_{-\infty}[u(y, z)-\Hb(y, z; h)-\bg(y,z;h)]\Hb_{\alpha}'(y, z; h_{\al})dz,
\end{align}
where $F_{\al}(h)$ denotes the $\cy_{\al}$ component of $F(h)$ and $(y,z)$ denotes the twisted Fermi coordinates with respect to $\Sia$. We restrict $F$ on the closed unit ball in $\cx$. We note that, for $y\in \Sia$ and $\mathsf{y}\in \bar{B}_1^{\al}(y)$,
\((g_{\be}(\mathsf{y},z;h_{\be})+\operatorname{sgn}(\be - \al)(-1)^{\be}) \Hb_{\alpha}'(\mathsf{y}, z; h_{\al}) \neq 0\), $\be \neq \al$, implies \(|z|< 9 |\log \varepsilon|\) and \(\be \in \mathtt{I}_{\al,y}\);  hence

\begin{equation}
	F_{\al}(h)(\mathtt{y})=\int^{\infty}_{-\infty}\big[u(\mathtt{y}, z)-g_{\al}(\mathtt{y}, z; h_{\al})-\sum_{\be\in \mathtt{I}_{\al,y}} (g_{\be}(\mathtt{y},z;h_{\be})+\operatorname{sgn}(\be - \al)(-1)^{\be})\big] \Hb_{\alpha}'(\mathtt{y}, z; h_{\al})dz.
\end{equation}

Notice that $F$ is a $C^1$ map, and its derivative $(DF_{\al}(h)\xi)(y)$ is given by
		\begin{align}\label{e.DF_alpha}
		 & (-1)^{\al}\xi_{\al}(y)\int_{-\infty}^{\infty}\left[\Hb'_{\alpha}(y, z; h_{\al})^2+\bga'(y, z; h_{\al})\Hb'_{\alpha}(y, z; h_{\al})\right]dz\\
		 & - (-1)^{\al}\xi_{\al}(y)\int_{-\infty}^{\infty}\big[u(y, z)-\Hb(y, z; h)-\bg(y,z;h)\big]\Hb''_{\alpha}(y, z; h_{\al})dz\\
		& +\sum_{\be\in \mathtt{I}_{\al,y}}(-1)^{\beta}\int_{-\infty}^{\infty}\xi_{\beta}(\Pi_{\beta}(y,z))\bhb'(y, z; h_{\beta})\bha'(y, z; h_{\al})dz\\
		& +\sum_{\be\in \mathtt{I}_{\al,y}}(-1)^{\beta} \int_{-\infty}^{\infty}\xi_{\beta}(\Pi_{\beta}(y,z))\bg'_{\beta}(y, z; h_{\beta})\bha'(y, z; h_{\al})dz.
	\end{align}
	Thus, there exist \cts\ linear maps $D_{\be}F_{\al}(h):\cx_{\be}\to \cy_{\al}$ \st
	\[DF_{\al}(h)\xi=D_{\alpha}F_{\alpha}(h)\xi_{\alpha}+\sum_{\be\in \mathtt{I}_{\al,y}}D_{\be}F_{\al}(h)\xi_{\be}.\]
	Let $\mathsf{T} = DF(0)$, $\mathsf{T}_{\al\al}=D_{\al}F_{\al}(0)$, $\mathsf{T}_{\be\al}=D_{\be}F_{\al}(0)$ for $\beta\not=\alpha$. Then, $(\mathsf{T}_{\al\al}\xi_{\al})(y)=\vth_{\al}(y)\xi_{\al}(y)$, where 
	\begin{align}\label{e.expr.vth}
		\vth_{\al}(y) = & (-1)^{\al}\int_{-\infty}^{\infty} \left[\Hb'_{\alpha}(y, z; 0)^2+\bga'(y, z;0)\Hb'_{\alpha}(y, z; 0)\right]dz\\
		& -(-1)^{\al}\int_{-\infty}^{\infty}\big[u(y, z)-\Hb_{\al}(y, z; 0)-\bg_{\al}(y,z;0)\big]\Hb''_{\alpha}(y, z; 0)dz\\
		&+(-1)^{\al}\sum_{\be\in \mathtt{I}_{\al,y}}\int_{-\infty}^{\infty}\big[\Hb_{\be}(y, z; 0)+\bg_{\be}(y,z;0)\big]\Hb''_{\alpha}(y, z; 0)dz.
	\end{align}

\begin{claim}\label{c.DF.F}	
We have the following properties.
\begin{enumerate}[(i)]
	\item \Tes\ $DF(0)^{-1}$, which is a right inverse of $DF(0)$, \st\  the operator norm of $DF(0)^{-1}$ satisfies  $\|DF(0)^{-1}\|\leq C(\si_0)$, where $\sigma_0$ is the energy of the one-dimensional solution, as in \eqref{e.sigma0}.
	\item $h \mapsto DF(h)$ is Lipschitz \cts\ at $h=0$ and its Lipschitz constant at $0$ is $ \La c_{\varepsilon}$, where $c_{\varepsilon}=o(1)$ as $\ve\to 0$.
	\item $\|F(0)\|_{\cy}\leq C \La^{-1}$ as $\ve \to 0$.	 
\end{enumerate}
\end{claim}	

First, let us assume Claim \ref{c.DF.F} and let us finish the proof of Proposition \ref{p.optimal_h_shift}. We will show  the existence of $h\in \cx$ satisfying \eqref{e.ortho.H'_al} and estimates \eqref{e.h=o1 stronger} by the contraction mapping principle, as in \cite[Theorem 1.2.4]{chang2005methodsnonlinear}.

Note that it is enough to show that \tes\ $h\in \cx$ \st\ $\|h\|_{\cx}\leq 1$ and $F(h)=0$.  Let $\mathsf{R}:\overline{B^{\cx}_1(0)}\to \cy$ be defined by $\mathsf{R}(h)=F(h)-DF(0)h$. $F(h)=0$ if and only if $DF(0)h=-\mathsf{R}(h)$. Let $\bar{\mathsf{R}}(h):=-DF(0)^{-1}\mathsf{R}(h)$. If $h$ is a fixed point of $\bar{\mathsf{R}}$, then $F(h)=0$. 

We note that
\begin{equation}
	\mathsf{R}(h)-\mathsf{R}(\tilde{h})=\int_{0}^{1} \left(DF(\ga(t))-DF(0)\right)(h-\tilde{h})dt,\;\text{ where }\ga(t) = th+(1-t)\tilde{h}.
\end{equation}
\hn, denoting the Lipschitz constant of $DF$ at $0$ by $\La c_{\varepsilon}$ ($c_{\varepsilon}=o(1)$ as \(\ve\to 0\)) and using (i) and (ii) of Claim \ref{c.DF.F}, 
\begin{equation}\label{e.R.bar.1}
	\|\bar{\mathsf{R}}(h)-\bar{\mathsf{R}}(\tilde{h})\| \leq C(\si_0)\La c_\varepsilon \|h-\tilde{h}\|,
\end{equation}
which implies
\begin{equation}\label{e.R.bar.2}
	\|\bar{\mathsf{R}}(h)\| \leq C(\si_0) \La c_\varepsilon \|h\|+\|\bar{\mathsf{R}}(0)\|.
\end{equation}
By  (i) and (iii) of Claim \ref{c.DF.F}, we  can fix $\La$ sufficiently large so that 
\begin{equation}
    \|\bar{\mathsf{R}}(0)\|_{\cx}\leq  1/2.
\end{equation}
Then, \eqref{e.R.bar.1} and \eqref{e.R.bar.2} imply that if $\ve$ is sufficiently small, $\bar{\mathsf{R}}:\overline{B^{\cx}_1(0)}\to\overline{B^{\cx}_1(0)}$ is a contraction. Hence, $\bar{\mathsf{R}}$ has a fixed point $h$ in $\overline{B^{\cx}_1(0)}$ satisfying estimates \eqref{e.h=o1 stronger} and $F(h)=0$. This finishes the proof of Proposition \ref{p.optimal_h_shift}.

Let us now prove Claim \ref{c.DF.F}.


\noindent\textit{\textbf{Proof of (i).}} 
As a preliminary to the construction of a right inverse of $DF(0)$ satisfying the required operator norm estimate, we shall establish the following two facts:
\begin{enumerate}[(a):]
    \item There exists $\mathsf{T}_{\al\al}^{-1}:\cy_{\al}\to\cx_{\al}$, which is a right inverse of $\mathsf{T}_{\al\al}$, \st\ the operator norm of $\mathsf{T}_{\al\al}^{-1}$ satisfies $\|\mathsf{T}_{\al\al}^{-1}\|\leq C(\si_0)$.
    \item The operators $\mathsf{T}_{\be\al}:\cx_{\beta}\to\cy_{\al}$ for $\beta\neq \alpha$ satisfy $\|\sum_{\be \in \mathtt{I}_{\al,y} }\mathsf{T}_{\be\al}\xi_{\be}\|_{*; 3, \tht}=o(1)\|\xi\|_{\cx}$ as $\varepsilon \to 0$.
\end{enumerate}

\noindent\textit{\textbf{Proof of (a).}}  Let $\mathscr{R}:C^{3, \theta}(\Sia\cap C^+_{9R/10}(0))\to C^{3, \theta}(\Sia\cap C^+_{7R/8}(0))$ denote the restriction map. By \cite[Lemma 6.37]{gilbarg1977elliptic},  \tes\ an extension operator $\mathscr{R}^{-1}$, which is a continuous right inverse of $\mathscr{R}$. 

By Lemmas \ref{l.1D.soln}, \ref{l.interaction.int}, \ref{l.D_al.infty}, \ref{l.sum.of.exp} and \eqref{e.G.est}, 
\begin{equation}\label{e.vth.lower.bd}
|\vth_{\al}(y)|\geq \frac{\si_0}{2}.
\end{equation}
Then we can define the right inverse of $\mathsf{T}_{\al\al}$ by $(\mathsf{T}_{\al\al}^{-1}\et_{\al})(y):=\vth_{\al}(y)^{-1}\tilde{\et}_{\al}(y)$ where $\tilde{\et}_{\al}=\mathscr{R}^{-1}{\et_{\al}}$.  We need to show that $\|\vth_{\al}^{-1}\tilde{\et}_{\al}\|_{\star}\leq C(\si_0) \|\et_{\al}\|_{*}$. Again by Lemmas \ref{l.1D.soln}, \ref{l.interaction.int}, \ref{l.D_al.infty}, \ref{l.sum.of.exp} and \eqref{e.G.est}, we also have
\begin{equation}\label{e.vth.bd}
 |\vth_{\al} - (-1)^{\alpha}\sigma_0|'_0 + |\vth_{\al}|_{0, \tht} \lesssim \ve + \om_{\al}^{(0)} + \mathrm{a}_{\al}^{0.9}.
\end{equation}
\hn, by \eqref{e.|phi.psi|_tht},
\begin{equation}\label{e.vth^-1.tet.0.tht}
	|\vthai \teta|_0 + |\vthai \teta|_{0, \tht} \leq C(\si_0) (|\teta|_0 + |\teta|_{0, \tht}).
\end{equation}
Differentiating \eqref{e.expr.vth} \wrt\ $y_i$, we obtain
\begin{align}
	\del_{y_i}\vth_{\al}(y)= 
	& (-1)^{\al}\int_{-\infty}^{\infty} \bg_{\al,i}'(y, z;0)\Hb'_{\alpha}(y, z; 0)dz \\
	& -(-1)^{\al}\int_{-\infty}^{\infty}\big[\del_{y_i}v_{\al}(y,z)-\bg_{\al,i}(y,z;0)\big]\Hb''_{\alpha}(y, z; 0)dz\\ \label{e.del_y_i.vth}
	& +(-1)^{\al}\sum_{\be\in \mathtt{I}_{\al,y}}(-1)^{\be}\int_{-\infty}^{\infty}\big[\Hb'_{\be}(y, z;0) +\bg'_{\be}(y,z;0)\big] \del_{y_i}t_{\be}(y,z)\Hb''_{\alpha}(y, z; 0)dz \\ 
	& +(-1)^{\al}\sum_{\be\in \mathtt{I}_{\al,y}}\sum_{m=1}^{n} \int_{-\infty}^{\infty} \bg_{\be,m}(y,z;0) \del_{y_i}\Pi^m_{\be}(y,z) \Hb''_{\alpha}(y, z; 0)dz,
\end{align}
which implies
\begin{equation}\label{e.del.vth.est}
	|\vth_{\al}|'_1 + |\vth_{\al}|_{1, \tht}\lesssim \ve + \om_{\al}^{(1)}+\mathrm{a}_{\al}^{0.9}.
\end{equation} 
Since
\begin{equation}\label{e.del_y_i.vthai.teta}
	\del_{y_i} (\vthai \teta) = \vthai \del_{y_i}\teta - \vth_{\al}^{-2}\del_{y_i} \vth_{\al} \teta, 
\end{equation}
by \eqref{e.vth.lower.bd}, \eqref{e.vth.bd}, \eqref{e.del.vth.est}, and \eqref{e.|phi.psi|_tht},
\begin{equation}\label{e.vth^-1.tet.1.tht}
	|\vthai \teta|_1 + |\vthai \teta|_{1, \tht} \leq C(\si_0) (|\teta|_0 + |\teta|_{0, \tht} + |\teta|_1 + |\teta|_{1, \tht}).
\end{equation}
Using \eqref{e.del_y_i.vthai.teta} and \eqref{e.|phi.psi|_tht} we obtain
\begin{multline}\label{e.vth^-1.tet.nu.0.tht}
	 |(\vthai \teta)_{\nu}|_0 + |(\vthai \teta)_{\nu}|_{0, \tht} \leq C(\si_0)\big(|(\teta)_{\nu}|_{0, \tht} + |(\teta)_{\nu}|_{0} \\
	  + |\teta|_{0, \tht} |\vth_{\al}|'_1 + |\teta|_{0} |\vth_{\al}|_{1, \tht} +  |\teta|_{0} |\vth_{\al}|_{1} \big).
\end{multline}
\eqref{e.vth^-1.tet.0.tht}, \eqref{e.vth^-1.tet.1.tht}, \eqref{e.vth^-1.tet.nu.0.tht}, \eqref{e.R^-1.ph}, \eqref{e.R.bar.ph} imply
\begin{equation}
	[\vthai \teta]_{\star; 0} + [\vthai \teta]_{\star; 1} + [\vthai \teta]'_{\star; 1} \le C(\si_0) ([\et_{\al}]_{*; 0} + [\et_{\al}]_{*; 1} + [\et_{\al}]'_{*; 1} )
\end{equation}
Further differentiating \eqref{e.del_y_i.vth}, one can show that 
\begin{equation}\label{e.del_y_i.vth.est}
	|\vth_{\al}|'_{j} + |\vth_{\al}|_{j, \tht}\lesssim \ve + \om_{\al}^{(j)}+\mathrm{a}_{\al}^{0.9},\quad 1 \le j \le k.
\end{equation}
\eqref{e.|phi.psi|_tht}, \eqref{e.R^-1.ph}, \eqref{e.R.bar.ph}, \eqref{e.vth.lower.bd}, \eqref{e.vth.bd} and \eqref{e.del_y_i.vth.est} imply 
\begin{equation}
\|\vth_{\al}^{-1}\tilde{\et}_{\al}\|_{\star}\leq C(\si_0) \|\et_{\al}\|_{*},
\end{equation}
which is the desired estimate for $\|\mathsf{T}^{-1}_{\al\al}\|$.

\noindent\textit{\textbf{Proof of (b).}} We recall that
\begin{equation}\label{e.T_{be.al}.1}
	(\mathsf{T}_{\be\al}\xi_{\be})(y)= (-1)^{\beta}\int_{-\infty}^{\infty}\xi_{\beta}(\Pi_{\beta}(y,z))(\bhb'(y, z; 0)+\bg'_{\beta}(y, z; 0))\bha'(y, z; 0)dz.
\end{equation}
If \(y' \in \bar{B}^{\al}_1(y)\), then 
\begin{align}\label{e.holder.chain.rule}
	& \frac{|\hat{\xi}_{\be}(\Pi_{\be}(y,z)) - \hat{\xi}_{\be}(\Pi_{\be}(y',z))|}{|y - y'|^{\tht}} \\
	& \le
		\left[|\hat{\xi}_{\be}|_{0, \tht}(\Pi_{\be}(y,z)) + |\hat{\xi}_{\be}(\Pi_{\be}(y,z))| + |\hat{\xi}_{\be}(\Pi_{\be}(y',z))| \right]\|D \Pi_{\be}\|^{\tht},
\end{align}
for any \(\hat{\xi}_{\be} \in C^{0, \tht}(\Sib)\). \hn\ \eqref{e.G.est}, \eqref{e.|phi.psi|_tht}, and Lemmas \ref{l.interaction.int}, \ref{l.sum.of.exp} imply 
\begin{equation}
	|\la|_0 + |\la|_{0, \tht} \lesssim (\ve + \mathfrak{w}_{\al}^{(0)} + \mathfrak{a}_{\al}) \mathrm{a}_{\al}^{0.95} \max_{\be \neq \al} [\xi_{\be}]_{\star;0}, \text{ where }\la = \sum_{\be\neq \al }\mathsf{T}_{\be\al}\xi_{\be}.
\end{equation}
Differentiating \eqref{e.T_{be.al}.1} with respect to \(y_i\), we obtain
\begin{align}\label{e.T_{be.al}.2}
	\del_{y_i}(\mathsf{T}_{\be\al}\xi_{\be}) = & (-1)^{\beta}\sum_{m=1}^n\int_{-\infty}^{\infty} \xi_{\beta,m}(\Pi_{\beta}) \del_{y_i}\Pi_{\beta}^m(\bhb'(-; 0) +\bg'_{\beta}(-; 0))\bha'(-; 0)dz \\
	& + \int_{-\infty}^{\infty}\xi_{\beta}(\Pi_{\beta})(\bhb''(-; 0)+\bg''_{\beta}(-; 0)) \del_{y_i}t_{\be}   \bha'(-; 0)dz\\
	& + (-1)^{\beta}\sum_{m=1}^n\int_{-\infty}^{\infty} \xi_{\beta}(\Pi_{\beta})\bg'_{\beta,m}(-; 0) \del_{y_i}\Pi_{\beta}^m \bha'(-; 0)dz.
\end{align}
\hn\ \eqref{e.holder.chain.rule}, \eqref{e.G.est}, \eqref{e.|phi.psi|_tht}, and Lemmas \ref{l.interaction.int}, \ref{l.sum.of.exp} imply 
\begin{equation}
	|\la|_{1} + |\la|_{1, \tht} + |\la_{\nu}|_{0} + |\la_{\nu}|_{0, \tht} \lesssim (\ve + \mathfrak{w}_{\al}^{(1)} + \mathfrak{a}_{\al}) \mathrm{a}_{\al}^{0.95} \max_{\be \neq \al} \left([\xi_{\be}]_{\star;0} + [\xi_{\be}]_{\star;1}\right).
\end{equation}
Differentiating \eqref{e.T_{be.al}.2}, one can show that
\begin{equation}
	|\la|_{j} + |\la|_{j, \tht} + |\la_{\nu}|_{j-1} + |\la_{\nu}|_{(j-1), \tht} \lesssim (\ve + \mathfrak{w}_{\al}^{(j)} + \mathfrak{a}_{\al}) \mathrm{a}_{\al}^{0.95} \max_{\be \neq \al} \left(\sum_{i=0}^j [\xi_{\be}]_{\star;i}\right)\quad \forall 1 \le j \le k.
\end{equation}

Let us now use (a) and (b) to construct the desired right inverse of $DF(0)$.
Given $\et\in \cy$, as in \cite[page 17] {wang-weiadv} we define $S_{\et}:\cx\to \cx$ by
\[S_{\et}(\xi)_{\al}=\mathsf{T}_{\al\al}^{-1}\et_{\al}-\sum_{\be\in \mathtt{I}_{\al,y} }\mathsf{T}_{\al\al}^{-1}\mathsf{T}_{\be\al}\xi_{\be}. \]
As a consequence of (a) and (b), the map $S_{\et}$ is a contraction for $\varepsilon>0$ small enough. By the contraction mapping principle, $S_{\et}$ has a unique fixed point, which we denote by $\tilde{S}(\et)$ and 
$$\big\|\tilde{S}(\et)\big\|_{\cx}=\lim_{k\ra \infty}\big\|S^k_{\et}(0)\big\|_{\cx}\leq C(\si_0)\|\et\|_{\cy}.$$ 
By the definition of $S_{\et}$, it follows that $\tilde{S}=(DF(0))^{-1}$ is a right inverse of $DF(0)$. This concludes the proof of (i) in Claim \ref{c.DF.F}.

\noindent\textit{\textbf{Proof of (ii).}}  We fix $\xi\in \cx$ \w\ $\|\xi\|_{\cx}\leq 1$. Since $F$ was restricted to the closed the unit ball in $\cx$, we need to show that 
\begin{equation}\label{e.DF.lip}
	\|DF_{\al}(h)\xi-DF_{\al}(0)\xi\|_{*}\leq c_{\varepsilon}\La \max_{\be}\|h_{\be}\|_{\star},\quad \forall  h\in \overline{B^{\cx}_1(0)},
\end{equation}
where $c_{\varepsilon}>0$ is independent of $\La$ and \(c_{\varepsilon}=o(1)\) as $\ve \to 0$. Let $\et_{\al}=DF_{\al}(h)\xi-DF_{\al}(0)\xi$. By the $L^2$-orthogonality condition in the definition of $\bar{\bg}_{\al}$, 
\[\int_{-\infty}^{\infty}\bg_{\al}(y,z;h_{\al})\Hb_{\al}'(y,z;h_{\al})dz=0\quad \forall  h_{\al}\in \cx_{\al},\]
which implies
\[\int_{-\infty}^{\infty}\bg_{\al}'(y,z;h_{\al})\Hb_{\al}'(y,z;h_{\al})dz+\int_{-\infty}^{\infty}\bg_{\al}(y,z;h_{\al})\Hb_{\al}''(y,z;h_{\al})dz=0.\]

\hn\ (using \eqref{e.DF_alpha}) $\et_{\al}=\et_{\al\al}+\sum\limits_{\be\neq \al}\et_{\al\be}$, where 
\begin{align}
	 \et_{\al\al} = & (-1)^{\al}\xi_{\al}\int_{-\infty}^{\infty}\left[\Hb'_{\alpha}(-; h_{\al})^2-(u-\Hb_{\al}(-; h))\Hb''_{\alpha}(-; h_{\al}) \right] dz \\
	 & - (-1)^{\al}\xi_{\al}\int_{-\infty}^{\infty}\left[\Hb'_{\alpha}(-; 0)^2-(u-\Hb_{\al}(-; 0))\Hb''_{\alpha}(-; 0) \right] dz,\\
	\et_{\al\be}= & (-1)^{\al} \xi_{\al} \int_{-\infty}^{\infty}(g_{\be}(-; h_{\beta}) +\operatorname{sgn}(\be-\al) (-1)^{\be})\bha''(-; h_{\al})dz\\
	& - (-1)^{\al} \xi_{\al} \int_{-\infty}^{\infty} (g_{\be}(-; 0) +\operatorname{sgn}(\be-\al) (-1)^{\be})\bha''(-; 0) dz \\	
	& +(-1)^{\beta}\int_{-\infty}^{\infty}\xi_{\beta}(\Pi_{\beta})\left[g_{\be}'(-; h_{\beta})\bha'(-; h_{\al}) - g_{\be}'(-; 0) \bha'(-; 0)\right]dz.		
\end{align}
We note that
\begin{align}
	\et_{\al\al} = & (-1)^{\al}\xi_{\al}\int_{-\infty}^{\infty} \left[\bH'(z-h_{\al})^2-\bH'(z)^2 + \left(\bH(z-h_{\al})-\bH(z)\right) \bH''(z-h_{\al})\right]dz \\
	& -\xi_{\al}\int_{-\infty}^{\infty} (u-\Hb_{\al}(-;0))(\bH''(z-h_{\al})-\bH''(z))dz,	
\end{align}
which implies
\begin{align}
	\et_{\al\al} = & (-1)^{\al+1}\xi_{\al}h_{\al}\int_{0}^{1}\int_{-\infty}^{\infty}\left[2\bH''(z-sh_{\al})\bH'(z-sh_{\al})+ \bH'(z-sh_{\al})\bH''(z-h_{\al})\right]dzds\\
	&+ \xi_{\al}h_{\al}\int_{0}^{1}\int_{-\infty}^{\infty} v_{\al}\bH'''(z-sh_{\al})dzds.
\end{align}
\hn, \fa\ \(0\leq j \leq k\),
\begin{align}\label{e.et_al,al}
& |\et_{\al\al}|_j + |\et_{\al\al}|_{j, \tht} \lesssim \La^2(\ve + \om_{\al}^{(j)} + \mathrm{a}_{\al}^{0.9}) (\ve + |\om_{\al}^{(j)}|'_0 + |\mathrm{a}_{\al}^{0.9}|'_0) \|\xi_{\al}\|_{\star} \|h_{\al}\|_{\star},\\
& |(\et_{\al\al})_{\nu}|_{j-1} + |(\et_{\al\al})_{\nu}|_{(j-1), \tht} \lesssim \La^3(\ve+\om^{(j)}_{\al}+\mathrm{a}_{\al}^{0.9})(\ve+\mathfrak{w}^{(j)}_{\al} +\mathfrak{a}_{\al}^{0.9}) (\ve + |\om^{(j)}_{\al}|'_0 + |\mathrm{a}_{\al}^{0.9}|'_0) \|\xi_{\al}\|_{\star} \|h_{\al}\|_{\star}.
\end{align}
Similarly,
\begin{align*}
 \eta_{\alpha\beta}={}&(-1)^\alpha\xi_\alpha\int_{\mathbb R}
 (g_\beta(-;h_\beta)-g_\beta(-;0))\Hb_\alpha''(-;h_\alpha)\,dz\\
 &+(-1)^\alpha\xi_\alpha\int_{\mathbb R}
 (g_\beta(-;0)+\operatorname{sgn}(\be-\al))
 (\Hb_\alpha''(-;h_\alpha)-\Hb_\alpha''(-;0))\,dz\\
 &+(-1)^\beta\int_{\mathbb R}\xi_\beta(\Pi_\beta)
 (g_\beta'(-;h_\beta)-g_\beta'(-;0))\Hb_\alpha'(-;h_\alpha)\,dz\\
 &+(-1)^\beta\int_{\mathbb R}\xi_\beta(\Pi_\beta)g_\beta'(-;0)
 (\Hb_\alpha'(-;h_\alpha)-\Hb_\alpha'(-;0))\,dz.
\end{align*}
which implies, by Taylor's theorem, that
\begin{align*}
 \eta_{\alpha\beta}={}&(-1)^{\alpha+\beta+1}\xi_\alpha
 \int_0^1\!\int_{\mathbb R}h_\beta(\Pi_\beta)
 g_\beta'(-;th_\beta)\Hb_\alpha''(-;h_\alpha)\,dz\,dt\\
 &-\xi_\alpha h_\alpha\int_0^1\!\int_{\mathbb R}
 (g_\beta(-;0)+\operatorname{sgn}(\be-\al))\Hb_\alpha^{(3)}(-;th_\alpha)\,dz\,dt\\
 &-\int_0^1\!\int_{\mathbb R}\xi_\beta(\Pi_\beta)h_\beta(\Pi_\beta)
 g_\beta''(-;th_\beta)\Hb_\alpha'(-;h_\alpha)\,dz\,dt\\
 &+(-1)^{\alpha+\beta+1}h_\alpha\int_0^1\!\int_{\mathbb R}
 \xi_\beta(\Pi_\beta)g_\beta'(-;0)\Hb_\alpha''(-;th_\alpha)\,dz\,dt.
\end{align*}
\hn, (similar to Step 2) if $\bar{\et}_{\al}=\sum_{\be\neq \al}\et_{\al\be}$, \(0 \le j \le k\),
\begin{align}\label{e.et'_al}
	& |\bar{\et}_{\al}|_j + |\bar{\et}_{\al}|_{j, \tht} + \md{(\bar{\et}_{\al})_{\nu}}_{j-1} + \md{(\bar{\et}_{\al})_{\nu}}_{(j-1), \tht} \\
	& \lesssim \La^2  (\ve+|\mathfrak{w}_{\al}^{(j)}|'_0 + |\mathfrak{a}_{\al}^{0.9}|'_0 ) (\ve + \mathfrak{w}_{\al}^{(j)} + \mathfrak{a}_{\al}^{0.9})\mathrm{a}_{\al}^{0.9}\|\xi\|\|h\|.
\end{align}
Hence, \eqref{e.et_al,al} and \eqref{e.et'_al} imply that the constant $c_{\varepsilon}=o(1)$ in \eqref{e.DF.lip}.
This concludes the proof of (ii) in  Claim \ref{c.DF.F}.

\noindent\textit{\textbf{Proof of (iii).}} Let $\ffa=F_{\al}(0)$. Then
\begin{align}\label{e.ffa.expr}
	\ffa(y)= & \int_{-\infty}^{\infty}\big[u(y, z)-\Hb_{\al}(y, z; 0)-\bg_{\al}(y,z;0)\big]\Hb'_{\alpha}(y, z; 0)dz\\
	&-\sum_{\be\in \mathtt{I}_{\al,y}}\int_{-\infty}^{\infty}\big[\Hb_{\be}(y, z; 0) + \bg_{\be}(y,z;0) +\operatorname{sgn}(\be - \al)(-1)^{\be}\big]\Hb'_{\alpha}(y, z; 0)dz.
\end{align}
Hence, using Lemmas \ref{l.interaction.int}, \ref{l.D_al.infty}, \ref{l.sum.of.exp} and \eqref{e.G.est}, we obtain
\begin{equation}\label{e.ffa.est}
	|\ffa|_0 + |\ffa|_{0, \tht} \lesssim \ve + \om^{(0)}_{\al} + \mathrm{a}_{\al}^{0.9}.
\end{equation}
Differentiating \eqref{e.ffa.expr} \wrt\ $y_i$ we obtain
\begin{align}\label{e.del_y_i.ffa}
	\del_{y_i}\ffa(y)= 
	& \int_{-\infty}^{\infty}\big[\del_{y_i}v_{\al}(y,z)-\bg_{\al,i}(y,z;0)\big]\Hb'_{\alpha}(y, z; 0)dz\\
	& -\sum_{\be\in \mathtt{I}_{\al,y}}(-1)^{\be}\int_{-\infty}^{\infty}\big[\Hb'_{\be}(y, z;0) +\bg'_{\be}(y,z;0)\big] \del_{y_i}t_{\be}(y,z)\Hb'_{\alpha}(y, z; 0)dz \\ 
	& -\sum_{\be\in \mathtt{I}_{\al,y}}\sum_{m=1}^{n} \int_{-\infty}^{\infty} \bg_{\be,m}(y,z;0) \del_{y_i}\Pi^m_{\be}(y,z) \Hb'_{\alpha}(y, z; 0)dz.
\end{align}
Hence (using \eqref{e.|phi.psi|_tht})
\begin{equation}\label{e.del_y_i.ffa.est}
	|\ffa|_1 + |\ffa|_{1, \tht} \lesssim \ve + \om^{(1)}_{\al}+\mathrm{a}_{\al}^{0.9}.
\end{equation} 
By further differentiating \eqref{e.del_y_i.ffa}, one can show that
\begin{equation}\label{e.del_y_i.ffa.est.j}
	|\ffa|_j + |\ffa|_{j, \tht} \lesssim \ve + \om^{(j)}_{\al}+\mathrm{a}_{\al}^{0.9},\quad 0\le j \le k.
\end{equation}
On $\partial_0C_R^+(0)$, the Neumann condition for $u$, the definition of $v_\alpha$, and
\eqref{G-equation} give
\begin{align}\label{e.del_nu.ffa.1}
 \partial_{y_1}v_\alpha-\Gb_{\alpha,1}(-;0)
 & =\langle\nabla v_\alpha,\partial_{y_1}-\partial_{x_1}\rangle
 -(-1)^\alpha\Hb_\alpha^{(1)}(y,z;0)\widetilde J_\alpha(y,z)
 -\delta_\alpha(y,(-1)^\alpha z).
\end{align}
Indeed,
$\partial_{x_1}v_\alpha=-(-1)^\alpha\Hb_\alpha^{(1)}J_\alpha$, whereas the boundary
condition for $\bgab$ gives
$-\Gb_{\alpha,1}=(-1)^\alpha w_\alpha\Hb_\alpha^{(1)}z
-\delta_\alpha(y,(-1)^\alpha z)$; now use
$\widetilde J_\alpha=J_\alpha-w_\alpha z$.
We also recall that 
\begin{equation}
    |\delta_\alpha|_{3}\lesssim \varepsilon^3
\end{equation}
by \eqref{e.xi.delta.est}.

Let \(\varrho : \del_0 \Sia \cap C_{9R/10} \to \bbr \) be defined by 
\begin{equation}
	\varrho(y) = \int_{-\infty}^{\infty}\big[\del_{y_1}v_{\al}(y,z)-\bg_{\al,1}(y,z;0)\big]\Hb'_{\alpha}(y, z; 0)dz
\end{equation}
\eqref{e.del_nu.ffa.1}, \eqref{e.|phi.psi|_tht}, \eqref{e.del.x_1-del.y_1} and Lemma \ref{l.dz(nu)} imply 
\begin{equation}\label{e.del_nu.ffa.2}
	|\varrho|_{(j-1)} + |\varrho|_{(j-1), \tht} \lesssim \ve \om_{\al}^{(j)} + \ve^2\quad \forall  1\leq j \leq k.
\end{equation}
Therefore, using \eqref{e.del_y_i.ffa}, \eqref{e.del_nu.ffa.2}, \eqref{e.G.est}, \eqref{e.|phi.psi|_tht} and Lemmas \ref{l.interaction.int}, \ref{l.del.y_1.t_be}, \ref{l.del.y_1.t_be 2} we obtain
\begin{equation}\label{e.ffa.nu}
	|(\ffa)_{\nu}|_{j-1} + |(\ffa)_{\nu}|_{j-1, \tht} \lesssim \ve^2+ \ve\om_{\al}^{(j)} + \ve\mathrm{a}_{\al}^{0.9}\quad \forall  1\leq j \leq k.
\end{equation}
\eqref{e.del_y_i.ffa.est.j} and \eqref{e.ffa.nu} imply $\|\ffa\|_{*}=O(\La^{-1})$.

This finishes the proof (iii), and therefore the proof of Claim \ref{c.DF.F}.	Hence we  finish the proof of Proposition \ref{p.optimal_h_shift}.
\end{proof}

From now on, let $h=(h_{\al})$ be the optimal layer shift function given by Proposition \ref{p.optimal_h_shift}. We shall use the following notation:
\begin{align*}
\Hb_{\al}^{(k)}(y, z) & =\Hb_{\al}^{(k)}(y, z; h_{\al}), \qquad & \bga(y,z) =\bga(y,z;h_{\al}), \\
\bg_{\al}^{(k)}(y,z) &= \bg_{\al}^{(k)}(y,z;h_{\al}) \qquad & \bg_{\al,i}(y,z)= \bg_{\al,i}(y,z;h_{\al}) \\  \bg_{\al,ij}^{(k)}(y,z) & =\bg_{\al,ij}^{(k)}(y,z;h_{\al}) \qquad & \del \bga=(\bg_{\al,1},\dots, \bg_{\al,n}, \bg_{\al}') \\ 
\Hb_{*}(y, z)& =\Hb(y, z; h) \qquad & \Gb_{*}(y, z)=\Gb(y, z; h) \\
g_*(y, z)&=\Hb_*(y, z)+\bg_*(y, z) &  
\end{align*}
Finally, we denote by
\begin{equation}\label{phi-def} 
    \phi=u-g_{*}=u-(\Hb_{*}+\Gb_{*})
\end{equation}
the difference between the real solution and the approximate one. 

\begin{corollary}\label{improved G estimates}
Let $K>0$. For any polynomial $p(z)$ in $\vert z \vert$, $k\geq 0$ and $\sigma\in (0, 1)$, we have  
\begin{equation}\label{Galpha epsilon} 
\|p(z)e^{\sigma |z|}{\bg}_{\alpha}(y, z)\|_{C^{2, \theta}_{\mathrm{loc}}(\Sia \times (-2R, 2R))}\lesssim_{k,p,\sigma} \varepsilon 
\end{equation}
and
\begin{equation} \label{Galpha epsilon 2}
\|\Gb_*(y, z)\|_{C^{2, \theta}_{\mathrm{loc}}(\Sia \times (-2R, 2R))}\lesssim \varepsilon
\end{equation}
as well as 
\begin{align}\label{G* epsilon} 
\begin{split}
    & \|{\bg}(y, z; h) - {\bg}(y, z; 0)\|_{C^{2, \theta}_{\mathrm{loc}}(\Sia \times(-2R, 2R) )} \\
    & \qquad \lesssim \varepsilon \|e^{-\sigma|t_{\alpha+1}(y, z)|}+e^{-\sigma|t_{\alpha-1}(y, z)|}\|_{C^0_{\mathrm{loc}}(\Sia \times (-2R, 2R))}+\varepsilon \|h_{\alpha}\|_{C^{2, \theta}_{\mathrm{loc}}(\Sia)}.
    \end{split} 
\end{align}
\end{corollary}
\begin{proof}
Note first that since the pointwise estimates in \eqref{e.G.est} for $\hat{\mathbb{G}}_\alpha(y, z)$ are valid for every $\sigma \in (0, 1)$, then the same holds for $\bar{\mathbb{G}}_\alpha(y, z)$ (it is useful to note here that derivatives of the cut-off $\zeta_\varepsilon$ are supported in the annulus $(-8 \vert \log \varepsilon \vert, - 4 \vert \log \varepsilon \vert) \cup (4 \vert \log \varepsilon \vert, 8 \vert \log \varepsilon\vert)$). In particular, the size of the interval in the $z$-direction does not affect the estimates. 

To prove \eqref{Galpha epsilon}, fix $\sigma > 0$, and choose $\sigma^\prime \in (0, 1)$ in \eqref{e.G.est} satisfying $\sigma^\prime > \sigma$, so that the term $p(z) e^{\sigma \vert z \vert}$ can be absorbed by $e^{-\sigma^\prime \vert z \vert}$, leaving only the bound $\varepsilon.$ The bound in \eqref{Galpha epsilon 2} follows from \eqref{e.G.est}. Lastly, \eqref{G* epsilon} follows by expressing everything in twisted Fermi coordinates $(y, z)$ with respect to $\Sigma_\alpha$ and combining \eqref{exponential sum2} with the higher order version of the estimate \eqref{e.G.est}. 
\end{proof}
To better estimate the terms $\fwa$ and $\faa$ from Proposition \ref{p.optimal_h_shift} we need to introduce the following quantity (cf. \cite[Subsection 3.2]{wang-weiadv})
\begin{equation} \label{definition of A(r; x)}
    A(r; x) := \max_\alpha \max_{y \in \overline{\Sigma_\alpha \cap B_r(x)}} e^{-D_\alpha(y)}, 
\end{equation}
where the setup is an in Section \ref{section twisted Fermi}, while $D_\alpha(\cdot)$ has been defined in \eqref{definition D alpha}. 

\begin{lemma}\label{omegay estimate}
For $\fwa$ and $\faa$ defined in \eqref{m=m2, w=w2} and \eqref{notation wwaa}, 
we have the following estimate 
\begin{equation}\label{fwa estimate}
\begin{aligned}
       \fwa(y) & \lesssim \varepsilon+\|\phi\|_{C^{2,\theta}(B_2^\alpha (y)\times [-25|\log\varepsilon|, 25|\log\varepsilon|])}+A(40|\log\varepsilon|; (y, 0))^{3/4}\\
    & \qquad +\sup_{\be: |t_{\be}(y,0)|< 17 |\log \ve|}\|h_{\beta}\|_{C^{2,\theta}(B_3^\beta(\Pi_{\beta}(y, 0)))}\\
     \faa(y)&\lesssim A(40|\log\varepsilon|; (y, 0))
\end{aligned}
\end{equation}
\end{lemma}

\begin{proof}
In order to prove the lemma, we start by establishing the following pointwise estimate 
    \begin{align}\label{omegay bound}
\vert \omega_{\alpha}(y) \vert \lesssim \|\phi\|_{C^{2,\theta}(B^\alpha_{2}(y)\times [-10|\log\varepsilon|, 10|\log\varepsilon|])}+\|h_{\alpha}\|_{C^{2,\theta}(B^\alpha_2(y))}+\max_{B^\alpha_2(y)}e^{-\frac{3}{4}D_{\alpha}(y)}+\varepsilon,  
\end{align}
where we recall from \eqref{notation wwaa} the definition of $\omega_\alpha(y)=\omega_\alpha^{(2)}(y)$. To do so, we expand $v_\alpha$ as follows
\begin{align}\label{computation proof omegay bound}
\begin{split}
 v_\alpha(y,z)
 &=\phi(y,z)+\bigl(\Hb_\alpha(y,z;h_\alpha)-\Hb_\alpha(y,z;0)\bigr)+\Gb_*(y,z)\\
 &\quad+\sum_{\beta\neq\alpha}
 \left(\Hb_\beta(y,z;h_\beta)
 +\operatorname{sgn}(\beta-\alpha)(-1)^\beta\right).
\end{split}
\end{align}
where the second equality follows from the definition of $\phi$ in \eqref{phi-def}. To infer the desired bound, we need to estimate the terms $\vert v_\alpha\vert_0$, $\vert v_\alpha\vert_1$, $\vert v_\alpha\vert_2$, and $\vert v_\alpha\vert_{2, \theta}$ multiplied by $(1 + \vert z \vert^3)e^{-\vert z \vert}$. The first order term $\vert v_\alpha\vert_0$ can be dealt with by appealing to Lemma \ref{l.interaction.int} and Lemma \ref{l.interaction.int.2} directly. In particular, the first term on the right-hand side of \eqref{omegay bound} arises from estimating the first term on the right-hand side of \eqref{computation proof omegay bound}. Similarly, by Lemma \ref{l.sum.of.exp}, Remark \ref{r.optimal_h_shift}, the estimates of $\Gb_*$ in Corollary \ref{improved G estimates}   the second, third, and fourth terms in \eqref{omegay bound} come from (in order of appearance) from estimating the second, third, and fourth terms of \eqref{computation proof omegay bound}. Note that we used the estimate $h_\alpha = o(1)$ from Proposition \ref{p.optimal_h_shift}, see Remark \ref{r.optimal_h_shift}. 

The higher order estimates on $v_{\alpha}$ can be estimated in a similar fashion, after differentiating the above expression for $v_\alpha$. Hence we obtain \eqref{omegay bound}.

To conclude, recall the definition of $\mathfrak{w}_{\alpha}(y)$ and $\mathtt{I}_{\al,y}$ in \eqref{notation wwaa}. Combining them with \eqref{omegay bound}, we obtain 
\begin{align}
    \fwa(y)=\fwa^{(2)}(y)\lesssim & \; \varepsilon+\|\phi\|_{C^{2,\theta}(B^\alpha_{2}(y)\times [-25|\log\varepsilon|, 25|\log\varepsilon|])}\\
    &+\sup_{\be: |t_{\be}(y,0)|< 17 |\log \ve|}\left(\|h_{\beta}\|_{C^{2,\theta}(B^{\beta}_2(\Pi_{\beta}(y, 0)))}+\max_{B^{\beta}_2(\Pi_{\beta}(y, 0))}e^{-\frac{3}{4}D_{\beta}}\right). 
\end{align}
Using this, the definition of $\faa(y)$ and the definition of $A(r; x)$ in \eqref{definition of A(r; x)}, we obtain the two estimates in \eqref{fwa estimate}.
\end{proof}

\subsection{Controlling $h$ by $\phi$}
In order to derive the Toda system, we prove the following estimates on $h_{\alpha}$. Compare them with the ones in Remark \ref{r.optimal_h_shift}. In later sections, we will further improve on these. 
\begin{lemma}[First improved estimates]\label{phicontrolhprop}
    For each $\alpha$ and for $y\in \Sigma_{\alpha}\cap B^{+}_{\frac{7}{8}R-2}(0)$, we have
\begin{equation}\label{phicontrolh}
\|h_{\alpha}\|_{C^{2,\theta}(B^\alpha_1(y))}\leq C\|\phi(\cdot, 0)\|_{C^{2,\theta}(B^{\alpha}_1(y, 0))}+C\max_{B^\alpha_1(y)}e^{-D_{\alpha}}+C\varepsilon^2 
    \end{equation}
and 
\begin{equation}\label{horizontalphicontrolh}
\begin{split}
    \|\nabla_{\alpha, 0} h_{\alpha}\|_{C^{1,\theta}(B^\alpha_1(y))}\leq & \;   C\varepsilon^2+C\|\nabla_{\alpha, 0} \phi(\cdot, 0)\|_{C^{1,\theta}(B^{\alpha}_1(y, 0))}+C\varepsilon^{1/6}\max_{B^{\alpha}_1(y)}e^{-D_{\alpha}}\\
    &+C\left(\max_{\beta: |t_{\beta}(y, 0)|\leq 9|\log\varepsilon|}  \|\nabla_{\beta, 0} h_{\beta}\|_{C^{1,\theta}(B^{\beta}_2(\Pi^{\beta}(y,0)))}\right)\max_{B^{\alpha}_1(y)}e^{-D_{\alpha}}.
\end{split}
\end{equation}
If, in addition, $y\in\partial_0\Sigma_\alpha$, then
\begin{equation}\label{neumann h improved relate to phi}
\|\partial_{y_1}h_{\alpha}\|_{C^{1,\theta}(B^{0,\alpha}_1(y))}\leq C\|\phi\|^2_{C^{2,\theta}(B^{\alpha}_{4}(y, 0)\times [-50|\log\varepsilon|, 50|\log\varepsilon|])}+CA(50|\log\varepsilon|;(y, 0))^{3/2}+C\varepsilon^2
\end{equation}
where the constants $C$ in the above estimates are uniform in $\alpha$ and $y$.
\end{lemma}
\begin{proof}
In twisted Fermi coordinates with respect to $\Sigma_{\alpha}$ we have $     \phi(y, 0)=-\Hb_*(y, 0)-\Gb_*(y, 0)$. Since $\|h_{\beta}\|_{C^{2, \theta}}=o(1)$, arguing as in \cite[Lemma 4.6]{wang-weiadv}, we obtain
\begin{equation}
   \|h_{\alpha}\|_{C^{2,\theta}(B^\alpha_1(y))}\lesssim\|\phi\|_{C^{2,\theta}(B^\alpha_1(y, 0))}+\max_{B^\alpha_1(y)}e^{-D_{\alpha}(\cdot)} + \|\mathbb{G}_{*}\|_{C^{2,\theta}(B^\alpha_1(y))}.
\end{equation}
Combining now \eqref{G* epsilon} (from Corollary \ref{improved G estimates}) with $\Gb_{\alpha}(y, 0; 0)=0$, which in turn follows from $\overline{\Gb}_{\alpha}$ being odd, we have
\begin{equation}
    \|\mathbb{G}_{*}\|_{C^{2,\theta}(B^\alpha_1(y))}\lesssim \varepsilon\|h_{\alpha}\|_{C^{2,\theta}(B^\alpha_1(y))}+\varepsilon \max_{B^\alpha_1(y)} e^{-\sigma D_{\alpha}} \qquad \text{where $\sigma\in (1/2, 1)$}.
\end{equation}
The two inequalities above, together with Cauchy's inequality, imply the desired bound \eqref{phicontrolh}. Differentiating the expression $\phi(y, 0)=-\Hb_*(y, 0)-\Gb_*(y, 0)$, appealing to Proposition \ref{p.conseq.almost.parallel} to handle the terms $\nabla_{\alpha, 0} t_\beta$ and to Lemma \ref{l.t.Pi} to have uniform boundedness of $\vert \nabla_{\alpha, 0} \Pi_\beta \vert$, and arguing in a similar fashion, we obtain \eqref{horizontalphicontrolh} as well. Second derivatives $\nabla_{\alpha, 0}^2 h_{\alpha}$ are handled in the exact same way, while for the H\"older terms we need to use the bound $\vert \nabla_{\alpha, 0}^2 t_\beta \vert \lesssim \varepsilon$. 

To obtain \eqref{neumann h improved relate to phi}, recall the estimate on the Neumann derivative of $h_{\alpha}$ from Remark \ref{r.optimal_h_shift}. This, together with the estimates for $\fwa, \faa$ in \eqref{fwa estimate} of Lemma \ref{omegay estimate}, and the improved estimates on $h_{\al}$ just shown in \eqref{phicontrolh}, give \eqref{neumann h improved relate to phi}.
\end{proof}
\begin{remark}
Compared with \cite[(4.13) in Lemma 4.6]{wang-weiadv}, the influence of the error term of order $C\varepsilon^2$ in \eqref{phicontrolh} on our final curvature estimates is negligible. Furthermore, the $\varepsilon^{1/6}$ appearing in estimate \eqref{horizontalphicontrolh} and later sections can be improved to $\varepsilon^{1/4}|\log\varepsilon|^{3/4}$ by appealing to Proposition \ref{p.conseq.almost.parallel}.
\end{remark}

\section{Equation for the error and mean curvature approximation by Toda system}\label{Equation of error and mean curvature approximation by Toda system of layers}
We derive the equation satisfied by the error function $\phi$ defined in \eqref{phi-def}, and then project this equation onto $\Hb'_{\alpha}$ to obtain the Toda system.  We use the geometric estimates of Section \ref{section twisted Fermi}, the heteroclinic estimates in Section \ref{s.approx.toda}, and the optimal shifts of Proposition \ref{p.optimal_h_shift}.  In particular, the shifts are $o(1)$ in $C^{3,\theta}$ as in Remark \ref{r.optimal_h_shift}.  

$B_s^+(y,z)$ denotes the ambient coordinate ball $B_s^+(Y_{\Sigma_\alpha}(y,z))$, intersected with the half-domain. A sheet ball $B_s^\alpha(y)$ is the coordinate ball on $\Sigma_\alpha$ defined in Section \ref{s.approx.toda}. We identify a base point $y\in\Sigma_\alpha$ with the ambient point $Y_{\Sigma_\alpha}(y,0)$.

\subsection{Equation for the error $\phi$}
The following computation is the counterpart of \cite[(4.11)]{wang-weiadv}, with the boundary correction and the twisted-coordinate terms retained.

\begin{proposition}\label{p.s6.error.equation}
	In the $\alpha$-coordinates, the error satisfies
	\begin{align}\label{phieq}
		\begin{split}
			\Delta_{\alpha,z}\phi-&H_\alpha(y,z)\partial_z\phi+\partial_{zz}\phi
			+\mathscr E_\alpha\phi\\
			&={}W'(\Hb_*+\Gb_*+\phi)-\sum_\beta W'(\Hb_\beta)
			-\sum_\beta W''(\Hb_\beta)\Gb_\beta\\
			&\quad+(-1)^\alpha\Hb'_\alpha
			\bigl(H_\alpha(y,z)+\Delta_{\alpha,z}h_\alpha\bigr)
			-\Hb''_\alpha|\nabla_{\alpha,z}h_\alpha|^2\\
			&\quad+\sum_{\beta\ne\alpha}
			\bigl[(-1)^\beta\Hb'_\beta\mathcal R_{\beta,1}
			-\Hb''_\beta\mathcal R_{\beta,2}\bigr]\\
			&\quad-\sum_\beta\xi_\beta-\sum_\beta\mathscr E_\beta\Hb_\beta
			+\sum_\beta\mathcal G_\beta,
		\end{split}
	\end{align}
	where  $\mathscr{E}_{\alpha}(y, z)$ is the second order linear differential operator defined in \eqref{diff operator E=E1+E2}, $H_\alpha(y, z)$ is the mean curvature of $\Sigma_{\alpha, z}$ (cf. Proposition \ref{geometric approximation}), and 
	\begin{align}\label{xiRi}
		\begin{split}
			\xi_\beta(y,z)
			={}&\bar\xi_1\bigl((-1)^\beta(z-h_\beta(y))\bigr)
			+\bar\xi_{2,\beta}\bigl(y,(-1)^\beta(z-h_\beta(y))\bigr),\\
			\mathcal R_{\beta,1}(y,z)
			={}&H_\beta(y,z)+\Delta_{\beta,z}h_\beta(y),\\
			\mathcal R_{\beta,2}(y,z)
			={}&|\nabla_{\beta,z}h_\beta(y)|^2,\\
			\mathcal R_{\beta,3}(y,z)
			={}&-2(-1)^\beta\sum_{i,j=1}^n
			g_\beta^{ij}(y,z)\Gb'_{\beta,i}(y,z)\partial_{y_j}h_\beta(y),\\
			\mathcal R_{\beta,4}(y,z)
			={}&\sum_{i,j=1}^n
			\bigl(g_\beta^{ij}(y,z)-g_\beta^{ij}(y,0)\bigr)\Gb_{\beta,ij}(y,z)\\
			&-\sum_{i,j,k=1}^n
			\bigl(g_\beta^{ij}(y,z)\Gamma^k_{\beta,ij}(y,z)
			-g_\beta^{ij}(y,0)\Gamma^k_{\beta,ij}(y,0)\bigr)
			\Gb_{\beta,k}(y,z).
		\end{split}
	\end{align}
	The correction terms are
	\begin{align}\label{bold Gbeta}
		\begin{split}
			\mathcal G_\beta \ &= \mathcal G_{\beta,1}+\mathcal G_{\beta,2}
			+\mathcal G_{\beta,3}+\mathcal G_{\beta,4}
			+\mathcal G_{\beta,5},\\
			\mathcal G_{\beta,1} &= (-1)^\beta\Gb'_\beta\mathcal R_{\beta,1},\\
			\mathcal G_{\beta,2} &= -\Gb''_\beta\mathcal R_{\beta,2},\\
			\mathcal G_{\beta,3} &= -\mathcal R_{\beta,3},\\
			\mathcal G_{\beta,4} &= -\mathcal R_{\beta,4},\\
			\mathcal G_{\beta,5} &= -\mathscr E_\beta\Gb_\beta.
		\end{split}
	\end{align}
	At every base point to which Proposition \ref{p.optimal_h_shift} applies, one also has
	\begin{equation}\label{phi perp}
		\int_{\mathbb R}\phi(y,z)\Hb'_\alpha(y,z)\,dz=0.
	\end{equation}
\end{proposition}

\begin{proof}
In twisted Fermi coordinates with respect to $\Sigma_\beta$, one has $\widetilde Z=\partial_z-\sum_iT^i\partial_{y_i}$.  Expanding $\Delta=\Delta_{\beta,z}+\widetilde Z^2-H_\beta\widetilde Z$ gives
	\begin{align}\label{e.s6.operator}
		\begin{split}
			\mathscr E_\beta\phi
			={}&-2\sum_iT^i\partial_{y_i z}\phi
			+\sum_{i,j}T^iT^j\partial_{y_i y_j}\phi+\sum_j\Bigl(\sum_iT^i\partial_{y_i}T^j
			-\partial_zT^j+H_\beta T^j\Bigr)\partial_{y_j}\phi.
		\end{split}
	\end{align}
and the coefficients can be estimated by \eqref{e.difference.z.tz}.
The chain rule gives
	\begin{align}\label{e.s6.shift.derivatives}
		\begin{split}
			\partial_{y_i}\Gb_\beta
			& =\Gb_{\beta,i}-(-1)^\beta\Gb'_\beta\partial_{y_i}h_\beta,\\
			\partial_{y_i y_j}\Gb_\beta
			& =\Gb_{\beta,ij}
			-(-1)^\beta\bigl(\Gb'_{\beta,i}\partial_{y_j}h_\beta
			+\Gb'_{\beta,j}\partial_{y_i}h_\beta
			+\Gb'_\beta\partial_{y_i y_j}h_\beta\bigr)\\
			&\qquad+\Gb''_\beta\partial_{y_i}h_\beta\partial_{y_j}h_\beta,\\
			\partial_z\Gb_\beta&=(-1)^\beta\Gb'_\beta,
			\qquad \partial_{zz}\Gb_\beta=\Gb''_\beta.
		\end{split}
	\end{align}
After contraction with $g_\beta^{ij}(y,z)$, the two terms containing
$\Gb'_{\beta,i}\partial_{y_j}h_\beta$ give $\mathcal R_{\beta,3}$, while the difference between the first-argument Laplacians at heights $z$ and $0$ is $\mathcal R_{\beta,4}$.  Hence the equation for $\bar\Gb_\beta$ in \eqref{G-equation} becomes
\begin{align}\label{delta-G}
\begin{split}
\Delta\Gb_\beta
={}&W''(\Hb_\beta)\Gb_\beta+\xi_{2,\beta}
-(-1)^\beta\Gb'_\beta\mathcal R_{\beta,1}
+\Gb''_\beta\mathcal R_{\beta,2}\\
&+\mathcal R_{\beta,3}+\mathcal R_{\beta,4}
+\mathscr E_\beta\Gb_\beta.
\end{split}
\end{align}

For the heteroclinic, we have
	\[
	\partial_{y_i}\Hb_\beta=-(-1)^\beta\Hb'_\beta\partial_{y_i}h_\beta,
	\qquad
	\partial_{y_i y_j}\Hb_\beta
	=\Hb''_\beta\partial_{y_i}h_\beta\partial_{y_j}h_\beta
	-(-1)^\beta\Hb'_\beta\partial_{y_i y_j}h_\beta.
	\]
	Since $\partial_z\Hb_\beta=(-1)^\beta\Hb'_\beta$ and
	$\partial_{zz}\Hb_\beta=\Hb''_\beta$, we obtain
	\begin{align}\label{delta-H}
		\begin{split}
			\Delta\Hb_\beta
			={}&W'(\Hb_\beta)
			+\bar\xi_1\bigl((-1)^\beta(z-h_\beta(y))\bigr)\\
			&-(-1)^\beta\Hb'_\beta\mathcal R_{\beta,1}
			+\Hb''_\beta\mathcal R_{\beta,2}
			+\mathscr E_\beta\Hb_\beta.
		\end{split}
	\end{align}
Finally, by the Allen--Cahn equation
	\begin{equation}\label{unewpde form}
		\Delta u=W'(\Hb_*+\Gb_*+\phi).
	\end{equation}
 Subtracting the sum of \eqref{delta-H} and \eqref{delta-G} from \eqref{unewpde form} gives \eqref{phieq}.  Finally, \eqref{phi perp} is \eqref{e.ortho.H'_al}. 
\end{proof}

The nonlinear terms admit the following exact decomposition
\begin{align}\label{interaction two forms=}
	\begin{split}
		W'(u)-\sum_\beta W'(\Hb_\beta)-\sum_\beta W''(\Hb_\beta)\Gb_\beta
		=W''(\Hb_*)\phi+\mathcal R(\Gb_*+\phi)+\mathcal I_1+\mathcal I_2,
	\end{split}
\end{align}
where
\begin{equation}\label{Rgphidef}
	\mathcal R(\Gb_*+\phi)
	=W'(\Hb_*+\Gb_*+\phi)-W'(\Hb_*)-W''(\Hb_*)(\Gb_*+\phi),
\end{equation}
\begin{equation}\label{defI1-}
	\mathcal I_1=W'(\Hb_*)-\sum_\beta W'(\Hb_\beta),
\end{equation}
and
\begin{equation}\label{defI2-}
	\mathcal I_2=\sum_\beta\mathcal I_{2,\beta}
	=\sum_\beta\bigl(W''(\Hb_*)-W''(\Hb_\beta)\bigr)\Gb_\beta.
\end{equation}
Moreover, by Taylor's formula
\[
\mathcal R(\Gb_*+\phi)
=(\Gb_*+\phi)^2\int_0^1(1-t)
W'''\bigl(\Hb_*+t(\Gb_*+\phi)\bigr)\,dt.
\]
The sums $\Hb_*$ and $\Gb_*$ have uniform $C^{2,\theta}(B_1^+(p))$ bounds for $p\in B^+_{7R/8-2}(0)$, by exponential decay and separation of the sheets. Interior and Neumann estimates for $u$ on $B_2^+(p)$ give the same bound for $u$, and hence for $\phi$, on $B_1^+(p)$.
Hence,
\begin{equation}\label{RGphi}
	\|\mathcal R(\Gb_*+\phi)\|_{C^{k,\theta}(\Omega)}
	\leq C\bigl(\|\Gb_*\|_{C^{k,\theta}(\Omega)}^2
	+\|\phi\|_{C^{k,\theta}(\Omega)}^2\bigr),
	\quad k=0,1,2.
\end{equation}
Constants in these estimates do not depend on the number of sheets.

\begin{lemma}[Interaction terms]\label{interaction_terms}
	For $(y,z)\in\mathcal M_\alpha^4\cap B^+_{7R/8-2}(0)$,
	\begin{equation}\label{I1estimates}
		|\mathcal I_1(y,z)|\leq C\bigl(e^{-D_\alpha(y)}+\varepsilon^2\bigr),
	\end{equation}
	\begin{equation}\label{I2estimates}
		|\mathcal I_2(y,z)|\leq C\bigl(e^{-3D_\alpha(y)/2}+\varepsilon^2\bigr).
	\end{equation}
	For $(y,z)\in\mathcal M_\alpha^3\cap B^+_{7R/8-2}(0)$, the corresponding norm estimates are
	\begin{align}\label{I1estimateslip}
		\|\mathcal I_1\|_{C^\theta(B_1^+(y,z))}
		&\leq C\|\mathcal I_1\|_{\mathrm{Lip}(B_1^+(y,z))}
		\lesssim \sup_{B_2^\alpha(y)}e^{-D_\alpha}+\varepsilon^2,\\
		\|\mathcal I_2\|_{C^\theta(B_1^+(y,z))}
		&\leq C\|\mathcal I_2\|_{\mathrm{Lip}(B_1^+(y,z))}
		\lesssim \sup_{B_2^\alpha(y)}e^{-3D_\alpha/2}+\varepsilon^2.
		\label{I2estimateslip}
	\end{align}
\end{lemma}

\begin{proof}
The proof follows the interaction estimates in \cite[Section 4.2]{wang-weiadv}.
By Proposition \ref{l.projection.error}, separation, and Lemma \ref{l.sum.of.exp}, we have
\begin{equation}\label{e.s6.sum.at.point}
 \sum_{\beta:\,|t_\beta(x)|\leq10|\log\varepsilon|}
 e^{-\sigma|t_\beta(x)|}\leq C_\sigma
 \qquad(\sigma>0).
\end{equation}
The same geometric-series argument bounds the sum over $|t_\beta(x)|>s$ by $C_\sigma e^{-\sigma s}$.
On $\mathcal M_\alpha^\lambda$, the sum over all the sheets $\beta \neq \alpha$ is bounded by
$C_\sigma e^{-\sigma|z|}$.

Let us begin with the term $\mathcal I_1=W'(\Hb_*)-\sum_\beta W'(\Hb_\beta)$. For $\beta\ne\alpha$, exponential decay gives
\begin{align}\label{e.s6.phase.tails}
 |\Hb_\beta-\operatorname{sgn}(\alpha-\beta)(-1)^\beta|
       +|\nabla\Hb_\beta|&\leq Ce^{-|t_\beta|},\notag\\
 |W''(\Hb_\alpha)-1|+|\nabla W''(\Hb_\alpha)|&\leq Ce^{-|z|}.
\end{align}
Hence, Taylor expansion using $W'(\pm1)=0$ and $W''(\pm1)=1$ yields
\begin{equation}\label{e.s6.I1.expansion}
 \begin{split}
 \mathcal I_1={}&[W''(\Hb_\alpha)-1]
 \sum_{\beta\ne\alpha}
 [\Hb_\beta-\operatorname{sgn}(\alpha-\beta)(-1)^\beta]+O\Bigl(\bigl(\sum_{\beta\ne\alpha}e^{-|t_\beta|}\bigr)^2\Bigr).
 \end{split}
\end{equation}
Differentiating the integral Taylor remainder and using \eqref{e.s6.phase.tails} gives the same quadratic bound for the first derivatives. Hence
\begin{equation}\label{e.s6.I1.product}
 |\mathcal I_1|+|\nabla\mathcal I_1|
 \leq Ce^{-|z|}\sum_{\beta\ne\alpha}e^{-|t_\beta|}.
\end{equation}
The terms with $|t_\beta|>2|\log\varepsilon|$ sum to $O(\varepsilon^2)$.  For the other indices, Proposition \ref{l.projection.error} gives $|z|+|t_\beta(y,z)|\geq|t_\beta(y,0)|-o(1)$.
Summing these contributions proves \eqref{I1estimates}.

For $\mathcal I_2=\sum_\beta[W''(\Hb_*)-W''(\Hb_\beta)]\Gb_\beta$, let us first separate the $\alpha$-sheet contribution from the contributions of the other sheets. 

For the $\alpha$ sheet, by Taylor expansion, 
\begin{equation}\label{W''H-W''Halpha}
|W''(\Hb_*)-W''(\Hb_\alpha)|
+|\nabla(W''(\Hb_*)-W''(\Hb_\alpha))|
\leq C\sum_{\beta\ne\alpha}e^{-|t_\beta|},
\end{equation}
while for $\beta\ne\alpha$,
\begin{equation}\label{W''Hbeta-W''pm1}
|W''(\Hb_\beta)-1|+|\nabla W''(\Hb_\beta)|
\leq Ce^{-|t_\beta|}.
\end{equation}
Combining these inequalities with Corollary \ref{improved G estimates}, for any fixed
$\sigma\in[3/4,1)$,
\begin{align}\label{I2 formula}
\begin{split}
|\mathcal I_2|+|\nabla\mathcal I_2|
&\leq C_\sigma\varepsilon\sum_{\beta\ne\alpha}
\bigl(e^{-\sigma|z|-|t_\beta|}
+e^{-|z|-\sigma|t_\beta|}\bigr)\leq C_\sigma\varepsilon\sum_{\beta\ne\alpha}
e^{-\sigma(|z|+|t_\beta|)}.
\end{split}
\end{align}
The second inequality uses the preceding distance splitting; the last is Young's inequality.
This proves \eqref{I2estimates}.

If either distance is larger than $2|\log\varepsilon|$, then
\begin{equation}\label{I2 formula1}
C_\sigma\varepsilon e^{-2\sigma|\log\varepsilon|}
=C_\sigma\varepsilon^{1+2\sigma}
\leq C_\sigma\varepsilon^2.
\end{equation}
Otherwise the distance comparison and Lemma \ref{l.sum.of.exp} give
\begin{equation}\label{I2 formula2}
C_\sigma\varepsilon\sum_{\beta\ne\alpha}
e^{-\sigma|t_\beta(y,0)|}
\leq C_\sigma\varepsilon e^{-\sigma D_\alpha(y)}
\leq C_\sigma\bigl(\varepsilon^2+e^{-2\sigma D_\alpha(y)}\bigr).
\end{equation}
Taking $\sigma=3/4$ proves \eqref{I2estimates} and its first-derivative analogue. 

Fix $p=Y_{\Sigma_\alpha}(y,z)\in\mathcal M_\alpha^3\cap B^+_{7R/8-2}(0)$. Then $B_1^+(p)\subset\mathcal M_\alpha^6$ and $\Pi_\alpha(B_1^+(p))\subset B_2^\alpha(y)$. Apply the preceding pointwise and first-derivative estimates at the points of $B_1^+(p)$, taking the supremum of $e^{-D_\alpha}$ and $e^{-3D_\alpha/2}$ on $B_2^\alpha(y)$. The Lipschitz bound on $B_1^+(p)$ controls its $C^\theta$ norm and proves \eqref{I1estimateslip} and \eqref{I2estimateslip}.
\end{proof}

\begin{lemma}\label{error phi equation}
	Let $(y,z)\in B^+_{7R/8-2}(0)$.  Then
	\begin{equation}\label{error phi equation1}
		\Bigl\|\sum_\beta\xi_\beta\Bigr\|_{C^\theta(B_1^+(y,z))}
		\leq C\varepsilon^3.
	\end{equation}
	For $|z|\leq9|\log\varepsilon|$ one has
	\begin{align}\label{MEphi}
		\|\mathscr E_\alpha\phi\|_{C^\theta(B_1^+(y,z))}
		&\leq C\varepsilon(1+|z|^2)
		\|\phi\|_{C^{2,\theta}(B_1^+(y,z))}\notag\\
		&\leq C\varepsilon^2(1+|z|^4)
		+C\|\phi\|_{C^{2,\theta}(B_1^+(y,z))}^2.
	\end{align}
	The other geometric and correction terms satisfy
	\begin{align}\label{error phi equation2}
		\Bigl\|\sum_\beta\mathscr E_\beta\Hb_\beta\Bigr\|_{C^\theta(B_1^+(y,z))}
		&\leq C\varepsilon
		\max_{\beta:\,|t_\beta(y,z)|\leq10|\log\varepsilon|}
		\|h_\beta\|_{C^{2,\theta}(B_2^\beta(\Pi_\beta(y,z)))}
		+C\varepsilon^2,\\
		\Bigl\|\sum_\beta\mathcal G_\beta\Bigr\|_{C^\theta(B_1^+(y,z))}
		&\leq C\varepsilon
		\max_{\beta:\,|t_\beta(y,z)|\leq10|\log\varepsilon|}
		\|h_\beta\|_{C^{2,\theta}(B_2^\beta(\Pi_\beta(y,z)))}
		+C\varepsilon^2.
		\label{error phi equation3}
	\end{align}
\end{lemma}

\begin{proof}
We estimate the four expressions in the order in which they occur in the statement.

First, by \eqref{xiestimates} and \eqref{e.xi.delta.est}, we have
\[
\|\xi_\beta\|_{C^\theta(B_1^+(y,z))}
\leq C\varepsilon^3e^{-|t_\beta(y,z)|/4}, 
\]
and summing with \eqref{e.s6.sum.at.point} proves \eqref{error phi equation1}.

Next, on $|z|\leq9|\log\varepsilon|$, the coefficients in \eqref{e.s6.operator}, together with their local $C^\theta$ norms, are bounded by $C\varepsilon(1+|z|^2)$, so the product estimate gives the first line of \eqref{MEphi}; Young's inequality gives the second line.
	
For \eqref{error phi equation2}, the proof proceeds as in the previous paragraph. We compute
  \begin{align}
\Bigl|\sum_{\beta}\mathscr{E}_{\alpha }(y, z)(\Hb_{\beta})\Bigr|&\lesssim \sum_{\beta}e^{-\vert t_{\beta}(y, z)\vert} [\varepsilon |z|(|\nabla^2_{\beta}h_{\beta}|+ |\nabla_{\beta}h_{\beta}|)]\\
&\lesssim \varepsilon  \max_{\beta: |t_{\be}(y,0)|< 8 |\log \ve|}(|\nabla^2_{\beta}h_{\beta}|+ |\nabla_{\beta}h_{\beta}|)+\varepsilon^2, 
    \end{align}
where we appealed to \eqref{e.Dpi.bound} and \eqref{exponential sum2} to estimate $D\Pi_\beta$ and to control the exponential sum. The $C^{\theta}$-H\"older norm estimate follows in a similar way. 
	
	Finally, to prove \eqref{error phi equation3}, we estimate the five terms in \eqref{bold Gbeta} before summing them.  Notice that their $C^\theta$ norms are bounded respectively by
	\begin{align}\label{e.s6.G.five.bounds}
		\begin{split}
			\|\mathcal G_{\beta,1}\|_{C^\theta}
			&\leq C_\sigma\varepsilon e^{-\sigma|z|}
			\bigl(\varepsilon+\|h_\beta\|_{C^{2,\theta}}\bigr),\\
			\|\mathcal G_{\beta,2}\|_{C^\theta}
			&\leq C_\sigma\varepsilon e^{-\sigma|z|}
			\|h_\beta\|_{C^{2,\theta}}^2,\\
			\|\mathcal G_{\beta,3}\|_{C^\theta}
			&\leq C_\sigma\varepsilon e^{-\sigma|z|}
			\|h_\beta\|_{C^{1,\theta}},\\
			\|\mathcal G_{\beta,4}\|_{C^\theta}
			+\|\mathcal G_{\beta,5}\|_{C^\theta}
			&\leq C_\sigma\varepsilon^2 e^{-\sigma|z|},
		\end{split}
	\end{align}
where the first line uses $H_\beta=O(\varepsilon)$; the second uses
$\mathcal R_{\beta,2}=|\nabla_{\beta,z}h_\beta|^2$; and the third is the explicit term $\mathcal R_{\beta,3}$ in \eqref{xiRi}. For the fourth line, the coefficients in $\mathcal R_{\beta,4}$ are $O(\varepsilon(1+|z|^2))$ in $C_\mathrm{loc}^\theta$, while the  derivatives of $\bar\Gb_\beta$ are $O(\varepsilon e^{-\sigma|z|})$.  The same reasoning applies to $\mathscr E_\beta\Gb_\beta$, using \eqref{e.s6.shift.derivatives}.  Summing \eqref{e.s6.G.five.bounds} by \eqref{e.s6.sum.at.point} proves \eqref{error phi equation3}.
\end{proof}

\subsection{Mean curvature and the Toda system with rough error}
Substituting \eqref{interaction two forms=} into \eqref{phieq} gives
\begin{align}\label{phieq-interaction form}
	\begin{split}
		\Delta_{\alpha,z}\phi-H_\alpha\partial_z\phi+\partial_{zz}\phi
		={}&W''(\Hb_*)\phi+\mathcal I_1+\mathcal I_2
		+\mathcal R(\Gb_*+\phi)\\
		&+(-1)^\alpha\Hb'_\alpha(H_\alpha+\Delta_{\alpha,z}h_\alpha)
		-\Hb''_\alpha|\nabla_{\alpha,z}h_\alpha|^2\\
		&+\sum_{\beta\ne\alpha}
		\bigl[(-1)^\beta\Hb'_\beta\mathcal R_{\beta,1}
		-\Hb''_\beta\mathcal R_{\beta,2}\bigr]
		-\sum_\beta\xi_\beta\\
		&-\mathscr E_\alpha\phi-\sum_\beta\mathscr E_\beta\Hb_\beta
		+\sum_\beta\mathcal G_\beta,
	\end{split}
\end{align}
where the estimates for $\mathcal{R}(\Gb_*)$, $\mathcal{I}_1$, $\mathcal{I}_2$ and the last four terms are given in \eqref{RGphi}, Lemma \ref{interaction_terms} and Lemma \ref{error phi equation}. 

We can now derive the Toda system satisfied by the layer distances. 

\begin{proposition}[Toda system]\label{toda system rough}
	For sufficiently small $\varepsilon$, on $\Sigma_\alpha\cap B^+_{7R/8-2}(0)$,
	\begin{equation}\label{H+lh}
		H_\alpha(y,0)+\Delta_{\alpha,0}h_\alpha(y)
		=\frac{2}{\sigma_0}
		\bigl(A_{(-1)^{\alpha-1}}^2e^{-|t_{\alpha-1}(y,0)|}
		-A_{(-1)^\alpha}^2e^{-|t_{\alpha+1}(y,0)|}\bigr)
		+E_\alpha^0(y).
	\end{equation}
	Here $A_{\pm1}$ are defined in \eqref{Hsolution asymptotics}.  For every
	$x\in B^+_{6R/7}(0)$ and $0<r<R/60$, the remainder satisfies
	\begin{align}\label{E^0_alphaestimates}
		\begin{split}
			\max_\alpha\|E_\alpha^0\|_{C^\theta(\Sigma_\alpha\cap B_r^+(x))}
			\lesssim{}&\varepsilon^2+\varepsilon^{1/3}A(r+25|\log\varepsilon|;x)
			+A(r+25|\log\varepsilon|;x)^{3/2}\\
			&+\max_\alpha
			\|H_\alpha+\Delta_{\alpha,0}h_\alpha\|_
			{C^\theta(\Sigma_\alpha\cap B^+_{r+25|\log\varepsilon|}(x))}^2\\
			&+\|\phi\|_{C^{2,\theta}(B^+_{r+25|\log\varepsilon|}(x))}^2.
		\end{split}
	\end{align}
\end{proposition}

\begin{proof}
Fix $x\in B^+_{6R/7}(0)$ and $0<r<R/60$. 
Differentiating \eqref{phi perp} twice and taking the horizontal trace gives
\begin{align}\label{0=ONintez}
\begin{split}
 \int_{\mathbb R}\Delta_{\alpha,0}\phi\,\Hb'_\alpha\,dz
 ={}&(-1)^\alpha\Delta_{\alpha,0}h_\alpha
                      \int_{\mathbb R}\phi\Hb''_\alpha\,dz
       -|\nabla_{\alpha,0}h_\alpha|^2
                      \int_{\mathbb R}\phi\Hb'''_\alpha\,dz\\
 &+2(-1)^\alpha\int_{\mathbb R}
       \langle\nabla_{\alpha,0}\phi,\nabla_{\alpha,0}h_\alpha\rangle
                                           \Hb''_\alpha\,dz.
\end{split}
\end{align}
Multiply \eqref{phieq-interaction form} by $\Hb'_\alpha$ and integrate in $z$ with respect to $dz$.  The identity
$\bar\Hb'''=W''(\bar\Hb)\bar\Hb'+\bar\xi'_1$
cancels the principal normal term, up to the cutoff error.

Now set 
\begin{equation}\label{bold{E}}
 \mathbf E_\alpha(y)=\int_{\mathbb R}
 \bigl[\mathcal I_2+\mathcal R(\Gb_*+\phi)-\mathscr E_\alpha\phi
       -\sum_\beta\mathscr E_\beta\Hb_\beta+\sum_\beta\mathcal G_\beta\bigr]
                                                     \Hb'_\alpha\,dz.
\end{equation}
Lemmas \ref{interaction_terms} and \ref{error phi equation}, \eqref{RGphi}, and the correction bounds in Corollary \ref{improved G estimates} give
\begin{equation}\label{bold{E}-estimate}
 \max_\alpha\|\mathbf E_\alpha\|_{C^\theta(\Sigma_\alpha\cap B_r^+(x))}
 \leq C\varepsilon^2+CA(r+25|\log\varepsilon|;x)^{3/2}
                                      +C\|\phi\|_{C^{2,\theta}(B^+_{r+25|\log\varepsilon|}(x))}^2.
\end{equation}
Then, we estimate the remaining terms as in \cite[proof of Lemma 5.1]{wang-weiadv}, using Proposition \ref{geometric approximation} and Lemma \ref{phicontrolhprop} . 
In particular, for the $\beta \neq \alpha$ curvature terms, retain
$H_\beta(\cdot,0)+\Delta_{\beta,0}h_\beta$
as a single $C^\theta$ factor and bound the overlap using
\[
 \int_{\mathbb R}e^{-|z|-|z+t|}\,dz
       =(1+|t|)e^{-|t|}\leq Ce^{-3|t|/4}.
\]
Together with the distance comparison and exponential summability, this bounds their contribution by a constant times $A(r+25|\log\varepsilon|;x)^{3/4}$ multiplied by $\max_\beta\|H_\beta+\Delta_{\beta,0}h_\beta\|_{C^\theta}$.
The potential-difference term is bounded in the same way, with the curvature norm replaced by $\|\phi\|_{C^{2,\theta}}$.  Young's inequality, the cutoff estimates, and \eqref{bold{E}-estimate} therefore give
\begin{align}\label{phi toda integral equation}
 &\max_\alpha\Bigl\|
    (-1)^\alpha(H_\alpha(\cdot,0)+\Delta_{\alpha,0}h_\alpha)
                       \int_{\mathbb R}(\Hb'_\alpha)^2\,dz
                  +\int_{\mathbb R}\mathcal I_1\Hb'_\alpha\,dz
                         \Bigr\|_{C^\theta}\notag\\
 &\qquad\leq C\varepsilon^2+CA(r+25|\log\varepsilon|;x)^{3/2}
       +C\max_\beta\|H_\beta+\Delta_{\beta,0}h_\beta\|_{C^\theta}^2
       +C\|\phi\|_{C^{2,\theta}}^2.
\end{align}
All these estimates hold in $C^\theta$.

In order to estimate the leading interaction, we use the decomposition in \cite[Lemma B.1]{wang-weiadv}, where only the two adjacent linear tails contribute at leading order.  The quadratic tails contribute $CA(r+25|\log\varepsilon|;x)^{3/2}$, while nonadjacent tails contribute $CA(r+25|\log\varepsilon|;x)^2$.  On the common logarithmic strips, the estimates of Section \ref{section twisted Fermi} give $t_\beta(\cdot,z)-t_\beta(\cdot,0)-z=O_{C^1}(\varepsilon^{1/3})$; the first-derivative bound follows by differentiating along the twisted flow.  We can then proceed as in \cite[Lemma 5.1]{wang-weiadv}.
Using \eqref{Hsolution asymptotics} gives
\begin{align}\label{e.s6.projected.interaction}
 &\max_\alpha\left\|
    \int_{\mathbb R}\mathcal I_1\Hb'_\alpha\,dz
     +2(-1)^\alpha\left(
       A_{(-1)^{\alpha-1}}^2e^{-|t_{\alpha-1}(\cdot,0)|}
       -A_{(-1)^\alpha}^2e^{-|t_{\alpha+1}(\cdot,0)|}
                       \right)\right\|_{C^\theta}\notag\\
 &\qquad\leq C\varepsilon^2
       +C\varepsilon^{1/3}A(r+25|\log\varepsilon|;x)
       +CA(r+25|\log\varepsilon|;x)^{3/2}
       +C\|\phi\|_{C^{2,\theta}}^2.
\end{align}

Finally,
$\int_{\mathbb R}(\Hb'_\alpha)^2\,dz
 =\int_{\mathbb R}(\bar\Hb')^2\,dz
 =\sigma_0+O(\varepsilon^3)$,
independently of $y$ and $\alpha$.  Combining \eqref{phi toda integral equation} and \eqref{e.s6.projected.interaction}, dividing by this positive mass and by $(-1)^\alpha$, and using $A^2\leq A^{3/2}$ for $0\leq A\leq1$, proves \eqref{H+lh} with the remainder estimate \eqref{E^0_alphaestimates}.
\end{proof}

\begin{corollary}\label{H+htheta}
	For sufficiently small $\varepsilon$, every $x\in B^+_{6R/7}(0)$ and $0<r<R/60$ satisfy
	\begin{align}\label{H+deltahestimates}
		\begin{split}
			&\max_\alpha
			\|H_\alpha+\Delta_{\alpha,0}h_\alpha\|_{C^\theta(\Sigma_\alpha\cap B_r^+(x))}\\
			&\quad\leq\frac14\bigl[
			\max_\alpha\|H_\alpha+\Delta_{\alpha,0}h_\alpha\|_
			{C^\theta(\Sigma_\alpha\cap B^+_{r+25|\log\varepsilon|}(x))}
			+\|\phi\|_{C^{2,\theta}(B^+_{r+25|\log\varepsilon|}(x))}\bigr]\\
			&\qquad\quad+C\varepsilon^2+CA(r+25|\log\varepsilon|;x).
		\end{split}
	\end{align}
\end{corollary}

\begin{proof}
	The mean value theorem and the uniform bound for the tangential derivatives of the distance functions give
	\[
	\|e^{-|t_{\alpha\pm1}(\cdot,0)|}\|_
	{C^\theta(\Sigma_\alpha\cap B_r^+(x))}
	\leq C A(r+25|\log\varepsilon|;x).
	\]
Combining this with \eqref{H+lh}-\eqref{E^0_alphaestimates} gives
	\begin{align}\label{H+deltaahactheta}
		\begin{split}
			&\max_\alpha\|H_\alpha+\Delta_{\alpha,0}h_\alpha\|_
			{C^\theta(\Sigma_\alpha\cap B_r^+(x))}\leq C\varepsilon^2+CA(r+25|\log\varepsilon|;x)\\
			&\qquad+C\max_\alpha\|H_\alpha+\Delta_{\alpha,0}h_\alpha\|_
			{C^\theta(\Sigma_\alpha\cap B^+_{r+25|\log\varepsilon|}(x))}^2
			+C\|\phi\|_{C^{2,\theta}(B^+_{r+25|\log\varepsilon|}(x))}^2.
		\end{split}
	\end{align}
Since $H_\alpha=O(\varepsilon)$ in $C^\theta$ and $h_\alpha=o(1)$ in $C^{3,\theta}$,	the $C^\theta$ norm of $H_\alpha+\Delta_{\alpha,0}h_\alpha$ on $B^+_{r+25|\log\varepsilon|}(x))$ is $o(1)$ uniformly, therefore we can absorb it and obtain \eqref{H+deltahestimates}.
\end{proof}

\section{Inner-outer gluing estimates for the error function $\phi$}\label{Inner-outer gluing regularity estimates of error}
We estimate the error $\phi=u-\Hb_*-\Gb_*$ by combining the equation and the orthogonality condition obtained in Section \ref{Equation of error and mean curvature approximation by Toda system of layers}.  The argument follows the inner--outer decomposition of \cite[Section 6]{wang-weiadv}.  At the   boundary, however, the approximate solution has nonzero Neumann derivative, which we shall estimate in Section \ref{sec:neumann.phi}.

\begin{proposition}[$C^{2,\theta}$ estimates]\label{aprior C2alpha Hhphi}
	For every fixed $\theta\in(0,1)$ there are $\varepsilon_0>0$ and $C<\infty$, depending only on the standing data and $\theta$, such that, whenever $0<\varepsilon<\varepsilon_0$, $x\in B^+_{6R/7}(0)$ and $0<r<R/60$,
	\begin{align}\label{aprior C2alpha Hhphieq}
		\begin{split}
			&\|\phi\|_{C^{2,\theta}(B_r^+(x))}
			+\max_\alpha\|h_\alpha\|_{C^{2,\theta}(\Sigma_\alpha\cap B_r^+(x))}\\
			&\qquad+\max_\alpha
			\|H_\alpha(\cdot,0)+\Delta_{\alpha,0}h_\alpha\|_{C^\theta(\Sigma_\alpha\cap B_r^+(x))}
			\leq C\varepsilon^2+CA(r+100|\log\varepsilon|^2;x).
		\end{split}
	\end{align}
	The constants are independent of the number of sheets.
\end{proposition}

We shall use the following \emph{qualitative smallness} statement
\begin{equation}\label{e.s7.qualitative}
\begin{split}
 &\|\varphi\|_{C^{2,\theta}(B^+_{7R/8-2}(0))}+\max_\alpha
 \|H_\alpha(\cdot,0)+\Delta_{\alpha,0}h_\alpha\|_
 {C^\theta(\Sigma_\alpha\cap B^+_{7R/8-2}(0))}=o(1)
 \quad\text{as }\varepsilon\to0,
\end{split}
\end{equation}
which follows from Lemma \ref{l.sum.of.exp}, Lemma \ref{l.1D.soln}, \eqref{Galpha epsilon 2}, and the fact that $H_\alpha=O(\varepsilon)$ in $C^\theta$ and $h_\alpha=o(1)$ in $C^{3,\theta}$ (see \cite[Remark 4.2]{wang-weiadv}).

\subsection{Neumann derivative estimates for $\phi$}\label{sec:neumann.phi}

\begin{proposition}[Neumann derivative estimate]\label{improved neumann phi}
For each ambient ball $B_r^+(x)$ used in this section, with $B^+_{r+80|\log\varepsilon|}(x)\subset B^+_{7R/8-2}(0)$,
	\begin{align}\label{neumann C1thetaphi}
		\|\partial_\nu\phi\|_{C^{1,\theta}(\partial_0B_r^+(x))}
		&\lesssim \varepsilon^2
		+\|\phi\|_{C^{2,\theta}(B^+_{r+80|\log\varepsilon|}(x))}^2
		+A(r+80|\log\varepsilon|;x)^{3/2}.
	\end{align}
\end{proposition}
\begin{proof}
	We first compute the contribution of one sheet in its own coordinates.  Set
	\begin{equation}\label{parnuphi=N}
		N^\beta=\partial_\nu(\Hb_\beta+\Gb_\beta),
		\qquad -\partial_\nu\phi=\sum_\beta N^\beta.
	\end{equation}
By \eqref{sign}, we have the exact formula
	\begin{align}\label{Nalpha-G-H}
		\begin{split}
			N^\beta(y,z)
			={}&(-1)^\beta(\Hb_\beta'+\Gb_\beta')
			\bigl(J_\beta-\sum_{j=1}^n
			(\partial_{y_j}h_\beta)\,\partial_\nu y_j\bigr)
			+\sum_{j=1}^n\Gb_{\beta,j}\,\partial_\nu y_j .
		\end{split}
	\end{align}
By definition of $\Gb_{\beta,1}$ and \eqref{G-equation}, we have
	\[
	\Gb_{\beta,1}(y,z)
	=-(-1)^\beta w_\beta(y)(z-h_\beta(y))\Hb_\beta'(y,z)
	+\delta_\beta\bigl(y,(-1)^\beta(z-h_\beta(y))\bigr).
	\]
	Substituting this and $J_\beta=w_\beta z+\widetilde J_\beta$ into \eqref{Nalpha-G-H} yields
	\begin{align}\label{parnuG+H'(+0)}
		\begin{split}
			N^\beta
			={}&(-1)^\beta\Hb_\beta'
			\bigl(\widetilde J_\beta+w_\beta h_\beta-\partial_{y_1}h_\beta\bigr)
			+\delta_\beta\bigl(y,(-1)^\beta(z-h_\beta)\bigr)+(-1)^\beta\Gb_\beta'
			\bigl(J_\beta-\partial_{y_1}h_\beta\bigr)\\
			&+\sum_{j=1}^n(\partial_\nu y_j-\delta_{1j})
			\left[\Gb_{\beta,j}
			-(-1)^\beta(\Hb_\beta'+\Gb_\beta')\partial_{y_j}h_\beta\right].
		\end{split}
	\end{align}
This provides the cancellation of the order $\varepsilon$ term $z\bar\Hb'$. 	
We estimate the four displayed contributions separately. First, 
\[
\|\partial_\nu y_j-\delta_{1j}\|_{C^{1,\theta}}
\leq C\varepsilon(1+|z|^3),
\qquad
\|\widetilde J_\beta\|_{C^{1,\theta}}
\leq C\varepsilon^2(1+|z|^6)
\]
by the flow estimates and Lemma \ref{l.dz(nu)}.

Moreover,
\[
\delta_\beta(y,z)=(1-\zeta_\varepsilon(z))w_\beta(y)z\bar\Hb'(z)
\]
is supported where $|z|\geq4|\log\varepsilon|$ and has local
$C^{1,\theta}$ norm at most $C\varepsilon^2e^{-|z|/2}$.  Finally,
Corollary \ref{improved G estimates} bounds $\Gb_\beta'$ and
$\Gb_{\beta,j}$ by $C\varepsilon e^{-\sigma|z|}$ for any fixed
$\sigma\in(1/2,1)$.  Absorbing the polynomial factors in $z$ into this exponential decay and applying the product inequality to \eqref{parnuG+H'(+0)} gives
\begin{align}\label{parnuG+H'(+)}
\begin{split}
\|N^\beta\|_{C^{1,\theta}(B_1^{0,\beta}(y)\times(z-1,z+1))}
\leq Ce^{-|z|/2}\bigl(&\varepsilon^2
+\|h_\beta\|_{C^{2,\theta}(B_2^\beta(y))}^2\\
&+\|\partial_{y_1}h_\beta\|_{C^{1,\theta}(B_2^{0,\beta}(y))}\bigr).
\end{split}
\end{align}

Apply \eqref{phicontrolh} and \eqref{neumann h improved relate to phi} to the two norms of $h_\beta$ in \eqref{parnuG+H'(+)}.  Their right-hand sides are bounded by
$C\varepsilon^2$, the squared ambient $C^{2,\theta}$ norm of $\phi$, and the corresponding $A^{3/2}$ term, since the $A^2$ produced by squaring \eqref{phicontrolh} is smaller than $A^{3/2}$.  Lastly,
\begin{equation}\label{parnuG+H'(+)beta}
\sum_{\beta:\,|t_\beta|\leq9|\log\varepsilon|}
e^{-|t_\beta|/2}\leq C
\end{equation}
by \eqref{e.s6.sum.at.point}.  Summing \eqref{parnuG+H'(+)} in \eqref{parnuphi=N} proves \eqref{neumann C1thetaphi}. 
\end{proof}

Fix $L>1$, to be chosen large before $\varepsilon$ is chosen small, and retain the notation
\begin{equation}\label{e.s7.regions}
	\begin{split}
		\Omega_\alpha^1&=\{|t_\alpha|<L\}\cap\mathcal M_\alpha^0,\\
		\Omega_\alpha^2&=\{|t_\alpha|>L/2\}\cap\mathcal M_\alpha^0,\\
		\Omega_\alpha^3&=\{|t_\alpha|\leq L/2\}\cap\mathcal M_\alpha^0.
	\end{split}
\end{equation}
By separation, after $L$ is fixed we may assume that the strips $|t_\alpha|\leq2L$ are contained in $\mathcal M_\alpha^0$ on every neighbourhood used below.

\subsection{$C^{1,\theta}$-estimates for $\phi$}

\begin{proposition}[$C^{1,\theta}$ estimates]\label{C1,theta}
	There are constants $c,C>0$ independent of $L$ such that, for each sufficiently large fixed $L$ and all sufficiently small $\varepsilon$ depending on $L$,
	\begin{align}\label{e.s7.first.order}
		\begin{split}
			\|\phi\|_{C^{1,\theta}(B_r^+(x))}
			\leq{}&Ce^{-cL}\bigl(\|\phi\|_{C^{2,\theta}(B^+_{r+95|\log\varepsilon|}(x))}+\max_\alpha\|H_\alpha+\Delta_{\alpha,0}h_\alpha\|_
			{C^\theta(\Sigma_\alpha\cap B^+_{r+95|\log\varepsilon|}(x))}\bigr)\\
			&+C(L)\varepsilon^2+C(L)A(r+95|\log\varepsilon|;x).
		\end{split}
	\end{align}
\end{proposition}

The proof will follow from combining Lemma \ref{outer} and Lemma \ref{C1thetainner -estimates-lemma0} below.

\subsubsection{$C^{1,\theta}$-outer estimates for $\phi$}\label{outer phi}

\begin{lemma}\label{outer}
	For each fixed sufficiently large $L$ and sufficiently small $\varepsilon$,
	\begin{align}\label{c1thetaouter}
		\begin{split}
			\|\phi\|_{C^{1,\theta}(B_r^+(x)\cap\Omega_\alpha^2)}
			\leq{}&Ce^{-cL}\bigl(\|\phi\|_{C^{2,\theta}(B^+_{r+85|\log\varepsilon|}(x))}+\max_\beta\|H_\beta+\Delta_{\beta,0}h_\beta\|_
			{C^\theta(\Sigma_\beta\cap B^+_{r+85|\log\varepsilon|}(x))}\bigr)\\
			&+C\varepsilon^2+CA(r+85|\log\varepsilon|;x)
			+C\|\phi\|_{C^{2,\theta}(B^+_{r+85|\log\varepsilon|}(x))}^2.
		\end{split}
	\end{align}
	The constants in this estimate are independent of $L$ once $L$ is sufficiently large.
\end{lemma}

\begin{proof}
We follow \cite[Lemma 6.2 and (6.2)]{wang-weiadv}. On
$\{\min_\beta|t_\beta|>L/4\}$ set $U_L=W''(\Hb_*+\Gb_*)$.
The exponential estimates and \eqref{e.s6.sum.at.point} give
$U_L\geq1/2$ and $U_L+|\nabla U_L|\leq C$, uniformly for large $L$
and small $\varepsilon$. Equation \eqref{phieq} can therefore be written as
\begin{equation}\label{outerpde}
 \Delta\phi=U_L\phi+E_\alpha^2,
\end{equation}
where we claim
\begin{equation}\label{E2estimate}
\begin{aligned}
 &\|E_\alpha^2\|_{L^\infty(B_r^+(x)\cap\{\min_\beta|t_\beta|>L/4\})}\\
 &\quad\leq C\varepsilon^2+CA(r+80|\log\varepsilon|;x)
       +C\|\phi\|_{C^{2,\theta}(B^+_{r+80|\log\varepsilon|}(x))}^2\\
 &\qquad\quad+Ce^{-cL}\max_\beta
 \|H_\beta+\Delta_{\beta,0}h_\beta\|_
 {C^\theta(\Sigma_\beta\cap B^+_{r+80|\log\varepsilon|}(x))}.
\end{aligned}
\end{equation}
Let us verify this estimate term by term at points $p\in B_r^+(x)\cap\{\min_\beta|t_\beta|>L/4\}$. In the bounds below, the norms of $h_\beta$ are on $B_2^\beta(\Pi_\beta(p))$. The norms of $\phi$ and the quantity $A$ refer to $B^+_{r+80|\log\varepsilon|}(x)$ and $A(r+80|\log\varepsilon|;x)$, respectively; norms of $H_\beta+\Delta_{\beta,0}h_\beta$ are on $\Sigma_\beta\cap B^+_{r+80|\log\varepsilon|}(x)$.
\begin{itemize}
\item
The difference
\[
W'(\Hb_*+\Gb_*+\phi)-W'(\Hb_*+\Gb_*)-U_L\phi
\]
is $O(\phi^2)$.  The expression
\[
\mathcal I_1+\mathcal I_2
+W'(\Hb_*+\Gb_*)-W'(\Hb_*)-W''(\Hb_*)\Gb_*
\]
is bounded by Lemma \ref{interaction_terms}, \eqref{RGphi}, and Corollary \ref{improved G estimates}.  These estimates give
$C\varepsilon^2+CA+C\|\phi\|_{C^{2,\theta}}^2$.
\item
For every $\beta$, Propositions \ref{geometric approximation}
and \ref{prope.comparing laplacian} give
\[
\begin{aligned}
 \Hb'_\beta\mathcal R_{\beta,1}
 ={}&\Hb'_\beta(H_\beta(\cdot,0)+\Delta_{\beta,0}h_\beta)+O\bigl(e^{-|z|/2}
          (\varepsilon^2+\varepsilon\|h_\beta\|_{C^{2,\theta}})\bigr),\\
 |\Hb''_\beta\mathcal R_{\beta,2}|
 \leq{}&Ce^{-|z|/2}\|h_\beta\|_{C^{2,\theta}}^2.
\end{aligned}
\]
Polynomial factors are absorbed by exponential decay. Summation using
\eqref{e.s6.sum.at.point} gives $e^{-cL}$ in the outer region; \eqref{phicontrolh}, and Young's inequality give the last line of \eqref{E2estimate} and the stated $C\varepsilon^2+CA+C\|\phi\|^2$ error.
\item
Lemma \ref{error phi equation} gives
\begin{equation}\label{last four terms}
 \bigl\| -\sum_\beta\xi_\beta-\sum_\beta\mathscr E_\beta\Hb_\beta
                 +\sum_\beta\mathcal G_\beta\bigr\|_{C^\theta}
 \leq C\varepsilon^2+C\varepsilon\max\|h_\beta\|_{C^{2,\theta}}.
\end{equation}
After \eqref{phicontrolh} and Young's inequality, the right-hand side of
\eqref{last four terms} is bounded by the first three terms on the right of \eqref{E2estimate}.
\end{itemize}
This proves \eqref{E2estimate} uniformly in the number of sheets.

For $p\in B_r^+(x)\cap\Omega_\alpha^2$, every sheet distance exceeds
$L/4$ on $B^+_{L/32}(p)$, so we can apply Lemma \ref{global gluing estimates 1}(i) iteratively on the concentric balls $B^+_{1+j\rho}(p)$ with Neumann condition on $\partial_0B^+_{1+j\rho}(p)$. Choose $\rho>1$, independently of $L$, so that $C/\rho\leq1/2$ in \eqref{e.est1.De-1}, and iterate for $\lfloor L/(64\rho)\rfloor$ steps. The remaining norm is $\|\phi\|_{C^{1,\theta}(B^+_{1+\rho\lfloor L/(64\rho)\rfloor}(p))}$, multiplied by at most $Ce^{-cL}$. All these balls lie in $B^+_{L/32}(p)$ for sufficiently large $L$. The contributions of $E_\alpha^2$ and $\partial_\nu\phi$ form a bounded geometric sum. Use \eqref{E2estimate} and Proposition \ref{improved neumann phi} on these balls. The largest data ball is contained in $B^+_{r+1+L/64+80|\log\varepsilon|}(x)\subset B^+_{r+85|\log\varepsilon|}(x)$ for fixed $L$ and small $\varepsilon$. This estimates $\phi$ on every $B_1^+(p)$; restricting to $B_r^+(x)\cap\Omega_\alpha^2$ proves \eqref{c1thetaouter}.
\end{proof}

\subsubsection{$C^{1,\theta}$-inner estimates for $\phi$}\label{inner phi}

\begin{lemma}\label{C1thetainner -estimates-lemma0}
	For each sufficiently large fixed $L$ and all sufficiently small $\varepsilon$,
	\begin{align}\label{c1theta}
		\begin{split}
			\|\phi\|_{C^{1,\theta}(B_r^+(x)\cap\Omega_\alpha^1)}
			\leq{}&Ce^{-cL}\bigl(\|\phi\|_{C^{2,\theta}(B^+_{r+95|\log\varepsilon|}(x))}\\
			&\quad+\max_\beta\|H_\beta+\Delta_{\beta,0}h_\beta\|_
			{C^\theta(\Sigma_\beta\cap B^+_{r+95|\log\varepsilon|}(x))}\bigr)\\
			&\quad+C(L)\bigl(\varepsilon^2+A(r+95|\log\varepsilon|;x)
			+\|\phi\|_{C^{2,\theta}(B^+_{r+95|\log\varepsilon|}(x))}^2\bigr).
		\end{split}
	\end{align}
\end{lemma}
We follow the cutoff--projection argument of \cite[Section 6.2]{wang-weiadv}, and keep track the extra terms coming from the Neumann derivative at the boundary and the twisted coordinates. Fix $p\in B_r^+(x)\cap\Omega_\alpha^1$ and put $y=\Pi_\alpha(p)$. We estimate $\phi$ on $B_1^+(p)\cap\Omega_\alpha^1$ by estimating its pullback to the $\alpha$-coordinates. The base of $Y_{\Sigma_\alpha}^{-1}(B_1^+(p))$ is contained in $B_2^\alpha(y)$.
\begin{claim}[Equation before projection]\label{innerpdewithoutprojection}
In $\Omega^1_\alpha$
	\begin{equation}\label{e.s7.inner.raw}
		(\Delta_{\alpha,0}+\partial_{zz}-W''(\Hb_\alpha))\phi
		=(-1)^\alpha\Hb_\alpha'
		(H_\alpha(\cdot,0)+\Delta_{\alpha,0}h_\alpha)+E_\alpha^1.
	\end{equation}
For $B_r^+(x)$ with $B^+_{r+30|\log\varepsilon|}(x)\subset B^+_{7R/8-2}(0)$
	\begin{align}\label{E^1estimates}
		\begin{split}
			\|E_\alpha^1\|_{L^\infty(B_r^+(x)\cap\{|t_\alpha|<2L\})}
			\leq C(L)\Bigl(&\varepsilon^2+A(r+30|\log\varepsilon|;x)+\|\phi\|_{C^{2,\theta}(B^+_{r+30|\log\varepsilon|}(x))}^2\Bigr).
		\end{split}
	\end{align}
\end{claim}
\begin{proof}
We follow \cite[Lemma 6.3]{wang-weiadv} and estimate the additional correction and boundary terms.
Using \eqref{phieq-interaction form}, the exact error in \eqref{e.s7.inner.raw} is
	\begin{align}\label{e.s7.E1.formula}
		\begin{split}
			E_\alpha^1={}&\mathcal I_1+\mathcal I_2+\mathcal R(\Gb_*+\phi)
			+(W''(\Hb_*)-W''(\Hb_\alpha))\phi\\
			&+(\Delta_{\alpha,0}-\Delta_{\alpha,z})\phi
			+H_\alpha(\cdot,z)\partial_z\phi-\mathscr E_\alpha\phi\\
			&+(-1)^\alpha\Hb_\alpha'
			\bigl(H_\alpha(\cdot,z)-H_\alpha(\cdot,0)
			+(\Delta_{\alpha,z}-\Delta_{\alpha,0})h_\alpha\bigr)
			-\Hb_\alpha''|\nabla_{\alpha,z}h_\alpha|^2\\
			&+\sum_{\beta\ne\alpha}
			\bigl((-1)^\beta\Hb_\beta'\mathcal R_{\beta,1}
			-\Hb_\beta''\mathcal R_{\beta,2}\bigr)
			-\sum_\beta\xi_\beta-\sum_\beta\mathscr E_\beta\Hb_\beta
			+\sum_\beta\mathcal G_\beta.
		\end{split}
	\end{align}
For fixed $L$, separation places $B_r^+(x)\cap\{|t_\alpha|<2L\}$ in $\mathcal M_\alpha^0$ once $\varepsilon$ is small. To obtain the pointwise estimates below, take $p\in B_r^+(x)\cap\{|t_\alpha|<2L\}$. Norms of $\phi$ are taken on $B_1^+(p)$, and norms of $h_\beta$ on $B_2^\beta(\Pi_\beta(p))$; their bounds are then replaced by those on $B^+_{r+30|\log\varepsilon|}(x)$ and its sheet intersections. In these bounds $A$ means $A(r+30|\log\varepsilon|;x)$.
\begin{itemize}
\item The sum
\[
 \mathcal I_1+\mathcal I_2+\mathcal R(\Gb_*+\phi)
 +(W''(\Hb_*)-W''(\Hb_\alpha))\phi
\]
is bounded by $C(L)(\varepsilon^2+A+\|\phi\|_{C^0}^2)$, by Lemma \ref{interaction_terms}, \eqref{RGphi}, and $|\Hb_*-\Hb_\alpha|\leq C(L)e^{-D_\alpha}$.
\item The terms
\[
 (\Delta_{\alpha,0}-\Delta_{\alpha,z})\phi
 +H_\alpha(\cdot,z)\partial_z\phi-\mathscr E_\alpha\phi
\]
and
\[
 (-1)^\alpha\Hb'_\alpha
 \bigl(H_\alpha(\cdot,z)-H_\alpha(\cdot,0)
 +(\Delta_{\alpha,z}-\Delta_{\alpha,0})h_\alpha\bigr)
 -\Hb''_\alpha|\nabla_{\alpha,z}h_\alpha|^2
\]
are bounded by
\[
 C(L)\bigl(\varepsilon\|\phi\|_{C^2}+\varepsilon^2
 +\varepsilon\|h_\alpha\|_{C^2}+\|h_\alpha\|_{C^1}^2\bigr).
\]
Use Propositions \ref{geometric approximation} and \ref{prope.comparing laplacian},
\eqref{e.s6.operator}, then \eqref{phicontrolh} and Young's inequality.
\item The sum
\[
 \sum_{\beta\ne\alpha}
 \bigl((-1)^\beta\Hb'_\beta\mathcal R_{\beta,1}
       -\Hb''_\beta\mathcal R_{\beta,2}\bigr)
\]
is bounded by $C(L)A$: its coefficients are uniformly bounded and the exponentially decaying factors sum to at most $C(L)e^{-D_\alpha}$.
Finally,
\[
 -\sum_\beta\xi_\beta-\sum_\beta\mathscr E_\beta\Hb_\beta
 +\sum_\beta\mathcal G_\beta
\]
is controlled by Lemma \ref{error phi equation} and \eqref{last four terms}, followed by \eqref{phicontrolh} and Young's inequality.
\end{itemize}
\end{proof}

\subsubsection{Cut-off and project}\label{s.cutoff.project} 
Choose $\tilde\xi\in C_c^\infty((-2L,2L))$ with $0\leq\tilde\xi\leq1$, $\tilde\xi=1$ on $[-L,L]$, and $|\tilde\xi'|\leq C_*/L$, $|\tilde\xi''|\leq C_*/L^2$.
Define
\begin{equation}\label{phi_alpha}
	\phi_\alpha(y,z)=\tilde\xi(z)\phi(y,z)-c_\alpha(y)\Hb_\alpha'(y,z),
\end{equation}
where
\begin{equation}\label{def calpha}
	c_\alpha(y)=
	\frac{\int_{\mathbb R}\phi(y,z)(\tilde\xi(z)-1)\Hb_\alpha'(y,z)\,dz}
	{\int_{\mathbb R}(\Hb_\alpha'(y,z))^2\,dz},
\end{equation}
so that
\begin{equation}\label{ON condition}
	\int_{\mathbb R}\phi_\alpha(y,z)\Hb_\alpha'(y,z)\,dz=0.
\end{equation}

\begin{claim}[Estimates for the projection coefficient]\label{calphamanyestimates}
For any $y\in\Sigma_\alpha$, we have
	\begin{align}\label{cphi1}
		|c_\alpha(y)|&\leq Ce^{-L}
		\sup_{L<|z|<9|\log\varepsilon|}|\phi(y,z)|,\\
		\label{cphi2}
		|\nabla_{\alpha,0}c_\alpha(y)|&\leq Ce^{-L}
		\sup_{L<|z|<9|\log\varepsilon|}
		\bigl(|\phi(y,z)|+|\nabla_{\alpha,0}\phi(y,z)|\bigr),\\
		\label{cphi3}
		\|c_\alpha\|_{C^{1,\theta}(B_1^\alpha(y))}
		&\leq Ce^{-L}\|\phi\|_
		{C^{1,\theta}(B_2^\alpha(y)\times[-9|\log\varepsilon|,9|\log\varepsilon|])}.
	\end{align}
For $y\in \partial_0\Sigma_\alpha$,
	\begin{align}\label{cphi4}
		\begin{split}
			(\partial_{y_1}c_\alpha)\int_{\mathbb R}(\Hb_\alpha')^2\,dz
			={}&\int_{\mathbb R}(\partial_{y_1}\phi)(\tilde\xi-1)\Hb_\alpha'\,dz\\
			&-(-1)^\alpha(\partial_{y_1}h_\alpha)
			\int_{\mathbb R}\phi( \tilde\xi-1)\Hb_\alpha''\,dz.
		\end{split}
	\end{align}
	In particular,
	\begin{align}\label{cphi5}
		\begin{split}
			\|\partial_{y_1}c_\alpha\|_{C^\theta(B_1^{0,\alpha}(y))}
			\leq Ce^{-L}\bigl(&\|\partial_{y_1}\phi\|_
			{C^\theta(B_2^{0,\alpha}(y)\times[-9|\log\varepsilon|,9|\log\varepsilon|])}\\
			&+\|\partial_{y_1}h_\alpha\|_{C^\theta(B_2^{0,\alpha}(y))}
			\|\phi\|_{C^\theta(B_2^{0,\alpha}(y)\times[-9|\log\varepsilon|,9|\log\varepsilon|])}\bigr).
		\end{split}
	\end{align}
\end{claim}
\begin{proof}
	The factor $\tilde{\xi}-1$ vanishes for $|z|\leq L$, while
	$|\Hb_\alpha'|+|\Hb_\alpha''|+|\Hb_\alpha'''|\leq Ce^{-|z|}$.
	Integrating this bound gives \eqref{cphi1}.  For any $i=1,\ldots,n$, differentiation of \eqref{def calpha} under the integral gives
	\begin{align}\label{e.s7.c.derivative}
		\begin{split}
			(\partial_{y_i}c_\alpha)\int_{\mathbb R}(\Hb_\alpha')^2\,dz
			={}&\int_{\mathbb R}(\partial_{y_i}\phi)(\tilde\xi-1)\Hb_\alpha'\,dz\\
			&-(-1)^\alpha(\partial_{y_i}h_\alpha)
			\int_{\mathbb R}\phi(\tilde\xi-1)\Hb_\alpha''\,dz.
		\end{split}
	\end{align}
	This proves \eqref{cphi2} and \eqref{cphi4}. Now subtract \eqref{e.s7.c.derivative} at two points $y, \bar y$ and use the H\"older product
inequality. The differences of the correction factors are bounded by
	$Ce^{-|z|}|h_\alpha(y)-h_\alpha(\bar y)|$, and the differences of the other factors are estimated by their $C^\theta$ seminorms.  Integration over $|z|>L$, using the bounded $C^{2,\theta}$ norm of $h_\alpha$, proves \eqref{cphi3} and \eqref{cphi5}.  
\end{proof}

\begin{claim}[Equation after projection]\label{innerpdephialpha}
For base points $y\in\Sigma_\alpha$ on which $c_\alpha$ is defined as above, and for every $z\in\mathbb R$,
	\begin{equation}\label{eqphi_alpha-}
		(\Delta_{\alpha,0}+\partial_{zz}-W''(\Hb_\alpha))\phi_\alpha
		=p_\alpha\Hb_\alpha'+F_\alpha,
	\end{equation}
	where
	\begin{equation}\label{eqpalpha=}
		p_\alpha=(-1)^\alpha(H_\alpha(\cdot,0)+\Delta_{\alpha,0}h_\alpha)
		-\Delta_{\alpha,0}c_\alpha
	\end{equation}
	and
	\begin{align}\label{e.s7.F.exact}
		\begin{split}
			F_\alpha={}&2(-1)^\alpha\Hb_\alpha''
			\langle\nabla_{\alpha,0}c_\alpha,\nabla_{\alpha,0}h_\alpha\rangle
			+(-1)^\alpha c_\alpha(\Delta_{\alpha,0}h_\alpha)\Hb_\alpha''\\
			&-c_\alpha|\nabla_{\alpha,0}h_\alpha|^2\Hb_\alpha'''
			-c_\alpha\bar\xi_1'\bigl((-1)^\alpha(z-h_\alpha)\bigr)\\
			&+2\tilde\xi'\partial_z\phi+\tilde\xi''\phi
			+\tilde\xi E_\alpha^1\\
			&+(-1)^\alpha(H_\alpha(\cdot,0)+\Delta_{\alpha,0}h_\alpha)
			(\tilde\xi-1)\Hb_\alpha'.
		\end{split}
	\end{align}
Moreover,
	\begin{align}\label{eqpalpha=on}
		\begin{split}
			p_\alpha\int_{\mathbb R}(\Hb_\alpha')^2\,dz
			={}&-\int_{\mathbb R}F_\alpha\Hb_\alpha'\,dz
			+\int_{\mathbb R}\phi_\alpha
			\bar\xi_1'\bigl((-1)^\alpha(z-h_\alpha)\bigr)\,dz\\
			&+2(-1)^\alpha
			\left\langle\nabla_{\alpha,0}h_\alpha,
			\int_{\mathbb R}\nabla_{\alpha,0}\phi_\alpha\,
			\Hb_\alpha''\,dz\right\rangle\\
			&+(-1)^\alpha(\Delta_{\alpha,0}h_\alpha)
			\int_{\mathbb R}\phi_\alpha\Hb_\alpha''\,dz
			-|\nabla_{\alpha,0}h_\alpha|^2
			\int_{\mathbb R}\phi_\alpha\Hb_\alpha'''\,dz.
		\end{split}
	\end{align}
On the boundary we have
	\begin{equation}\label{neumannphialpha 0}
		\partial_{y_1}\phi_\alpha
		=\tilde\xi\partial_{y_1}\phi
		-(\partial_{y_1}c_\alpha)\Hb_\alpha'
		+(-1)^\alpha c_\alpha(\partial_{y_1}h_\alpha)\Hb_\alpha''.
	\end{equation}
\end{claim}
\begin{proof}
Recall that
\begin{equation}\label{e.chain.rule}
\begin{gathered}
 \partial_{y_i}\Hb_\alpha'=-(-1)^\alpha h_{\alpha,i}\Hb_\alpha'',\quad
 \Delta_{\alpha,0}\Hb_\alpha'
 =-(-1)^\alpha(\Delta_{\alpha,0}h_\alpha)\Hb_\alpha''
       +|\nabla_{\alpha,0}h_\alpha|^2\Hb_\alpha''',\\
       (\partial_{zz}-W''(\Hb_\alpha))\Hb_\alpha'
=\bar\xi_1'((-1)^\alpha(z-h_\alpha)).
\end{gathered}
\end{equation}
Apply these identities to \eqref{phi_alpha}, grouping
$-(\Delta_{\alpha,0}c_\alpha)\Hb_\alpha'$ into $p_\alpha\Hb_\alpha'$ according to \eqref{eqpalpha=}.
This gives \eqref{eqphi_alpha-}-\eqref{e.s7.F.exact}.
	
	To compute its projection, differentiate \eqref{ON condition} twice and apply \eqref{e.chain.rule} to obtain
	\begin{align*}
		\int_{\mathbb R}(\Delta_{\alpha,0}\phi_\alpha)\Hb_\alpha'\,dz
		={}&2(-1)^\alpha\left\langle\nabla_{\alpha,0}h_\alpha,
		\int_{\mathbb R}\nabla_{\alpha,0}\phi_\alpha\Hb_\alpha''\,dz\right\rangle\\
		&+(-1)^\alpha(\Delta_{\alpha,0}h_\alpha)
		\int_{\mathbb R}\phi_\alpha\Hb_\alpha''\,dz
		-|\nabla_{\alpha,0}h_\alpha|^2
		\int_{\mathbb R}\phi_\alpha\Hb_\alpha'''\,dz.
	\end{align*}
Integration by parts in $z$ contributes $\int\phi_\alpha\bar\xi_1'((-1)^\alpha(z-h_\alpha))\,dz$, and there are no endpoint terms because of the cutoffs. Combining these computations proves \eqref{eqpalpha=on}.  In particular, the horizontal integral is not zero: the orthogonality vector depends on $y$.  Finally, differentiating \eqref{phi_alpha} in $y_1$ gives \eqref{neumannphialpha 0}.
\end{proof}

\begin{proof}[Proof of Lemma \ref{C1thetainner -estimates-lemma0}]
In the following estimates, the norms of $\phi$ on the right are on
$B^+_{r+95|\log\varepsilon|}(x)$, and sheet norms are on its
intersections with $\Sigma_\beta$.

First, we estimate \eqref{e.s7.F.exact}. Notice that 
\begin{itemize}
\item the first two lines of \eqref{e.s7.F.exact}, by \eqref{cphi1}-\eqref{cphi3} and $h_\alpha=o(1)$ in $C^2$, on $B_s^\alpha(y)\times\mathbb R$ by $Ce^{-L}\|\phi\|_{C^{1,\theta}(B^+_{r+95|\log\varepsilon|}(x))}$;
\item the last line of \eqref{e.s7.F.exact} is bounded by $Ce^{-L}\|H_\alpha(\cdot,0)+\Delta_{\alpha,0}h_\alpha\|_{C^0}$;
\item by \eqref{c1thetaouter}, we have 
\[
\|2\tilde\xi'\partial_z\phi+\tilde\xi''\phi\|_{L^\infty}
		\leq{}Ce^{-cL}\bigl(\|\phi\|_{C^{2,\theta}}
		+\max_\beta\|H_\beta+\Delta_{\beta,0}h_\beta\|_{C^\theta}\bigr)+C\bigl(\varepsilon^2+A+\|\phi\|_{C^{2,\theta}}^2\bigr);\]
\item  the remaining term $\tilde\xi E_\alpha^1$ is controlled by Claim \ref{innerpdewithoutprojection}.
\end{itemize}
Therefore, for every $s\geq2$ with $s\leq2+L/64$, take the norm of $F_\alpha$ below on $B_s^\alpha(y)\times\mathbb R$ and obtain 
\begin{align}\label{Festimates}
\begin{split}
 \|F_\alpha\|_{L^\infty}
 \leq{}&Ce^{-cL}\bigl(\|\phi\|_{C^{2,\theta}}
       +\max_\beta\|H_\beta+\Delta_{\beta,0}h_\beta\|_{C^\theta}\bigr)\\
 &+C(L)\bigl(\varepsilon^2+A(r+95|\log\varepsilon|;x)
                 +\|\phi\|_{C^{2,\theta}}^2\bigr).
\end{split}
\end{align}

We now estimate the Neumann derivative. First, let us rewrite \eqref{neumannphialpha 0} as
\begin{align*}
 \partial_{y_1}\phi_\alpha
 ={}&\tilde\xi\,\partial_\nu\phi
       +\tilde\xi(\partial_{y_1}-\partial_\nu)\phi-(\partial_{y_1}c_\alpha)\Hb'_\alpha
       +(-1)^\alpha c_\alpha(\partial_{y_1}h_\alpha)\Hb''_\alpha.
\end{align*}
The first term is controlled by Proposition \ref{improved neumann phi}. On the support of $\tilde\xi$, the second term on the right-hand side has $C^\theta$ norm at most $C(L)\varepsilon\|\phi\|_{C^{1,\theta}}$ by \eqref{e.del.x_1-del.y_1}.  Hence, Young's inequality bounds it by $C(L)(\varepsilon^2+\|\phi\|_{C^{1,\theta}}^2)$. Finally, \eqref{cphi3} bounds the last two terms in $C^\theta((B_s^\alpha(y)\cap\partial_0\Sigma_\alpha)\times\mathbb R)$ by $Ce^{-L}\|\phi\|_{C^{1,\theta}(B^+_{r+95|\log\varepsilon|}(x))}$.

Therefore,
\begin{align}\label{neumannphialpha 1}
\begin{split}
 \|\partial_{y_1}\phi_\alpha\|_{C^\theta}
 \leq{}&Ce^{-L}\|\phi\|_{C^{2,\theta}}\\
 &+C(L)\bigl(\varepsilon^2+A(r+95|\log\varepsilon|;x)^{3/2}
                  +\|\phi\|_{C^{2,\theta}}^2\bigr).
\end{split}
\end{align}

We also note that by the definition of $\phi_{\alpha}$ in \eqref{phi_alpha} $\phi_{\alpha}$ has uniform exponential decay in $z$ for all $y$. By choosing $L$ sufficiently large, we can conclude the proof by arguing as in \cite[page 26,  (6.6)]{wang-weiadv} using the inner gluing estimates for the linearized operator in Appendix \ref{s.gluing.a.apriori}.
Indeed, the equation in Claim \ref{innerpdephialpha} and the orthogonality \eqref{ON condition} allow us to use Proposition \ref{p.lin.ac}. 

Fix $\rho>1$, independently of $L$, so that $C/\rho\leq1/2$ and repeat the fixed-step iteration used in Lemma \ref{outer}, starting at radius 2.
After $\lfloor L/(64\rho)\rfloor$ steps, the norm $\|\phi\|_{C^{1,\theta}}$ has coefficient $2^{-\lfloor L/(64\rho)\rfloor}\leq Ce^{-cL}$, while the $F_\alpha$ and boundary terms on the right-hand side only have a bounded  factor coming from summing a geometric series. By \eqref{phi_alpha} and \eqref{cphi3},
\[
 \|\phi_\alpha\|_{C^{1,\theta}(B_{2+k\rho}^\alpha(y)\times\mathbb R)}
 \leq C\|\phi\|_{C^{1,\theta}(B^+_{r+95|\log\varepsilon|}(x))}.
\]
Substituting \eqref{Festimates}-\eqref{neumannphialpha 1} and using $A^{3/2}\leq A$ gives the required estimate for $\phi_\alpha$ on $B_2^\alpha(y)\times\mathbb R$. On $Y_{\Sigma_\alpha}^{-1}(B_1^+(p)\cap\Omega_\alpha^1)$ we have $|z|<L$, hence $\phi=\phi_\alpha+c_\alpha\Hb'_\alpha$. Equation \eqref{cphi3} bounds the added term there by $Ce^{-L}\|\phi\|_{C^{1,\theta}(B^+_{r+95|\log\varepsilon|}(x))}$. Taking the supremum over $p\in B_r^+(x)\cap\Omega_\alpha^1$ proves \eqref{c1theta}
\end{proof}

\begin{proof}[Proof of Proposition \ref{C1,theta}]
Notice that the inner and outer regions overlap and cover the domain.  Use \eqref{c1thetaouter} and \eqref{c1theta} after enlarging all domains in the right-hand side to $B^+_{r+95|\log\varepsilon|}(x)$.  Their sum gives \eqref{e.s7.first.order} with the additional term $C(L)\|\phi\|_{C^{2,\theta}(B^+_{r+95|\log\varepsilon|}(x))}^2$.
After fixing $L$, choose $\varepsilon$ so small that this term is at most $e^{-cL}\|\phi\|_{C^{2,\theta}(B^+_{r+95|\log\varepsilon|}(x))}$. This proves \eqref{e.s7.first.order}.
\end{proof}

\subsection{$C^{2,\theta}$-estimates for $\phi$}\label{C2thetaphiestimates}
Let us begin by recording the following estimate which will be needed later to apply Schauder theory.  
\begin{claim}\label{schauder ctheta error}
For each ball $B_r^+(x)$ with $B^+_{r+80|\log\varepsilon|}(x)\subset B^+_{7R/8-2}(0)$,
	\begin{align}\label{cthetaerror}
		\begin{split}
			\|\Delta\phi-W''(\Hb_*)\phi\|_{C^\theta(B_r^+(x))}
			\lesssim\; &\varepsilon^2+A(r+80|\log\varepsilon|;x)
			+\|\phi\|_{C^{2,\theta}(B^+_{r+80|\log\varepsilon|}(x))}^2\\
			&+\max_\alpha\|H_\alpha+\Delta_{\alpha,0}h_\alpha\|_
			{C^\theta(\Sigma_\alpha\cap B^+_{r+80|\log\varepsilon|}(x))}^2.
		\end{split}
	\end{align}
\end{claim}
\begin{proof}
We follow \cite[Appendix C, (2)--(6)]{wang-weiadv} for the terms already present in the ordinary Fermi-coordinate calculation, and estimate the boundary-correction and twisted-coordinate contributions separately.

By \eqref{phieq-interaction form}, we have
	\begin{align}\label{e.s7.ambient.residual}
		\begin{split}
			\Delta\phi-W''(\Hb_*)\phi
			={}&\mathcal I_1+\mathcal I_2+\mathcal R(\Gb_*+\phi)
			+\sum_\beta\bigl((-1)^\beta\Hb_\beta'\mathcal R_{\beta,1}
			-\Hb_\beta''\mathcal R_{\beta,2}\bigr)\\
			&-\sum_\beta\xi_\beta-\sum_\beta\mathscr E_\beta\Hb_\beta
			+\sum_\beta\mathcal G_\beta.
		\end{split}
	\end{align}
Split $\mathcal R_{\beta,1}$ into
	\begin{align}\label{e.s6.R1.split}
		\mathcal R_{\beta,1}(y,z)
		={}&H_\beta(y,0)+\Delta_{\beta,0}h_\beta(y)\notag\\
		&+[H_\beta(y,z)-H_\beta(y,0)]
		+(\Delta_{\beta,z}-\Delta_{\beta,0})h_\beta(y).
	\end{align}
 The estimates in
\cite[Appendix C, (2)--(6)]{wang-weiadv} apply to $\mathcal I_1$ and the sums
in $\beta$ of
\[
 \Hb'_\beta[H_\beta(y,z)-H_\beta(y,0)],\qquad
 \Hb'_\beta(\Delta_{\beta,z}-\Delta_{\beta,0})h_\beta,\qquad
 \Hb''_\beta|\nabla_{\beta,z}h_\beta|^2,
\]
using here Lemma \ref{interaction_terms},
Propositions \ref{geometric approximation} and \ref{prope.comparing laplacian},
and \eqref{phicontrolh}. Their $C^\theta(B_r^+(x))$ norms are bounded by
\[
 C\bigl(\varepsilon^2+A(r+80|\log\varepsilon|;x)
 +\|\phi\|_{C^{2,\theta}(B^+_{r+80|\log\varepsilon|}(x))}^2\bigr).
\]
For $H_\beta(y,z)-H_\beta(y,0)$, Proposition \ref{geometric approximation}
controls the tangential derivatives; the Riccati equation and
$\partial_z=\widetilde Z_\beta+T$ control its normal derivative.
Multiplication by $\Hb'_\beta$ absorbs the polynomial factors in $z$.
All sums are controlled uniformly by \eqref{e.s6.sum.at.point}.

For $p\in B_r^+(x)$ and each index with $|t_\beta(p)|<9|\log\varepsilon|$, put $(y,z)=(\Pi_\beta(p),t_\beta(p))$. For the first line of \eqref{e.s6.R1.split}, use
\[
 \|\Hb'_\beta(H_\beta(\cdot,0)+\Delta_{\beta,0}h_\beta)\|_
 {C^\theta(B_1^+(y,z))}
 \leq Ce^{-|z|/2}
 \|H_\beta(\cdot,0)+\Delta_{\beta,0}h_\beta\|_{C^\theta(B_2^\beta(y))}.
\]
Apply \eqref{H+deltaahactheta}, including both terms on its last line,
to the norm on the right, and sum using \eqref{e.s6.sum.at.point}. This bounds
\[
 \Bigl\|\sum_\beta(-1)^\beta\Hb'_\beta
 (H_\beta(\cdot,0)+\Delta_{\beta,0}h_\beta)\Bigr\|_{C^\theta(B_r^+(x))}
\]
by the right-hand side of \eqref{cthetaerror}. Here $\Pi_\beta(p)$ is joined to $p$ by a twisted flow segment with time length less than $9|\log\varepsilon|$. The sheet ball $B_2^\beta(\Pi_\beta(p))$ and the balls required when applying \eqref{H+deltaahactheta} on that set, with its additional $25|\log\varepsilon|$ radius, lie in $B^+_{r+80|\log\varepsilon|}(x)$.

Finally, \eqref{RGphi}, Lemma \ref{interaction_terms},
Corollary \ref{improved G estimates}, and \eqref{last four terms},
combined with \eqref{phicontrolh} and Young's inequality, give
\[
\begin{aligned}
 &\Bigl\|\mathcal R(\Gb_*+\phi)+\mathcal I_2
 -\sum_\beta\xi_\beta-\sum_\beta\mathscr E_\beta\Hb_\beta
 +\sum_\beta\mathcal G_\beta\Bigr\|_{C^\theta(B_r^+(x))}\\
 &\quad\leq C\varepsilon^2+CA(r+80|\log\varepsilon|;x)^{3/2}
        +C\|\phi\|_{C^{2,\theta}(B^+_{r+80|\log\varepsilon|}(x))}^2.
\end{aligned}
\]
Combining these bounds in \eqref{e.s7.ambient.residual}, using $A^{3/2}\leq A$, proves
\eqref{cthetaerror}.
\end{proof}

\begin{lemma}[$C^{2,\theta}$ estimates]\label{C2,theta}
	For sufficiently small $\varepsilon$, the estimate \eqref{aprior C2alpha Hhphieq} holds.  Equivalently, for $x\in B^+_{6R/7}(0)$ and $0<r<R/60$,
	\begin{align}\label{C2thetaphiestimateseq}
		\begin{split}
			&\|\phi\|_{C^{2,\theta}(B_r^+(x))}
			+\max_\alpha\|H_\alpha+\Delta_{\alpha,0}h_\alpha\|_
			{C^\theta(\Sigma_\alpha\cap B_r^+(x))}
			+\max_\alpha\|h_\alpha\|_{C^{2,\theta}(\Sigma_\alpha\cap B_r^+(x))}\\
			&\hspace{35mm}\leq C\varepsilon^2+CA(r+100|\log\varepsilon|^2;x).
		\end{split}
	\end{align}
\end{lemma}
\begin{proof}
The argument follows the proof of \cite[Section 6.3, especially (6.8)]{wang-weiadv}. For each $p\in B_r^+(x)$, apply the interior or Neumann Schauder estimate in Proposition \ref{t.sch} on $B_2^+(p)$, with data on $\partial_0B_2^+(p)$ when that set is nonempty, to estimate $\phi$ on $B_1^+(p)$. Since $B_2^+(p)\subset B^+_{r+2}(x)$, taking the supremum over $p$ gives
	\begin{align}\label{e.s7.schauder.step}
		\begin{split}
			\|\phi\|_{C^{2,\theta}(B_r^+(x))}
			\leq C\bigl(&\|\phi\|_{C^0(B^+_{r+2}(x))}
			+\|\Delta\phi-W''(\Hb_*)\phi\|_{C^\theta(B^+_{r+2}(x))}\\
			&+\|\partial_\nu\phi\|_{C^{1,\theta}(\partial_0B^+_{r+2}(x))}\bigr).
		\end{split}
	\end{align}
The coefficient $W''(\Hb_*)$ has a uniform $C^\theta(B_2^+(p))$ bound for these centres $p$.

Fix $L$ so large that the coefficient $Ce^{-cL}$ in \eqref{e.s7.first.order}, after multiplication by the Schauder constant, is less than $1/128$.
The qualitative convergence supplied by Lemmas \ref{l.1D.soln}-\ref{l.D_al.infty}, Proposition \ref{p.optimal_h_shift}, and Corollary \ref{improved G estimates} implies that, after decreasing $\varepsilon$ with this $L$ fixed,
\[
 C\|\phi\|_{C^{2,\theta}(B^+_{r+100|\log\varepsilon|}(x))}
 +C\max_\alpha\|H_\alpha+\Delta_{\alpha,0}h_\alpha\|_{C^\theta(\Sigma_\alpha\cap B^+_{r+100|\log\varepsilon|}(x))}
 \leq \frac1{128}.
\]

Now apply \eqref{e.s7.first.order} to $\|\phi\|_{C^0}$ in \eqref{e.s7.schauder.step}, Claim \ref{schauder ctheta error} to $\Delta\phi-W''(\Hb_*)\phi$, and Proposition \ref{improved neumann phi} to $\partial_\nu\phi$.  Then enlarge the right-hand domains to $B^+_{r+100|\log\varepsilon|}(x)$.  We have
\begin{equation}\label{e.prelim.bound.7.10}
 C\|\phi\|_{C^{2,\theta}}^2
 \leq \frac1{128}\|\phi\|_{C^{2,\theta}},
 \quad
 C\max_\alpha\|H_\alpha+\Delta_{\alpha,0}h_\alpha\|_{C^\theta}^2
 \leq \frac1{128}\max_\alpha\|H_\alpha+\Delta_{\alpha,0}h_\alpha\|_{C^\theta},
\end{equation}
with the norms of $\phi$ on $B^+_{r+100|\log\varepsilon|}(x)$ and the norms of $H_\alpha+\Delta_{\alpha,0}h_\alpha$ on $\Sigma_\alpha\cap B^+_{r+100|\log\varepsilon|}(x)$.  Since $A^{3/2}\leq A$, we obtain
\begin{align}\label{phic2testima}
\begin{split}
 \|\phi\|_{C^{2,\theta}(B_r^+(x))}
 \leq{}&\frac1{32}\Bigl(
 \|\phi\|_{C^{2,\theta}(B^+_{r+100|\log\varepsilon|}(x))}\\
 &\qquad+\max_\alpha\|H_\alpha+\Delta_{\alpha,0}h_\alpha\|_{C^\theta(\Sigma_\alpha\cap B^+_{r+100|\log\varepsilon|}(x))}\Bigr)\\
 &+C\varepsilon^2+CA(r+100|\log\varepsilon|;x).
\end{split}
\end{align}

Next apply \eqref{H+deltaahactheta} and consider 
\[
 C\max_\alpha\|H_\alpha+\Delta_{\alpha,0}h_\alpha\|_{C^\theta}^2
 \quad\text{and}\quad
 C\|\phi\|_{C^{2,\theta}}^2
\]
where the norms are taken, respectively, on $\Sigma_\alpha\cap B^+_{r+100|\log\varepsilon|}(x)$ and $B^+_{r+100|\log\varepsilon|}(x)$. By \eqref{e.prelim.bound.7.10}, followed by adding the resulting estimate to \eqref{phic2testima}, we obtain
\begin{align}\label{phi+H+deltaaha}
\begin{split}
 &\|\phi\|_{C^{2,\theta}(B_r^+(x))}
 +\max_\alpha\|H_\alpha+\Delta_{\alpha,0}h_\alpha\|_{C^\theta(\Sigma_\alpha\cap B_r^+(x))}\\
 &\quad\leq\frac1{16}\Bigl(
 \|\phi\|_{C^{2,\theta}(B^+_{r+100|\log\varepsilon|}(x))}
 +\max_\alpha\|H_\alpha+\Delta_{\alpha,0}h_\alpha\|_{C^\theta(\Sigma_\alpha\cap B^+_{r+100|\log\varepsilon|}(x))}\Bigr)\\
 &\qquad+C\varepsilon^2+CA(r+100|\log\varepsilon|;x).
\end{split}
\end{align}
This is the analogue of \cite[(6.8)]{wang-weiadv}.

Iterate \eqref{phi+H+deltaaha} $k=\lceil3|\log\varepsilon|/4\rceil$ times, increasing the radius by $100|\log\varepsilon|$ at each step.  For sufficiently small $\varepsilon$,
\[
 16^{-k}\leq\varepsilon^2,
 \qquad
 2+100k|\log\varepsilon|\leq50|\log\varepsilon|^2.
\]
Start the iteration at radius $r+2$. After $k$ steps, the remaining norms are $\|\phi\|_{C^{2,\theta}(B^+_{r+2+100k|\log\varepsilon|}(x))}$ and $\max_\alpha\|H_\alpha+\Delta_{\alpha,0}h_\alpha\|_{C^\theta(\Sigma_\alpha\cap B^+_{r+2+100k|\log\varepsilon|}(x))}$, both uniformly bounded, while $A$ is nondecreasing in its radius and $\sum_{j\geq0}16^{-j}=16/15$.  Hence
\begin{align}\label{e.s7.iteration.result}
\begin{split}
 &\|\phi\|_{C^{2,\theta}(B_{r+2}^+(x))}
 +\max_\alpha\|H_\alpha+\Delta_{\alpha,0}h_\alpha\|_{C^\theta(\Sigma_\alpha\cap B_{r+2}^+(x))}\\
 &\hspace{30mm}\leq C\varepsilon^2+CA(r+100|\log\varepsilon|^2;x).
\end{split}
\end{align}
The intermediate balls $B^+_{r+2+100j|\log\varepsilon|}(x)$, $0\leq j\leq k$, remain in $B^+_{7R/8-2}(0)$ by the fixed outer margin. Finally, apply \eqref{phicontrolh} on $B_1^\alpha(y)$ for each $y\in\Sigma_\alpha\cap B_r^+(x)$.
Equation \eqref{e.s7.iteration.result}, followed by the supremum over $y$, gives the $C^{2,\theta}$ bound for $h_\alpha$ and proves \eqref{C2thetaphiestimateseq}.
\end{proof}

\begin{corollary}[Toda system with the first error improvement]\label{toda system rough error first improved}
	For $x\in B^+_{6R/7}(0)$ and $0<r<R/60$, the Toda system is
	\begin{align}\label{H+lhimproved}
		\begin{split}
			H_\alpha(y,0)+\Delta_{\alpha,0}h_\alpha(y)
			={}&\frac2{\sigma_0}\left(
			A_{(-1)^{\alpha-1}}^2e^{-|t_{\alpha-1}(y,0)|}
			-A_{(-1)^\alpha}^2e^{-|t_{\alpha+1}(y,0)|}\right)+E_\alpha^0(y),
		\end{split}
	\end{align}
	where
	\begin{align}\label{improvetoda1}
		\begin{split}
			\max_\alpha\|E_\alpha^0\|_{C^\theta(\Sigma_\alpha\cap B_r^+(x))}
			\lesssim\; &\varepsilon^2+A(r+120|\log\varepsilon|^2;x)^{3/2}+\varepsilon^{1/3}A(r+120|\log\varepsilon|^2;x).
		\end{split}
	\end{align}
\end{corollary}
\begin{proof}
Equation \eqref{H+lhimproved} is the rough Toda equation \eqref{H+lh}.  It remains to improve the estimate for $E_\alpha^0$.
From \eqref{E^0_alphaestimates},
\begin{align*}
 \max_\alpha\|E_\alpha^0\|_{C^\theta(\Sigma_\alpha\cap B_r^+(x))}
 \leq C\bigl(&\varepsilon^2+\varepsilon^{1/3}A(r+25|\log\varepsilon|;x)
 +A(r+25|\log\varepsilon|;x)^{3/2}\\
 &+\max_\alpha\|H_\alpha+\Delta_{\alpha,0}h_\alpha\|_{C^\theta(\Sigma_\alpha\cap B^+_{r+25|\log\varepsilon|}(x))}^2\\
 &+\|\phi\|_{C^{2,\theta}(B^+_{r+25|\log\varepsilon|}(x))}^2\bigr).
\end{align*}
Apply Lemma \ref{C2,theta} with centre radius $r+25|\log\varepsilon|$.  The sum of the two displayed squared norms is estimated on a ball contained in the one below, since
\[
 r+25|\log\varepsilon|+100|\log\varepsilon|^2
 \leq r+120|\log\varepsilon|^2
\]
for small $\varepsilon$.  Lemma \ref{C2,theta} then gives
\begin{align*}
 &\max_\alpha\|H_\alpha+\Delta_{\alpha,0}h_\alpha\|_{C^\theta(\Sigma_\alpha\cap B^+_{r+25|\log\varepsilon|}(x))}^2
 +\|\phi\|_{C^{2,\theta}(B^+_{r+25|\log\varepsilon|}(x))}^2\\
 &\qquad\leq C\bigl(\varepsilon^2+A(r+120|\log\varepsilon|^2;x)\bigr)^2
 \leq C\varepsilon^2+CA(r+120|\log\varepsilon|^2;x)^{3/2}.
\end{align*}
Monotonicity of $A$ in the radius now gives \eqref{improvetoda1}.  
\end{proof}

For the differentiated form of the Toda system, restrict to the already defined logarithmic strip by setting
\[
\widetilde{\mathcal M}_\alpha^0
=\mathcal M_\alpha^0\cap\{|t_\alpha|<10|\log\varepsilon|\}.
\]
This does not change the definition of $\mathcal M_\alpha^0$.  As usual, write
$h_{\alpha,i}=\partial_{y_i}h_\alpha$ and
$H_{\alpha,i}(y,0)=\partial_{y_i}H_\alpha(y,0)$.

\begin{corollary}[Differentiated Toda system]\label{diff toda}
	For $i=1,\ldots,n$, in $\widetilde{\mathcal M}_\alpha^0$ one has
	\begin{equation}\label{e.s7.differentiated.toda}
		\Hb_\alpha'\bigl(H_{\alpha,i}(y,0)+\Delta_{\alpha,0}h_{\alpha,i}(y)\bigr)
		=\partial_{y_i}\mathcal E_{\alpha,1,i}+\mathcal E_{\alpha,2,i},
	\end{equation}
	where
	\begin{align}\label{Ea12i}
		\begin{split}
			&\|\mathcal E_{\alpha,1,i}\|_{C^\theta(\widetilde{\mathcal M}_\alpha^0\cap B_r^+(x))}
			+\|\mathcal E_{\alpha,2,i}\|_{C^0(\widetilde{\mathcal M}_\alpha^0\cap B_r^+(x))}\\
			&\qquad\leq C\bigl(\varepsilon^2+A(r+120|\log\varepsilon|^2;x)^{3/2}
			+\varepsilon^{1/6}A(r+120|\log\varepsilon|^2;x)\bigr).
		\end{split}
	\end{align}
\end{corollary}
\begin{proof}
	The graphical ordering gives $t_{\alpha-1}(y,0)>0$ and $t_{\alpha+1}(y,0)<0$.  Differentiating \eqref{H+lhimproved} in $y_i$  gives
\begin{align}\label{e.s7.toda.derivative.raw}
\begin{split}
 H_{\alpha,i}+\Delta_{\alpha,0}h_{\alpha,i}
 ={}&-\frac2{\sigma_0}\Bigl[
 A_{(-1)^{\alpha-1}}^2(\partial_{y_i}t_{\alpha-1})e^{-t_{\alpha-1}}
 +A_{(-1)^\alpha}^2(\partial_{y_i}t_{\alpha+1})e^{t_{\alpha+1}}\Bigr]\\
 &+\partial_{y_i}E_\alpha^0
 +\Delta_{\alpha,0}h_{\alpha,i}-\partial_{y_i}\Delta_{\alpha,0}h_\alpha.
\end{split}
\end{align}
Set
\begin{align}\label{e.s7.E12.exact}
 \mathcal E_{\alpha,1,i}={}&\Hb_\alpha' E_\alpha^0,\\
 \mathcal E_{\alpha,2,i}={}&-\frac2{\sigma_0}\Hb_\alpha'
 \Bigl[A_{(-1)^{\alpha-1}}^2(\partial_{y_i}t_{\alpha-1})e^{-t_{\alpha-1}}
 +A_{(-1)^\alpha}^2(\partial_{y_i}t_{\alpha+1})e^{t_{\alpha+1}}\Bigr]\notag\\
 &+\Hb_\alpha'\bigl(\Delta_{\alpha,0}h_{\alpha,i}
 -\partial_{y_i}\Delta_{\alpha,0}h_\alpha\bigr)
 +(-1)^\alpha\Hb_\alpha''h_{\alpha,i}E_\alpha^0.\notag
\end{align}
Since $\partial_{y_i}\Hb_\alpha'=-(-1)^\alpha h_{\alpha,i}\Hb_\alpha''$, the product rule in
$\partial_{y_i}(\Hb_\alpha'E_\alpha^0)$ and \eqref{e.s7.toda.derivative.raw} prove \eqref{e.s7.differentiated.toda}.

In the estimates that follow, ambient norms are on $\widetilde{\mathcal M}_\alpha^0\cap B_r^+(x)$ and sheet norms on $\Pi_\alpha(\widetilde{\mathcal M}_\alpha^0\cap B_r^+(x))$. 
We estimate separately the exponential-distance expression in the first line of $\mathcal E_{\alpha,2,i}$, the commutator $\Delta_{\alpha,0}h_{\alpha,i}-\partial_{y_i}\Delta_{\alpha,0}h_\alpha$, and the factors $\Hb_\alpha'E_\alpha^0$ and $\Hb_\alpha''h_{\alpha,i}E_\alpha^0$.
\begin{itemize}
 \item If $|t_{\alpha\pm1}(y,0)|\leq10|\log\varepsilon|$, Proposition \ref{p.conseq.almost.parallel} gives
 $|\partial_{y_i}t_{\alpha\pm1}(y,0)|\leq C\varepsilon^{1/6}$.
 If $|t_{\alpha\pm1}(y,0)|>10|\log\varepsilon|$, then
 $e^{-|t_{\alpha\pm1}(y,0)|}\leq\varepsilon^{10}$ and
 $|\partial_{y_i}t_{\alpha\pm1}|\leq C$.
 Hence
 \begin{equation}\label{dfftoda1}
 \sum_{\beta=\alpha-1,\alpha+1}
 |\partial_{y_i}t_\beta|e^{-|t_\beta|}
 \leq C\varepsilon^2+C\varepsilon^{1/6}A(r+120|\log\varepsilon|^2;x).
 \end{equation}

 \item By the commutator estimate \eqref{e.exchange.derivative},
 \[
 |\Delta_{\alpha,0}h_{\alpha,i}-\partial_{y_i}\Delta_{\alpha,0}h_\alpha|
 \leq C\varepsilon\bigl(|\nabla_{\alpha,0}h_\alpha|+|\nabla_{\alpha,0}^2h_\alpha|\bigr).
 \]
 Lemma \ref{C2,theta} therefore gives
 \begin{equation}\label{dfftoda2}
 |\Delta_{\alpha,0}h_{\alpha,i}-\partial_{y_i}\Delta_{\alpha,0}h_\alpha|
 \leq C\varepsilon^3+C\varepsilon A(r+120|\log\varepsilon|^2;x)
 \leq C\varepsilon^2+C\varepsilon^{1/6}A(r+120|\log\varepsilon|^2;x).
 \end{equation}

 \item Corollary \ref{toda system rough error first improved} and the bounded $C^\theta$ norms of
 $\Hb_\alpha'$ and $\Hb_\alpha''$ give
 \[
 \|\Hb_\alpha'E_\alpha^0\|_{C^\theta}
 \leq C\bigl(\varepsilon^2+A^{3/2}+\varepsilon^{1/3}A\bigr),
 \]
 where $A=A(r+120|\log\varepsilon|^2;x)$.
 
 Moreover, Proposition \ref{p.optimal_h_shift} gives $\|h_{\alpha,i}\|_{C^0}=o(1)$ uniformly, so
 \[
 \|(-1)^\alpha\Hb_\alpha''h_{\alpha,i}E_\alpha^0\|_{C^0}
 \leq C\bigl(\varepsilon^2+A^{3/2}+\varepsilon^{1/3}A\bigr).
 \]
\end{itemize}
Combining \eqref{dfftoda1}, \eqref{dfftoda2}, and the two displayed estimates for $\Hb_\alpha'E_\alpha^0$ and $\Hb_\alpha''h_{\alpha,i}E_\alpha^0$, and using $\varepsilon^{1/3}\leq\varepsilon^{1/6}$, proves \eqref{Ea12i}.
\end{proof}

\section{Improved horizontal estimates for \texorpdfstring{$\phi_{y_i}$}{phi_yi} and improved Toda system}\label{improved estimates section}
We improve the $C^{1,\theta}$ estimates for the horizontal derivatives of the error.  Throughout this section,
$\phi_{y_i}=\partial_{y_i}\phi$, $1\leq i\leq n$, in the twisted Fermi coordinates with respect to the sheet under consideration. The argument follows the inner--outer decomposition of \cite[Section 7]{wang-weiadv}, with the addition of the boundary terms and the twisted-coordinate errors.

Let $K\geq10$ be fixed. The constants may depend on $K$ and $\theta$, but not on the number of sheets.  By decreasing $\varepsilon_0$, the same estimates hold on every fixed larger logarithmic strip.

\begin{proposition}[Improved horizontal derivative estimates]\label{p.s8.horizontal.derivatives}
For $x\in B^+_{5R/6}(0)$, $0<r<R/70$, and all sufficiently small $\varepsilon$,
\begin{align}\label{e.s8.main}
\max_{\substack{\alpha\\1\leq i\leq n}}
\|\phi_{y_i}\|_{C^{1,\theta}(B_r^+(x)\cap\mathcal M_\alpha^0\cap\{|t_\alpha|\leq K|\log\varepsilon|\})}
\lesssim\; &\varepsilon^2+A(r+200|\log\varepsilon|^2;x)^{3/2}\notag\\
&+\varepsilon^{1/6}A(r+200|\log\varepsilon|^2;x).
\end{align}
In particular, \eqref{e.s8.main} controls every second derivative of $\phi$ containing at least one horizontal derivative.
\end{proposition}

\begin{corollary}[Horizontal derivatives of the shifts]\label{c.s8.shift.derivatives}
Under the hypotheses of Proposition \ref{p.s8.horizontal.derivatives},
\begin{align}\label{e.s8.shift}
\max_\alpha\|\nabla_{\alpha,0}h_\alpha\|_{C^{1,\theta}(\Sigma_\alpha\cap B_r^+(x))}
\lesssim\; &\varepsilon^2+A(r+200|\log\varepsilon|^2;x)^{3/2}\notag\\
&+\varepsilon^{1/6}A(r+200|\log\varepsilon|^2;x).
\end{align}
\end{corollary}

\begin{corollary}[Toda system with improved error]\label{c.s8.improved.toda}
Under the same hypotheses,
\begin{align}\label{e.s8.toda}
H_\alpha(y,0)
={}&\frac{2}{\sigma_0}\Bigl(
A_{(-1)^{\alpha-1}}^2e^{-|t_{\alpha-1}(y,0)|}
-A_{(-1)^\alpha}^2e^{-|t_{\alpha+1}(y,0)|}\Bigr)\notag\\
&+E_\alpha^0(y)-\Delta_{\alpha,0}h_\alpha(y),
\end{align}
where
\begin{align}\label{e.s8.toda.error}
\max_\alpha\|E_\alpha^0-\Delta_{\alpha,0}h_\alpha\|_{C^\theta(\Sigma_\alpha\cap B_r^+(x))}
\lesssim\; &\varepsilon^2+A(r+200|\log\varepsilon|^2;x)^{3/2}\notag\\
&+\varepsilon^{1/6}A(r+200|\log\varepsilon|^2;x).
\end{align}
\end{corollary}

The proofs are completed after we obtain the inner estimate Lemma \ref{l.s8.inner}. Until then, let $x\in B^+_{5R/6}(0)$ and $0<r<R/65$; the smaller range $r<R/70$ in the conclusions leaves room for iteration.

\subsection{Boundary estimates and the equation for $\phi_{y_i}$}

\begin{proposition}[Boundary data for horizontal derivatives]\label{p.s8.boundary.data}
For every $p\in\partial_0B_r^+(x)\cap\mathcal M_\alpha^0\cap\{|t_\alpha|\leq K|\log\varepsilon|\}$,
\begin{align}\label{e.s8.boundary.tangential}
\|\partial_\nu\phi_{y_i}\|_{C^\theta(\partial_0B_1^+(p))}
+\|\partial_{y_1}\phi_{y_i}\|_{C^\theta(\partial_0B_1^+(p))}
\lesssim\; &\varepsilon^2+A(r+120|\log\varepsilon|^2;x)^{3/2}\notag\\
&+\varepsilon^{1/6}A(r+120|\log\varepsilon|^2;x),
\qquad 2\leq i\leq n,
\end{align}
and
\begin{align}\label{e.s8.boundary.normal}
\|\phi_{y_1}\|_{C^{1,\theta}(\partial_0B_1^+(p))}
\lesssim\; &\varepsilon^2+A(r+120|\log\varepsilon|^2;x)^{3/2}\notag\\
&+\varepsilon^{1/6}A(r+120|\log\varepsilon|^2;x).
\end{align}
These are boundary norms: derivatives are taken only in the variables tangent to the boundary.  If $n=1$, only \eqref{e.s8.boundary.normal} is needed.
\end{proposition}

\begin{proof}
Proposition \ref{improved neumann phi} and Proposition \ref{aprior C2alpha Hhphi}, applied on the enlarged balls occurring in \eqref{neumann C1thetaphi}, give
\begin{equation}\label{e.s8.neumann.residual}
\|\partial_\nu\phi\|_{C^{1,\theta}(\partial_0B_{r+2}^+(x))}
\leq C\bigl(\varepsilon^2+A(r+120|\log\varepsilon|^2;x)^{3/2}\bigr).
\end{equation}
Indeed, since $0<A\leq1$ by definition, the square in \eqref{neumann C1thetaphi} is bounded by
$C(\varepsilon^2+A)^2\leq C(\varepsilon^2+A^{3/2})$, with $A$ evaluated at the radius in \eqref{e.s8.neumann.residual}. 
On the boundary, the differentiated flow estimates and \eqref{e.exchange.zneumannderivative0}-\eqref{e.exchange.neumannderivative0} yield
\begin{equation}\label{e.s8.boundary.commutators}
\|\partial_{y_1}-\partial_\nu\|_{C^{1,\theta}(\partial_0B_1^+(p))}
+\sum_{j=2}^n\|[\partial_\nu,\partial_{y_j}]\|_{C^\theta(\partial_0B_1^+(p))}
+\|[\partial_\nu,\partial_z]\|_{C^\theta(\partial_0B_1^+(p))}
\leq C\varepsilon(1+|z|^2).
\end{equation}
Here a norm of a vector field means the corresponding norm of its coefficients.  At $z=0$, the coefficients of $\partial_{y_1}-\partial_\nu$ vanish.  Their $C^0$ bound follows by integrating in $z$; the first and second derivatives of all coefficients in \eqref{e.s8.boundary.commutators} follow from Lemmas \ref{l.Tht} and \ref{l.na.y_i}, together with the rescaled metric bounds.  The mean value theorem gives the stated H\"older seminorms.

For later use, on $|z|\leq K|\log\varepsilon|+2$,
\begin{equation}\label{e.s8.log.absorption}
\varepsilon^3(1+|\log\varepsilon|^6)\leq C\varepsilon^2,
\qquad
\varepsilon(1+|\log\varepsilon|^6)\leq C\varepsilon^{1/6}.
\end{equation}

If $i\geq2$, then $\partial_{y_i}$ is tangent to the boundary and
\[
\partial_\nu\phi_{y_i}
=\partial_{y_i}(\partial_\nu\phi)+[\partial_\nu,\partial_{y_i}]\phi.
\]
The first term is bounded by \eqref{e.s8.neumann.residual}; the second is bounded by \eqref{e.s8.boundary.commutators}, Proposition \ref{aprior C2alpha Hhphi}, and \eqref{e.s8.log.absorption}.  This proves the estimate for $\partial_\nu\phi_{y_i}$ in \eqref{e.s8.boundary.tangential}.  Since
\[
\partial_{y_1}\phi_{y_i}
=\partial_\nu\phi_{y_i}+(\partial_{y_1}-\partial_\nu)\phi_{y_i},
\]
the same estimates prove the second term in \eqref{e.s8.boundary.tangential}.

For $i=1$ we do not differentiate the Neumann condition in the transverse direction.  Instead, on the boundary,
\begin{equation}\label{e.s8.boundary.trace}
\phi_{y_1}=\partial_\nu\phi+(\partial_{y_1}-\partial_\nu)\phi.
\end{equation}
For $2\leq j\leq n$,
\[
\partial_{y_j}\bigl((\partial_{y_1}-\partial_\nu)\phi\bigr)
=(\partial_{y_1}-\partial_\nu)\phi_{y_j}+[\partial_\nu,\partial_{y_j}]\phi,
\]
and
\[
\partial_z\bigl((\partial_{y_1}-\partial_\nu)\phi\bigr)
=(\partial_{y_1}-\partial_\nu)\partial_z\phi+[\partial_\nu,\partial_z]\phi.
\]
Thus \eqref{e.s8.neumann.residual}-\eqref{e.s8.log.absorption} control the full boundary $C^{1,\theta}$ norm of the right-hand side of \eqref{e.s8.boundary.trace}, proving \eqref{e.s8.boundary.normal}.
\end{proof}

We next isolate the derivative of a $C^\theta$ remainder, as in the interior calculation of \cite[Lemma 7.2]{wang-weiadv}. 
\begin{lemma}[Equation for the horizontal derivatives]\label{l.s8.derivative.equation}
Write $H_{\alpha,i}(y,0)=\partial_{y_i}H_\alpha(y,0)$ and $h_{\alpha,i}=\partial_{y_i}h_\alpha$.  Set
\begin{align}\label{e.s8.E0}
\widetilde E_\alpha={}&\widetilde E_{0,\alpha}+\widetilde E_{e,\alpha},\notag\\
\widetilde E_{0,\alpha}={}&H_\alpha(y,z)\partial_z\phi
+(-1)^\alpha\Hb_\alpha'\bigl(H_\alpha(y,z)-H_\alpha(y,0)
 +(\Delta_{\alpha,z}-\Delta_{\alpha,0})h_\alpha\bigr)\notag\\
&-\Hb_\alpha''|\nabla_{\alpha,z}h_\alpha|^2
+\sum_{\beta\neq\alpha}\bigl((-1)^\beta\Hb_\beta'\mathcal R_{\beta,1}
 -\Hb_\beta''\mathcal R_{\beta,2}\bigr)-\sum_\beta\xi_\beta,
\end{align}
where, in the sum over $\beta\neq\alpha$, $\mathcal R_{\beta,1}$ and $\mathcal R_{\beta,2}$ are those in \eqref{xiRi}, and
\begin{align}\label{e.s8.Ee}
\widetilde E_{e,\alpha}={}&(\Delta_{\alpha,0}-\Delta_{\alpha,z})\phi
-\mathscr E_\alpha\phi+\mathcal R(\Gb_*+\phi)+\mathcal I_2
-\sum_\beta\mathscr E_\beta\Hb_\beta+\sum_\beta\mathcal G_\beta.
\end{align}
Then
\begin{equation}\label{e.s8.undifferentiated}
(\Delta_{\alpha,0}+\partial_{zz})\phi
=W''(\Hb_*)\phi+(-1)^\alpha\Hb_\alpha'
\bigl(H_\alpha(y,0)+\Delta_{\alpha,0}h_\alpha\bigr)
+\mathcal I_1+\widetilde E_\alpha,
\end{equation}
and, in the distributional sense,
\begin{equation}\label{e.s8.differentiated}
(\Delta_{\alpha,0}+\partial_{zz}-W''(\Hb_\alpha))\phi_{y_i}
=(-1)^\alpha\Hb_\alpha'
\bigl(H_{\alpha,i}(y,0)+\Delta_{\alpha,0}h_{\alpha,i}\bigr)
+\partial_{y_i}\widetilde E_\alpha+E^i+\partial_{y_i}\mathcal I_1,
\end{equation}
where
\begin{align}\label{e.s8.Ei}
E^i={}&[\Delta_{\alpha,0},\partial_{y_i}]\phi
+\bigl(W''(\Hb_*)-W''(\Hb_\alpha)\bigr)\phi_{y_i}\notag\\
&+W'''(\Hb_*)\phi\sum_\beta(-1)^\beta\Hb_\beta'
\Bigl(\partial_{y_i}t_\beta
-\sum_{j=1}^n(h_{\beta,j}\circ\Pi_\beta)\partial_{y_i}\Pi_\beta^j\Bigr)\notag\\
&+(-1)^\alpha\Hb_\alpha'
\bigl(\partial_{y_i}\Delta_{\alpha,0}h_\alpha-\Delta_{\alpha,0}h_{\alpha,i}\bigr)
-\Hb_\alpha''h_{\alpha,i}
\bigl(H_\alpha(y,0)+\Delta_{\alpha,0}h_\alpha\bigr).
\end{align}
The sum in the second line of \eqref{e.s8.Ei} includes $\beta=\alpha$.
\end{lemma}

\begin{proof}
Move the mean-curvature term and the difference $\Delta_{\alpha,z}-\Delta_{\alpha,0}$ in \eqref{phieq} to the right, and use \eqref{interaction two forms=}.  The coefficient of $(-1)^\alpha\Hb_\alpha'$ at height zero is
$H_\alpha(y,0)+\Delta_{\alpha,0}h_\alpha$; the remaining terms are exactly \eqref{e.s8.E0} and \eqref{e.s8.Ee}.  This proves \eqref{e.s8.undifferentiated}.

Differentiating \eqref{e.s8.undifferentiated} gives
\[
\partial_{y_i}\bigl(W''(\Hb_*)\phi\bigr)
=W''(\Hb_*)\phi_{y_i}+W'''(\Hb_*)\phi\,\partial_{y_i}\Hb_*,
\]
with
\[
\partial_{y_i}\Hb_*
=\sum_\beta(-1)^\beta\Hb_\beta'
\Bigl(\partial_{y_i}t_\beta
-\sum_{j=1}^n(h_{\beta,j}\circ\Pi_\beta)\partial_{y_i}\Pi_\beta^j\Bigr).
\]
Moreover,
\begin{align*}
\partial_{y_i}\Bigl[(-1)^\alpha\Hb_\alpha'
\bigl(H_\alpha(y,0)+\Delta_{\alpha,0}h_\alpha\bigr)\Bigr]
={}&(-1)^\alpha\Hb_\alpha'
\bigl(H_{\alpha,i}(y,0)+\Delta_{\alpha,0}h_{\alpha,i}\bigr)\\
&+(-1)^\alpha\Hb_\alpha'
\bigl(\partial_{y_i}\Delta_{\alpha,0}h_\alpha-\Delta_{\alpha,0}h_{\alpha,i}\bigr)\\
&-\Hb_\alpha''h_{\alpha,i}
\bigl(H_\alpha(y,0)+\Delta_{\alpha,0}h_\alpha\bigr).
\end{align*}
Here $\partial_{y_i}t_\alpha=0$, $\partial_{y_i}\Pi_\alpha^j=\delta_i^j$, and
$\partial_{y_i}\Hb_\alpha'=-(-1)^\alpha h_{\alpha,i}\Hb_\alpha''$.  Commuting $\partial_{y_i}$ past $\Delta_{\alpha,0}$ and subtracting $W''(\Hb_\alpha)\phi_{y_i}$ proves \eqref{e.s8.differentiated}-\eqref{e.s8.Ei}.  
\end{proof}

\begin{lemma}[Bounds for the differentiated equation]\label{l.s8.error.bounds}
On $B_r^+(x)\cap\mathcal M_\alpha^0\cap\{|t_\alpha|\leq K|\log\varepsilon|\}$,
\begin{equation}\label{e.s8.E.bound}
\|\widetilde E_\alpha\|_{C^\theta(\mathcal{M}^{0}_{\alpha}\cap B^+_r(x))}
\lesssim \varepsilon^2+A(r+120|\log\varepsilon|^2;x)^{3/2},
\end{equation}
\begin{align}\label{e.s8.Ei.bound}
\|E^i\|_{C^0(\mathcal{M}^{0}_{\alpha}\cap B^+_r(x))}
\lesssim\;&\varepsilon^2+A(r+120|\log\varepsilon|^2;x)^{3/2}+\varepsilon^{1/6}A(r+120|\log\varepsilon|^2;x),
\end{align}
and
\begin{align}\label{e.s8.I1.derivative.bound}
\|\partial_{y_i}\mathcal I_1\|_{C^0(\mathcal{M}^{0}_{\alpha}\cap B^+_r(x))}
\lesssim\; &\varepsilon^2+A(r+120|\log\varepsilon|^2;x)^2+\varepsilon^{1/6}A(r+120|\log\varepsilon|^2;x).
\end{align}
The same estimates hold on every fixed enlargement $\mathcal M_\alpha^\lambda$, with constants also depending on $\lambda$.
\end{lemma}

\begin{proof}
Fix $p\in B_r^+(x)\cap\mathcal M_\alpha^0\cap\{|t_\alpha|\leq K|\log\varepsilon|\}$. We estimate the expressions below on $B_1^+(p)$. Notice that Proposition \ref{aprior C2alpha Hhphi} gives bounds for $\phi$, $h_\beta$, and
$H_\beta(\cdot,0)+\Delta_{\beta,0}h_\beta$, respectively in $C^{2,\theta}$, $C^{2,\theta}$, and $C^\theta$, by
\begin{equation}\label{e.s8.rough.local}
C\varepsilon^2+CA(r+120|\log\varepsilon|^2;x).
\end{equation}

For the first three linear expressions in \eqref{e.s8.E0}-\eqref{e.s8.Ee},
\[
H_\alpha(y,z)\partial_z\phi,
\qquad (\Delta_{\alpha,0}-\Delta_{\alpha,z})\phi,
\qquad -\mathscr E_\alpha\phi,
\]
Proposition \ref{geometric approximation} and \eqref{e.s6.operator} give the bound
$C\varepsilon(1+|z|^2)\|\phi\|_{C^{2,\theta}}$.  Substitution of \eqref{e.s8.rough.local} uses
\begin{equation}\label{e.s8.young}
\varepsilon(1+|\log\varepsilon|^2)A(r+120|\log\varepsilon|^2;x)
\leq \frac23A(r+120|\log\varepsilon|^2;x)^{3/2}
+\frac13\varepsilon^3(1+|\log\varepsilon|^2)^3,
\end{equation}
whose last term is $O(\varepsilon^2)$.

The two expressions
\[
(-1)^\alpha\Hb_\alpha'
\bigl(H_\alpha(y,z)-H_\alpha(y,0)+(\Delta_{\alpha,z}-\Delta_{\alpha,0})h_\alpha\bigr),
\qquad
-\Hb_\alpha''|\nabla_{\alpha,z}h_\alpha|^2,
\]
have local $C^\theta$ norm at most
$C\varepsilon^2+C\varepsilon\|h_\alpha\|_{C^{2,\theta}}
+C\|h_\alpha\|_{C^{2,\theta}}^2$; the polynomial factors in $z$ are absorbed by the exponential decay of $\Hb_\alpha'$ and $\Hb_\alpha''$.

For $\beta\neq\alpha$, write in the $\beta$-coordinates
\[
\mathcal R_{\beta,1}
=H_\beta(y,0)+\Delta_{\beta,0}h_\beta
+H_\beta(y,z)-H_\beta(y,0)
+(\Delta_{\beta,z}-\Delta_{\beta,0})h_\beta.
\]
On $\mathcal M_\alpha^0$, separation and the closest-sheet property give
\[
\sum_{\beta\neq\alpha}
\bigl(\|\Hb_\beta'\|_{C^\theta}+\|\Hb_\beta''\|_{C^\theta}\bigr)
\leq C A(r+120|\log\varepsilon|^2;x)^{1/2}+C\varepsilon^2.
\]
The contribution
\[
\sum_{\beta\neq\alpha}\Hb_\beta'
\bigl(H_\beta(\cdot,0)+\Delta_{\beta,0}h_\beta\bigr)
\]
is bounded by the product of \eqref{e.s8.rough.local} and the sum above.  The terms containing
$H_\beta(y,z)-H_\beta(y,0)$, $(\Delta_{\beta,z}-\Delta_{\beta,0})h_\beta$, and
$\Hb_\beta''\mathcal R_{\beta,2}$ are estimated as for the two $\alpha$-sheet expressions above.  Their sum is
$C(\varepsilon^2+A^{3/2})$ at the radius in \eqref{e.s8.E.bound}.

Finally,
\[
\|\mathcal R(\Gb_*+\phi)\|_{C^\theta}
\leq C(\varepsilon+\|\phi\|_{C^{2,\theta}})^2,
\quad
\|\mathcal I_2\|_{C^\theta}\leq C(\varepsilon^2+A^{3/2}),
\]
while Lemma \ref{error phi equation} gives
\[
\Bigl\|\sum_\beta\mathscr E_\beta\Hb_\beta\Bigr\|_{C^\theta}
+\Bigl\|\sum_\beta\mathcal G_\beta\Bigr\|_{C^\theta}
\leq C\varepsilon^2+C\varepsilon\max_\beta\|h_\beta\|_{C^{2,\theta}},
\qquad
\Bigl\|\sum_\beta\xi_\beta\Bigr\|_{C^\theta}\leq C\varepsilon^2.
\]
Together with \eqref{e.s8.rough.local}-\eqref{e.s8.young}, these estimates prove \eqref{e.s8.E.bound}.

We next estimate the five expressions in \eqref{e.s8.Ei}.  The commutator estimate \eqref{e.exchange.derivative} gives
\[
\|[\Delta_{\alpha,0},\partial_{y_i}]\phi\|_{C^0}
\leq C\varepsilon\|\phi\|_{C^{2,\theta}}.
\]
The closest-sheet estimate and Taylor expansion give
\[
\|W''(\Hb_*)-W''(\Hb_\alpha)\|_{C^0}
\leq C A(r+120|\log\varepsilon|^2;x)^{1/2}+C\varepsilon^2.
\]
In
\[
W'''(\Hb_*)\phi\sum_\beta(-1)^\beta\Hb_\beta'
\Bigl(\partial_{y_i}t_\beta
-\sum_j(h_{\beta,j}\circ\Pi_\beta)\partial_{y_i}\Pi_\beta^j\Bigr),
\]
the $\beta=\alpha$ summand in parentheses is $-h_{\alpha,i}$, so the corresponding heteroclinic  factor is
$-\Hb_\alpha'h_{\alpha,i}$.  The sum over $\beta\neq\alpha$ is controlled by separation and, when both sheets lie in the logarithmic strip, Proposition \ref{p.conseq.almost.parallel}.  Thus this expression is bounded by
$C(\varepsilon^2+A^{3/2}+\varepsilon^{1/6}A)$.
Moreover,
\[
|\partial_{y_i}\Delta_{\alpha,0}h_\alpha-\Delta_{\alpha,0}h_{\alpha,i}|
\leq C\varepsilon\bigl(|\nabla_{\alpha,0}h_\alpha|+|\nabla^2_{\alpha,0}h_\alpha|\bigr)
\]
by \eqref{e.exchange.derivative}, while
$\Hb_\alpha''h_{\alpha,i}(H_\alpha(y,0)+\Delta_{\alpha,0}h_\alpha)$ is bounded by the product of the two corresponding norms in \eqref{e.s8.rough.local}.  This proves \eqref{e.s8.Ei.bound}.

It remains to differentiate $\mathcal I_1$.  On $\mathcal M_\alpha^0$, the expansion used in the proof of Lemma \ref{interaction_terms} is
\[
\mathcal I_1
=\bigl(W''(\Hb_\alpha)-1\bigr)
\sum_{\beta\neq\alpha}
\bigl(\Hb_\beta-\operatorname{sgn}(\alpha-\beta)(-1)^\beta\bigr)
+O\Bigl(\bigl(\sum_{\beta\neq\alpha}e^{-|t_\beta|}\bigr)^2\Bigr).
\]
If $\partial_{y_i}$ hits $h_\alpha$, $h_\beta$, or the quadratic remainder, the result is bounded by
$C A(r+120|\log\varepsilon|^2;x)(\varepsilon^2+A(r+120|\log\varepsilon|^2;x))$.
If it hits $t_\beta$ or $\Pi_\beta$, Proposition \ref{p.conseq.almost.parallel} gives
$C\varepsilon^{1/6}A(r+120|\log\varepsilon|^2;x)$; the part outside the common logarithmic strip is $O(\varepsilon^2)$.  This proves \eqref{e.s8.I1.derivative.bound}.  This is the analogue of the interaction differentiation in \cite[Section 7]{wang-weiadv}.
\end{proof}

On $|t_\alpha|<10|\log\varepsilon|$, use Corollary \ref{diff toda} in \eqref{e.s8.differentiated}.  On its complement in the fixed $K|\log\varepsilon|$ strip, the factors $\Hb_\alpha'$ and $\Hb_\alpha''$, and so do both terms in \eqref{e.s7.differentiated.toda}.  Hence the same substitution is valid throughout the strip.  Define, for the fixed index $i$,
\begin{equation}\label{e.s8.Ecircle}
\widetilde E_\alpha^\circ
=\widetilde E_\alpha+(-1)^\alpha\mathcal E_{\alpha,1,i},
\end{equation}
\begin{equation}\label{e.s8.Eicircle}
E^{i,\circ}=E^i+\partial_{y_i}\mathcal I_1+(-1)^\alpha\mathcal E_{\alpha,2,i}.
\end{equation}
Then
\begin{equation}\label{e.s8.divergence.equation}
(\Delta_{\alpha,0}+\partial_{zz}-W''(\Hb_\alpha))\phi_{y_i}
=\partial_{y_i}\widetilde E_\alpha^\circ+E^{i,\circ},
\end{equation}
Since $0<A\leq1$, Lemma \ref{l.s8.error.bounds} and \eqref{Ea12i} imply
\begin{align}\label{e.s8.source.bound}
\|\widetilde E_\alpha^\circ\|_{C^\theta}
+\|E^{i,\circ}\|_{C^0}
\lesssim\; &\varepsilon^2+A(r+140|\log\varepsilon|^2;x)^{3/2}\notag\\
&+\varepsilon^{1/6}A(r+140|\log\varepsilon|^2;x).
\end{align}
The $20|\log\varepsilon|^2$ increase absorbs the displacement between the ambient point and the base points used in \eqref{Ea12i}.

We record two points needed in both the inner and outer estimates.  First, at a point written in the $\alpha$-coordinates and then in the $\beta$-coordinates,
\begin{equation}\label{e.s8.coordinate.change}
\phi_{y_i}(y,z)
=\sum_{j=1}^n(\partial_{y_i}\Pi_\beta^j)(y,z)
\phi_{y_j}(\Pi_\beta(y,z),t_\beta(y,z))
+(\partial_{y_i}t_\beta)(y,z)
\partial_z\phi(\Pi_\beta(y,z),t_\beta(y,z)).
\end{equation}
Fix a ball $B_1^+(p)$ on which $|t_\alpha|+|t_\beta|\leq C|\log\varepsilon|$. Then, on $B_1^+(p)$, Proposition \ref{p.conseq.almost.parallel} and the higher derivative bounds in Lemma \ref{l.t.Pi} give
$\|\partial_{y_i}t_\beta\|_{C^{1,\theta}}\leq C\varepsilon^{1/6}$, while
$\|D\Pi_\beta\|_{C^{1,\theta}}\leq C$.  Proposition \ref{aprior C2alpha Hhphi} therefore gives
\begin{equation}\label{e.s8.coordinate.error}
C\|\partial_{y_i}t_\beta\|_{C^{1,\theta}(B_1^+(p))}\|\phi\|_{C^{2,\theta}(B_2^+(p))}
\leq C\bigl(\varepsilon^2+\varepsilon^{1/6}A(r+120|\log\varepsilon|^2;x)\bigr).
\end{equation}
For each centre $p$, choose $\beta$ with $p\in\mathcal M_\beta^0$ and keep this $\beta$ in \eqref{e.s8.coordinate.change} on all of $B_1^+(p)$.

Second, with $\lambda_\alpha(y,0)$ denoting the volume density of $g_{\alpha,0}$,
\begin{equation}\label{e.s8.density.divergence}
\lambda_\alpha(y,0)\partial_{y_i}\widetilde E_\alpha^\circ
=\partial_{y_i}\bigl(\lambda_\alpha(y,0)\widetilde E_\alpha^\circ\bigr)
-(\partial_{y_i}\lambda_\alpha(y,0))\widetilde E_\alpha^\circ.
\end{equation}
For $i\geq2$, the flux in \eqref{e.s8.density.divergence} is tangent to $y_1=0$.  Apply Proposition \ref{t.sch}(a) in the interior, (b) for $i=1$ at the boundary, and (c) for $i\geq2$, with the traces in Proposition \ref{p.s8.boundary.data}.  These estimates use $\|E^{i,\circ}\|_{C^0}$ and $\|\widetilde E_\alpha^\circ\|_{C^\theta}$, not a norm of $\partial_{y_i}\widetilde E_\alpha^\circ$.

\subsection{Improved horizontal outer estimates}

Let $\Omega_\alpha^2$ be the outer region in \eqref{e.s7.regions}.

\begin{lemma}[Outer localization]\label{l.s8.outer}
Fix $\rho>1$, and then take $L>8\rho$.  For every $p\in B_r^+(x)\cap\Omega_\alpha^2\cap\{|t_\alpha|\leq K|\log\varepsilon|\}$, 
\begin{align}\label{e.s8.outer.contraction}
\|\phi_{y_i}\|_{C^{1,\theta}(B_1^+(p))}\leq C\bigl(\varepsilon^2 &+A(r+140|\log\varepsilon|^2;x)^{3/2}+\varepsilon^{1/6}A(r+140|\log\varepsilon|^2;x)\bigr)\\
&+C_*(\rho^{-1}+e^{-c_*L})\max_{\substack{\beta,\,1\leq j\leq n}}
\sup_{\substack{q\in B^+_{r+C_*\rho}(x)\cap\mathcal M_\beta^0\\
|t_\beta(q)|\leq K|\log\varepsilon|+C_*\rho}}
\|\phi_{y_j}\|_{C^{1,\theta}(B_1^+(q))}.
\end{align}
The constants $C_*,c_*$ in the coefficient $C_*(\rho^{-1}+e^{-c_*L})$ are independent of $L$ and $\rho$. The constant $C$ multiplying the right-hand side of \eqref{e.s8.source.bound} may depend on $L$ and $\rho$; once these are fixed, it is independent of $\varepsilon$, the layer indices, and the admissible centres and radii.

\end{lemma}

\begin{proof}
The argument is the outer gluing argument used in \cite[Section 7]{wang-weiadv}. Fix $p\in B_r^+(x)\cap\Omega_\alpha^2\cap\{|t_\alpha|\leq K|\log\varepsilon|\}$ and write $p=Y_{\Sigma_\alpha}(y,z)$.  Fix $p$ as in the statement.  Choose a cutoff
$\chi$ in the $\alpha$-coordinates such that $\chi=1$ on
$Y_{\Sigma_\alpha}^{-1}(B_1^+(p))$, its support is contained in the
product-coordinate ball of radius $2\rho$ centred at
$(\Pi_\alpha(p),t_\alpha(p))$, and
\[
 |D\chi|\leq C_\ast\rho^{-1},\qquad
 |D^2\chi|\leq C_\ast\rho^{-2},\qquad
 \partial_{y_1}\chi=0\quad\hbox{on }\{y_1=0\}.
\]
Since $p\in\Omega_\alpha^2$ and $L>8\rho$, one has $|z|>L/4$ on
$\operatorname{spt}\chi$.  Hence
\[
 \|W''(\Hb_\alpha)-1\|_{C^\theta(\operatorname{spt}\chi)}
 \leq C_\ast e^{-c_\ast L}.
\]
We use the outer gluing estimate (see Appendix \ref{s.gluing.a.apriori})
\begin{equation}
 \|\sfu\|_{C^{1,\theta}}
 \leq C_\ast\bigl(
 \|\snf\|_{C^0}+\|\snG\|_{C^\theta}+\|\psi\|_{C^{1,\theta}}
 \bigr),
\label{e.s8.massive.estimate.full}
\end{equation}
for
$(\Delta_{\alpha,0}+\partial_{zz}-1)\sfu=\snf+\partial_{y_i}\snG^i$.
For $i\geq2$, the last norm in
\eqref{e.s8.massive.estimate.full} is replaced by
$\|\partial_{y_1}\sfu\|_{C^\theta}$.  The constant $C_\ast$ in
\eqref{e.s8.massive.estimate.full} is independent of $L$ and $\rho$.
This is the corresponding consequence of Lemma F.4 for the product metric
$g_{\alpha,0}+dz^2$.  Identity \eqref{e.s8.density.divergence} gives the required
divergence-form interpretation at $y_1=0$.

Apply \eqref{e.s8.massive.estimate.full} to $\sfu=\chi\phi_{y_i}$.  From
\eqref{e.s8.divergence.equation},
\begin{align*}
(\Delta_{\alpha,0}+\partial_{zz}-1)(\chi\phi_{y_i})
={}&\partial_{y_i}(\chi\widetilde E_\alpha^\circ)
   +\chi E^{i,\circ}-(\partial_{y_i}\chi)\widetilde E_\alpha^\circ\\
&+\chi\bigl(W''(\Hb_\alpha)-1\bigr)\phi_{y_i}
  +2\langle\nabla\chi,\nabla\phi_{y_i}\rangle
  +(\Delta_{\alpha,0}+\partial_{zz})\chi\,\phi_{y_i}.
\end{align*}
By \eqref{e.s8.source.bound}, the first line and the corresponding boundary
data contribute at most
\[
 C_{L,\rho}\Bigl(
 \varepsilon^2
 +A(r+140|\log\varepsilon|^2;x)^{3/2}
 +\varepsilon^{1/6}A(r+140|\log\varepsilon|^2;x)
 \Bigr).
\]
The second line is bounded by
\[
 C_\ast\bigl(\rho^{-1}+e^{-c_\ast L}\bigr)
 \|\phi_{y_i}\|_{C^{1,\theta}
 (Y_{\Sigma_\alpha}(\operatorname{spt}\chi))}.
\]
On $y_1=0$, Proposition \ref{p.s8.boundary.data} controls
$\chi\phi_{y_1}$ when $i=1$ and
\[
 \partial_{y_1}(\chi\phi_{y_i})
 =\chi\partial_{y_1}\phi_{y_i},\qquad i\geq2.
\]

Finally,
$Y_{\Sigma_\alpha}(\operatorname{spt}\chi)\subset B_{C_\ast\rho}^+(p)$.
For each $q$ in this set, choose a sheet $\Sigma_\beta$ closest to $q$ and
use its coordinates throughout $B_1^+(q)$.  Equations
\eqref{e.s8.coordinate.change}-\eqref{e.s8.coordinate.error} give
\begin{align*}
&\|\phi_{y_i}\|_{C^{1,\theta}
 (Y_{\Sigma_\alpha}(\operatorname{spt}\chi))}\\
&\quad\leq C_\ast
 \max_{\substack{\beta\\1\leq j\leq n}}
 \sup_{\substack{
 q\in B_{r+C_\ast\rho}^+(x)\cap\mathcal M_\beta^0\\
 |t_\beta(q)|\leq K|\log\varepsilon|+C_\ast\rho}}
 \|\phi_{y_j}\|_{C^{1,\theta}(B_1^+(q))}\\
&\qquad+C_{L,\rho}\Bigl(
 \varepsilon^2+A(r+140|\log\varepsilon|^2;x)^{3/2}
 +\varepsilon^{1/6}A(r+140|\log\varepsilon|^2;x)
 \Bigr).
\end{align*}
Substitution in \eqref{e.s8.massive.estimate.full} proves \eqref{e.s8.outer.contraction}.
\end{proof}

\subsection{Improved horizontal inner estimates}
By taking derivatives of the orthogonality condition \eqref{phi perp} is not zero we get \begin{equation}\label{e.s8.first.moment}
\int_{\mathbb R}\phi_{y_i}\Hb_\alpha'\,dz
=(-1)^\alpha h_{\alpha,i}\int_{\mathbb R}\phi\Hb_\alpha''\,dz,
\end{equation}
\begin{align}\label{e.s8.second.moment}
\int_{\mathbb R}\phi_{y_i y_j}\Hb_\alpha'\,dz
={}&(-1)^\alpha h_{\alpha,i}\int_{\mathbb R}\phi_{y_j}\Hb_\alpha''\,dz
+(-1)^\alpha h_{\alpha,j}\int_{\mathbb R}\phi_{y_i}\Hb_\alpha''\,dz\notag\\
&+(-1)^\alpha h_{\alpha,ij}\int_{\mathbb R}\phi\Hb_\alpha''\,dz
-h_{\alpha,i}h_{\alpha,j}\int_{\mathbb R}\phi\Hb_\alpha'''\,dz.
\end{align}
They follow by differentiating \eqref{phi perp} once and twice and using
$\partial_{y_i}\Hb_\alpha'=-(-1)^\alpha h_{\alpha,i}\Hb_\alpha''$ and
$\partial_{y_j}\Hb_\alpha''=-(-1)^\alpha h_{\alpha,j}\Hb_\alpha'''$.
The integrals are supported in $|z|\leq8|\log\varepsilon|+1$.  For each $y\in\Sigma_\alpha\cap B_r^+(x)$, the norms are taken on $B_1^\alpha(y)$.  Proposition \ref{aprior C2alpha Hhphi}, the product rule, and exponential decay give
\begin{align}\label{e.s8.moment.bound}
\Bigl\|\int_{\mathbb R}\phi_{y_i}\Hb_\alpha'\,dz\Bigr\|_{C^{1,\theta}}
+\Bigl\|\int_{\mathbb R}\phi_{y_i y_j}\Hb_\alpha'\,dz\Bigr\|_{C^\theta}
\leq C\bigl(\varepsilon^2+A(r+120|\log\varepsilon|^2;x)^2\bigr).
\end{align}
Indeed, every term in \eqref{e.s8.first.moment}-\eqref{e.s8.second.moment} is bounded by
$C\|h_\alpha\|_{C^{2,\theta}}\|\phi\|_{C^{2,\theta}}$, except the last term in \eqref{e.s8.second.moment}, which has an additional factor $\|\nabla h_\alpha\|_{C^0}$.

Let $\Omega_\alpha^3$ be the inner region in \eqref{e.s7.regions}.

\begin{lemma}[Inner localization]\label{l.s8.inner}
After $\rho$ and $L$ have been fixed and $\varepsilon$ is sufficiently small, for every $p\in B_r^+(x)\cap\Omega_\alpha^3$
\begin{align}\label{e.s8.inner.contraction}
\|\phi_{y_i}\|_{C^{1,\theta}(B_1^+(p))}\leq &C\bigl(
 \varepsilon^2+A(r+140|\log\varepsilon|^2;x)^{3/2}
 +\varepsilon^{1/6}A(r+140|\log\varepsilon|^2;x)\bigr)\\
&+C_*(\rho^{-1}+L^{-1}+e^{-L})
\max_{\substack{\beta,\,1\leq j\leq n}}
\sup_{\substack{q\in B^+_{r+12|\log\varepsilon|}(x)\cap\mathcal M_\beta^0\\
|t_\beta(q)|\leq K|\log\varepsilon|+C\rho}}
\|\phi_{y_j}\|_{C^{1,\theta}(B_1^+(p))},
\end{align}
where $C_*$ is independent of $L$ and $\rho$.
\end{lemma}

\begin{proof}
We use the following inner gluing estimate: if
\[
 (\Delta_{\alpha,0}+\partial_{zz}-W''(\Hb_\alpha))\sfu
 =\snf+\partial_{y_i}\snG^i,
\]
then
\begin{equation}
 \|\sfu\|_{C^{1,\theta}}
 \leq C_\ast\left(
 \|\snf\|_{C^0}+\|\snG\|_{C^\theta}+\|\psi\|_{C^{1,\theta}}
 +\left\|\int_{\mathbb R}\sfu\Hb_\alpha'\,dz\right\|_{C^0}
 \right),
\label{e.s8.inner.apriori.full}
\end{equation}
where $\psi$ is the Dirichlet trace when $i=1$; for $i\geq2$, replace its
norm by $\|\partial_{y_1}\sfu\|_{C^\theta}$.  The constant $C_\ast$ in
\eqref{e.s8.inner.apriori.full} is independent of $L$, $\rho$, and the size
of the normal support.  This is the standard inner gluing estimate obtained
from Proposition \ref{t.sch}  and Lemmas \ref{l.entire.lin.ac}-\ref{l.entire.lin.ac'}.

Fix $p$ as in the statement and put $y=\Pi_\alpha(p)$.  Let
$\tilde\xi$ be the normal cutoff used in Section \ref{s.cutoff.project} , with
\[
 \tilde\xi=1\quad\hbox{on }[-L,L],\qquad
 \tilde\xi=0\quad\hbox{outside }[-2L,2L],
\]
\[
 |\tilde\xi'|\leq C_\ast L^{-1},\qquad
 |\tilde\xi''|\leq C_\ast L^{-2}.
\]
Choose a base cutoff $\chi$ equal to one on $\Pi_\alpha(B_1^+(p))$, with
$\operatorname{spt}\chi\subset B_{C_\ast\rho}^\alpha(y)$ and
\[
 |D\chi|\leq C_\ast\rho^{-1},\qquad
 |D^2\chi|\leq C_\ast\rho^{-2},\qquad
 \partial_{y_1}\chi=0\quad\hbox{on }\{y_1=0\}.
\]
For sufficiently large fixed $L$, the condition $|t_\alpha(p)|\leq L/2$
implies that $\chi\tilde\xi=1$ on
$Y_{\Sigma_\alpha}^{-1}(B_1^+(p))$.  Set
\[
 \sfu=\chi\tilde\xi\phi_{y_i}.
\]
Then
\begin{align*}
(\Delta_{\alpha,0}+\partial_{zz}-W''(\Hb_\alpha))\sfu &= \partial_{y_i}(\chi\tilde\xi\widetilde E_\alpha^\circ)
 +\chi\tilde\xi E^{i,\circ}
 -(\partial_{y_i}\chi)\tilde\xi\widetilde E_\alpha^\circ\\
&\qquad+2\tilde\xi
  \langle\nabla_{\alpha,0}\chi,\nabla_{\alpha,0}\phi_{y_i}\rangle
 +(\Delta_{\alpha,0}\chi)\tilde\xi\phi_{y_i}
 +\chi\bigl(2\tilde\xi'\partial_z\phi_{y_i}
             +\tilde\xi''\phi_{y_i}\bigr).
\end{align*}
By \eqref{e.s8.source.bound}, the first line on the right and the required
boundary data contribute at most
\[
 C\Bigl(
 \varepsilon^2
 +A(r+140|\log\varepsilon|^2;x)^{3/2}
 +\varepsilon^{1/6}A(r+140|\log\varepsilon|^2;x)
 \Bigr).
\]
The remaining terms are bounded by
\[
 C_\ast(\rho^{-1}+L^{-1})
 \|\phi_{y_i}\|_{C^{1,\theta}
 (B_{C_\ast\rho}^\alpha(y)\times[-2L,2L])}.
\]

It remains to estimate the integral in
\eqref{e.s8.inner.apriori.full}.  Since $\chi$ is independent of $z$,
\begin{align*}
 \int_{\mathbb R}\sfu\Hb_\alpha'\,dz
 ={}&\chi\int_{\mathbb R}\phi_{y_i}\Hb_\alpha'\,dz
 +\chi\int_{\mathbb R}(\tilde\xi-1)
                         \phi_{y_i}\Hb_\alpha'\,dz.
\end{align*}
By \eqref{e.s8.moment.bound}, the first term is bounded by
\[
 C_\ast\bigl(\varepsilon^2
 +A(r+120|\log\varepsilon|^2;x)^2\bigr),
\]
and hence by the expression in
\eqref{e.s8.inner.contraction}, since $A^2\leq A^{3/2}$. By exponential decay of $\Hb_\alpha'$, the second
term is bounded by $C_\ast e^{-L}$ times the supremum of
$|\phi_{y_i}|$ on $B_{C_\ast\rho}^\alpha(y)\times \mathbb{R}$.  At each point, choose a closest sheet and use
\eqref{e.s8.coordinate.change}-\eqref{e.s8.coordinate.error}.  This gives
\begin{align*}
&\left\|\int_{\mathbb R}(\tilde\xi-1)
             \phi_{y_i}\Hb_\alpha'\,dz\right\|_{C^0}\\
&\quad\leq C_\ast e^{-L}
 \max_{\substack{\beta\\1\leq j\leq n}}
 \sup_{\substack{
 q\in B_{r+12|\log\varepsilon|}^+(x)\cap\mathcal M_\beta^0\\
 |t_\beta(q)|\leq K|\log\varepsilon|+C_\ast\rho}}
 \|\phi_{y_j}\|_{C^{1,\theta}(B_1^+(q))}\\
&\qquad+C\Bigl(
 \varepsilon^2+A(r+140|\log\varepsilon|^2;x)^{3/2}
 +\varepsilon^{1/6}A(r+140|\log\varepsilon|^2;x)
 \Bigr).
\end{align*}

On $y_1=0$,
\[
 \partial_{y_1}\sfu
 =\chi\tilde\xi\,\partial_{y_1}\phi_{y_i}
 \quad(i\geq2),
 \qquad
 \sfu=\chi\tilde\xi\phi_{y_1}
 \quad(i=1),
\]
because $\partial_{y_1}\chi=0$.  Proposition \ref{p.s8.boundary.data}
therefore supplies exactly the Neumann and Dirichlet norms required in
\eqref{e.s8.inner.apriori.full}.
Applying \eqref{e.s8.inner.apriori.full} with the preceding estimates proves
\eqref{e.s8.inner.contraction}.
\end{proof}

\begin{proof}[Proof of Proposition \ref{p.s8.horizontal.derivatives}]
Choose $\rho$ sufficiently large and then $L$ sufficiently large so that the coefficients in Lemmas \ref{l.s8.outer} and \ref{l.s8.inner} are at most $1/16$.  Since $\Omega_\alpha^2\cup\Omega_\alpha^3=\mathcal M_\alpha^0$, those lemmas give
\begin{align}\label{e.s8.iteration}
&\max_{\substack{\alpha,\,1\leq i\leq n}}
\sup_{\substack{p\in B_r^+(x)\cap\mathcal M_\alpha^0\\|t_\alpha(p)|\leq K|\log\varepsilon|}}
\|\phi_{y_i}\|_{C^{1,\theta}(B_1^+(p))}\notag\\
&\quad\leq\frac1{16}
\max_{\substack{\beta,\,1\leq j\leq n}}
\sup_{\substack{p\in B^+_{r+12|\log\varepsilon|}(x)\cap\mathcal M_\beta^0\\
|t_\beta(p)|\leq K|\log\varepsilon|+C\rho}}
\|\phi_{y_j}\|_{C^{1,\theta}(B_1^+(p))}\notag\\
&\qquad+C\Bigl(\varepsilon^2+A(r+140|\log\varepsilon|^2;x)^{3/2}
+\varepsilon^{1/6}A(r+140|\log\varepsilon|^2;x)\Bigr).
\end{align}

The same iteration argument as in the proof of Lemma \ref{C1thetainner -estimates-lemma0} concludes the proof. 
\end{proof}

\begin{proof}[Proof of Corollary \ref{c.s8.shift.derivatives}]
Fix $y\in\Sigma_\alpha\cap B_r^+(x)$ and apply \eqref{horizontalphicontrolh} on $B_1^\alpha(y)$.  Proposition \ref{p.s8.horizontal.derivatives} controls $\|\nabla_{\alpha,0}\phi(\cdot,0)\|_{C^{1,\theta}(B_1^\alpha(y))}$.

In the product
\[
\max_{\beta:\,|t_\beta(y,0)|\leq9|\log\varepsilon|}
\|\nabla_{\beta,0}h_\beta\|_{C^{1,\theta}}
\bigl(\max_{B_1^\alpha(y)}e^{-D_\alpha}\bigr)
\]
from \eqref{horizontalphicontrolh} the norm indexed by $\beta$ is taken on $B_2^\beta(\Pi_\beta(y,0))$. Use the rough bound of Proposition \ref{aprior C2alpha Hhphi} to conclude that this product is at most
\[
C\bigl(\varepsilon^2+A(r+200|\log\varepsilon|^2;x)\bigr)
A(r+200|\log\varepsilon|^2;x)
\leq C\bigl(\varepsilon^2+A(r+200|\log\varepsilon|^2;x)^{3/2}\bigr).
\]
The explicit term $\varepsilon^{1/6}e^{-D_\alpha}$ in \eqref{horizontalphicontrolh} gives the last term in \eqref{e.s8.shift}.  Taking the supremum proves the corollary.
\end{proof}

\begin{proof}[Proof of Corollary \ref{c.s8.improved.toda}]
Subtract $\Delta_{\alpha,0}h_\alpha$ from \eqref{H+lhimproved}.  The coefficient $2/\sigma_0$ is the coefficient established by the projected interaction calculation in the proof of Proposition \ref{toda system rough}; it is unchanged by the shift estimate.  Uniform coordinate bounds give
\[
\|\Delta_{\alpha,0}h_\alpha\|_{C^\theta}
\leq C\|\nabla_{\alpha,0}h_\alpha\|_{C^{1,\theta}}.
\]
Use \eqref{e.s8.shift} for this term and \eqref{improvetoda1} for $E_\alpha^0$.  Since
$\varepsilon^{1/3}A\leq\varepsilon^{1/6}A$, and the radius in \eqref{improvetoda1} is smaller than the radius in \eqref{e.s8.toda.error}, we obtain \eqref{e.s8.toda}-\eqref{e.s8.toda.error}.
\end{proof}


\part{Stability of Toda system implies curvature estimates} 
\section{Stability inequality for Toda system}\label{s.stability.toda}

We transfer the stability inequality for $u$ to the Toda system governing the sheet distances.  As in Section 8 of \cite{wang-weiadv}, the test function associated with $\Sigma_\alpha$ is the product of a function on $\Sigma_\alpha$ and the full translation derivative
\[
g_\alpha'=\Hb_\alpha'+\Gb_\alpha'
\]
introduced in Section \ref{s.approx.toda}.  Since $\Gb_\alpha$ depends also on its first argument, horizontal derivatives of $\Gb_\alpha'$ occur in the equation for $g_\alpha'$.  We therefore combine the horizontal and normal parts of the diagonal energy by the divergence theorem before estimating either part.  The interaction integrals and the ordered cross terms are then treated as in Sections 8.3--8.4 of \cite{wang-weiadv}, with the additional terms produced by the twisted Fermi coordinates and by the boundary.

Fix
\[
x\in B^+_{5R/6}(0),\qquad 0<r\leq R/80,\qquad 1<\tau<2.
\]
All functions on a sheet are smooth up to $\partial_0\Sigma_\alpha$, compactly supported in $\Sigma_\alpha\cap B_{r+8|\log \varepsilon|}^+(x)$, and extended by zero outside that support.  They need not vanish on $\partial_0\Sigma_\alpha$. Unless another measure is displayed, an integral over $\Sigma_\alpha$ is taken with respect to $dA_{\alpha,0}$.

\begin{proposition}[Stability inequality]\label{p.stability.toda}
Let $\eta_\alpha\in C_c^\infty(\Sigma_\alpha\cap B_{r}^+(x))$.
For all sufficiently small $\varepsilon$, there is a nonnegative quadratic form $Q$ such that
	\begin{align}\label{e.s9.stability}
		\sum_\alpha\int_{\Sigma_\alpha}|\nabla_{\alpha,0}\eta_\alpha|^2+Q(\eta)
		\geq{}&\sum_\alpha\frac{2A_{(-1)^{\alpha-1}}^2}{\tau\sigma_0}
		\int_{\Sigma_\alpha}e^{-|t_{\alpha-1}(y,0)|}\bigl[\eta_\alpha(y)+\eta_{\alpha-1}(\Pi_{\alpha-1}(y,0))\bigr]^2.
	\end{align}
	Moreover,
	\begin{align}\label{e.s9.Q}
		0\leq Q(\eta)\leq{}&C(N,\tau)\Bigl(\varepsilon^{1/4}
		+A(r+240|\log\varepsilon|^2;x)^{1/2}\Bigr)
		\sum_\alpha\int_{\Sigma_\alpha}|\nabla_{\alpha,0}\eta_\alpha|^2\notag\\
		&+C(N,\tau)\Bigl(\varepsilon^2
		+A(r+240|\log\varepsilon|^2;x)^{7/6}
		+\varepsilon^{1/7}A(r+240|\log\varepsilon|^2;x)\Bigr)
		\sum_\alpha\int_{\Sigma_\alpha}\eta_\alpha^2,
	\end{align}
	where $N$ is the number of non-vanishing $\eta_{\alpha}$.  
\end{proposition}

\subsection{The diagonal terms}

Set
\begin{equation}\label{e.s9.test}
	\varphi=\sum_\alpha\varphi_\alpha,\qquad
	\varphi_\alpha=\eta_\alpha g_\alpha',\qquad
	g_\alpha'=\Hb_\alpha'+\Gb_\alpha',
\end{equation}
so that each $\varphi_{\alpha}$ is supported in $B^+_{r+20|\log\varepsilon|}(x)$.
Here $h_\alpha$ and $\eta_\alpha$ are extended constantly along the $\alpha$-fibres.  Every quantity carrying an index $\beta$ is first calculated in the twisted Fermi coordinates of $\Sigma_\beta$ and then evaluated at the ambient point through $(\Pi_\beta,t_\beta)$. 

For $k\geq0$ we retain the conventions of Section \ref{s.approx.toda}:
\begin{equation}\label{e.s9.profile.conventions}
	\Hb^{(k)}_{\alpha}=
	\Hb^{(k)}_{\alpha}(y, z)=\bar{\Hb}^{(k)}((-1)^{\alpha}(z-h_{\alpha}(y)))\quad \Gb^{(k)}_{\alpha}=\Gb^{(k)}_{\alpha}(y, z)=\bar{\Gb}^{(k)}((-1)^{\alpha}(z-h_{\alpha}(y)))
\end{equation}
 The cut-offs introduced in Section \ref{s.approx.toda} imply that all differentiated one-dimensional factors vanish at the endpoints of every $z$-integration below.  Thus integrations by parts in $z$ have no endpoint terms.

For every fixed $k$, every fixed polynomial factor, and every $0<\sigma<1$,
\begin{equation}\label{e.s9.profile.decay}
	|\Hb_\alpha^{(k)}|\leq C_ke^{-|z-h_\alpha|}\ \ (k\geq 1),\quad
	|\Gb_\alpha^{(k)}|+|\nabla_{\alpha,0}\Gb_\alpha^{(k)}|
	\leq C_{k,\sigma}\varepsilon e^{-\sigma|z-h_\alpha|}\ \ (k\geq 0).
\end{equation}
Taking $\sigma=31/32$ and decreasing $\varepsilon_0$ gives, on the support of the differentiated factors,
\[
|\Gb_\alpha^{(k)}|+|\nabla_{\alpha,0}\Gb_\alpha^{(k)}|
\leq C_k\varepsilon^{3/4}e^{-|z-h_\alpha|}.
\]
for $k\geq 0$.
Indeed, $|z-h_\alpha|\leq8|\log\varepsilon|+1$ there and
$\varepsilon e^{|z-h_\alpha|/32}\leq C\varepsilon^{3/4}$.  Consequently, for every fixed $m\geq0$,
\begin{equation}\label{e.s9.profile.moments}
	\int_{\mathbb R}(g_\alpha')^2\,dz=\sigma_0+O(\varepsilon),\qquad
	\int_{\mathbb R}(1+|z|)^m\bigl(|g_\alpha'|^2+|\partial_zg_\alpha'|^2+|\partial_{zz}g_\alpha'|^2\bigr)\,dz\leq C_m.
\end{equation}

Proposition \ref{aprior C2alpha Hhphi}, applied on the enlarged domains used below, gives
\begin{align}\label{e.s9.rough.estimates}
	&\|\phi\|_{C^{2,\theta}}+\max_\alpha\|h_\alpha\|_{C^{2,\theta}}
	+\max_\alpha\|H_\alpha(\cdot,0)+\Delta_{\alpha,0}h_\alpha\|_{C^\theta}\leq C\bigl(\varepsilon^2+A(r+240|\log\varepsilon|^2;x)\bigr),
\end{align}
where each norm is taken on the corresponding portion of the enlarged ball.  The improved horizontal estimates proved in the preceding section yield
\begin{align}\label{e.s9.horizontal.estimates}
	\max_{\alpha,i}\|\phi_{y_i}\|_{C^{1,\theta}}
	+&\max_\alpha\|\nabla_{\alpha,0}h_\alpha\|_{C^{1,\theta}}\notag\\
	&\leq C\bigl(\varepsilon^2
	+A(r+240|\log\varepsilon|^2;x)^{3/2}
	+\varepsilon^{1/6}A(r+240|\log\varepsilon|^2;x)\bigr).
\end{align}
Throughout the remainder of this section, the two scalar functions produced by differentiating the equations for $\Hb_\alpha$ and $\Gb_\alpha$ are denoted by
\[
 \mathbb{E}_{\alpha}(\Hb,h),\qquad \mathbb{E}_{\alpha}(\Gb,h).
\]

\begin{lemma}\label{l.s9.gprime}
We have the following
	\begin{equation}\label{e.s9.Hbprime}
		\Hb_\alpha'''=-\Delta_{\alpha,0}\Hb_\alpha'
		+W''(\Hb_\alpha)\Hb_\alpha'+\mathbb E_\alpha(\Hb,h),
	\end{equation}
	\begin{equation}\label{e.s9.Gbprime}
		\Gb_\alpha'''=-\Delta_{\alpha,0}\Gb_\alpha'
		+W''(\Hb_\alpha)\Gb_\alpha'
		+W'''(\Hb_\alpha)\Hb_\alpha'\Gb_\alpha+\mathbb E_\alpha(\Gb,h),
	\end{equation}
	where
	\begin{align}\label{e.s9.profile.errors}
		\mathbb E_\alpha(\Hb,h)={}&|\nabla_{\alpha,0}h_\alpha|^2\Hb_\alpha'''
		-(-1)^\alpha(\Delta_{\alpha,0}h_\alpha)\Hb_\alpha''
		+\bar\xi_1'\bigl((-1)^\alpha(z-h_\alpha)\bigr),\notag\\
		\mathbb E_\alpha(\Gb,h)={}&|\nabla_{\alpha,0}h_\alpha|^2\Gb_\alpha'''
		-(-1)^\alpha(\Delta_{\alpha,0}h_\alpha)\Gb_\alpha''\notag\\
		&-2(-1)^\alpha\sum_{i,j}g_\alpha^{ij}(y,0)\Gb_{\alpha,i}''\,\partial_{y_j}h_\alpha
		+\xi_{2,\alpha}'.
	\end{align}
Hence
	\begin{equation}\label{e.s9.gprime.equation}
		(\Delta_{\alpha,0}+\partial_{zz}-W''(\Hb_\alpha))g_\alpha'
		=W'''(\Hb_\alpha)\Hb_\alpha'\Gb_\alpha
		+\mathbb E_\alpha(\Hb,h)+\mathbb E_\alpha(\Gb,h).
	\end{equation}
If $\Delta$ is the ambient Laplacian, then
\begin{align}
 (-\Delta+W''(u))g'_\alpha
 &=\bigl(W''(u)-W''(\Hb_\alpha)\bigr)g'_\alpha
 -W'''(\Hb_\alpha)\Hb'_\alpha \Gb_\alpha
 \notag\\
 &\quad-\mathbb E_\alpha(\Hb,h)-\mathbb E_\alpha(\Gb,h)
 -(\Delta_{\alpha,z}-\Delta_{\alpha,0})g'_\alpha
 +H_\alpha\partial_zg'_\alpha-\mathscr E_\alpha g'_\alpha.
 \label{e.s9.gprime.ambient}
\end{align}

\end{lemma}

\begin{proof}
Fix an arbitrary base point and choose geodesic coordinates for $g_{\alpha,0}$ at that point.  There the Christoffel symbols vanish, so differentiating the equations defining $\Hb_\alpha$ and $\Gb_\alpha$ and applying the chain rule to $z-h_\alpha(y)$ gives
	\[
	\partial_{y_i}g_\alpha'=\Gb_{\alpha,i}'-(-1)^\alpha g_\alpha''h_{\alpha,i},
	\]
	\[
	\partial_{y_iy_j}g_\alpha'=\Gb_{\alpha,ij}'-(-1)^\alpha
	\bigl(\Gb_{\alpha,i}''h_{\alpha,j}+\Gb_{\alpha,j}''h_{\alpha,i}+g_\alpha''h_{\alpha,ij}\bigr)
	+g_\alpha'''h_{\alpha,i}h_{\alpha,j}.
	\]
	Contracting the second identity with $g_\alpha^{ij}(y,0)$ proves \eqref{e.s9.Hbprime}-\eqref{e.s9.profile.errors}; adding \eqref{e.s9.Hbprime} and \eqref{e.s9.Gbprime} proves \eqref{e.s9.gprime.equation}.  This is the analogue of differentiating \cite[(4.1), p. 15]{wang-weiadv} and applying the resulting identity as in \cite[Section 8.2, p. 34]{wang-weiadv}; unlike in that Euclidean calculation, the first-argument derivatives of $\Gb_\alpha$ and the term $\mathscr E_\alpha g_\alpha'$ must also be retained.
\end{proof}

Recall that by the definition of twisted Fermi coordinates and the geometric approximation identities of Section \ref{section twisted Fermi}, in twisted Fermi coordinate with respect to $\Sigma_\alpha$, we have
	\begin{equation}\label{e.s9.gradient.split}
		|\nabla f|^2=|\nabla_{\alpha,z}f|^2
		+\bigl(\partial_zf-\sum_iT^i\partial_{y_i}f\bigr)^2,
		\qquad d\operatorname{vol}_g=\lambda_\alpha(y,z)\,dy\,dz,
	\end{equation}
	and
	\begin{equation}\label{e.s9.laplacian.split}
	\begin{split}
\Delta f={}&\Delta_{\alpha,z}f+\partial_{zz}f-H_\alpha(y,z)\partial_zf
 +\mathscr{E}_{\alpha}f,\notag\\
 \mathscr{E}_{\alpha}f={}&\lambda_\alpha^{-1}\partial_{y_i}
 \bigl(\lambda_\alpha T_\alpha^iT_\alpha^j\partial_{y_j}f\bigr)
 -T_\alpha^i\partial_{y_i z}f
 -\lambda_\alpha^{-1}\partial_z
 \bigl(\lambda_\alpha T_\alpha^j\partial_{y_j}f\bigr).
\end{split}
	\end{equation}
	Moreover,
	\begin{equation}\label{e.s9.density.derivative}
		\partial_z\lambda_\alpha
		=\bigl(-H_\alpha(y,z)+\operatorname{div}_{\alpha,z}T\bigr)\lambda_\alpha,
	\end{equation}
	and, for a fixed integer $m$,
	\begin{equation}\label{e.s9.Hlambda.derivative}
		\bigl|\partial_z\bigl(H_\alpha(y,z)\lambda_\alpha(y,z)\bigr)\bigr|
		\leq C\varepsilon^2(1+|z|^m)\lambda_\alpha(y,0).
	\end{equation}
	Finally,
	\begin{align}\label{e.s9.trace}
		\int_{\partial_0\Sigma_\alpha}\eta_\alpha^2\,d\sigma_{\alpha,0}
		\leq{}&C\left(\int_{\Sigma_\alpha}\eta_\alpha^2\,dA_{\alpha,0}\right)^{1/2}
		\left(\int_{\Sigma_\alpha}|\nabla_{\alpha,0}\eta_\alpha|^2\,dA_{\alpha,0}\right)^{1/2}+C\varepsilon\int_{\Sigma_\alpha}\eta_\alpha^2\,dA_{\alpha,0}.
	\end{align}
To prove this, integrate $ \operatorname{div}_{\alpha,0}(\eta_\alpha^2\partial_{y_1})$ over $\Sigma_\alpha$, and use
$|\mathrm{div}_{\alpha,0}\partial_{y_1}|\leq C\varepsilon$.

\begin{lemma}[Flux through the boundary]\label{l.s9.boundary}
	For every $\delta>0$, the sum of the absolute values of all diagonal and cross boundary integrals
	\[
	\int_{\partial_0B^+_{r+8|\log\varepsilon|}(x)}
	\eta_\alpha\eta_\beta g_\alpha'\partial_\nu g_\beta'\,d\sigma_g
	\]
	that arise from \eqref{e.s9.test} is bounded by
	\begin{align}\label{e.s9.boundary.flux}
		&\delta\sum_\alpha\int_{\Sigma_\alpha}|\nabla_{\alpha,0}\eta_\alpha|^2\,dA_{\alpha,0}\notag\\
		&\quad+C(N,\delta)\bigl(\varepsilon^2
		+A(r+240|\log\varepsilon|^2;x)^{7/6}
		+\varepsilon^{1/7}A(r+240|\log\varepsilon|^2;x)\bigr)
		\sum_\alpha\int_{\Sigma_\alpha}\eta_\alpha^2\,dA_{\alpha,0}.
	\end{align}
\end{lemma}

\begin{proof}
The following identity holds on $\del_0 \Sia$.
\begin{equation}\label{e.s9.normal.chain}
		\partial_\nu g_\alpha'
		=\sum_i(\partial_\nu y_i)
		\bigl(\Gb_{\alpha,i}'-(-1)^\alpha g_\alpha''\partial_{y_i}h_\alpha\bigr)
		+(-1)^\alpha J_\alpha g_\alpha''.
	\end{equation}
	Differentiating the boundary equation for $\Gb_\alpha$ gives
	\[
	\Gb_{\alpha,1}'
	=-w_\alpha\bigl(\Hb_\alpha'+(-1)^\alpha(z-h_\alpha)\Hb_\alpha''\bigr)
	+\delta_\alpha'\bigl(y,(-1)^\alpha(z-h_\alpha)\bigr).
	\]
	Using $J_\alpha=w_\alpha z+(J_\alpha-w_\alpha z)$ in \eqref{e.s9.normal.chain}, we obtain
	\begin{align}\label{e.s9.normal.expansion}
		\partial_\nu g_\alpha'={}&-w_\alpha\Hb_\alpha'
		+(-1)^\alpha\bigl(w_\alpha h_\alpha+J_\alpha-w_\alpha z\bigr)\Hb_\alpha''
		+(-1)^\alpha J_\alpha\Gb_\alpha''\notag\\
		&-(-1)^\alpha g_\alpha''\partial_\nu h_\alpha
		+\sum_i(\partial_\nu y_i-\delta_{1i})\Gb_{\alpha,i}'
		+\delta_\alpha'\bigl(y,(-1)^\alpha(z-h_\alpha)\bigr).
	\end{align}
	The first term is $O(\varepsilon)$.  Moreover,
	\[
	\partial_\nu h_\alpha=\partial_{y_1}h_\alpha
	+\sum_i(\partial_\nu y_i-\delta_{1i})\partial_{y_i}h_\alpha,
	\qquad |\partial_\nu y_i-\delta_{1i}|\leq C\varepsilon(1+|z|^3).
	\]
	Equations \eqref{e.s9.profile.decay}, \eqref{e.s9.horizontal.estimates}, and \eqref{e.s9.normal.expansion} show that the $z$-integral of each boundary product is bounded by
	\[
	C\bigl(\varepsilon+\varepsilon^2
	+A(r+240|\log\varepsilon|^2;x)^{3/2}
	+\varepsilon^{1/6}A(r+240|\log\varepsilon|^2;x)\bigr)
	\]
	times the corresponding squared boundary test functions.  For a cross product, use
	$2|\eta_\alpha\eta_\beta|\leq\eta_\alpha^2+\eta_\beta^2$.  The boundary area measure is uniformly comparable to $d\sigma_{\alpha,0}\,dz$.
	
	
	Apply \eqref{e.s9.trace} and Young's inequality, choosing the coefficient of $|\nabla_{\alpha, 0}\eta_\alpha|^2$ to be $\delta$.  Since $0\leq A(\cdot;\cdot)\leq1$,
	\[
	(\varepsilon+\varepsilon^2+A^{3/2}+\varepsilon^{1/6}A)^2
	+\varepsilon(\varepsilon+\varepsilon^2+A^{3/2}+\varepsilon^{1/6}A)
	\leq C(\varepsilon^2+A^{7/6}+\varepsilon^{1/7}A).
	\]
	Here every $A$ in the numerical inequality is evaluated at $r+240|\log\varepsilon|^2$ and $x$.  This proves \eqref{e.s9.boundary.flux}.
\end{proof}

\begin{lemma}[Combined diagonal contribution]\label{l.s9.diagonal.energy}
	For every $\delta>0$, the difference between
	\[
	\int_{B^+_{r+20|\log\varepsilon|}(x)}
	\bigl(|\nabla\varphi_\alpha|^2+W''(u)\varphi_\alpha^2\bigr)\,d\operatorname{vol}_g
	\]
	and
	\begin{align}\label{e.s9.diagonal.reference}
		&\sigma_0\int_{\Sigma_\alpha}|\nabla_{\alpha,0}\eta_\alpha|^2\,dA_{\alpha,0}\notag\\
		&\quad+\int_{\Sigma_\alpha}\eta_\alpha^2
		\int_{\mathbb R}\Bigl[
		\bigl(W''(u)-W''(\Hb_\alpha)\bigr)(g_\alpha')^2
		-W'''(\Hb_\alpha)\Hb_\alpha'\Gb_\alpha g_\alpha'
		\Bigr]\lambda_\alpha(y,z)\,dz\,dy
	\end{align}
	is bounded in absolute value by
	\begin{align}\label{e.s9.diagonal.energy.error}
		&(C\varepsilon^{1/4}+\delta)
		\int_{\Sigma_\alpha}|\nabla_{\alpha,0}\eta_\alpha|^2\,dA_{\alpha,0}\notag\\
		&\quad+C(\delta)\bigl(\varepsilon^2
		+A(r+240|\log\varepsilon|^2;x)^{7/6}
		+\varepsilon^{1/7}A(r+240|\log\varepsilon|^2;x)\bigr)
		\int_{\Sigma_\alpha}\eta_\alpha^2\,dA_{\alpha,0}.
	\end{align}
\end{lemma}

\begin{proof}
	Green's identity, with $\nu$ the inward normal to the boundary, gives
	\begin{align}\label{e.s9.green.diagonal}
		\int_{B^+_{r+8|\log \varepsilon|}(x)}\bigl(|\nabla(\eta_\alpha g_\alpha')|^2
		+W''(u)\eta_\alpha^2(g_\alpha')^2\bigr)\,d\operatorname{vol}_g
		={}&\int_{B^+_{r+8|\log \varepsilon|}(x)}(g_\alpha')^2|\nabla\eta_\alpha|^2\,d\operatorname{vol}_g\notag\\
		&+\int_{B^+_{r+8|\log \varepsilon|}(x)}\eta_\alpha^2g_\alpha'(-\Delta+W''(u))g_\alpha'\,d\operatorname{vol}_g\notag\\
		&-\int_{\partial_0B^+_{r+8|\log \varepsilon|}(x)}\eta_\alpha^2g_\alpha'\partial_\nu g_\alpha'\,d\sigma_g.
	\end{align}
	
	Since $\eta_\alpha$ is constant along the coordinate fibres, \eqref{e.s9.gradient.split}, \eqref{e.s9.profile.moments}, and the coefficient estimates give
	\[
	\int_{B^+_{r+8|\log \varepsilon|}(x)}(g_\alpha')^2|\nabla\eta_\alpha|^2\,d\operatorname{vol}_g
	=\sigma_0\int_{\Sigma_\alpha}|\nabla_{\alpha,0}\eta_\alpha|^2\,dA_{\alpha,0}
	+O(\varepsilon)\int_{\Sigma_\alpha}|\nabla_{\alpha,0}\eta_\alpha|^2\,dA_{\alpha,0}.
	\]
	The terms containing the tangential component $T$ obey the same bound.
	
	Substitution of \eqref{e.s9.gprime.ambient} in the second line of
\eqref{e.s9.green.diagonal} leaves the two terms in
\eqref{e.s9.diagonal.reference} and the four integrals
\begin{align}
 &-\int_{B^+_{r+8|\log \varepsilon|}(x)}\eta_\alpha^2g'_\alpha
 \bigl(E_\alpha(\Hb,h)+E_\alpha(\Gb,h)\bigr)\,d\operatorname{vol}_g,
 \notag\\
 &-\int_{B^+_{r+8|\log \varepsilon|}(x)}\eta_\alpha^2g'_\alpha
 (\Delta_{\alpha,z}-\Delta_{\alpha,0})g'_\alpha\,d\operatorname{vol}_g,
 \notag\\
 &\int_{B^+_{r+8|\log \varepsilon|}(x)}\eta_\alpha^2H_\alpha g'_\alpha\partial_zg'_\alpha\,d\operatorname{vol}_g,
 \qquad
 -\int_{B^+_{r+8|\log \varepsilon|}(x)}\eta_\alpha^2g'_\alpha\mathscr E_\alpha g'_\alpha\,d\operatorname{vol}_g.
 \label{e.s9.diagonal.remainders}
\end{align}
The first, second and fourth lines are estimated by
\eqref{e.s9.profile.decay}-\eqref{e.s9.horizontal.estimates} and the geometric coefficient bounds.  In the
third line, integrate in $z$ before estimating:
\[
 \int_{\mathbb R}H_\alpha g'_\alpha\partial_zg'_\alpha\lambda_\alpha\,dz
 =-\frac12\int_{\mathbb R}(g'_\alpha)^2
 \partial_z(H_\alpha\lambda_\alpha)\,dz
 =O(\varepsilon^2)\lambda_\alpha(y,0),
\]
where \eqref{e.s9.Hlambda.derivative} is used. Finally apply
Lemma \ref{l.s9.boundary} to the last line of \eqref{e.s9.green.diagonal}.  This proves the
lemma.
\end{proof}

\subsection{Interaction estimates}

\begin{lemma}[Overlap and projection estimates]\label{l.s9.overlap}
	For $\beta\neq\alpha$ and every fixed nonnegative integer $m$,
	\begin{align}\label{e.s9.overlap}
		&\int_{\mathbb R}(1+|z|+|t_\beta(y,z)|)^m
		e^{-|z|-|t_\beta(y,z)|}\,dz\notag\\
		&\hspace{25mm}\leq C_m(1+|t_\beta(y,0)|)^{m+1}e^{-|t_\beta(y,0)|}.
	\end{align}
	The sum of the right-hand side over $\beta\neq\alpha$ is bounded by
	$C_mA(r+240|\log\varepsilon|^2;x)^{11/12}$, while the sum over $|\beta-\alpha|\geq2$ is bounded by
	\[
	C_mA(r+240|\log\varepsilon|^2;x)^{3/2}+C_m\varepsilon^2.
	\]
	For $\beta=\alpha\pm1$ and $|t_\beta(y,0)|\leq100|\log\varepsilon|$,
	\begin{align}\label{e.s9.projection}
		|t_\alpha(\Pi_\beta(y,0),0)+t_\beta(y,0)|&\leq C\varepsilon^{1/3},\notag\\
		(\Pi_\beta|_{\Sigma_\alpha})^*dA_{\beta,0}&=(1+O(\varepsilon^{1/3}))dA_{\alpha,0},\notag\\
		|\Pi_\alpha(\Pi_\beta(y,0),0)-y|&\leq C\varepsilon^{1/3}.
	\end{align}
\end{lemma}

\begin{proof}
	The signed-distance comparison of Proposition \ref{l.projection.error} gives
	\[
	|t_\beta(y,z)-t_\beta(y,0)-z|
	\leq C\varepsilon^{1/2}|z|(1+|z|).
	\]
The two estimates of Lemma \ref{l.interaction.int}, applied with $\sigma=\varsigma=1$, therefore give
\eqref{e.s9.overlap}.  Exponential separation of the ordered sheets and
	$(1+s)^me^{-s/12}\leq C_m$ give the two asserted sums.
	
The first and third lines of \eqref{e.s9.projection} follow from the signed-distance and projection comparison in Proposition \ref{l.projection.error}.  Differentiating
	\[
	Y_{\Sigma_\beta}(\Pi_\beta(y,0),t_\beta(y,0))=y
	\]
	along $\Sigma_\alpha$ gives the area-form comparison.  Indeed, the derivative of the normal flow changes the base metric by $O(\varepsilon(1+|\log\varepsilon|^2))$, whereas the tangential component of $\nabla t_\beta$ along $\Sigma_\alpha$ is bounded by $C\varepsilon^{1/4}|\log\varepsilon|^{3/4}$.  The pulled-back metric therefore changes by $O(\varepsilon^{1/3})$ for small $\varepsilon$.  If the separation exceeds $100|\log\varepsilon|$, the exponential factor is $O(\varepsilon^2)$, and the coarse bounded-Jacobian estimates suffice.
\end{proof}

By the one-dimensional computation used in Section 8.3 of \cite{wang-weiadv}, we have the following lemma.

\begin{lemma}[Adjacent one-dimensional interaction]\label{l.s9.adjacent}
	For $\beta=\alpha\pm1$,
	\begin{align}\label{e.s9.adjacent}
		&\int_{\mathbb R}\bigl(W''(\Hb_*)-W''(\Hb_\beta)\bigr)
		\Hb_\alpha'\Hb_\beta'\,dz\notag\\
		&\quad=-2A_{(-1)^{\min\{\alpha,\beta\}}}^2e^{-|t_\beta(y,0)|}
		+O\Bigl(\varepsilon^2+A(r+240|\log\varepsilon|^2;x)^{7/6}
		+\varepsilon^{1/7}A(r+240|\log\varepsilon|^2;x)\Bigr).
	\end{align}
	The sum of the absolute values of the same integrals over $|\beta-\alpha|\geq2$ is bounded by
	$C(\varepsilon^2+A(r+240|\log\varepsilon|^2;x)^{3/2})$.
\end{lemma}

\begin{lemma}[Corrected diagonal interaction]\label{l.s9.diagonal.interaction}
We have
	\begin{align}\label{e.s9.diagonal.interaction}
		&\Bigl|\int_{\Sigma_\alpha}\eta_\alpha^2
		\int_{\mathbb R}\bigl[
		\bigl(W''(u)-W''(\Hb_\alpha)\bigr)(g_\alpha')^2
		-W'''(\Hb_\alpha)\Hb_\alpha'\Gb_\alpha g_\alpha'
		\bigr]\lambda_\alpha\,dz\,dy\notag\\
		&\quad+2\int_{\Sigma_\alpha}\eta_\alpha^2
		\bigl[A_{(-1)^{\alpha-1}}^2e^{-|t_{\alpha-1}(y,0)|}
		+A_{(-1)^\alpha}^2e^{-|t_{\alpha+1}(y,0)|}\bigr]dA_{\alpha,0}\Bigr|\notag\\
		&\hspace{15mm}\leq C\bigl(\varepsilon^2
		+A(r+240|\log\varepsilon|^2;x)^{7/6}
		+\varepsilon^{1/7}A(r+240|\log\varepsilon|^2;x)\bigr)
		\int_{\Sigma_\alpha}\eta_\alpha^2\,dA_{\alpha,0}.
	\end{align}
\end{lemma}

\begin{proof}
	A direct Taylor estimate of the term linear in $\phi$ gives only an $O(A)$ bound.  To recover the required error, we first remove $\Gb_*$ from the nonlinear argument and use the equation for $\phi$.
	
	The identity
	\[
	\Delta(u-\Gb_*)-W'(u-\Gb_*)
	=W'(u)-W'(u-\Gb_*)-\sum_\beta W''(\Hb_\beta)\Gb_\beta
	-\sum_\beta\bigl(\Delta\Gb_\beta-W''(\Hb_\beta)\Gb_\beta\bigr)
	\]
	combined with the equation for $\phi=u-\Hb_*-\Gb_*$ gives (recall the notation introduced in \eqref{xiRi})
		\begin{align}\label{e.s9.phi.without.Gb}
		\phi_{zz}-W'(u-\Gb_*)+\sum_\beta W'(\Hb_\beta)
		={}&\Delta(u-\Gb_*)-W'(u-\Gb_*)\notag\\
		&-\Delta_{\alpha,z}\phi+H_\alpha(y,z)\phi_z-\mathscr E_\alpha\phi\notag\\
		&+\sum_\beta\Bigl[(-1)^\beta\Hb_\beta'\mathcal R_{\beta,1}
		-\Hb_\beta''\mathcal R_{\beta,2}-\mathscr E_\beta\Hb_\beta\Bigr]\notag\\
		&-\sum_\beta\bar\xi_1\bigl((-1)^\beta(t_\beta-h_\beta\circ\Pi_\beta)\bigr).
	\end{align}
	Multiply \eqref{e.s9.phi.without.Gb} by $\Hb_\alpha''$ and integrate in $z$.  The horizontal derivatives of $\phi$ are controlled by \eqref{e.s9.horizontal.estimates}.  For the term involving the $\alpha$-sheet, use
	$\int_{\mathbb R}\Hb_\alpha'\Hb_\alpha''\,dz=0$ and subtract
	$H_\alpha(y,0)+\Delta_{\alpha,0}h_\alpha(y)$ from $\mathcal R_{\alpha,1}$.  The difference is bounded by
	$C\varepsilon^2(1+|z|^m)+C\varepsilon(1+|z|^m)\|\nabla_{\alpha,0}h_\alpha\|_{C^1}$.  The terms with $\beta\neq\alpha$ are summed using Lemma \ref{l.s9.overlap}.  The cut-off terms, the operator terms, and the term involving $\Gb_*$ obey the same bound.  Therefore,
	\begin{align}\label{e.s9.Hbsecond.test}
		&\Bigl|\int_{\mathbb R}
		\bigl[\phi_{zz}-W'(u-\Gb_*)+\sum_\beta W'(\Hb_\beta)\bigr]\Hb_\alpha''\,dz\Bigr|\notag\\
		&\quad\leq C\bigl(\varepsilon^2
		+A(r+240|\log\varepsilon|^2;x)^{3/2}
		+\varepsilon^{1/6}A(r+240|\log\varepsilon|^2;x)\bigr).
	\end{align}
	
	Integrating the left-hand side of \eqref{e.s9.Hbsecond.test} by parts once gives
	\begin{align}\label{e.s9.differentiated.interaction}
		&\int_{\mathbb R}
		\bigl[\phi_{zz}-W'(u-\Gb_*)+\sum_\beta W'(\Hb_\beta)\bigr]\Hb_\alpha''\,dz\notag\\
		={}&\int_{\mathbb R}\bigl(W''(u-\Gb_*)-W''(\Hb_\alpha)\bigr)(\Hb_\alpha')^2\,dz\notag\\
		&+(-1)^\alpha\int_{\mathbb R}
		\bigl(W''(u-\Gb_*)-W''(\Hb_\alpha)\bigr)\Hb_\alpha'\phi_z\,dz\notag\\
		&+\sum_{\beta\neq\alpha}(-1)^{\alpha+\beta}
		\int_{\mathbb R}\bigl(W''(u-\Gb_*)-W''(\Hb_\beta)\bigr)
		\Hb_\alpha'\Hb_\beta'\,\partial_z(t_\beta-h_\beta\circ\Pi_\beta)\,dz\notag\\
		&-(-1)^\alpha\int_{\mathbb R}
		\bar\xi_1'\bigl((-1)^\alpha(z-h_\alpha)\bigr)\phi_z\,dz.
	\end{align}
	Indeed,
	\[
	\partial_z\Hb_\alpha''=(-1)^\alpha
	\Bigl(W''(\Hb_\alpha)\Hb_\alpha'
	+\bar\xi_1'\bigl((-1)^\alpha(z-h_\alpha)\bigr)\Bigr),
	\]
	\[
	\partial_z(u-\Gb_*)=\phi_z+
	\sum_\beta(-1)^\beta\Hb_\beta'\partial_z(t_\beta-h_\beta\circ\Pi_\beta).
	\]	
The second and fourth lines on the right of \eqref{e.s9.differentiated.interaction} are $O(\varepsilon^2+A^{3/2})$.  In the third line, replacing $W''(u-\Gb_*)$ by $W''(\Hb_*)$ has the same cost.  Moreover, by Proposition \ref{p.conseq.almost.parallel}
	\[
	\partial_z(t_\beta-h_\beta\circ\Pi_\beta)-1
	=O(\varepsilon^{1/3})+O(\varepsilon^{1/5})\|\nabla h_\beta\|_{C^0}.
	\]
	Lemma \ref{l.s9.overlap} bounds the cost of removing this factor by
	$C(\varepsilon^2+A^{7/6}+\varepsilon^{1/7}A)$.  Since $(-1)^{\alpha+\beta}=-1$ for adjacent sheets, Lemma \ref{l.s9.adjacent} yields
	\begin{align}\label{e.s9.Hb.diagonal}
		&\int_{\mathbb R}\bigl(W''(u-\Gb_*)-W''(\Hb_\alpha)\bigr)(\Hb_\alpha')^2\,dz\notag\\
		&\quad=-2\bigl[A_{(-1)^{\alpha-1}}^2e^{-|t_{\alpha-1}(y,0)|}
		+A_{(-1)^\alpha}^2e^{-|t_{\alpha+1}(y,0)|}\bigr]\notag\\
		&\qquad+O\bigl(\varepsilon^2+A(r+240|\log\varepsilon|^2;x)^{7/6}
		+\varepsilon^{1/7}A(r+240|\log\varepsilon|^2;x)\bigr).
	\end{align}
	
	It remains to restore $\Gb_*$ and $g_\alpha'$ and to replace $\lambda_\alpha(y,z)$ by $\lambda_\alpha(y,0)$.  Taylor expansion gives
	\[
	W''(u)=W''(u-\Gb_*)+W'''(u-\Gb_*)\Gb_*+O(\varepsilon^2),
	\qquad
	(g_\alpha')^2=(\Hb_\alpha')^2+2\Hb_\alpha'\Gb_\alpha'+(\Gb_\alpha')^2.
	\]
	After subtracting \eqref{e.s9.Hb.diagonal}, the terms linear in the own-sheet correction are
	\[
	2\bigl(W''(u-\Gb_*)-W''(\Hb_\alpha)\bigr)\Hb_\alpha'\Gb_\alpha',
	\qquad
	\bigl(W'''(u-\Gb_*)-W'''(\Hb_\alpha)\bigr)\Gb_\alpha(\Hb_\alpha')^2.
	\]
	The subtraction of
	$W'''(\Hb_\alpha)\Hb_\alpha'\Gb_\alpha g_\alpha'$ in \eqref{e.s9.diagonal.reference} is precisely what produces this cancellation.  The part involving $\phi$ is
	$O(\varepsilon(\varepsilon^2+A))$.  Notice that
	\[
	\int_{\mathbb R}(1+|z|^m)e^{-(1+\sigma)|z|}\min\{1,e^{z-t}\}\,dz
	\leq C_{m,\sigma}e^{-t}.
	\]
	Taking $\sigma=7/8$ and using
	$\varepsilon A^{7/8}\leq(\varepsilon^2+A^{7/4})/2\leq(\varepsilon^2+A^{7/6})/2$
	bounds all linear foreign-sheet terms.  Terms containing two corrections are $O(\varepsilon^2)$.  Finally,
	$|\lambda_\alpha(y,z)-\lambda_\alpha(y,0)|\leq C\varepsilon(1+|z|^m)\lambda_\alpha(y,0)$, and the same estimate controls the density error.  Multiplying by $\eta_\alpha^2$ and integrating in $y$ proves \eqref{e.s9.diagonal.interaction}.
\end{proof}

\subsection{Cross terms}

\begin{lemma}[Cross terms]\label{l.s9.cross}
	For the ordered sum over $\alpha\neq\beta$,
	\begin{align}\label{e.s9.cross}
		&\Bigl|\sum_{\alpha\neq\beta}
		\int_{B^+_{r+20|\log\varepsilon|}(x)}
		\bigl(\langle\nabla\varphi_\alpha,\nabla\varphi_\beta\rangle
		+W''(u)\varphi_\alpha\varphi_\beta\bigr)\,d\operatorname{vol}_g\notag\\
		&\quad+2\sum_\alpha\sum_{\beta=\alpha\pm1}
		A_{(-1)^{\min\{\alpha,\beta\}}}^2
		\int_{\Sigma_\alpha}e^{-|t_\beta(y,0)|}
		\eta_\alpha(y)\eta_\beta(\Pi_\beta(y,0))\,dA_{\alpha,0}\Bigr|\notag\\
		&\quad\leq C(N)\Bigl(\varepsilon^{1/4}
		+A(r+240|\log\varepsilon|^2;x)^{1/2}\Bigr)
		\sum_\alpha\int_{\Sigma_\alpha}|\nabla_{\alpha,0}\eta_\alpha|^2\,dA_{\alpha,0}\notag\\
		&\qquad+C(N)\Bigl(\varepsilon^2
		+A(r+240|\log\varepsilon|^2;x)^{7/6}
		+\varepsilon^{1/7}A(r+240|\log\varepsilon|^2;x)\Bigr)
		\sum_\alpha\int_{\Sigma_\alpha}\eta_\alpha^2\,dA_{\alpha,0}.
	\end{align}
\end{lemma}

\begin{proof}
	For a fixed ordered pair, Green's identity gives
	\begin{align}\label{e.s9.cross.green}
		\int_{B^+_{r+8|\log \varepsilon|}(x)}\bigl(\langle\nabla\varphi_\alpha,\nabla\varphi_\beta\rangle
		+W''(u)\varphi_\alpha\varphi_\beta\bigr)
		={}&\int_{B^+_{r+8|\log \varepsilon|}(x)} g_\alpha'g_\beta'\langle\nabla\eta_\alpha,\nabla\eta_\beta\rangle\notag\\
		&+\int_{B^+_{r+8|\log \varepsilon|}(x)}\eta_\alpha
		\langle g_\beta'\nabla g_\alpha'-g_\alpha'\nabla g_\beta',\nabla\eta_\beta\rangle\notag\\
		&+\int_{B^+_{r+8|\log \varepsilon|}(x)}\eta_\alpha\eta_\beta g_\alpha'(-\Delta+W''(u))g_\beta'\notag\\
		&-\int_{\partial_0 B^+_{r+8|\log \varepsilon|}(x)}\eta_\alpha\eta_\beta g_\alpha'\partial_\nu g_\beta'.
	\end{align}
	All unlabelled integrals in this proof are ambient integrals over the common support.  Formula \eqref{e.s9.cross.green} follows by expanding both gradients and integrating only the derivative falling on $g_\beta'$.	
	The decay estimates imply
	\[
	|g_\alpha'|+|\nabla g_\alpha'|\leq Ce^{-|t_\alpha|},
	\qquad |\nabla\eta_\alpha|\leq C|\nabla_{\alpha,0}\eta_\alpha|\circ\Pi_\alpha.
	\]
	Lemma \ref{l.s9.overlap}, applied in either sheet's coordinates, gives
	\[
	\int_{B^+_{r+8|\log \varepsilon|}(x)} e^{-|t_\alpha|-|t_\beta|}|\eta_\alpha\circ\Pi_\alpha|^2\,d\operatorname{vol}_g
	\leq CA^{11/12}\int_{\Sigma_\alpha}\eta_\alpha^2\,dA_{\alpha,0},
	\]
	\[
	\int_{B^+_{r+8|\log \varepsilon|}(x)} e^{-|t_\alpha|-|t_\beta|}
	|\nabla_{\beta,0}\eta_\beta|^2\circ\Pi_\beta\,d\operatorname{vol}_g
	\leq CA^{11/12}\int_{\Sigma_\beta}|\nabla_{\beta,0}\eta_\beta|^2\,dA_{\beta,0},
	\]
	where, in this proof, $A=A(r+240|\log\varepsilon|^2;x)$.  These bounds also hold with fixed polynomial factors.  The first line of \eqref{e.s9.cross.green} is therefore bounded by $CA^{11/12}$ times the two gradient energies.  Cauchy's inequality bounds its second line by
	\[
	CA^{1/2}\int_{\Sigma_\beta}|\nabla_{\beta,0}\eta_\beta|^2\,dA_{\beta,0}
	+CA^{4/3}\int_{\Sigma_\alpha}\eta_\alpha^2\,dA_{\alpha,0}.
	\]
	Since $A^{11/12}\leq A^{1/2}$ and $A^{4/3}\leq A^{7/6}$ for $0\leq A\leq1$, these are allowed errors.  Lemma \ref{l.s9.boundary} controls the boundary integral.
	
	Substituting the $\beta$-version of \eqref{e.s9.gprime.equation} and \eqref{e.s9.laplacian.split} in the third line of \eqref{e.s9.cross.green}, its only principal part is
	\begin{equation}\label{e.s9.cross.principal}
		\int_{B^+_{r+8|\log \varepsilon|}(x)}\eta_\alpha\eta_\beta
		\bigl(W''(u)-W''(\Hb_\beta)\bigr)g_\alpha'g_\beta'\,d\operatorname{vol}_g.
	\end{equation}
	The remaining scalar factors multiplying $\eta_\alpha\eta_\beta g_\alpha'$ are
	\[
	-W'''(\Hb_\beta)\Hb_\beta'\Gb_\beta,
	\quad -\mathbb E_\beta(\Hb,h)-\mathbb E_\beta(\Gb,h),
	\quad -(\Delta_{\beta,t_\beta}-\Delta_{\beta,0})g_\beta',
	\quad H_\beta(\cdot,t_\beta)\partial_{t_\beta}g_\beta',
	\quad -\mathscr E_\beta g_\beta'.
	\]
Equations \eqref{e.s9.profile.decay}-\eqref{e.s9.horizontal.estimates}, the coefficient bounds, Lemma \ref{l.s9.overlap}, and Young's inequality give the mass error in \eqref{e.s9.cross}.  For example,
	\[
	\varepsilon A^{11/12}\leq C(\varepsilon^2+A^{11/6}),\qquad
	\varepsilon^{3/4}A^{11/12}\leq C(\varepsilon^2+A^{22/15})
	\leq C(\varepsilon^2+A^{7/6}).
	\]
	
	In \eqref{e.s9.cross.principal}, replace $u$ by $\Hb_*$ and both $g'$ factors by the corresponding $\Hb'$ factors.  The preceding estimates control the error, and replacing $\lambda_\alpha(y,z)$ by $\lambda_\alpha(y,0)$ has the same bound.  It remains to replace the transported test function by its value at height zero.  The exact identity is
	\begin{equation}\label{e.s9.freeze.test}
		\eta_\beta(\Pi_\beta(y,z))-\eta_\beta(\Pi_\beta(y,0))
		=\int_0^z\sum_j(\partial_{y_j}\eta_\beta)(\Pi_\beta(y,s))
		\partial_s\Pi_\beta^j(y,s)\,ds.
	\end{equation}
On the overlap,
$|\partial_s\Pi_\beta|\leq C\varepsilon^{1/5}$.
The $s$- and $z$-intervals have length
$O(|\log\varepsilon|)$, and the maps
$y\mapsto\Pi_\beta(y,s)$ have uniformly bounded Jacobians.
Therefore, Cauchy's inequality in $y$, $s$, and $z$ gives
\[
\begin{aligned}
&C\varepsilon^{1/5}|\log\varepsilon|^2
A(r+240|\log\varepsilon|^2;x)
\left(
\int_{\Sigma_\alpha}\eta_\alpha^2\,dA_{\alpha,0}
+
\int_{\Sigma_\beta}
|\nabla_{\beta,0}\eta_\beta|^2\,dA_{\beta,0}
\right)
\\
&\qquad\leq
C\varepsilon^{1/7}
A(r+240|\log\varepsilon|^2;x)
\left(
\int_{\Sigma_\alpha}\eta_\alpha^2\,dA_{\alpha,0}
+
\int_{\Sigma_\beta}
|\nabla_{\beta,0}\eta_\beta|^2\,dA_{\beta,0}
\right),
\end{aligned}
\]
where the last inequality holds for sufficiently small $\varepsilon$,
since
$\varepsilon^{2/35}|\log\varepsilon|^2\leq C$.
Lemma \ref{l.s9.adjacent} now gives the neighbouring principal terms in \eqref{e.s9.cross}, and the nonadjacent terms satisfy the stated error estimate.  Summing over the finitely many ordered pairs proves the lemma.
\end{proof}

\begin{lemma}[Changing the base sheet]\label{l.s9.change.base}
	In the terms involving a neighbouring pair $\{\alpha-1,\alpha\}$, one may change the base from $\Sigma_{\alpha-1}$ to $\Sigma_\alpha$, transport the exponential weight and the area measure, and replace
	$\eta_\alpha(\Pi_\alpha(\Pi_{\alpha-1}(y,0),0))$ by $\eta_\alpha(y)$.  The sum of all resulting errors is bounded by
	\begin{align}\label{e.s9.change.base}
		&C(N)\Bigl(\varepsilon^{1/4}+A(r+240|\log\varepsilon|^2;x)^{1/2}\Bigr)
		\sum_\gamma\int_{\Sigma_\gamma}|\nabla_{\gamma,0}\eta_\gamma|^2\,dA_{\gamma,0}\notag\\
		&\quad+C(N)\Bigl(\varepsilon^2+A(r+240|\log\varepsilon|^2;x)^{7/6}
		+\varepsilon^{1/7}A(r+240|\log\varepsilon|^2;x)\Bigr)
		\sum_\gamma\int_{\Sigma_\gamma}\eta_\gamma^2\,dA_{\gamma,0}.
	\end{align}
	This applies to the squared factors in Lemma \ref{l.s9.diagonal.interaction} and the mixed factors in Lemma \ref{l.s9.cross}.
\end{lemma}

\begin{proof}
	First suppose that the separation is at most $100|\log\varepsilon|$.  Use the inverse of
	$\Pi_{\alpha-1}|_{\Sigma_\alpha}$ for the change of variables.  The first two estimates in \eqref{e.s9.projection} show that changing the exponential weight and the area measure has relative error $O(\varepsilon^{1/3})$, hence absolute error
	$C\varepsilon^{1/3}A(r+240|\log\varepsilon|^2;x)$ times the two test-function masses.
	
	The reverse projection need not be the inverse map.  By \eqref{e.s9.projection}, its composition with the first projection differs from the identity by $O(\varepsilon^{1/3})$.  Its derivative in graph coordinates differs from the identity by $O(\varepsilon^{1/5})$.  Therefore
	\begin{equation}\label{e.s9.projection.composition}
		\int_{\{y\in \Sigma_\alpha:|t_{\alpha-1}(y,0)|\leq 100|\log \varepsilon|\}}\bigl|\eta_\alpha(\Pi_\alpha(\Pi_{\alpha-1}(y,0),0))-\eta_\alpha(y)\bigr|^2\,dA_{\alpha,0}
		\leq C\varepsilon^{2/3}\int_{\Sigma_\alpha}|\nabla_{\alpha,0}\eta_\alpha|^2\,dA_{\alpha,0}.
	\end{equation}
	The integral on the left is over the logarithmic comparison region; the test function is zero outside its support.  To compare squares, use $a^2-b^2=(a-b)(a+b)$; for a mixed product, leave the other factor unchanged.  Cauchy's inequality, \eqref{e.s9.projection.composition}, and
	$e^{-|t_{\alpha-1}|}\leq CA(r+240|\log\varepsilon|^2;x)$ prove \eqref{e.s9.change.base}.
	
	If the separation exceeds $100|\log\varepsilon|$, the corresponding term is $O(\varepsilon^2)$ times the masses before and after the base change, by the coarse Jacobian bound in Lemma \ref{l.s9.overlap}.  Restricting the comparison to the logarithmic region differentiates no cut-off, because only the integrals are being compared.  Summing over the finitely many pairs proves the lemma.
\end{proof}

\begin{proof}[Proof of Proposition \ref{p.stability.toda}]
Since $\operatorname{supp}\varphi_\alpha\subset B^+_{r+20|\log\varepsilon|}(x)$ for every $\alpha$, the stability assumption gives
	\begin{equation}\label{e.s9.stability.test}
		0\leq\int_{B^+_{r+20|\log\varepsilon|}(x)}
		\bigl(|\nabla\varphi|^2+W''(u)\varphi^2\bigr)\,d\operatorname{vol}_g.
	\end{equation}
	If necessary, the compactly supported $C^1$ test is approximated in $H^1$ by smooth functions up to the boundary. 
	
	Expand the right-hand side of \eqref{e.s9.stability.test} into diagonal terms and the ordered cross sum.  Apply Lemmas \ref{l.s9.diagonal.energy}, \ref{l.s9.diagonal.interaction}, and \ref{l.s9.cross}.  Before changing the base of a neighbouring term, the negative principal contribution is
	\begin{align}\label{e.s9.principal.before.change}
		&-2\sum_\alpha\sum_{\beta=\alpha\pm1}
		A_{(-1)^{\min\{\alpha,\beta\}}}^2
		\int_{\Sigma_\alpha}e^{-|t_\beta(y,0)|}\eta_\alpha(y)^2\,dA_{\alpha,0}\notag\\
		&\quad-2\sum_\alpha\sum_{\beta=\alpha\pm1}
		A_{(-1)^{\min\{\alpha,\beta\}}}^2
		\int_{\Sigma_\alpha}e^{-|t_\beta(y,0)|}
		\eta_\alpha(y)\eta_\beta(\Pi_\beta(y,0))\,dA_{\alpha,0}.
	\end{align}
	Fix the unordered neighbouring pair $\{\alpha-1,\alpha\}$.  By Lemma \ref{l.s9.change.base}, all four terms may be written on $\Sigma_\alpha$.  Their common coefficient is $A_{(-1)^{\alpha-1}}^2$.  The two diagonal terms and the two ordered mixed terms then combine as
	\[
	-2A_{(-1)^{\alpha-1}}^2e^{-|t_{\alpha-1}(y,0)|}
	\bigl[\eta_\alpha(y)+\eta_{\alpha-1}(\Pi_{\alpha-1}(y,0))\bigr]^2.
	\]
	Each unordered neighbouring pair is counted exactly once after this regrouping.
	
	Choose the finitely many constants $\delta$ in Lemmas \ref{l.s9.boundary} and \ref{l.s9.diagonal.energy} so that the sum of their gradient contributions is at most
	$(\tau-1)\sigma_0$ times the total gradient energy.  These choices are made after $N$ and $\tau$ are fixed and before $\varepsilon$ is decreased.  Moving the negative square in \eqref{e.s9.principal.before.change} to the other side of \eqref{e.s9.stability.test}, dividing by $\tau\sigma_0$, and defining $Q(\eta)$ to be the sum of the two resulting nonnegative error terms give \eqref{e.s9.stability} and \eqref{e.s9.Q}.
\end{proof}

\section{A separation estimate from local bounds for the Emden equation}\label{s.separation.emden}

Recall that $R=\varepsilon^{-1}$ and
\[
A(r;x)=\max_\alpha\max_{y\in\Sigma_\alpha\cap B_r^+(x)}e^{-D_\alpha(y)}.
\]
We use the geometric and separation estimates of Section \ref{section twisted Fermi}, the improved Toda equation in Corollary \ref{c.s8.improved.toda}, and the stability inequality in Proposition \ref{p.stability.toda}.  In particular, Lemma \ref{l.D_al.infty} implies
\[
A(5R/6;0)=o(1)
\]
uniformly as $\varepsilon\to0$. Fix once and for all a constant
\[
0<\delta<\frac1{48}.
\]
This value of \(\delta\) remains fixed throughout this section.

\begin{proposition}[Decay above the logarithmic comparison scale]\label{p.s10.decay}
	Suppose that $1\leq n\leq9$, where $n$ is the dimension of a sheet.  There are fixed constants $M\gg K\gg1$ such that the following holds for all sufficiently small $\varepsilon$.  Let $2R/3\leq r\leq5R/6$ and $\kappa=A(r;0)$.  If
	\begin{equation}\label{e.s10.threshold.general}
		\kappa\geq M\varepsilon^2|\log\varepsilon|^2,
	\end{equation}
	then
	\begin{equation}\label{e.s10.decay}
		A(r-KR_*;0)\leq\frac{\kappa}{2},
		\qquad
		R_*=\max\{\kappa^{-1/2},200|\log\varepsilon|^2\}.
	\end{equation}
\end{proposition}

\begin{remark}\label{r.s10.threshold}
		If the boundary is totally geodesic, the same conclusion holds under the weaker hypothesis
	\begin{equation}\label{e.s10.threshold.tg}
		\kappa\geq M\varepsilon^2|\log\varepsilon|.
	\end{equation}
The distinction between \eqref{e.s10.threshold.general} and \eqref{e.s10.threshold.tg} occurs in the comparison between the signed Riemannian distance and the Euclidean distance in a coordinate chart.  In a boundary Fermi chart, one has
	\[
	g-\mathrm{Id}=O(\varepsilon|y|+\varepsilon|z|).
	\]
	On a base ball of radius $K\kappa^{-1/2}$, replacing a distance of order $|\log\varepsilon|$ by the coordinate difference can therefore produce an error of order
	\[
	C(K)\varepsilon\kappa^{-1/2}|\log\varepsilon|.
	\]
	Hypothesis \eqref{e.s10.threshold.general} makes this error small.  When the boundary is totally geodesic, all first derivatives of the normalised metric vanish at the centre of a boundary Fermi chart, and the corresponding metric error is quadratic.  This yields \eqref{e.s10.threshold.tg}.  
\end{remark}

The proof follows the reduction in \cite[Section 9]{wang-weiadv}: derive a scalar inequality for the gap between two adjacent graphs, transfer the stability inequality to the resulting exponential term, obtain an $L^p$ estimate with $p>n/2$, and apply a local boundedness argument. 

\subsection{The geometric reduction}\label{ss.s10.geometric}

To prove \eqref{e.s10.decay}, suppose that $x\in\Sigma_\alpha\cap B_{r-KR_*}^+(0)$ satisfies
\begin{equation}\label{e.s10.badpoint}
	e^{-D_\alpha(x)}>\frac{\kappa}{2}.
\end{equation}
Choose a point on a neighbouring sheet that attains $D_\alpha(x)$.  Reversing the graphical ordering if necessary, denote it by $\Sigma_{\alpha+1}$.  Since $W$ is even, $A_1=A_{-1}$, so this relabelling does not change the interaction coefficient.  At this point,
\begin{equation}\label{e.s10.logdistance}
	|t_{\alpha+1}(x)|<\log\frac{2}{\kappa}\leq2|\log\varepsilon|
\end{equation}
for all sufficiently small $\varepsilon$.

The two sheets are the graphs of the functions $f_\alpha$ and $f_{\alpha+1}$ already introduced in Section \ref{section twisted Fermi}, and $x=(x', f_\alpha(x'))$.  Recall that
\begin{equation}\label{e.s10.graph.C2}
	|\nabla f_\alpha|+|\nabla f_{\alpha+1}|\leq C,
	\qquad
	|\nabla^2 f_\alpha|+|\nabla^2 f_{\alpha+1}|\leq C\varepsilon.
\end{equation}
In a boundary chart,
\begin{equation}\label{e.s10.graph.boundary}
	\partial_{y_1}f_\alpha=\partial_{y_1}f_{\alpha+1}=0
	\qquad\text{on }\{y_1=0\},
\end{equation}
because $\partial_{y_1}$ is the ambient inward unit normal there.  At the centre of the normalized chart,
\begin{equation}\label{e.s10.normalization}
	f_\alpha(0)=0,
	\qquad
	\nabla f_\alpha(0)=0,
	\qquad
	g(0)=\mathrm{Id}.
\end{equation}
Consistently with Sections \ref{s.preliminary.setup}-\ref{s.approx.toda}, $\mathbb B_s(x')\subset\mathbb R^n$ denotes the Euclidean ball, and $\mathbb B_s^+(x')=\mathbb B_s(x')\cap\{y_1\geq0\}$ in a boundary chart; in an interior chart we use $\mathbb B_s(x')$.  We reserve $B_s^+(\cdot)$ for ambient Fermi coordinate balls, $B_s^\alpha(\cdot)$ for the sheet coordinate balls, and $B_s^{0,\alpha}(\cdot)$ for the boundary coordinate balls of Definition \ref{def.notation.balls}.  In the graphical and scalar calculations, $\nabla$ and $\operatorname{div}$ denote Euclidean coordinate operations on the base; intrinsic derivatives on a sheet retain their earlier subscripts.

The estimates below are used on $\mathbb B_{2K\kappa^{-1/2}}^+(x')$, together with a fixed multiple of $|\log\varepsilon|$ in the normal and tangential directions.  

Set
\[
f_\alpha^t=(1-t)f_\alpha+tf_{\alpha+1},
\qquad
t_\beta(y)=t_\beta(y,f_\alpha(y)).
\]
The vector $\nu=e_1$ denotes the inward normal to the boundary.  

\begin{lemma}[Riemannian gap inequality]\label{l.s10.gap}
Under \eqref{e.s10.threshold.general}, for fixed $K$, $M$, and all sufficiently small $\varepsilon$, there is a symmetric matrix $B_\alpha$ such that
\begin{equation}\label{e.s10.B.bounds}
|B_\alpha-\mathrm{Id}|\leq C(K)M^{-\delta},
\qquad
|\nabla B_\alpha|\leq C\varepsilon,
\qquad
\frac12\mathrm{Id}\leq B_\alpha\leq2\mathrm{Id}.
\end{equation}
For every $y\in\mathbb B_{2K\kappa^{-1/2}}^+(x')$ such that
$f_{\alpha+1}(y)-f_\alpha(y)<4|\log\varepsilon|+1$,
\begin{equation}\label{e.s10.gap.inequality}
\operatorname{div}\!\left(B_\alpha\nabla(f_{\alpha+1}-f_\alpha)\right)
\leq\frac{4A_1^2}{\sigma_0}\bigl(1+C(K)M^{-\delta}\bigr)e^{-(f_{\alpha+1}-f_\alpha)}
+C(K)M^{-\delta}\kappa.
\end{equation}
On the boundary $\partial_0\mathbb B_{2K\kappa^{-1/2}}^+(x')$, we have
\[
\left\langle B_\alpha\nabla(f_{\alpha+1}-f_\alpha),\nu\right\rangle=0.
\]
\end{lemma}

For completeness, we record the matrix in Lemma \ref{l.s10.gap}.  If $p=(p_1,\ldots,p_n)$ is the covector defining a graph, then
\begin{equation}\label{e.s10.varpi}
\varpi(y,z,p)^2
=g^{n+1,n+1}(y,z)-2\sum_i g^{i,n+1}(y,z)p_i
+\sum_{i,j}g^{ij}(y,z)p_ip_j.
\end{equation}
Thus the graph-area density relative to $dy$ is $\sqrt{\det g(y,z)}\,\varpi(y,z,p)$, where the determinant is that of the ambient $(n+1)\times(n+1)$ metric matrix.  Define
\begin{equation}\label{e.s10.B.definition}
B_{\alpha,ij}(y)
=\int_0^1
\left[\partial_{p_ip_j}^2\!\left(\sqrt{\det g}\,\varpi\right)\right]
\bigl(y,f_\alpha^t(y),\nabla f_\alpha^t(y)\bigr)\,dt.
\end{equation}

\begin{remark}\label{r.s10.conormal}
At $y_1=0$, boundary Fermi coordinates give $g^{1j}=0$ for $j\neq1$, $g^{11}=1$, and $g^{1,n+1}=0$.  Equation \eqref{e.s10.graph.boundary} gives $p_1=0$ for $p=\nabla f_\alpha^t$.  Differentiating \eqref{e.s10.varpi} therefore yields
\[
\partial_{p_1}\!\left(\sqrt{\det g}\,\varpi\right)=0,
\qquad
\partial_{p_1p_j}^2\!\left(\sqrt{\det g}\,\varpi\right)=0
\quad(j\geq2).
\]
Consequently $B_\alpha\nu=B_{\alpha,11}\nu$, and \eqref{e.s10.graph.boundary} gives the asserted conormal condition. 
\end{remark}

\begin{proof}[Proof of Lemma \ref{l.s10.gap}]
We begin with the two geometric estimates used in the proof.

\begin{claim}\label{c.s10.graph.estimates}
For every $y\in\mathbb B_{2K\kappa^{-1/2}+20|\log\varepsilon|}^+(x')$, we have
\begin{align}
|\nabla f_\alpha|
&\leq C(K)\varepsilon\kappa^{-1/2}+C\varepsilon|\log\varepsilon|,
\label{e.s10.lower.slope}\\
|\nabla f_{\alpha+1}|
&\leq C(K)\varepsilon\kappa^{-1/2}
+C\sqrt{\varepsilon|\log\varepsilon|}
+C\varepsilon|\log\varepsilon|,
\label{e.s10.upper.slope}\\
|g-\mathrm{Id}|
&\leq C(K)(\varepsilon\kappa^{-1/2}+\varepsilon|\log\varepsilon|),
\qquad |\partial g|\leq C\varepsilon,
\qquad |\partial^2g|\leq C\varepsilon^2.
\label{e.s10.metric}
\end{align}
\end{claim}

\begin{proof}
The first estimate follows from \eqref{e.s10.normalization}, \eqref{e.s10.graph.C2}, and the distance between 0 and $x'$.  At $x'$, the uniform metric and graph bounds together with \eqref{e.s10.logdistance} first give
$f_{\alpha+1}-f_\alpha\leq C|\log\varepsilon|$.  The positive difference of two disjoint graphs whose Hessians are bounded by $C\varepsilon$ satisfies
\begin{equation}\label{e.s10.gap.gradient}
|\nabla(f_{\alpha+1}-f_\alpha)|\leq C\sqrt{\varepsilon|\log\varepsilon|}
\end{equation}
at that point.  Indeed, Taylor's inequality in a unit direction chosen to decrease the difference gives
\[
0\leq(f_{\alpha+1}-f_\alpha)(y)
-s|\partial(f_{\alpha+1}-f_\alpha)(y)|+C\varepsilon s^2.
\]
Taking $s$ comparable to $\sqrt{|\log\varepsilon|/\varepsilon}$ gives \eqref{e.s10.gap.gradient}.  

Combining \eqref{e.s10.gap.gradient} with \eqref{e.s10.lower.slope} at $x'$  gives \eqref{e.s10.upper.slope} there.  Integrating \eqref{e.s10.graph.C2} along base segments of length $O(K\kappa^{-1/2}+|\log\varepsilon|)$ proves \eqref{e.s10.lower.slope}-\eqref{e.s10.upper.slope} on $B_{2K\kappa^{-1/2}+20|\log\varepsilon|}^+(x')$.  The graph values on this set remain in an ambient coordinate neighbourhood of radius $C(K)(\kappa^{-1/2}+|\log\varepsilon|)$.  Taylor's theorem for the rescaled metric then gives \eqref{e.s10.metric}.
\end{proof}

\begin{claim}\label{c.s10.distance.comparison}
For every $y\in\mathbb B_{2K\kappa^{-1/2}}^+(x')$ satisfying
$f_{\alpha+1}(y)-f_\alpha(y)\leq4|\log\varepsilon|+1$,
\begin{align}
&\left||t_{\alpha+1}(y,f_\alpha(y))|-(f_{\alpha+1}-f_\alpha)(y)\right|
+\left||t_\alpha(y,f_{\alpha+1}(y))|-(f_{\alpha+1}-f_\alpha)(y)\right|
\notag\\
&\quad\leq C(K)\left(
\varepsilon\kappa^{-1/2}|\log\varepsilon|
+\varepsilon^2\kappa^{-1}|\log\varepsilon|
+\varepsilon|\log\varepsilon|^2\right)
\leq C(K)M^{-\delta}.
\label{e.s10.distance.gap}
\end{align}
In particular,
\begin{equation}\label{e.s10.weights.compare}
e^{-(f_{\alpha+1}-f_\alpha)}
=\bigl(1+O(C(K)M^{-\delta})\bigr)
 e^{-|t_{\alpha+1}(y,f_\alpha(y))|},
\end{equation}
and the analogous formula holds with the two sheets interchanged.  If
$(y,f_\alpha(y))\in\Sigma_\alpha\cap B_r^+(0)$ and
$(y,f_{\alpha+1}(y))\in\Sigma_{\alpha+1}\cap B_r^+(0)$, then
\[
e^{-|t_{\alpha+1}(y,f_\alpha(y))|}\leq\kappa,
\qquad
e^{-|t_\alpha(y,f_{\alpha+1}(y))|}\leq\kappa.
\]
\end{claim}

\begin{proof}
A vertical segment gives the upper distance bound
\[
|t_{\alpha+1}(y,f_\alpha(y))|
\leq(1+C|g-\mathrm{Id}|)(f_{\alpha+1}-f_\alpha)(y),
\]
and similarly from the upper sheet.  For the reverse inequality, a Euclidean graph of slope at most $s$ has distance at least $d/\sqrt{1+s^2}$ from a point whose vertical gap is $d$.  Indeed, at horizontal displacement $a$ the vertical difference is at least $\max\{d-s|a|,0\}$, and minimizing
\[
|a|^2+\max\{d-s|a|,0\}^2
\]
gives $d^2/(1+s^2)$.  The comparison applies to the extended graphs used in the definition of the signed distances in Section \ref{section twisted Fermi}.  The vertical upper bound confines every minimizing segment to a coordinate neighbourhood of radius $C|\log\varepsilon|$.

By Claims \ref{c.s10.graph.estimates}, the squared slope and the metric error on this neighbourhood are bounded by
\[
C(K)(\varepsilon^2\kappa^{-1}+\varepsilon|\log\varepsilon|),
\qquad
C(K)(\varepsilon\kappa^{-1/2}+\varepsilon|\log\varepsilon|),
\]
respectively.  Multiplying by a gap of order $|\log\varepsilon|$ proves the first inequality in \eqref{e.s10.distance.gap}.  Under \eqref{e.s10.threshold.general}, its three terms are bounded by
\[
C(K)M^{-1/2},
\qquad C(K)M^{-1}|\log\varepsilon|^{-1},
\qquad C(K)\varepsilon|\log\varepsilon|^2.
\]
Choose $M$ after $K$ and then choose $\varepsilon$ sufficiently small.  Since $\delta<1/48$, this proves the second inequality in \eqref{e.s10.distance.gap}; \eqref{e.s10.weights.compare} follows by exponentiation.  If $(y,f_\alpha(y))\in\Sigma_\alpha\cap B_r^+(0)$, then the definition of $A(r;0)$ gives
$e^{-|t_{\alpha+1}(y,f_\alpha(y))|}\leq e^{-D_\alpha(y)}\leq\kappa$.  At $(y,f_{\alpha+1}(y))\in\Sigma_{\alpha+1}\cap B_r^+(0)$, the definition of $A(r;0)$ gives
$e^{-|t_\alpha(y,f_{\alpha+1}(y))|}\leq\kappa$.
\end{proof}

We now need to compute the operator in \eqref{e.s10.gap.inequality}.  
The proof of Lemma \ref{mean curvature expansion tf} in Appendix \ref{ap:geometric.approx} already contains the essential computations to check that by Claims \ref{c.s10.graph.estimates} for every $y\in\mathbb B_{2K\kappa^{-1/2}}^+(x')$,
\[
|B_\alpha-\mathrm{Id}|
\leq C(K)(\varepsilon\kappa^{-1/2}+\varepsilon^2\kappa^{-1}
+\varepsilon|\log\varepsilon|),
\qquad
|\nabla B_\alpha|\leq C\varepsilon.
\]
After $M$ is fixed and $\varepsilon$ is decreased, these inequalities give \eqref{e.s10.B.bounds}.  Remark \ref{r.s10.conormal} gives the boundary condition.

Fix
\[
y\in\mathbb B_{2K\kappa^{-1/2}}^+(x')
\quad\text{with}\quad
f_{\alpha+1}(y)-f_\alpha(y)<4|\log\varepsilon|+1.
\]
Apply Corollary 8.3 at $(y,f_\alpha(y))$ and $(y,f_{\alpha+1}(y))$.  By the radius choice in the proof of Proposition \ref{p.s10.decay},
\[
B_{300|\log\varepsilon|^2}^+\bigl((y,f_\gamma(y))\bigr)
\subset B_r^+(0),
\qquad \gamma\in\{\alpha,\alpha+1\}.
\]
The remainder in each of the two Toda equations is
\[
C\bigl(\varepsilon^2+\kappa^{3/2}+\varepsilon^{1/6}\kappa\bigr).
\]

In the equation for the lower sheet, discard the nonnegative interaction with its lower neighbour; in the equation for the upper sheet, discard the nonpositive interaction with its upper neighbour.  Since $A_1=A_{-1}$,
\[
H_\alpha\geq-\frac{2A_1^2}{\sigma_0}
 e^{-|t_{\alpha+1}(y,f_\alpha(y))|}
-C(\varepsilon^2+\kappa^{3/2}+\varepsilon^{1/6}\kappa),
\]
while
\[
H_{\alpha+1}\leq\frac{2A_1^2}{\sigma_0}
 e^{-|t_\alpha(y,f_{\alpha+1}(y))|}
+C(\varepsilon^2+\kappa^{3/2}+\varepsilon^{1/6}\kappa).
\]
Multiply these inequalities by the corresponding positive densities and subtract the lower-sheet inequality from the upper-sheet inequality to obtain \eqref{e.s10.gap.inequality}. Indeed, after division by $\kappa$, all remaining terms are bounded by
\[
C(K)\left(
\frac{\varepsilon^2(1+|\log\varepsilon|)}{\kappa}
+\kappa^{1/2}+\varepsilon^{1/6}\right).
\]
The first term is at most $C(K)/M$ under \eqref{e.s10.threshold.general}, and the other two tend to zero uniformly because $A(5R/6;0)=o(1)$.  This proves the lemma.
\end{proof}

\subsection{The rescaled inequality and its stability}\label{ss.s10.rescaled}

Use the following rescaling:
\[
v_\alpha(y)
=(f_{\alpha+1}-f_\alpha)(\kappa^{-1/2}y)-|\log\kappa|,
\qquad
\widetilde B_\alpha(y)=B_\alpha(\kappa^{-1/2}y).
\]
After rescaling, the base coordinate of $x$ is $\kappa^{1/2}x'$.  Then set
\begin{equation}\label{e.s10.V.definition}
V_\alpha(y)=\max\left\{e^{-v_\alpha(y)},\frac{\varepsilon^4}{\kappa}\right\}.
\end{equation}
At every point
\[
y\in\mathbb B_K^+(\kappa^{1/2}x')
\quad\text{such that}\quad
(f_{\alpha+1}-f_\alpha)(\kappa^{-1/2}y)
\leq4|\log\varepsilon|,
\]
one has $V_\alpha(y)=e^{-v_\alpha(y)}$.
The proof of Proposition \ref{p.s10.decay} shows that
$(f_{\alpha+1}-f_\alpha)(x')<4|\log\varepsilon|$, and hence $V_\alpha(\kappa^{1/2}x')=e^{-v_\alpha(\kappa^{1/2}x')}$.

For every $y\in\mathbb B_K^+(\kappa^{1/2}x')$ satisfying
$(f_{\alpha+1}-f_\alpha)(\kappa^{-1/2}y)<4|\log\varepsilon|$, rescaling \eqref{e.s10.gap.inequality} gives
\[
\operatorname{div}(\widetilde B_\alpha\nabla v_\alpha)
\leq\frac{4A_1^2}{\sigma_0}\bigl(1+C(K)M^{-\delta}\bigr)e^{-v_\alpha}
+C(K)M^{-\delta}.
\]
Consequently, $V_\alpha$ is a positive, continuous, locally Lipschitz weak subsolution of
\begin{equation}\label{e.s10.V.inequality}
\left\{
\begin{aligned}
-\operatorname{div}(\widetilde B_\alpha\nabla V_\alpha)
&\leq\frac{4A_1^2}{\sigma_0}\bigl(1+C(K)M^{-\delta}\bigr)V_\alpha^2
-V_\alpha^{-1}\left\langle\widetilde B_\alpha\nabla V_\alpha,
\nabla V_\alpha\right\rangle
+C(K)M^{-\delta}V_\alpha,\\
\left\langle\widetilde B_\alpha\nabla V_\alpha,\nu\right\rangle&=0
\qquad\text{on }\{y_1=0\}.
\end{aligned}
\right.
\end{equation}
The rescaled coefficient satisfies
\[
|\widetilde B_\alpha-\mathrm{Id}|\leq C(K)M^{-\delta},
\qquad
|\nabla\widetilde B_\alpha|\leq C\varepsilon\kappa^{-1/2},
\qquad
\frac12\mathrm{Id}\leq\widetilde B_\alpha\leq2\mathrm{Id}.
\]

\begin{lemma}[Transfer of almost stability]\label{l.s10.stability.transfer}
Fix $\tau\in(1,2)$ independently of $K$, $M$, and $\varepsilon$.  For every $\eta\in C^1(\mathbb B_K^+(\kappa^{1/2}x'))$ supported away from $\partial_+\mathbb B_K^+(\kappa^{1/2}x')$,
\begin{align}
\frac{4A_1^2}{\sigma_0}\int_{\mathbb B_K^+(\kappa^{1/2}x')}V_\alpha\eta^2
&\leq\bigl(\tau+C(K,\tau)M^{-\delta}\bigr)
\int_{\mathbb B_K^+(\kappa^{1/2}x')}\left\langle\widetilde B_\alpha\nabla\eta,\nabla\eta\right\rangle
\notag\\
&\quad+C(K,\tau)M^{-\delta}\int_{\mathbb B_K^+(\kappa^{1/2}x')}\eta^2.
\label{e.s10.almost.stability}
\end{align}
and similarly in a full ball in the interior.
\end{lemma}

\begin{proof}
We put the same base test function on the two adjacent sheets, compare its
projected values, and rescale, as in \cite[Lemma 9.5 and (9.26)]{wang-weiadv}.
Throughout the proof, $\eta$ denotes the rescaled test function, extended
by zero outside $\mathbb B_K^+(\kappa^{1/2}x')$.  Set
\[
\eta_\gamma(y,f_\gamma(y))=\eta(\kappa^{1/2}y)
\quad(\gamma=\alpha,\alpha+1),
\qquad \eta_\gamma=0\quad(\gamma\notin\{\alpha,\alpha+1\}).
\]
Proposition \ref{p.stability.toda}, with $N=2$, gives
\begin{align*}
\frac{2A_1^2}{\tau\sigma_0}
\int_{\Sigma_{\alpha+1}}e^{-|t_\alpha|}
 (\eta_{\alpha+1}+\eta_\alpha\circ\Pi_\alpha)^2\,dA_{\alpha+1,0}
\leq{}&\sum_{\gamma=\alpha,\alpha+1}
\int_{\Sigma_\gamma}|\nabla_{\gamma,0}\eta_\gamma|^2\,dA_{\gamma,0}
+Q((\eta_\gamma)_\gamma),
\end{align*}
Claim \ref{c.s10.graph.estimates},
Remark \ref{r.tZ-Z}, and the coordinate-flow estimates used in the proof of
Lemma \ref{l.s9.change.base} give, whenever $y\in\mathbb B_{K\kappa^{-1/2}}^+(x')$ is such that $f_{\alpha+1}(y)-f_\alpha(y)\leq4|\log\varepsilon|$,
\begin{align}
|\Pi_\alpha(y,f_{\alpha+1}(y))-y|
&\leq C(K)\Bigl(
\varepsilon\kappa^{-1/2}|\log\varepsilon|
+\sqrt{\varepsilon}\,|\log\varepsilon|^{3/2}
+\varepsilon|\log\varepsilon|^3\Bigr).
\label{e.s10.projection.displacement}
\end{align}

\noindent \textit{Claim}: Let $\mathcal{S}\coloneqq \{y\in\mathbb B_{K\kappa^{-1/2}}^+(x'): f_{\alpha+1}(y)-f_\alpha(y)\leq4|\log\varepsilon|\}$. Then 
\begin{align}\label{e.s10.test.transport}
&\int_{\mathcal{S}}
 e^{-|t_\alpha(y,f_{\alpha+1}(y))|}
 \left|\eta_\alpha(\Pi_\alpha(y,f_{\alpha+1}(y)))
       -\eta_{\alpha+1}(y,f_{\alpha+1}(y))\right|^2\,dy
\notag\\
&\qquad\qquad\leq C(K)M^{-2\delta}\kappa^{1-n/2}
 \int_{\mathbb B_K^+(\kappa^{1/2}x')}|\nabla\eta|^2\,dy.
 \end{align}
\noindent \textit{Proof of Claim.} To obtain \eqref{e.s10.test.transport}, let $\Pi_\alpha(y, f_{\alpha + 1}(y))=(P(y), f_\alpha(P(y)))$ and notice that 
\begin{equation}\label{e.estimate.difference.projections}
\begin{aligned}
	|\eta_\alpha(\Pi_\alpha(y,f_{\alpha+1}(y)))
       -&\eta_{\alpha+1}(y,f_{\alpha+1}(y))|^2 = |\eta(\kappa^{1/2}P(y))-\eta(\kappa^{1/2}y)|^2\\
       &\leq \kappa |P(y)-y|^2\int_0^1|\nabla\eta(\kappa^{1/2}((1-s)y+sP(y)))|^2\, ds.
\end{aligned}       
\end{equation}
By \eqref{e.s10.projection.displacement}, 
\[
|P(y)-y|^2\leq C(K)\bigl(\varepsilon^2\kappa^{-1}|\log \varepsilon|^2+\varepsilon \kappa |\log \varepsilon|^3+\varepsilon^2\kappa |\log \varepsilon|^6\bigr), 
\]
so that, for fixed $K$ and $M$, we can pick $\varepsilon$ small enough that 
\[
C(K)\bigl(\varepsilon^2\kappa^{-1}|\log \varepsilon|^2+\varepsilon \kappa |\log \varepsilon|^3+\varepsilon^2\kappa |\log \varepsilon|^6\bigr)\leq C(K)M^{-2\delta}
\]
Using $e^{-|t_\alpha|}\leq\kappa$ and integrating the left-hand side of \eqref{e.estimate.difference.projections} over $\mathcal{S}$ and the right-hand side over $\mathbb{B}^+_{K\kappa^{-1/2}(x')}$ gives \eqref{e.s10.test.transport}. \qed

Now notice that $(\eta_{\alpha+1}+\eta_\alpha\circ\Pi_\alpha)^2\geq 4(1-M^{-\delta})\eta_{\alpha+1}^2-M^\delta|\eta_{\alpha+1}-\eta_\alpha\circ\Pi_\alpha|^2$, so the previous claim gives
\begin{align*}
\int_{\mathcal{S}}
 e^{-|t_\alpha(y,f_{\alpha+1}(y))|}(\eta_{\alpha+1}+\eta_\alpha\circ\Pi_\alpha)^2\,dy\geq{}& 4(1-M^{-\delta})
 \int_{\mathcal{S}}
 e^{-|t_\alpha(y,f_{\alpha+1}(y))|}\eta(\kappa^{1/2}y)^2\,dy\\
&\quad-C(K)M^{-\delta}\kappa^{1-n/2}
 \int_{\mathbb B_K^+(\kappa^{1/2}x')}|\nabla\eta|^2\,dy.
\end{align*}
Under the graph parametrizations, Claim \ref{c.s10.graph.estimates} gives
\[
dA_{\gamma,0}=\bigl(1+O(C(K)M^{-\delta})\bigr)\,dy
\quad(\gamma=\alpha,\alpha+1),
\]
and
\[
\int_{\Sigma_\gamma}|\nabla_{\gamma,0}\eta_\gamma|^2\,dA_{\gamma,0}
\leq\bigl(1+C(K)M^{-\delta}\bigr)
\int_{\mathbb B_{K\kappa^{-1/2}}^+(x')}
\left|\nabla_y[\eta(\kappa^{1/2}y)]\right|^2\,dy.
\]
The weight comparison in \eqref{e.s10.weights.compare} and the formula
immediately following it give
\[
e^{-|t_\alpha(y,f_{\alpha+1}(y))|}
=\bigl(1+O(C(K)M^{-\delta})\bigr)
 e^{-(f_{\alpha+1}-f_\alpha)(y)}
\]
when $f_{\alpha+1}(y)-f_\alpha(y)\leq4|\log\varepsilon|$.
Now set $Y=\kappa^{1/2}y$.  On this set,
$e^{-(f_{\alpha+1}-f_\alpha)(y)}=\kappa V_\alpha(Y)$, so
\begin{align*}
&\int_{\mathcal S}
 e^{-(f_{\alpha+1}-f_\alpha)(y)}\eta(\kappa^{1/2}y)^2\,dy=\kappa^{1-n/2}
 \int_{\mathcal{S}'}
 V_\alpha(Y)\eta(Y)^2\,dY,
\end{align*}
where $\mathcal{S'}=\{Y\in\mathbb B_K^+(\kappa^{1/2}x'):(f_{\alpha+1}-f_\alpha)(\kappa^{-1/2}Y)\leq4|\log\varepsilon|\}$, whereas
\[
\int_{\mathbb B_{K\kappa^{-1/2}}^+(x')}
\left|\nabla_y[\eta(\kappa^{1/2}y)]\right|^2\,dy
=\kappa^{1-n/2}
 \int_{\mathbb B_K^+(\kappa^{1/2}x')}|\nabla\eta|^2\,dY.
\]
Substitute these comparisons and the lower bound for the square into the
stability inequality, and cancel $\kappa^{1-n/2}$.  We obtain
\begin{align*}
&\frac{8A_1^2}{\tau\sigma_0}\bigl(1-C(K)M^{-\delta}\bigr)
 \int_{\mathcal{S}'}
 V_\alpha\eta^2\,dY\leq\bigl(2+C(K,\tau)M^{-\delta}\bigr)
 \int_{\mathbb B_K^+(\kappa^{1/2}x')}|\nabla\eta|^2\,dY
 +\kappa^{n/2-1}Q((\eta_\gamma)_\gamma).
\end{align*}
Moreover,
\[
\int_{\Sigma_\gamma}\eta_\gamma^2\,dA_{\gamma,0}
\leq\bigl(1+C(K)M^{-\delta}\bigr)\kappa^{-n/2}
\int_{\mathbb B_K^+(\kappa^{1/2}x')}\eta^2\,dY
\quad(\gamma=\alpha,\alpha+1).
\]
Thus \eqref{e.s9.Q}, with $N=2$, gives
\begin{align*}
\kappa^{n/2-1}Q((\eta_\gamma)_\gamma)
&\leq C(2,\tau)(\varepsilon^{1/4}+\kappa^{1/2})
 \int_{\mathbb B_K^+(\kappa^{1/2}x')}|\nabla\eta|^2\,dY\\
&\quad+C(2,\tau)\left(\frac{\varepsilon^2}{\kappa}
 +\kappa^{1/6}+\varepsilon^{1/7}\right)
 \int_{\mathbb B_K^+(\kappa^{1/2}x')}\eta^2\,dY\\
&\leq C(K,\tau)M^{-\delta}
 \left(\int_{\mathbb B_K^+(\kappa^{1/2}x')}|\nabla\eta|^2\,dY
       +\int_{\mathbb B_K^+(\kappa^{1/2}x')}\eta^2\,dY\right).
\end{align*}
Insert this bound for $Q$ into the inequality with coefficient
$8A_1^2/(\tau\sigma_0)$.  For $M$ sufficiently large,
$C(K)M^{-\delta}\leq1/2$.  Divide by $2$ and multiply by
$\tau/(1-C(K)M^{-\delta})$ to obtain
\begin{align}\label{e.s10.stability.S}
&\frac{4A_1^2}{\sigma_0}
 \int_{\mathcal{S}'}
 V_\alpha\eta^2\,dY\\
&\qquad\leq\bigl(\tau+C(K,\tau)M^{-\delta}\bigr)
 \int_{\mathbb B_K^+(\kappa^{1/2}x')}|\nabla\eta|^2\,dY
 +C(K,\tau)M^{-\delta}
 \int_{\mathbb B_K^+(\kappa^{1/2}x')}\eta^2\,dY.
\end{align}
By \eqref{e.s10.B.bounds},
\begin{equation}\label{e.s.10.ellipticity}
|\nabla\eta|^2\leq\bigl(1+C(K)M^{-\delta}\bigr)
\langle\widetilde B_\alpha\nabla\eta,\nabla\eta\rangle.
\end{equation}
Moreover, on $\mathbb{B}^+_K(\kappa^{1/2}x')\setminus \mathcal{S}'$, we have $V_\alpha=\varepsilon^4/\kappa$. Adding the resulting contribution to \eqref{e.s10.stability.S} and using \eqref{e.s.10.ellipticity} gives \eqref{e.s10.almost.stability}.
\end{proof}

\subsection{Local bounds for the scalar inequality}\label{ss.s10.scalar}

We now state the scalar estimates for an arbitrary base-ball centre $y_0$.  In the application to the function $V_\alpha$ constructed above, take $y_0=\kappa^{1/2}x'$.  We regard \eqref{e.s10.V.inequality}--\eqref{e.s10.almost.stability} as scalar weak inequalities on $\mathbb B_K^+(y_0)$.  If $\mathbb B_K(y_0)\subset\{y_1>0\}$, replace every $\mathbb B_s^+(y_0)$ below by $\mathbb B_s(y_0)$ and omit the conormal condition.  The matrix $\widetilde{B}_\alpha$ is symmetric and has ellipticity constants between $1/2$ and $2$.  Test functions may meet $\mathbb B_K(y_0)\cap\{y_1=0\}$ but have support disjoint from $\partial\mathbb B_K(y_0)\cap\{y_1>0\}$.  All constants are uniform for every centre $y_0$ with $\mathbb B_K^+(y_0)\neq\varnothing$.  Once $\tau$ is fixed, its dependence is included in $C(K)$.

\begin{proposition}[Smallness from almost stability]\label{p.s10.scalar.smallness}
Suppose that $1\leq n\leq9$.  Choose, once and for all,
\begin{equation}\label{e.s10.exponents}
\frac74<q<2,
\qquad
1<\tau<\frac2q,
\qquad
p=2q+1.
\end{equation}
The only properties used below are
\[
p>\frac n2
\qquad\text{and}\qquad
2q-\tau q^2>0.
\]
For $K$ sufficiently large and then $M$ sufficiently large, every positive locally Lipschitz weak subsolution of \eqref{e.s10.V.inequality} on $\mathbb B_K^+(y_0)$ that satisfies \eqref{e.s10.almost.stability} obeys
\[
\|V_\alpha\|_{L^\infty(\mathbb B_1^+(y_0))}\leq\frac14.
\]
\end{proposition}

The proof is divided into the following three lemmas.

\begin{lemma}[Bounded local subsolutions]\label{l.s10.bounded.moser}
Suppose that $C(K)M^{-\delta}\leq1$, and let $V_\alpha$ satisfy \eqref{e.s10.V.inequality} on $\mathbb B_2^+(y_0)$ with $0<V_\alpha\leq\Lambda$.  For every $p\geq2$,
\begin{equation}\label{e.s10.local.moser}
\|V_\alpha\|_{L^\infty(\mathbb B_1^+(y_0))}
\leq c_0(n,p,W,\Lambda)\|V_\alpha\|_{L^p(\mathbb B_2^+(y_0))}.
\end{equation}
More generally, for every fixed $0<r_1<r_2\leq2$,
\[
\|V_\alpha\|_{L^\infty(\mathbb B_{r_1}^+(y_0))}
\leq c_0(n,p,W,\Lambda,r_1,r_2)
\|V_\alpha\|_{L^p(\mathbb B_{r_2}^+(y_0))}.
\]
\end{lemma}

\begin{proof}
Discard the nonpositive quadratic-gradient term in \eqref{e.s10.V.inequality}.  Since $C(K)M^{-\delta}\leq1$ and $V_\alpha\leq\Lambda$,
\[
-\operatorname{div}(\widetilde B_\alpha\nabla V_\alpha)
\leq\left(1+\frac{8A_1^2}{\sigma_0}\Lambda\right)V_\alpha.
\]
For $s\geq2$ and $\psi\in C^1(\mathbb B_2^+(y_0))$ satisfying $\operatorname{supp}\psi\cap
\bigl(\partial\mathbb B_2(y_0)\cap\{y_1>0\}\bigr)=\emptyset$, test this inequality with $\psi^2V_\alpha^{s-1}$.  

All integrals through \eqref{e.s10.moser.energy} are over $\mathbb B_2^+(y_0)$.  The conormal condition removes the integral over
$\mathbb B_2(y_0)\cap\{y_1=0\}$.  Uniform ellipticity and Young's inequality give
\begin{align*}
(s-1)\int\psi^2V_\alpha^{s-2}
\left\langle\widetilde B_\alpha\nabla V_\alpha,\nabla V_\alpha\right\rangle
&\leq C(W)(1+\Lambda)\int\psi^2V_\alpha^s\\
&\quad+2\, \Bigl|\int\psi V_\alpha^{s-1}
\left\langle\widetilde B_\alpha\nabla V_\alpha,\nabla\psi\right\rangle\Bigr|\\
&\leq C(W)(1+\Lambda)\int\psi^2V_\alpha^s
+\frac{s-1}{2}\int\psi^2V_\alpha^{s-2}
\left\langle\widetilde B_\alpha\nabla V_\alpha,\nabla V_\alpha\right\rangle\\
&\quad+\frac{C}{s-1}\int V_\alpha^s|\nabla\psi|^2.
\end{align*}
Subtract the repeated gradient term and use
$\nabla V_\alpha^{s/2}=(s/2)V_\alpha^{s/2-1}\nabla V_\alpha$.  It follows that
\begin{equation}\label{e.s10.moser.energy}
\int\bigl|\nabla(\psi V_\alpha^{s/2})\bigr|^2
\leq C(W)s^2\int V_\alpha^s
\left((1+\Lambda)\psi^2+|\nabla\psi|^2\right).
\end{equation}

Apply the Euclidean Sobolev inequality to $\psi V_\alpha^{s/2}$ after extending it by zero across
$\partial\mathbb B_2(y_0)\cap\{y_1>0\}$.  If
$\operatorname{supp}\psi\cap\{y_1=0\}\neq\emptyset$, take its even $H^1$ extension across $\{y_1=0\}$. For $n>2$ use the exponent $2n/(n-2)$; for $n=1,2$ use exponent $4$, including the $L^2$ term on the bounded support.

At the $j$-th step, use the nested half-balls
$\mathbb B_{1+2^{-j}}^+(y_0)$ and
$\mathbb B_{1+2^{-j-1}}^+(y_0)$, begin with exponent $s=p$, and choose a cutoff supported in
$\mathbb B_{1+2^{-j}}^+(y_0)$, equal to one on
$\mathbb B_{1+2^{-j-1}}^+(y_0)$, with derivative bounded by $C2^j$.  Equation \eqref{e.s10.moser.energy} multiplies the next norm by at most
\[
\bigl(C(W)s^2(1+\Lambda+4^j)\bigr)^{1/s}.
\]
The exponents grow geometrically, by $n/(n-2)$ when $n>2$ and by $2$ when $n=1,2$.  Hence the sums of $1/s$ and $j/s$ over the iteration are finite, and so is the product of the displayed factors.  This proves \eqref{e.s10.local.moser}.
\end{proof}

\begin{lemma}[$L^p$ estimate]\label{l.s10.Lp}
For the choices in \eqref{e.s10.exponents}, if $K$ is sufficiently large and then $M$ is sufficiently large, every subsolution satisfying \eqref{e.s10.V.inequality} and \eqref{e.s10.almost.stability} on $\mathbb B_K^+(y_0)$ satisfies
\begin{equation}\label{e.s10.Lp.bound}
\|V_\alpha\|_{L^p(\mathbb B_4^+(y_0))}^p
\leq CK^{n-2p}+C(K)M^{-\delta p}K^n.
\end{equation}
The constant in the first term is independent of $K$ and $M$.  Consequently, the left side can be made arbitrarily small by choosing $K$ first and then $M$.
\end{lemma}

\begin{proof}
Let $\psi\geq0$ belong to $C^1(\mathbb B_K^+(y_0))$ and satisfy $\operatorname{supp}\psi\cap
\bigl(\partial\mathbb B_K(y_0)\cap\{y_1>0\}\bigr)=\emptyset.$
All integrals below are over $\mathbb B_K^+(y_0)$.  Since $q$ and $\tau$ are fixed, the constants denoted by $C$ in this proof may depend on them, but not on $K$ or $M$.

Test \eqref{e.s10.V.inequality} with $V_\alpha^{2q-1}\psi^2$.  The conormal condition removes the boundary integral on $\{y_1=0\}$, and the negative quadratic-gradient term changes the coefficient of the first integral from $2q-1$ to $2q$.  Thus
\begin{align}
2q\int V_\alpha^{2q-2}\psi^2
\left\langle\widetilde B_\alpha\nabla V_\alpha,\nabla V_\alpha\right\rangle
+2\int V_\alpha^{2q-1}\psi
\left\langle\widetilde B_\alpha\nabla V_\alpha,\nabla\psi\right\rangle
&\leq\frac{4A_1^2}{\sigma_0}\bigl(1+C(K)M^{-\delta}\bigr)
\int V_\alpha^{2q+1}\psi^2
\notag\\
&\quad+C(K)M^{-\delta}\int V_\alpha^{2q}\psi^2.
\label{e.s10.power.test}
\end{align}
Apply \eqref{e.s10.almost.stability} with $\eta=V_\alpha^q\psi$:
\begin{align}
\frac{4A_1^2}{\sigma_0}\int V_\alpha^{2q+1}\psi^2
&\leq\bigl(\tau+C(K)M^{-\delta}\bigr)
\Big[q^2\int V_\alpha^{2q-2}\psi^2
\left\langle\widetilde B_\alpha\nabla V_\alpha,\nabla V_\alpha\right\rangle
\notag\\
&\qquad+2q\int V_\alpha^{2q-1}\psi
\left\langle\widetilde B_\alpha\nabla V_\alpha,\nabla\psi\right\rangle
+\int V_\alpha^{2q}
\left\langle\widetilde B_\alpha\nabla\psi,\nabla\psi\right\rangle\Big]
\notag\\
&\quad+C(K)M^{-\delta}\int V_\alpha^{2q}\psi^2.
\label{e.s10.stability.power}
\end{align}
Substitute this estimate into \eqref{e.s10.power.test}.  After increasing $C(K)$, all perturbations of the fixed coefficients are bounded by $C(K)M^{-\delta}$, and hence
\begin{align}
&\bigl(2q-\tau q^2-C(K)M^{-\delta}\bigr)
\int V_\alpha^{2q-2}\psi^2
\left\langle\widetilde B_\alpha\nabla V_\alpha,\nabla V_\alpha\right\rangle
\notag\\
&\quad\leq C\left|\int V_\alpha^{2q-1}\psi
\left\langle\widetilde B_\alpha\nabla V_\alpha,\nabla\psi\right\rangle\right|
+C\int V_\alpha^{2q}
\left\langle\widetilde B_\alpha\nabla\psi,\nabla\psi\right\rangle
\notag\\
&\qquad+C(K)M^{-\delta}\int V_\alpha^{2q}\psi^2.
\label{e.s10.absorption.identity}
\end{align}
The number $2q-\tau q^2$ is positive by \eqref{e.s10.exponents}.  Choose $M$ so large that
\[
C(K)M^{-\delta}\leq\frac14(2q-\tau q^2).
\]
By Cauchy--Schwarz and Young's inequality, we have
\begin{align*}
C\left|\int V_\alpha^{2q-1}\psi
\left\langle\widetilde B_\alpha\nabla V_\alpha,\nabla\psi\right\rangle\right|
&\leq\frac14(2q-\tau q^2)
\int V_\alpha^{2q-2}\psi^2
\left\langle\widetilde B_\alpha\nabla V_\alpha,\nabla V_\alpha\right\rangle\\
&\quad+C\int V_\alpha^{2q}
\left\langle\widetilde B_\alpha\nabla\psi,\nabla\psi\right\rangle.
\end{align*}
Absorbing the first term and using uniform ellipticity yields
\[
\int V_\alpha^{2q-2}\psi^2|\nabla V_\alpha|^2
\leq C\int V_\alpha^{2q}|\nabla\psi|^2
+C(K)M^{-\delta}\int V_\alpha^{2q}\psi^2.
\]
Insert this estimate into \eqref{e.s10.stability.power}; the mixed term is handled by the same Cauchy--Schwarz and Young inequalities.  We obtain
\begin{equation}\label{e.s10.stable.power}
\int V_\alpha^{2q+1}\psi^2
\leq C\int V_\alpha^{2q}|\nabla\psi|^2
+C(K)M^{-\delta}\int V_\alpha^{2q}\psi^2.
\end{equation}

Now replace $\psi$ in \eqref{e.s10.stable.power} by $\psi^p$, where $p=2q+1$.  The fixed factor $p^2$ is absorbed into $C$.  H\"older's inequality gives
\[
\int V_\alpha^{p-1}\psi^{2p-2}|\nabla\psi|^2
\leq
\left(\int V_\alpha^p\psi^{2p}\right)^{(p-1)/p}
\left(\int|\nabla\psi|^{2p}\right)^{1/p}
\]
and
\[
\int V_\alpha^{p-1}\psi^{2p}
\leq
\left(\int V_\alpha^p\psi^{2p}\right)^{(p-1)/p}
\left(\int\psi^{2p}\right)^{1/p}.
\]
Dividing by the common factor, when it is nonzero, and raising to the power $p$ yields
\begin{equation}\label{e.s10.Holder}
\int V_\alpha^p\psi^{2p}
\leq C\int|\nabla\psi|^{2p}
+C(K)M^{-\delta p}\int\psi^{2p}.
\end{equation}

Finally, choose a radial cutoff about $y_0$ which is one on $\mathbb B_4^+(y_0)$, vanishes outside $\mathbb B_{K/2}^+(y_0)$, and satisfies $|\nabla\psi|\leq C/K$.  Its restriction to the half-space is admissible.  Since $|\mathbb B_{K/2}^+(y_0)|\leq CK^n$, equation \eqref{e.s10.Holder} gives \eqref{e.s10.Lp.bound}.  Finally,
\[
p=2q+1>\frac92\geq\frac n2,
\qquad\text{so}\qquad n-2p<0.
\]
Hence $CK^{n-2p}$ is made small by choosing $K$ large, and, after $K$ is fixed, $C(K)M^{-\delta p}K^n$ is made small by choosing $M$ large.  This proves the lemma.
\end{proof}

\begin{lemma}[Removal of the a priori upper bound]\label{l.s10.remove.bound}
Let $p>n/2$, $p\geq2$, and suppose that $V_\alpha$ satisfies \eqref{e.s10.V.inequality} on $\mathbb B_4^+(y_0)$, with $C(K)M^{-\delta}\leq1$ and
\[
\|V_\alpha\|_{L^p(\mathbb B_4^+(y_0))}\leq1.
\]
There is a constant $C_0=C_0(n,p,W)$ such that
\[
\|V_\alpha\|_{L^\infty(\mathbb B_2^+(y_0))}\leq C_0.
\]
\end{lemma}

\begin{proof}
Suppose otherwise.  Let $V_i$ be a sequence of subsolutions violating the conclusion.  For the $i$-th subsolution, $y_0$ denotes the centre of the half-ball $\mathbb B_4^+(y_0)$ in the statement and may depend on $i$.  Choose $x_i\in \mathbb B_2^+(y_0)$ such that $V_i(x_i)\to\infty$, and set
\[
r_i=V_i(x_i)^{-1/4}.
\]
For all large $i$,
$\overline{\mathbb B}_{r_i}(x_i)\cap\{y_1\geq0\}\subset\mathbb B_3^+(y_0)$.  Let $y_i$ maximize
\[
(r_i-|y-x_i|)^2V_i(y)
\]
on $\overline{\mathbb B}_{r_i}(x_i)\cap\{y_1\geq0\}$. The maximum is positive and is not attained on
$\partial\mathbb B_{r_i}(x_i)\cap\{y_1>0\}$, where
$(r_i-|y-x_i|)^2=0$.  Put
\[
r_i'=r_i-|y_i-x_i|>0,
\qquad
\lambda_i=V_i(y_i).
\]
Maximality gives
\begin{equation}\label{e.s10.point.selection}
r_i'\sqrt{\lambda_i}
\geq r_i\sqrt{V_i(x_i)}
=V_i(x_i)^{1/4}\to\infty.
\end{equation}
Since $r_i'\leq r_i\to0$, it follows that $\lambda_i\to\infty$.  If $|y-y_i|\leq r_i'/2$, then
$r_i-|y-x_i|\geq r_i'/2$; hence
\begin{equation}\label{e.s10.local.bound}
V_i(y)\leq4\lambda_i
\qquad\text{on }\mathbb B_{r_i'/2}(y_i)\cap\{y_1\geq0\}.
\end{equation}

Let $d_i$ be the Euclidean distance from $y_i$ to the boundary.  First suppose that $d_i\sqrt{\lambda_i}\geq4$.  Define
\[
\widetilde V_i(\widetilde y)
=\lambda_i^{-1}V_i(y_i+\lambda_i^{-1/2}\widetilde y).
\]
By \eqref{e.s10.point.selection}--\eqref{e.s10.local.bound}, this function is defined on $\mathbb B_2(0)$ for all large $i$ and satisfies
\[
\widetilde V_i(0)=1,
\qquad
0<\widetilde V_i\leq4.
\]
The coefficient in the rescaled equation is
\[
\widetilde B_i(y_i+\lambda_i^{-1/2}\widetilde y),
\]
not $\widetilde B_i(\widetilde y)$.  After discarding the negative quadratic-gradient term and dividing the rescaled equation by $\lambda_i^2$, we obtain
\begin{equation}\label{e.s10.blowup.equation}
-\operatorname{div}_{\widetilde y}\!\left(
\widetilde B_i(y_i+\lambda_i^{-1/2}\widetilde y)
\nabla_{\widetilde y}\widetilde V_i\right)
\leq\left(1+\frac{32A_1^2}{\sigma_0}\right)\widetilde V_i
\qquad\text{in }\mathbb B_2(0).
\end{equation}
Indeed, the term $C(K)M^{-\delta}V_i$ becomes
$C(K)M^{-\delta}\lambda_i^{-1}\widetilde V_i$ after division by $\lambda_i^2$, and its coefficient is at most one for all large $i$; ellipticity is unchanged.  Lemma \ref{l.s10.bounded.moser}, or its proof applied to \eqref{e.s10.blowup.equation}, now yields
\[
\begin{aligned}
1
&\leq c_0(n,p,W,4)\|\widetilde V_i\|_{L^p(\mathbb B_2(0))}\\
&=c_0(n,p,W,4)\lambda_i^{n/(2p)-1}
\|V_i\|_{L^p(\mathbb B_{2/\sqrt{\lambda_i}}(y_i))}\\
&\leq c_0(n,p,W,4)\lambda_i^{n/(2p)-1}\to 0.
\end{aligned}
\]
because $p>n/2$.  For all large $i$,
\[
\mathbb B_{2/\sqrt{\lambda_i}}(y_i)
\subset\mathbb B_{r_i'/2}(y_i)\cap\{y_1\geq0\}
\subset\mathbb B_4^+(y_0),
\]
which gives the last inequality.  This is a contradiction.

Now suppose that $d_i\sqrt{\lambda_i}<4$.  Let $y_i'$ be the Euclidean projection of $y_i$ onto the boundary and retain the same value $\lambda_i=V_i(y_i)$.  Define
\[
\widetilde V_i(\widetilde y)
=\lambda_i^{-1}V_i(y_i'+\lambda_i^{-1/2}\widetilde y),
\qquad
\widetilde y_i=\sqrt{\lambda_i}(y_i-y_i').
\]
Then $|\widetilde y_i|<4$ and $\widetilde V_i(\widetilde y_i)=1$.  By \eqref{e.s10.point.selection}, $r_i'\sqrt{\lambda_i}>40$ for all large $i$.  The inclusion
\[
y_i'+\lambda_i^{-1/2}\mathbb B_{16}^+(0)
\subset\mathbb B_{20/\sqrt{\lambda_i}}(y_i)\cap\{y_1\geq0\}
\subset\mathbb B_{r_i'/2}(y_i)\cap\{y_1\geq0\}
\]
follows from $|\widetilde y_i|<4$ and
$r_i'\sqrt{\lambda_i}>40$.  Thus \eqref{e.s10.local.bound} gives
$\widetilde V_i\leq4$ on $\mathbb B_{16}^+(0)$.  The coefficient is
\[
\widetilde B_i(y_i'+\lambda_i^{-1/2}\widetilde y),
\]
and
\[
-\operatorname{div}_{\widetilde y}\!\left(
\widetilde B_i(y_i'+\lambda_i^{-1/2}\widetilde y)
\nabla_{\widetilde y}\widetilde V_i\right)
\leq\left(1+\frac{32A_1^2}{\sigma_0}\right)\widetilde V_i
\qquad\text{in }\mathbb B_{16}^+(0),
\]
with zero conormal derivative on
$\mathbb B_{16}(0)\cap\{\widetilde y_1=0\}$.  The bounded Moser estimate on the radii $8$ and $16$, followed by the exact scaling of the $L^p$ norm, gives
\[
1\leq\|\widetilde V_i\|_{L^\infty(\mathbb B_8^+(0))}
\leq C(n,p,W)\|\widetilde V_i\|_{L^p(\mathbb B_{16}^+(0))}
\leq C(n,p,W)\lambda_i^{n/(2p)-1}
\|V_i\|_{L^p(\mathbb B_4^+(y_0))}\to 0.
\]
This is again a contradiction.  
\end{proof}

\begin{proof}[Proof of Proposition \ref{p.s10.scalar.smallness}]
Fix $q$, $\tau$, and $p$ as in \eqref{e.s10.exponents}.  Lemma \ref{l.s10.Lp} allows us to choose $K$ and then $M$ so that
\begin{equation}\label{e.s10.Lp.small}
\|V_\alpha\|_{L^p(\mathbb B_4^+(y_0))}
\leq\min\left\{1,\bigl(4c_0(n,p,W,C_0)\bigr)^{-1}\right\}.
\end{equation}
By the first bound, we can apply Lemma \ref{l.s10.remove.bound}; hence
\[
V_\alpha\leq C_0\qquad\text{on }\mathbb B_2^+(y_0).
\]
Apply Lemma \ref{l.s10.bounded.moser} on the radii $1$ and $2$ with $\Lambda=C_0$.  Then
\[
\|V_\alpha\|_{L^\infty(\mathbb B_1^+(y_0))}
\leq c_0(n,p,W,C_0)\|V_\alpha\|_{L^p(\mathbb B_2^+(y_0))}
\leq\frac14.
\]
\end{proof}

\subsection{Completion of the separation estimate}\label{ss.s10.completion}

\begin{proof}[Proof of Proposition \ref{p.s10.decay}]
We first assume \eqref{e.s10.threshold.general}.  Fix any $q$, $\tau$, and $p$ satisfying \eqref{e.s10.exponents}.  Choose the radius $K$ using \eqref{e.s10.Lp.bound}, and then choose $M$ sufficiently large for \eqref{e.s10.Lp.small}, Claims \ref{c.s10.graph.estimates}--\ref{c.s10.distance.comparison}, and Lemma \ref{l.s10.stability.transfer}.  Increase $M$ once more so that the final quantity in \eqref{e.s10.distance.gap} is at most $\log(4/3)$.  Finally choose $\varepsilon$ sufficiently small, uniformly in the layer index and in $r$.

For every
\[
y\in\mathbb B_{2K\kappa^{-1/2}+20|\log\varepsilon|}^+(x'),
\qquad
\min_{\gamma\in\{\alpha,\alpha+1\}}|z-f_\gamma(y)|
\leq20|\log\varepsilon|,
\]
the point $(y,z)$, the graph points $(y,f_\alpha(y))$ and
$(y,f_{\alpha+1}(y))$, and the normal-flow points used in
Claim \ref{c.s10.distance.comparison} and Lemma \ref{l.s10.stability.transfer}
all belong to
\[
B_{C(K)(\kappa^{-1/2}+|\log\varepsilon|)+300|\log\varepsilon|^2}^+(x).
\]
Under \eqref{e.s10.threshold.general}, for all sufficiently small $\varepsilon$,
\[
C(K)(\kappa^{-1/2}+|\log\varepsilon|)
+300|\log\varepsilon|^2<\frac{R}{80},
\qquad
x\in B_{r-KR_*}^+(0)\subset B_{5R/6}^+(0).
\]
The two-sheet test functions in Lemma \ref{l.s10.stability.transfer} are supported in
\[
B_{C(K)(\kappa^{-1/2}+|\log\varepsilon|)+120|\log\varepsilon|^2}^+(x)
\subset B_r^+(0),
\]
so Proposition \ref{p.stability.toda} applies.  For every
$y\in\mathbb B_{2K\kappa^{-1/2}}^+(x')$ and
$\gamma\in\{\alpha,\alpha+1\}$ apply Corollary \ref{c.s8.improved.toda} on
\[
B_{300|\log\varepsilon|^2}^+\bigl((y,f_\gamma(y))\bigr)
\subset B_r^+(0).
\]

Recall that $x=(x',f_\alpha(x'))\in\Sigma_\alpha\cap B_{r-KR_*}^+(0)$ is the point fixed in \eqref{e.s10.badpoint}.  The coordinate construction uses a full geodesic chart when
\[
B_{C(K)(\kappa^{-1/2}+|\log\varepsilon|)+300|\log\varepsilon|^2}^+(x)
\cap\partial_0B_R^+(0)=\varnothing.
\]
Apply the lower-distance calculation in the proof of Claim \ref{c.s10.distance.comparison} under the explicit condition
$f_{\alpha+1}(x')-f_\alpha(x')\leq C|\log\varepsilon|$, together with Claim \ref{c.s10.graph.estimates}.  Since \eqref{e.s10.logdistance} gives $|t_{\alpha+1}(x)|<2|\log\varepsilon|$, it follows that
$f_{\alpha+1}(x')-f_\alpha(x')<3|\log\varepsilon|$ for all sufficiently small $\varepsilon$.  The hypothesis of Claim \ref{c.s10.distance.comparison} therefore holds at $x'$, and
$V_\alpha(\kappa^{1/2}x')=\kappa^{-1}e^{-(f_{\alpha+1}-f_\alpha)(x')}$.  By \eqref{e.s10.badpoint}, \eqref{e.s10.weights.compare}, and the choice of $M$,
\begin{equation}\label{e.s10.final.lower}
V_\alpha(\kappa^{1/2}x')
=\kappa^{-1}e^{-(f_{\alpha+1}-f_\alpha)(x')}
\geq\frac34\kappa^{-1}e^{-|t_{\alpha+1}(x)|}
>\frac38.
\end{equation}
By Lemma \ref{l.s10.gap} and Lemma \ref{l.s10.stability.transfer}, we can apply Proposition \ref{p.s10.scalar.smallness} with $y_0=\kappa^{1/2}x'$, which gives $V_\alpha(\kappa^{1/2}x')\leq1/4$.  This contradicts \eqref{e.s10.final.lower}. Taking the maximum over the sheets proves \eqref{e.s10.decay} under \eqref{e.s10.threshold.general}.
\end{proof}

\begin{corollary}\label{c.s10.global.decay}
Under the assumptions of Theorem \ref{t.main.ball},
\begin{equation}\label{e.s10.global.decay}
A(2R/3;0)\leq C\varepsilon^2|\log\varepsilon|^2.
\end{equation}
\end{corollary}

\begin{proof}
Increase $M$ in Proposition \ref{p.s10.decay}, while keeping its $K$ fixed, so that
\[
\frac{\sqrt M\log2}{6K}>3.
\]
This preserves Proposition \ref{p.s10.decay}.  Suppose, to the contrary, that
\[
A(2R/3;0)>M\varepsilon^2|\log\varepsilon|^2.
\]
By monotonicity,
\[
A(r;0)>M\varepsilon^2|\log\varepsilon|^2
\qquad\text{for every }r\in[2R/3,5R/6].
\]  Hence, for all sufficiently small $\varepsilon$ and every such $r$,
\[
\max\{A(r;0)^{-1/2},500|\log\varepsilon|^2\}
\leq\frac{R}{\sqrt M|\log\varepsilon|}.
\]
Indeed,
\[
A(r;0)^{-1/2}\leq\frac{R}{\sqrt M|\log\varepsilon|}
\]
follows from $A(r;0)>M\varepsilon^2|\log\varepsilon|^2$, while
\[
500|\log\varepsilon|^2\leq\frac{R}{\sqrt M|\log\varepsilon|}
\]
holds after $\varepsilon$ is decreased.

Since $A(\,\cdot\,;0)$ is nondecreasing, Proposition \ref{p.s10.decay} therefore implies
\[
A\!\left(r-\frac{KR}{\sqrt M|\log\varepsilon|};0\right)
\leq\frac12A(r;0)
\qquad\text{whenever }r\in[2R/3,5R/6]
\]
provided that the radius on the left is at least $2R/3$.  Starting at $5R/6$, this estimate can be iterated at least
\[
\left\lfloor\frac{\sqrt M|\log\varepsilon|}{6K}\right\rfloor
\]
times before reaching $2R/3$.  Since $A(5R/6;0)\leq1$, it follows that
\[
\begin{aligned}
A(2R/3;0)
&\leq2^{-\left\lfloor\sqrt M|\log\varepsilon|/(6K)\right\rfloor}\\
&\leq2\varepsilon^{\sqrt M\log2/(6K)}
<M\varepsilon^2|\log\varepsilon|^2
\end{aligned}
\]
for all sufficiently small $\varepsilon$.  This contradicts
$A(2R/3;0)>M\varepsilon^2|\log\varepsilon|^2$ and proves \eqref{e.s10.global.decay}.
\end{proof}

\section{Final estimates and the boundary correction on nonzero levels}\label{s.final.estimates}

Throughout this section, $1\leq n\leq9$ and $R=\varepsilon^{-1}$.  We use the geometric estimates of Section \ref{section twisted Fermi}, the corrected approximate solution of Section \ref{s.approx.toda}, the estimates of Sections \ref{Inner-outer gluing regularity estimates of error} and \ref{improved estimates section}, and the separation estimate of Section \ref{s.separation.emden}.  Constants depend only on the fixed data in Theorem \ref{t.main.ball}; constants in statements concerning nonzero levels may also depend on $b'$.  

There are two differences from the interior conclusion in \cite[Proof of Theorem 1.1]{wang-weiadv}.  First, Proposition \ref{p.s10.decay} gives a squared logarithm for a general boundary.  Second, the correction $\overline{\Gb}_\alpha$ constructed in Section \ref{s.approx.toda} need not vanish at a nonzero value of the heteroclinic variable.  The latter term enters the second derivatives of the graph of $\{u=t\}$ and therefore cannot be absorbed into the error.

\subsection{Separation and estimates on the zero level}

\begin{proposition}[Separation bounds]\label{p.s11.separation}
Under the assumptions of Theorem \ref{t.main.ball}, 
\begin{equation}\label{e.s11.separation.general}
 A(2R/3;0)\leq C\varepsilon^2|\log\varepsilon|^2.
\end{equation}
Consequently, for every $y\in\Sigma_\alpha\cap B_{2R/3}^+(0)$ and every existing neighbouring sheet,
\begin{equation}\label{e.s11.distance.lower}
 D_\alpha(y)\geq2|\log\varepsilon|-2\log|\log\varepsilon|-C.
\end{equation}
\end{proposition}

\begin{remark}
If $\partial_0B_{3R}^+(0)$ is totally geodesic, then
\begin{equation}\label{e.s11.separation.tg}
 A(2R/3;0)\leq C\varepsilon^2|\log\varepsilon|.
\end{equation}
\end{remark}

\begin{remark}[The boundary condition in the intrinsic separation equation]\label{r.s11.intrinsic.boundary}
The intrinsic improvement discussed in \cite[Section 10.2 and Remarks 10.9--10.10]{wang-weiadv} cannot be imported without a new boundary argument.  To see the obstruction, let $\Sigma_{\alpha+1}$ be the upper neighbour of $\Sigma_\alpha$ and set
\[
 \rho(y)=|t_{\alpha+1}(y,0)|=-t_{\alpha+1}(y,0)
 \qquad (y\in\Sigma_\alpha).
\]
At a point of $\partial_0\Sigma_\alpha$, the inward normal $\nu$ is tangent to $\Sigma_\alpha$.  Lemma \ref{l.dz(nu)}, applied in the coordinates of $\Sigma_{\alpha+1}$, gives the exact identity
\begin{equation}\label{e.s11.gap.boundary}
\begin{split}
 \partial_\nu\rho(y)
 &=-J_{\alpha+1}\bigl(\Pi_{\alpha+1}(y,0),t_{\alpha+1}(y,0)\bigr)\\
 &=w_{\alpha+1}\bigl(\Pi_{\alpha+1}(y,0)\bigr)\rho(y)
 -\widetilde J_{\alpha+1}\bigl(\Pi_{\alpha+1}(y,0),-\rho(y)\bigr).
\end{split}
\end{equation}
The last term is bounded by
\[
 C\varepsilon^2\rho(y)^2\bigl(1+\rho(y)^3\bigr),
\]
whereas the first term is generally of order $\varepsilon\rho(y)$.  At the scale $\kappa\simeq\varepsilon^2|\log\varepsilon|$, the rescaling $y-y_0=\kappa^{-1/2}\widetilde y$ changes this boundary term into one of size $\varepsilon\kappa^{-1/2}\rho\simeq |\log\varepsilon|^{1/2}$.
It is therefore not a vanishing Neumann error.  Nonnegative ambient Ricci curvature does not imply $w_\alpha=0$.
\end{remark}

\begin{proof}
Equation \eqref{e.s11.separation.general} is Corollary \ref{c.s10.global.decay}.  Since $e^{-D_\alpha(y)}\leq A(2R/3;0)$, taking logarithms gives \eqref{e.s11.distance.lower}.  If $\Sigma_\alpha$ has no neighbouring sheet, then $D_\alpha=+\infty$ and there is nothing to prove.

\end{proof}

\begin{proposition}[Zero-level estimates]\label{p.s11.zero.level}
On $B_{5R/8}^+(0)$ one has
\begin{align}\label{e.s11.zero.basic}
&\|\phi\|_{C^{2,\theta}(B_{5R/8}^+(0))}
 +\max_\alpha\|h_\alpha\|_{C^{2,\theta}(\Sigma_\alpha\cap B_{5R/8}^+(0))}\notag\\
&\qquad
 +\max_\alpha\|H_\alpha(\cdot,0)+\Delta_{\alpha,0}h_\alpha\|_{C^\theta(\Sigma_\alpha\cap B_{5R/8}^+(0))}
 \leq C\varepsilon^2|\log\varepsilon|^2,
\end{align}
\begin{equation}\label{e.s11.zero.shift}
 \max_\alpha\|\nabla_{\alpha,0}h_\alpha\|_{C^{1,\theta}(\Sigma_\alpha\cap B_{5R/8}^+(0))}
 \leq C\varepsilon^2.
\end{equation}
For every fixed $D>0$,
\begin{equation}\label{e.s11.zero.horizontal}
 \max_{\substack{\alpha\\1\leq i\leq n}}
 \|\phi_{y_i}\|_{C^{1,\theta}(B_{5R/8}^+(0)\cap\mathcal M_\alpha^0\cap\{|t_\alpha|\leq D\})}
 \leq C(D)\varepsilon^2.
\end{equation}
Moreover,
\begin{align}\label{e.s11.zero.toda}
\max_\alpha\bigl\|&
 H_{\alpha,0}+\Delta_{\alpha,0}h_\alpha
 -\frac{2}{\sigma_0}\bigl(
 A_{(-1)^{\alpha-1}}^2e^{-|t_{\alpha-1}(\cdot,0)|}
 -A_{(-1)^\alpha}^2e^{-|t_{\alpha+1}(\cdot,0)|}
 \bigr)\bigr\|_{C^\theta(\Sigma_\alpha\cap B_{5R/8}^+(0))}\\
 &\quad\leq C\varepsilon^2.
\end{align}
In particular,
\begin{equation}\label{e.s11.zero.H}
 \max_\alpha\|H_\alpha(\cdot,0)\|_{C^\theta(\Sigma_\alpha\cap B_{5R/8}^+(0))}
 \leq C\varepsilon^2|\log\varepsilon|^2.
\end{equation}

Let $H_{\alpha,\varepsilon}$ and $\sff_{\alpha,\varepsilon}$ be the mean curvature and second fundamental form of the corresponding zero sheet in the original metric $\hat g$.  On $B_{1/2}^+(0)$,
\begin{equation}\label{e.s11.zero.original.H}
 \|H_{\alpha,\varepsilon}\|_{L^\infty}
 \leq C\varepsilon|\log\varepsilon|^2,
 \qquad
 [H_{\alpha,\varepsilon}]_{C^\theta}
 \leq C\varepsilon^{1-\theta}|\log\varepsilon|^2,
\end{equation}
and
\begin{equation}\label{e.s11.zero.original.II}
 \|\sff_{\alpha,\varepsilon}\|_{C^\theta}\leq C.
\end{equation}

If $n=1$, then
\[
 \|\sff_{\alpha,\varepsilon}\|_{L^\infty}
 \leq C\varepsilon|\log\varepsilon|^2,
 \qquad
 [\sff_{\alpha,\varepsilon}]_{C^\theta}
 \leq C\varepsilon^{1-\theta}|\log\varepsilon|^2.
\]
\end{proposition}

\begin{remark}
If the boundary is totally geodesic, the factors $|\log\varepsilon|^2$ in \eqref{e.s11.zero.basic}, \eqref{e.s11.zero.H}, and \eqref{e.s11.zero.original.H} may be replaced by $|\log\varepsilon|$; estimates \eqref{e.s11.zero.shift} and \eqref{e.s11.zero.horizontal} remain unchanged.
\end{remark}

\begin{proof}
Fix $x\in B_{5R/8}^+(0)$.  For all sufficiently small $\varepsilon$,
\[
 B_{2+200|\log\varepsilon|^2}^+(x)\subset B_{2R/3}^+(0).
\]
Apply Proposition \ref{aprior C2alpha Hhphi}, Corollary \ref{c.s8.shift.derivatives}, and Proposition \ref{p.s8.horizontal.derivatives} with centre $x$ and radius $2$.  Every interaction quantity on their right-hand sides is bounded by
\[
 A(2+200|\log\varepsilon|^2;x)\leq A(2R/3;0).
\]
Therefore
\begin{align*}
&\|\phi\|_{C^{2,\theta}(B_2^+(x))}
 +\max_\alpha\|h_\alpha\|_{C^{2,\theta}(\Sigma_\alpha\cap B_2^+(x))}\\
&\qquad
 +\max_\alpha\|H_\alpha(\cdot,0)+\Delta_{\alpha,0}h_\alpha\|_{C^\theta(\Sigma_\alpha\cap B_2^+(x))}
 \leq C\varepsilon^2|\log\varepsilon|^2.
\end{align*}
Moreover,
\[
 (\varepsilon^2|\log\varepsilon|^2)^{3/2}
 =\varepsilon^3|\log\varepsilon|^3\leq C\varepsilon^2,
 \qquad
 \varepsilon^{1/6}(\varepsilon^2|\log\varepsilon|^2)
 =\varepsilon^{13/6}|\log\varepsilon|^2\leq C\varepsilon^2.
\]
Since $D$ is fixed, $D<K|\log\varepsilon|$ for sufficiently small $\varepsilon$.  Proposition \ref{p.s8.horizontal.derivatives} and Corollary \ref{c.s8.shift.derivatives} then give
\[
 \max_\alpha\|\nabla_{\alpha,0}h_\alpha\|_{C^{1,\theta}(\Sigma_\alpha\cap B_2^+(x))}
 +\max_{\substack{\alpha\\1\leq i\leq n}}
 \|\phi_{y_i}\|_{C^{1,\theta}(B_2^+(x)\cap\mathcal M_\alpha^0\cap\{|t_\alpha|\leq D\})}
 \leq C\varepsilon^2.
\]
This proves \eqref{e.s11.zero.shift} and \eqref{e.s11.zero.horizontal}.  

Assume now that $n=1$.  At every point of the original zero sheet, choose a unit tangent vector $\tau$.  Since the tangent space is one-dimensional,
\[
 H_{\alpha,\varepsilon}
 =\sff_{\alpha,\varepsilon}(\tau,\tau),
 \qquad
 |\sff_{\alpha,\varepsilon}|_{\hat g}
 =|H_{\alpha,\varepsilon}|.
\]
In a graph coordinate, the covariant tensor $\sff_{\alpha,\varepsilon}$ is $H_{\alpha,\varepsilon}$ times the induced tangent metric.  The $C^{1,1}$ graph bound in (H2) and the fixed $C^4$ bound for $\hat g$ give a uniform $C^\theta$ bound for that induced metric.  Hence
\[
 [\sff_{\alpha,\varepsilon}]_{C^\theta}
 \leq C\left(
 [H_{\alpha,\varepsilon}]_{C^\theta}
 +\|H_{\alpha,\varepsilon}\|_{L^\infty}
 \right).
\]

Corollary \ref{c.s8.improved.toda}, together with \eqref{e.s11.zero.shift} and the preceding two numerical estimates, gives \eqref{e.s11.zero.toda}.  More precisely, for every $y_0\in\Sigma_\alpha\cap B_{5R/8}^+(0)$ and every existing neighbour $\Sigma_{\alpha\pm1}$,
\[
 \|e^{-|t_{\alpha\pm1}(\cdot,0)|}\|_{C^\theta(B_1^\alpha(y_0))}
 \leq C A(2R/3;0).
\]
Indeed, $B_1^\alpha(y_0)\subset\Sigma_\alpha\cap B_{2R/3}^+(0)$ for small $\varepsilon$, and the coordinate tangent derivatives of $t_{\alpha\pm1}(\cdot,0)$ are uniformly bounded on $B_1^\alpha(y_0)$.  The mean value theorem controls pairs of points in $B_1^\alpha(y_0)$ at base-coordinate distance less than one; pairs at distance at least one are controlled by the supremum $A(2R/3;0)$.  Equations \eqref{e.s11.zero.shift}, \eqref{e.s11.zero.toda}, and \eqref{e.s11.separation.general} prove \eqref{e.s11.zero.H}.

Under the dilation $x\mapsto\varepsilon x$,
\begin{equation}\label{e.s11.scaling.H}
 H_{\alpha,\varepsilon}(\varepsilon y)=\varepsilon^{-1}H_\alpha(y,0),
 \qquad
 [H_{\alpha,\varepsilon}]_{C^\theta}
 =\varepsilon^{-1-\theta}[H_\alpha(\cdot,0)]_{C^\theta},
\end{equation}
up to fixed equivalence constants for the graphical coordinates.  This proves \eqref{e.s11.zero.original.H}.

It remains to prove \eqref{e.s11.zero.original.II}.  In the graph coordinates from (H2), the mean-curvature equation for the original zero sheet is a uniformly elliptic quasilinear equation whose principal coefficient is the inverse graph metric divided by the graph area factor; see the computation in the proof of Lemma \ref{mean curvature expansion tf}.  The fixed slope and $C^{1,1}$ bounds in (H2) give uniform ellipticity and uniform $C^\theta$ bounds for the coefficients.  The boundary condition is $\partial_{x_1}\hat f_{\alpha,\varepsilon}=0$.  Since \eqref{e.s11.zero.original.H} gives a uniform $C^\theta$ bound for the right-hand side, the local interior and Neumann Schauder estimates of Appendix \ref{s.ell.est}, applied on coordinate neighbourhoods contained in $B_{5/8}^+(0)$, give a uniform $C^{2,\theta}$ bound for the graph on $B_{1/2}^+(0)$.  This is equivalent to \eqref{e.s11.zero.original.II}.  This last step is the boundary version of the closing elliptic argument in \cite[Proof of Theorem 1.1, equations (11.1)--(11.2)]{wang-weiadv}.

If the boundary is totally geodesic, use \eqref{e.s11.separation.tg} in the same argument.  The two error terms from Section \ref{improved estimates section} are still $O(\varepsilon^2)$.
\end{proof}

Notice that this concludes the proof of Theorem \ref{t.main.ball}.

\subsection{Nonzero levels}

For the remainder of the section, fix $t\in[-1+b',1-b']$ and set
\[
 s(t)=\Hb^{-1}(t).
\]
We orient every level graph by the normal having positive $z$-component.  This orientation agrees with the one used for $H_\alpha(\cdot,0)$ at $t=0$; the normal $\nabla u/|\nabla u|$ differs from it by the constant sign $(-1)^\alpha$ near $\Sigma_\alpha$.

\begin{lemma}[Graphs of the nonzero levels]\label{l.s11.level.graph}
For each $\alpha$, the portion of $\{u=t\}$ associated with $\Sigma_\alpha$ and lying over $\Sigma_\alpha\cap B_{3R/5}^+(0)$ is the unique graph
\begin{equation}\label{e.s11.level.height}
 z=\ell(y,t)=h_\alpha(y)+(-1)^\alpha s(t)+\eta(y,t),
 \qquad
 u(y,\ell(y,t))=t,
\end{equation}
in a normal strip whose width depends only on $b'$ and $c_0$.  Uniformly in $t$ and $\alpha$,
\begin{equation}\label{e.s11.eta.rough}
 \|\eta(\cdot,t)\|_{C^{2,\theta}(\Sigma_\alpha\cap B_{3R/5}^+(0))}
 \leq C(b')\varepsilon,
\end{equation}
and
\begin{equation}\label{e.s11.eta.correction}
 \left\|
 \eta(\cdot,t)
 +\frac{(-1)^\alpha\overline{\Gb}_\alpha(\cdot,s(t))}{\Hb'(s(t))}
 \right\|_{C^{2,\theta}(\Sigma_\alpha\cap B_{3R/5}^+(0))}
 \leq C(b')\varepsilon^2|\log\varepsilon|^2.
\end{equation}
At $t=0$ one has $\ell(\cdot,0)=0$ and $\eta(\cdot,0)=-h_\alpha$.  Every point of
$B_{R/2}^+(0)\cap\{|u|\leq1-b'\}$ belongs to one of the graphs \eqref{e.s11.level.height}.
\end{lemma}

\begin{proof}
In this proof, take the fixed number $D$ in \eqref{e.s11.zero.horizontal} to be
\[
 D=2+\max_{|\tau|\leq1-b'}|\Hb^{-1}(\tau)|
       +\frac{2(1-b')}{c_0}.
\]
Fix $y_0\in\Sigma_\alpha\cap B_{3R/5}^+(0)$.  Since $D$ is independent of $\varepsilon$, \eqref{e.s11.distance.lower} implies, for all sufficiently small $\varepsilon$,
\[
 Y_{\Sigma_\alpha}\bigl(B_2^\alpha(y_0)\times[-D,D]\bigr)\subset\mathcal M_\alpha^0.
\]
For every $(y,z)\in B_2^\alpha(y_0)\times[-D,D]$, the cutoff used in the definition of $\Hb_\alpha(y,z;h_\alpha)$ equals one.  Lemma \ref{l.sum.of.exp}, Corollary \ref{improved G estimates}, and the distance and projection estimates of Section \ref{section twisted Fermi} give
\begin{align}\label{e.s11.foreign.terms}
&\Bigl\|
 \sum_{\beta\neq\alpha}
 \Bigl(\Hb_\beta+(-1)^\beta\operatorname{sgn}(\beta-\alpha)\Bigr)
 \Bigr\|_{C^{2,\theta}(B_1^\alpha(y_0)\times[-D,D])}
 \leq C(b')A(2R/3;0),\notag\\
&\Bigl\|\sum_{\beta\neq\alpha}\Gb_\beta\Bigr\|_{C^{2,\theta}(B_1^\alpha(y_0)\times[-D,D])}
 \leq C(b')\varepsilon A(2R/3;0)^{3/4}
 \leq C(b')\varepsilon^2.
\end{align}
The first estimate follows by differentiating the exponentially decaying tails on $B_1^\alpha(y_0)\times[-D,D]$ and summing with Lemma \ref{l.sum.of.exp}.  The second follows from \eqref{e.G.est} with any fixed exponential weight larger than $3/4$.  On $B_1^\alpha(y_0)\times[-D,D]$, the distance and projection estimates control the maps $t_\beta$ and $\Pi_\beta$ that occur in every summand $\Hb_\beta(y,z)$ and $\Gb_\beta(y,z)$, together with their derivatives and $C^\theta$ seminorms.

Equations \eqref{e.s11.zero.basic}, \eqref{e.G.est}, and \eqref{e.s11.foreign.terms} yield
\[
 \left\|u(y,z)-\Hb\bigl((-1)^\alpha(z-h_\alpha(y))\bigr)\right\|_{C^{2,\theta}(B_1^\alpha(y_0)\times[-D,D])}
 \leq C(b')\varepsilon.
\]
Moreover,
\[
 \min_{|\tau|\leq1-b'}\Hb'(\Hb^{-1}(\tau))>0.
\]
The intermediate value theorem and the implicit function theorem therefore give a unique root $z=\ell(y,t)$ for every $y\in B_1^\alpha(y_0)$.  Uniqueness makes the locally defined roots agree on overlaps as $y_0$ varies, producing the graph \eqref{e.s11.level.height} over $\Sigma_\alpha\cap B_{3R/5}^+(0)$.  For every $y\in B_1^\alpha(y_0)$, evaluation at $(y,\ell(y,t))$ gives
\begin{equation}\label{e.s11.implicit.derivatives}
 \ell_i=-\frac{u_i}{u_z},
 \qquad
 \ell_{ij}=-\frac{u_{ij}+u_{iz}\ell_j+u_{jz}\ell_i+u_{zz}\ell_i\ell_j}{u_z}.
\end{equation}
The estimate
\[
 \left\|u(y,z)-\Hb\bigl((-1)^\alpha(z-h_\alpha(y))\bigr)\right\|_{C^{2,\theta}(B_1^\alpha(y_0)\times[-D,D])}
 \leq C(b')\varepsilon,
\]
together with the product and composition estimates in \eqref{e.s11.implicit.derivatives}, gives \eqref{e.s11.eta.rough} on $B_1^\alpha(y_0)$.  Taking the supremum over $y_0\in\Sigma_\alpha\cap B_{3R/5}^+(0)$ gives the $C^0$ bounds on $\Sigma_\alpha\cap B_{3R/5}^+(0)$.  If $p,q\in\Sigma_\alpha\cap B_{3R/5}^+(0)$ have base-coordinate distance less than one, apply the estimate on $B_1^\alpha(p)$; if their base-coordinate distance is at least one, bound the H\"older quotient by twice the $C^0$ norm on $\Sigma_\alpha\cap B_{3R/5}^+(0)$.  This gives the full $C^{2,\theta}$ norm in \eqref{e.s11.eta.rough}.

At the root \eqref{e.s11.level.height}, the exact identity $u=\Hb_*+\Gb_*+\phi$ becomes
\begin{align}\label{e.s11.root.identity}
0={}&\Hb\bigl(s(t)+(-1)^\alpha\eta\bigr)-\Hb(s(t))
 +\overline{\Gb}_\alpha\bigl(y,s(t)+(-1)^\alpha\eta\bigr)
 +\phi(y,\ell)\notag\\
&+\sum_{\beta\neq\alpha}
 \left(\Hb_\beta(y,\ell)+(-1)^\beta\operatorname{sgn}(\beta-\alpha)
 +\Gb_\beta(y,\ell)\right).
\end{align}
Taylor's theorem and the higher-order version of \eqref{e.G.est} give, with $s=s(t)$,
\[
 \left\|
 \Hb\bigl(s+(-1)^\alpha\eta\bigr)-\Hb(s)
 -(-1)^\alpha\Hb'(s)\eta
 \right\|_{C^{2,\theta}(B_1^\alpha(y_0))}
 \leq C(b')\|\eta\|_{C^{2,\theta}(B_1^\alpha(y_0))}^2,
\]
\[
 \left\|
 \overline{\Gb}_\alpha\bigl(\cdot,s+(-1)^\alpha\eta\bigr)
 -\overline{\Gb}_\alpha(\cdot,s)
 \right\|_{C^{2,\theta}(B_1^\alpha(y_0))}
 \leq C(b')\varepsilon\|\eta\|_{C^{2,\theta}(B_1^\alpha(y_0))}.
\]
Substituting \eqref{e.s11.eta.rough}, \eqref{e.s11.zero.basic}, and \eqref{e.s11.foreign.terms} into \eqref{e.s11.root.identity}, and dividing by $\Hb'(s(t))$, gives \eqref{e.s11.eta.correction} on $B_1^\alpha(y_0)$.  Taking the supremum over $y_0\in\Sigma_\alpha\cap B_{3R/5}^+(0)$ gives the $C^0$ part on $\Sigma_\alpha\cap B_{3R/5}^+(0)$.  If $p,q\in\Sigma_\alpha\cap B_{3R/5}^+(0)$ have base-coordinate distance less than one, apply the estimate on $B_1^\alpha(p)$; if their base-coordinate distance is at least one, bound the H\"older quotient by twice that $C^0$ norm.  This proves \eqref{e.s11.eta.correction}.

For $t=0$, $s(0)=0$, $\overline{\Gb}_\alpha(y,0)=0$, and the zero sheet is $z=0$; hence $\eta=-h_\alpha$.  Finally, let $p\in B_{R/2}^+(0)\cap\{|u|\leq1-b'\}$.  The rescaled form of (H1) is $|\nabla u|\geq c_0$ on $\{|u|\leq1-b\}$.  The integral curve of $-\operatorname{sign}(u(p))\nabla u/|\nabla u|$ starting at $p$ reaches a zero of $u$ after length at most $(1-b')/c_0$.  Let $q$ be the zero of $u$ reached by this integral curve and let $\Sigma_\alpha$ be the zero sheet containing $q$.  For sufficiently small $\varepsilon$, one has
\[
 q\in\Sigma_\alpha\cap B_{3R/5}^+(0),
 \qquad
 \Pi_\alpha(p)\in\Sigma_\alpha\cap B_{3R/5}^+(0).
\]
The uniqueness in \eqref{e.s11.level.height}, applied with $t=u(p)$ and base point $\Pi_\alpha(p)$, therefore gives $p\in\Sigma_\alpha(u(p))$.
\end{proof}

\begin{proposition}[Mean curvature of a nonzero level]\label{p.s11.level.H}
Let $H_{\Sigma_\alpha(t)}$ be the mean curvature of the graph \eqref{e.s11.level.height}, with the orientation fixed above.  Uniformly for $t\in[-1+b',1-b']$,
\begin{equation}\label{e.s11.level.H.correction}
 \left\|
 H_{\Sigma_\alpha(t)}-H_\alpha(\cdot,0)
 +\frac{(-1)^\alpha}{\Hb'(s(t))}
 \Delta_{\alpha,0}\overline{\Gb}_\alpha(\cdot,s(t))
 \right\|_{C^\theta(\Sigma_\alpha\cap B_{3R/5}^+(0))}
 \leq C(b')\varepsilon^2|\log\varepsilon|^2.
\end{equation}
Consequently,
\begin{equation}\label{e.s11.level.H.general}
 \|H_{\Sigma_\alpha(t)}\|_{C^\theta(\Sigma_\alpha\cap B_{3R/5}^+(0))}
 \leq C(b')\varepsilon.
\end{equation}
\end{proposition}

\begin{remark}\label{r.s11.enhanced.boundary}
Recall if $|\nabla u(x)|\neq 0$, then
\begin{equation}\label{e.s11.enhanced.identity}
 |\mathcal A(u)|^2
 =|\sff_{\{u=u(x)\}}|^2+|\nabla^T\log|\nabla u||^2.
\end{equation}
If in addition $x\in \partial_0B^+_{3R}(0)$,
\begin{equation}\label{e.s11.loggradient.boundary}
 \partial_\nu\log|\nabla u|
 =\sff_{\partial_0B_{3R}^+(0)}\left(
 \frac{\nabla u}{|\nabla u|},\frac{\nabla u}{|\nabla u|}\right).
\end{equation}
In particular, on $\partial_0\Sigma_\alpha$, we have $\partial_\nu\log|\nabla u|=-w_\alpha$, and
\begin{equation}\label{e.s11.zero.boundary.identity}
 |\mathcal A(u)|^2\geq |\sff_{\Sigma_\alpha}|^2+|w_\alpha|^2.
\end{equation}
If $n=1$, then at every regular boundary point, in the original metric,
\begin{equation}\label{e.s11.surface.boundary.identity}
 |\mathcal A(u_\varepsilon)|_{\hat g}^2
 =|\sff_{\{u_\varepsilon=u_\varepsilon(x)\}}|_{\hat g}^2
 +|\sff_{\partial_0B_3^+(0)}|_{\hat g}^2.
\end{equation}
Thus, in ambient dimension two, a decay estimate for $|\mathcal A(u_\varepsilon)|$ at boundary points in the transition layer requires the curvature of the boundary to vanish at the same rate.
\end{remark}

\begin{remark}
If $\partial_0B_3^+(0)$ is totally geodesic, then $\overline{\Gb}_\alpha=0$ for every sheet $\Sigma_\alpha$ meeting $B_{2R/3}^+(0)$. Then uniformly for $t\in[-1+b',1-b']$, 
\begin{equation}\label{e.s11.tg.level.H}
 \|H_{\Sigma_\alpha(t)}-H_\alpha(\cdot,0)\|_{C^\theta(\Sigma_\alpha\cap B_{3R/5}^+(0))}
 \leq C(b')\varepsilon^2,
\end{equation}
\begin{equation}\label{e.s11.tg.level.H.total}
 \|H_{\Sigma_\alpha(t)}\|_{C^\theta(\Sigma_\alpha\cap B_{3R/5}^+(0))}
 \leq C(b')\varepsilon^2|\log\varepsilon|.
\end{equation}
and its second fundamental form satisfies
\begin{equation}
 \|\sff_{\Sigma_\alpha(t)}\|_{L^\infty(B_{R/2}^+(0))} +\varepsilon^{-\theta}[\sff_{\Sigma_\alpha(t)}]_{C^\theta(B_{R/2}^+(0))}
 \leq C(b')\varepsilon.
\end{equation}
If, in addition, $n=1$, then each level set is a curve, so its second fundamental form has one component and its norm equals the absolute value of its mean curvature.  Consequently,
\begin{equation}\label{e.s11.tg.enhanced}
\begin{aligned}
 \sup_{p\in B_{R/2}^+(0)\cap\{|u|\leq1-b'\}}
 |\mathcal A(u)(p)|
 &\leq C(b')\varepsilon^2|\log\varepsilon|.
\end{aligned}
\end{equation}
Unlike the dimension-one conclusion in Proposition \ref{p.s11.zero.level}, this estimate still requires total geodesicity of the boundary, as noticed in Remark \ref{r.s11.enhanced.boundary}.
\end{remark}

\begin{proof}[Proof of Proposition \ref{p.s11.level.H}]
The graph in Lemma \ref{l.s11.level.graph} is a regular level contained in $\{|u|\leq1-b'\}$.  Hence (H3), together with the uniformly controlled twisted-coordinate map, gives $|\nabla_{\alpha,0}\ell|\leq C\eta_0$ and $|\nabla_{\alpha,0}^2\ell|\leq C\Lambda_0\varepsilon$.  The smallness of $\eta_0$ and $\Lambda_0$ is chosen so that Lemma \ref{mean curvature expansion tf} applies to $\ell(\cdot,t)$.  Since the operator in that lemma annihilates constants,
\begin{align}\label{e.s11.mc.expansion}
 H_{\Sigma_\alpha(t)}-H_\alpha(\cdot,0)
 ={}&\mathcal L_{\Sigma_\alpha[\ell]}(h_\alpha+\eta)
 +\bigl(|\sff_{\Sigma_\alpha}|^2
 +\Ric_g(\bn_{\Sigma_\alpha},\bn_{\Sigma_\alpha})\bigr)\ell
 +\mathcal Q(y,\ell,d\ell).
\end{align}
On the bounded strip containing the graph, the coefficient estimates of Proposition \ref{geometric approximation}, together with \eqref{e.s11.eta.rough}, imply
\begin{equation}\label{e.s11.operator.compare}
 \|\mathcal L_{\Sigma_\alpha[\ell]}(h_\alpha+\eta)
 -\Delta_{\alpha,0}(h_\alpha+\eta)\|_{C^\theta}
 \leq C(b')\varepsilon
 \|\nabla_{\alpha,0}(h_\alpha+\eta)\|_{C^{1,\theta}}
 \leq C(b')\varepsilon^2.
\end{equation}
The curvature coefficient in \eqref{e.s11.mc.expansion} is $O(\varepsilon^2)$ in $C^\theta$ by Proposition \ref{l.grad.A} and the rescaled metric bounds.  The estimate for $\mathcal Q$ in Lemma \ref{mean curvature expansion tf}, its differentiated form supplied by Proposition \ref{geometric approximation} and Appendix \ref{ap:geometric.approx}, and \eqref{e.s11.eta.rough} likewise give
\begin{equation}\label{e.s11.lower.order.mc}
 \left\|
 \bigl(|\sff_{\Sigma_\alpha}|^2
 +\Ric_g(\bn_{\Sigma_\alpha},\bn_{\Sigma_\alpha})\bigr)\ell
 +\mathcal Q(y,\ell,d\ell)
 \right\|_{C^\theta}
 \leq C(b')\varepsilon^2.
\end{equation}
Finally, \eqref{e.s11.eta.correction} and \eqref{e.s11.zero.shift} give
\begin{equation}\label{e.s11.laplacian.eta}
 \left\|
 \Delta_{\alpha,0}(h_\alpha+\eta)
 +\frac{(-1)^\alpha}{\Hb'(s(t))}
 \Delta_{\alpha,0}\overline{\Gb}_\alpha(\cdot,s(t))
 \right\|_{C^\theta}
 \leq C(b')\varepsilon^2|\log\varepsilon|^2.
\end{equation}
Combining \eqref{e.s11.mc.expansion}--\eqref{e.s11.laplacian.eta} proves \eqref{e.s11.level.H.correction}.  Equation \eqref{e.G.est} gives
$\|\Delta_{\alpha,0}\overline{\Gb}_\alpha(\cdot,s(t))\|_{C^\theta}\leq C(b')\varepsilon$; together with \eqref{e.s11.zero.H}, this proves \eqref{e.s11.level.H.general}.
\end{proof}
We conclude by making the dependence of the curvature estimates on the second fundamental form of the boundary explicit. The following remark compares the enhanced second fundamental form throughout the transition region with the second fundamental forms of the zero sheets.

\begin{remark}[Dependence on the boundary geometry]\label{r.s11.boundary.dependence}
Under the assumptions of Theorem \ref{t.main.ball}, fix $b<b'<1$. Applying \eqref{e.lin.ac.est} to \eqref{G-equationwithout cutoff}, with $\snf=0$ and $\psi(y,\tau)=-w_\alpha(y)\tau\bH'(\tau)$ and using the identity for $w_\alpha$ in Lemma \ref{l.dz(nu)} gives, for all sufficiently small $\varepsilon$ and uniformly for $|t|\leq1-b'$,
\begin{equation}
\label{e.s11.boundary.mean}
 \|H_{\{u_\varepsilon=t\}}\|_{L^\infty_{\hat g}(\{u_\varepsilon=t\}\cap B_{1/2}^+(0))}
 \leq C(b')\left(\varepsilon|\log\varepsilon|^2+
 \|\sff_{\partial_0B_3^+(0)}\|_{C^2_{\hat g}(\partial_0B_3^+(0))}\right),
\end{equation}
and
\begin{equation}\label{e.s11.boundary.enhanced}
\sup_{B_{1/2}^+(0)\cap\{|u_\varepsilon|\leq1-b'\}}
|\mathcal A(u_\varepsilon)|_{\hat g}\leq C\max_\alpha\|\sff_{\alpha,\varepsilon}\|_{L^\infty_{\hat g}(\Sigma_\alpha\cap B_{3/5}^+(0))}
+C(b')\left(\varepsilon|\log\varepsilon|^2+
\|\sff_{\partial_0B_3^+(0)}\|_{C^2_{\hat g}(\partial_0B_3^+(0))}\right), 
\end{equation}
where the metric $\hat g$ is as in Theorem \ref{t.main.ball}.

If $n=1$, Proposition~\ref{p.s11.zero.level} and
\eqref{e.s11.boundary.enhanced} give
\[
 \sup_{B_{1/2}^+(0)\cap\{|u_\varepsilon|\leq1-b'\}}
 |\mathcal A(u_\varepsilon)|_{\hat g}
 \leq C(b')\left(\varepsilon|\log\varepsilon|^2+
 \|\sff_{\partial_0B_3^+(0)}\|_{C^2_{\hat g}(\partial_0B_3^+(0))}\right).
\]
The boundary second fundamental form term cannot in general be omitted because of
\eqref{e.s11.surface.boundary.identity}.
Thus, as we pointed out in Remark \ref{e.dim.2}, unlike the interior decay estimates in
\cite[Theorem~3.6]{wang-wei} and
\cite[Proposition~C.4]{chodosh-mantoulidis-2d}, no uniform $o(1)$ estimate
for $|\mathcal A(u_\varepsilon)|$ is possible where $|\sff_{\partial_0B_3^+(0)}|_{\hat g}$ stays bounded away from zero.
This applies even on $\{u_\varepsilon=0\}$, although the ordinary curvature
of its components tends to zero by Proposition~\ref{p.s11.zero.level}.

Indeed, if we choose $D$ as in the proof of Lemma~\ref{l.s11.level.graph} and $y_0\in\Sigma_\alpha\cap B_{3R/5}^+(0)$, and note that $-w_\alpha(y)\tau\bH'(\tau)$ is odd in $\tau$, then the product estimate and \eqref{e.lin.ac.est}, with $\snf=0$, give
\begin{align*}
 \|\bgah\|_{C^{2,\theta}(B_1^\alpha(y_0)\times[-D,D])}
 &\leq C_D\|w_\alpha\|_{C^{1,\theta}(\partial_0\Sigma_\alpha)}\\
 &\leq C_D\varepsilon
 \|\sff_{\partial_0B_3^+(0)}\|_{C^2_{\hat g}(\partial_0B_3^+(0))},
\end{align*}
where the second inequality follows by differentiating the identity for $w_\alpha$ in Lemma \ref{l.dz(nu)} and using Proposition \ref{l.grad.A} and \eqref{e.rescaled.bound.sff}. The cutoff equals one for $|z|\leq D$, hence,
\eqref{e.s11.level.H.correction} and \eqref{e.s11.zero.H}, followed by rescaling, give \eqref{e.s11.boundary.mean}.

For \eqref{e.s11.boundary.enhanced}, use $u=\Hb_*+\Gb_*+\phi$,
\eqref{e.s11.zero.basic}, \eqref{e.s11.foreign.terms}, and the
estimate
\begin{align*}
 &\|u(y,z)-\Hb((-1)^\alpha z)\|_{C^2(B_1^\alpha(y_0)\times[-D,D])}\leq C_D\left(\varepsilon^2|\log\varepsilon|^2+
 \varepsilon\|\sff_{\partial_0B_3^+(0)}\|_{C^2_{\hat g}(\partial_0B_3^+(0))}\right).
\end{align*}
Since $\Hb'$ is bounded below on $[-D,D]$, this controls the difference
between $\nabla(\nabla u/|\nabla u|)$ and $(-1)^\alpha\nabla^2t_\alpha$.
By \eqref{e.Riccati} and Lemmas~\ref{l.t.Pi} and~\ref{l.Z.Z-tld},
$|\nabla_Z\nabla^2t_\alpha|\leq C_D\varepsilon^2$ on
$B_1^\alpha(y_0)\times[-D,D]$.
Comparing with the value at $z=0$ and rescaling proves \eqref{e.s11.boundary.enhanced}.
\end{remark}


\part{Appendices} 
\begin{appendix}
\section{Proof of derivative estimates for enhanced second fundamental form $\mathcal{A}$ in Proposition \ref{l.grad.A}}\label{Agradient} 
In this appendix we prove derivative estimates for the enhanced second fundamental form $\mathcal{A}$ in Proposition \ref{l.grad.A}.



\begin{proof}[Proof of Proposition \ref{l.grad.A}]
We prove the estimates on $B^+_{3R-1}(0)\cap\{|u|\leq1-b'\}$, using (H1) and (H3) of Theorem \ref{t.main.ball} on the larger set $\{|u|\leq1-b\}$, as in \cite[proof of Lemma 8.1]{wang-wei}.

Fix $x_0\in B^+_{3R-1}(0)\cap\{|u|\leq1-b'\}$. By the uniform local gradient upper bound and $b'>b$, there is a small ball (if $x_0$ is an interior point) or half-ball (if $x_0$ is a boundary point) centred at $x_0$, denoted by $B$, which lies in $\{|u|<1-b\}$. The radius of $B$ is uniform in $x_0\in \{|u|\leq 1-b'\}$ and $\varepsilon$.
On $B$, by (H1), (H3),  and the metric rescaling,  we have $|\nabla u|\geq c_0>0$.

Let $(x_1, \dots, x_{n+1})$ be the Fermi coordinates of Section \ref{s.graph}, and recall that $u$ satisfies $|u|\leq1$ and the equation $\Delta_g u=W'(u)$, with boundary condition $\partial_1u=0$ on $\{x_1=0\}$, so that by interior/Neumann elliptic regularity (see Appendix \ref{s.ell.est}), 
\begin{equation}\label{e.initial.elliptic.estimate}
 \|u\|_{C^{m+2,\theta}}
 \leq C_m\bigl(\|u\|_{C^0}+\|W'(u)\|_{C^{m,\theta}}\bigr)
 \leq C_m
\end{equation}
for $m\geq 0$ on a ball/half-ball $B'\Subset B$ centred at $x_0$ of, say, half the radius.

Set $w=|\nabla u|^2$, $N=w^{-1/2}\nabla u$, so that $\mathcal A=\nabla N$, and let $N_q=g(N, \partial_{x_q})=w^{-1/2}u_q$. Since $w\geq c_0^2$ on $B'$, then, by the product rule and the chain rule, by \eqref{e.initial.elliptic.estimate}, we obtain 
\begin{equation}\label{e.crude.estimate.A}
 \|w^{-1}\|_{C^{m,\theta}}+\|N\|_{C^{m,\theta}}\leq C_m \quad \text{on } B'.
\end{equation}
Notice that this implies $\|\mathcal A\|_{C^\theta}\leq C$ on $B'$.

Since $\partial_iN_q=\mathcal A_{iq}+\Gamma^k_{iq}N_k$ and $|\Gamma|\leq C\varepsilon$, together with the assumption $|\mathcal A|\leq C\varepsilon$, we conclude that 
\begin{equation}\label{e.bound.DN}
|DN|\leq C\varepsilon\quad\text{ on $B'$.}
\end{equation}
As in \cite[Lemma 8.1]{wang-wei}, Bochner's formula gives
\begin{equation}\label{eqn:div3}
 \operatorname{div}(w\nabla N)
 =-w|\mathcal A|^2N-w\Ric_g(N,N)N+w\Ric_g(N,\cdot)^\#.
\end{equation}
Let us differentiate the Allen--Cahn equation satisfied by $u$, and use \eqref{eqn:div3} to obtain the following scalar equation for $N_q=w^{-1/2}u_q$
\begin{equation}\label{e.v.1}
 \Delta_gN_q+\langle\nabla\log w,\nabla N_q\rangle=F_q,
\end{equation}
where
\begin{align*}
 F_q={}&-|\mathcal A|^2N_q-\Ric_g(N,N)N_q+R_q{}^kN_k
       +2g^{ij}\Gamma^k_{iq}\partial_jN_k\\
 &+g^{ij}\bigl(\partial_i\Gamma^k_{jq}
       -\Gamma^\ell_{ij}\Gamma^k_{\ell q}
       -\Gamma^\ell_{iq}\Gamma^k_{j\ell}
       +(\partial_i\log w)\Gamma^k_{jq}\bigr)N_k.
\end{align*}
Here, repeated coordinate indices are summed and $q$ is fixed.

By the bounds on curvature and the Christoffel coefficients in Theorem \ref{t.main.ball}, \eqref{e.crude.estimate.A}, and \eqref{e.bound.DN},
\[
 \||\mathcal A|^2N_q\|_{C^\theta}
 \leq C\|\mathcal A\|_{C^0}\|\mathcal A\|_{C^\theta}\leq C\varepsilon.
\]
Thus $\|F\|_{C^\theta}\leq C\bigl(\varepsilon+\|\mathcal A\|_{C^0}\|\mathcal A\|_{C^\theta}\bigr)\leq C\varepsilon$.

Now, on $B'$, the oscillation of $N_q$ is $O(\varepsilon)$ because $|DN|\leq C\varepsilon$. If $x_0$ is on the boundary $\partial_0B^+_{3R-1}$, then $u_1=u_{1q}=0$ for $q\geq2$, so we have
\begin{equation}\label{e.w_1}
 w_1=(\partial_1g^{ij})u_i u_j,\quad
 N_1=0,\quad
 \partial_1N_q=-\frac{(\partial_1g^{ij})u_i u_j}{2w}N_q
 \quad(q\geq2).
\end{equation}
In either case, if we let $v_q=N_q-N_q(x_0)$, then by the interior/Neumann elliptic estimates (see Appendix \ref{s.ell.est}) for $q\geq 2$, and the interior/Dirichlet elliptic estimates for $v_1$, then 
\begin{equation}
\|v_q\|_{C^{2,\theta}(B'')}\leq C \varepsilon
\end{equation}
on a ball $B''\Subset B'\Subset B$ of, say, half the radius of $B'$. This gives $\|\mathcal{A}\|_{C^{1, \theta}(B'')}\leq C \varepsilon$. The result then comes from differentiating \eqref{e.v.1} and iterating this argument.
\end{proof}

\section{Proof of geometric approximation estimates and mean curvature expansion in Proposition \ref{geometric approximation}}\label{ap:geometric.approx}

We collect in this appendix the various geometric approximation estimates in Proposition \ref{geometric approximation} that we obtain in twisted Fermi coordinates and the mean curvature expansion under twisted Fermi coordinates in Lemma \ref{mean curvature expansion tf}.

Throughout, $|z|\leq K|\log\varepsilon|$. Tensors on $\Sigma_{\alpha, z}$ are pulled back by $y\mapsto Y_{\Sigma_\alpha}(y,z)$ and measured using $g_{\alpha,0}$.

\begin{proof}[Proof of Proposition \ref{geometric approximation}]
Notice that by definition $Z=\partial_z$. Since $\{\partial_{y_1}, \dots, \partial_{y_n}, \partial_z\}$ are coordinate vector fields, then $0=[\partial_z, \partial_{y_i}]=[Z, \partial_{y_i}]$ for all $1\leq i\leq n$.

By definition $\{\partial_{y_1}|_{(y,z)}, \dots, \partial_{y_n}|_{(y,z)}\}$ are tangent to $\Sigma_{\alpha, z}$ at $(y,z)$. Therefore, $g_{\alpha,ij}(y,z)=g_{ij}(y,z)$, and
\begingroup
\mathtoolsset{showonlyrefs=false}
\begin{equation}\label{e.partial_z.g}
\partial_z g_{\alpha,ij}=Z(g(\partial_{y_i}, \partial_{y_j}))=g(\nabla_Z\partial_{y_i}, \partial_{y_j})+ g(\partial_{y_i},\nabla_Z\partial_{y_j})=g(\nabla_{\partial_{y_i}}Z, \partial_{y_j})+g(\partial_{y_i}, \nabla_{\partial_{y_j}}Z).
\end{equation}
\endgroup
By \eqref{e.na.Z} and \eqref{e.del.y_i.O(1)}, then $|\partial_z g_{\alpha,ij}(y,\tau)|\lesssim \varepsilon (1+|\tau|)$ for $|\tau|<K|\log\varepsilon|$. Therefore, after integrating between $0$ and $|z|$, we obtain the first half of \eqref{e.g_ij.approx}.

By the chain rule 
\begin{equation}\label{e.partial_z.g-1}
\partial_zg_{\alpha}^{ij}=-\sum_{l,k=1}^ng_{\alpha}^{li}g_{\alpha}^{jk}\partial_zg_{\alpha,lk}, 
\end{equation}
and by definition, 
\begin{equation}\label{e.def.Christoffel}
  \Gamma_{\al, ij}^k=\frac{1}{2}\sum_{l=1}^ng_{\alpha}^{kl}(\partial_{y_i}g_{\alpha,lj}+\partial_{y_j}g_{\alpha,il}-\partial_{y_l}g_{\alpha,ij}).  
\end{equation}
The coordinate-field bounds in Section \ref{sectwisted fermi def} give
\begin{align*}
 |\partial_zg_{\alpha,ij}|+|\partial_z\partial_{y_\ell}g_{\alpha,ij}|
   +|\partial_z\Gamma^k_{\alpha,ij}|
 &\leq C\varepsilon(1+|z|),\\
 \sum_{m=1}^2\bigl(|\partial_y^mg_{\alpha,ij}|
                         +|\partial_y^mg_\alpha^{ij}|\bigr)
       +|\partial_{y_\ell}\Gamma^k_{\alpha,ij}|
 &\leq C\varepsilon(1+|z|^2),
\end{align*}
which completes the proof of \eqref{e.g_ij.approx}, and proves \eqref{e.na.g.approx}, \eqref{e.gamma.alpha.approx}, \eqref{e.partial.gamma.alpha.approx}.

For the estimates involving the second fundamental form, recall that, by the Riccati equation in \cite[Appendix A]{chodosh-mantoulidis}, we have
\begin{equation}\label{e.Riccati}
 \nabla_{\widetilde Z}\nabla^2t_\alpha
 =-(\nabla^2t_\alpha)^2
   -g(R(\cdot,\widetilde Z)\widetilde Z,\cdot),
 \qquad \nabla^2t_\alpha(\widetilde Z,\cdot)=0.
\end{equation}
The initial bounds and $|\operatorname{Rm}_g|\leq C\varepsilon^2$ give $|\nabla^2t_\alpha|\leq C\varepsilon$, and $|\nabla_{\widetilde Z}\nabla^2t_\alpha|\leq C\varepsilon^2$. Then, tangential differentiation gives
\[
 \big| (\nabla^4t_\alpha)(\widetilde Z,\cdot,\cdot,\cdot)\big|
 \leq C|\sff_{\alpha,z}|\,|\nabla_{\alpha,z}\sff_{\alpha,z}|
       +C|\nabla\operatorname{Rm}_g|
       +C|\operatorname{Rm}_g|\,|\sff_{\alpha,z}|,
\]
with the remaining slots tangent and normally parallel transported. Since
\[
 \nabla_{\alpha,z}\sff_{\alpha,z}
 =-(\nabla^3t_\alpha)|_{T\Sigma_{\alpha,z}^{\otimes3}},
\]
an ODE argument using Gronwall's inequality for vector valued functions, together with Appendix \ref{Agradient}, yields
\begin{equation}\label{e.nabla4.ta.tZ}
 |\nabla_{\alpha,z}\sff_{\alpha,z}|\leq C\varepsilon,
 \qquad
 \big| (\nabla^4t_\alpha)(\widetilde Z,\cdot,\cdot,\cdot)\big|
 \leq C\varepsilon^2.
\end{equation}
Now recall that by definition
\[
 \partial_z=\widetilde Z+T,\qquad
 \nabla_Z\partial_{y_i}=\nabla_{\partial_{y_i}}Z,\qquad
 \sff_{\alpha,ij}=-(\nabla^2t_\alpha)(\partial_{y_i},\partial_{y_j}),
\]
so that by Lemmas \ref{l.t.Pi} and \ref{l.Z.Z-tld}, we finally obtain the following estimates:
\begin{align*}
 |\partial_z\sff_{\alpha,ij}|
 +|\partial_z(\nabla_{\alpha,z}\sff_{\alpha,z})_{kij}|
 &\leq C\varepsilon^2
  +C|T|\bigl(|\nabla^3t_\alpha|+|\nabla^4t_\alpha|\bigr)+C|\nabla Z|\bigl(|\nabla^2t_\alpha|+|\nabla^3t_\alpha|\bigr)\\
 &\leq C\varepsilon^2(1+|z|^3),
 \end{align*}
 and
 \begin{align*}
 |\partial_zH_\alpha|
 &=\big|(\partial_zg_\alpha^{ij})\sff_{\alpha,ij}
              +g_\alpha^{ij}\partial_z\sff_{\alpha,ij}\big|
 \leq C\varepsilon^2(1+|z|^3).
\end{align*}
Integration between 0 and $|z|$ completes the proof of \eqref{e.sff.approx},\eqref{e.mn.curv.approx}, and \eqref{e.nablasff}.

Finally, for the volume form
\[
 \partial_z\lambda_\alpha
 =\tfrac12\lambda_\alpha g_\alpha^{ij}\partial_zg_{\alpha,ij}
 =\lambda_\alpha(-H_\alpha+\operatorname{div}_{\alpha,z}T),
\]
so that $|\partial_z\lambda_\alpha|\leq C\varepsilon(1+|z|)$ by \eqref{e.difference.z.tz} and \eqref{e.mn.curv.approx}.
\end{proof}

We shall now give an expansion for the mean curvature of the graph of a function $\vp\in C^2(\Sigma_\alpha, )$. This is analogous to \cite[Lemma 2.9]{chodosh-mantoulidis}.

\begin{proof}[Proof of Lemma \ref{mean curvature expansion tf}]
By Proposition \ref{geometric approximation} and \eqref{e.Riccati},
$|\sff_{\alpha,z}|+|H_\alpha|+|\nabla^2t_\alpha|\leq C\varepsilon$ and
\begin{equation}\label{e.mc.normal.first}
 \widetilde ZH_\alpha
 =|\sff_{\alpha,z}|^2+\Ric_g(\widetilde Z,\widetilde Z),
 \qquad |\widetilde Z^2H_\alpha|\leq C\varepsilon^3.
\end{equation}
By the definition of $Z_2$ and $Z$,
\begin{gather*}
 |\nabla Z_2|\leq C\varepsilon,\qquad
 \langle\widetilde Z,Z_2\rangle(y,0)=0,\qquad
 |\partial_z\langle\widetilde Z,Z_2\rangle|\leq C\varepsilon,\\
 |\langle\widetilde Z,Z_2\rangle(y,z)|\leq C\varepsilon|z|
 \quad\Longrightarrow\quad |T(y,z)|\leq C\varepsilon|z|.
\end{gather*}
For $\widetilde y=\widetilde\Pi_\alpha(Y_{\Sigma_\alpha}(y,z))$,
Lemma \ref{l.tilde.t.Pi} yields
\[
 \frac{d\widetilde y}{dz}
 =D\widetilde\Pi_\alpha(T)(Y_{\Sigma_\alpha}(y,z)),
 \qquad |\widetilde y-y|\leq C\varepsilon z^2.
\]
Taylor expansion along the normal geodesic from $\widetilde y$ gives
\[
 H_\alpha(y,z)=H_\alpha(\widetilde y,0)
            +z(\widetilde ZH_\alpha)(\widetilde y,0)
            +O(\varepsilon^3z^2).
\]
By Appendix \ref{Agradient} and the rescaled metric bounds,
\[
 |\nabla_{\alpha,0}^2H_\alpha(\cdot,0)|\leq C\varepsilon,
 \qquad
 |\nabla_{\alpha,0}(\widetilde ZH_\alpha)(\cdot,0)|\leq C\varepsilon^2.
\]
Hence
\begin{align*}
 |H_\alpha(\widetilde y,0)-H_\alpha(y,0)|
 &\leq C\varepsilon z^2|\nabla_{\alpha,0}H_\alpha(y,0)|
           +C\varepsilon^3z^4,\\
 |z|\,|(\widetilde ZH_\alpha)(\widetilde y,0)
                -(\widetilde ZH_\alpha)(y,0)|
 &\leq C\varepsilon^3|z|^3,
\end{align*}
and therefore
\begin{equation}\label{e.mc.height.remainder}
\begin{split}
 &\Big|H_\alpha(y,z)-H_\alpha(y,0)
  -\bigl(|\sff_{\alpha,0}|^2+
        \Ric_g(\bn_{\Sigma_\alpha},\bn_{\Sigma_\alpha})\bigr)(y)z\Big|\\
 &\qquad\leq C\varepsilon z^2|\nabla_{\alpha,0}H_\alpha(y,0)|
           +C\varepsilon^3z^2(1+z^2).
\end{split}
\end{equation}

Write $\varphi_i=\partial_{y_i}\varphi$ and evaluate coefficients at
$(y,\varphi(y))$. The defining function $t_\alpha-\varphi\circ\Pi_\alpha$
gives
\[
 \varpi^2=(1+T^i\varphi_i)^2+|d\varphi|_{g_{\alpha,\varphi}}^2,
 \qquad
 \bn_{\Sigma_\alpha[\varphi]}
 =\frac{(1+T^i\varphi_i)\widetilde Z-
                  \nabla_{\alpha,\varphi}\varphi}{\varpi},
\]
where $\nabla_{\alpha,z}\varphi=\nabla_{\alpha,z}(\varphi\circ\Pi_\alpha)$
and repeated indices are summed. Since
\[
 |T|\,|d\varphi|=o(1),\qquad
 \langle\bn_{\Sigma_\alpha[\varphi]},\widetilde Z\rangle
 =\frac{1+T^i\varphi_i}{\varpi}\geq C^{-1},
\]
the graph metric is uniformly equivalent to $g_{\alpha,0}$.
The stated properties of $\mathcal L_{\Sigma_\alpha[\varphi]}$ follow from
its definition. Since $t_\alpha=\varphi\circ\Pi_\alpha$ on the graph,
the restriction formula for the Laplacian gives
\begin{equation}\label{e.mc.intrinsic.exact}
 H_{\Sigma_\alpha[\varphi]}-\mathcal L_{\Sigma_\alpha[\varphi]}\varphi
 =\frac{\varpi}{1+T^i\varphi_i}H_\alpha(y,\varphi)
 -\frac{\sff_{\alpha,\varphi}
    (\nabla_{\alpha,\varphi}\varphi,\nabla_{\alpha,\varphi}\varphi)}
       {(1+T^i\varphi_i)\varpi}.
\end{equation}
Since
\[
 \frac{\varpi}{1+T^i\varphi_i}-1
 =\frac{|d\varphi|_{g_{\alpha,\varphi}}^2}
 {(1+T^i\varphi_i)(\varpi+1+T^i\varphi_i)},
\]
we obtain
\[
 \big|H_{\Sigma_\alpha[\varphi]}
       -\mathcal L_{\Sigma_\alpha[\varphi]}\varphi-H_\alpha(y,\varphi)\big|
 \leq C\varepsilon|d\varphi|_{g_{\alpha,0}}^2.
\]
Together with \eqref{e.mc.height.remainder}, this proves the expansion.
\end{proof}

\section{The one-dimensional heteroclinic solution $\mathbb{H}$}\label{s.1d.sol}
We summarize in this appendix the results on the one-dimensional heteroclinic solution needed in this work. 

Let $W$ be the even double-well potential fixed in Section 2.  There is a unique increasing solution
$\mathbb H:\mathbb R\to(-1,1)$ of
\[
 \mathbb H''=W'(\mathbb H),
 \qquad
 \mathbb H(0)=0,
 \qquad
 \lim_{t\to\pm\infty}\mathbb H(t)=\pm1.
\]
This is called the \textit{heteroclinic solution}. Since $W$ is even and the normalization fixes the translation, $\mathbb H$ is odd and $\mathbb H'$ is positive and even.
Multiplying the equation by $\mathbb H'$ and using the limits at infinity gives
\begin{equation}\label{eq: heteroclinic.first.order}
    \mathbb{H}^\prime(t) = \sqrt{2W(\mathbb{H}(t))}, \quad \text{for all $t \in \mathbb{R}$}. 
\end{equation}
It is not difficult to check that, as $t \rightarrow \pm \infty$, the function $\mathbb{H}(t)$ converges exponentially to $\pm 1$. More precisely, as $t \rightarrow + \infty$, we have the following asymptotic expansions: there exists a constant $A_1$, depending on $W$, such that for all $t > 0$ large enough, there holds 
\begin{align}\label{Hsolution asymptotics}
    \quad\,\, \mathbb{H}(t) & = 1-  A_1 e^{- t} + O(e^{-2t}), \qquad \mathbb{H}^\prime(t) = A_1 e^{-t} + O(e^{-2t}), \qquad \mathbb{H}^{\prime \prime}(t) = - A_1 e^{- t} + O(e^{-2t}). 
\end{align}
Likewise, as $t\to-\infty$, the analogous expansion holds with a positive constant
$A_{-1}$.  Since $W$ is even and $\mathbb H$ is odd, in the present paper
$A_{-1}=A_1$.  

The correctly rescaled heteroclinic is
$\mathbb H_\varepsilon(t)=\mathbb H(t/\varepsilon)$; it satisfies
$\varepsilon^2\mathbb H_\varepsilon''=W'(\mathbb H_\varepsilon)$. Finally, 
\begin{equation}\label{e.sigma0}
    \sigma_0 = \int_{\mathbb{R}} \left( \frac{1}{2} (\mathbb{H}^\prime(t))^2 + W(\mathbb{H}(t)) \right) dt= \int_{\mathbb{R}}  (\mathbb{H}^\prime(t))^2  dt\in (0, + \infty). 
\end{equation}
\begin{lemma}\label{1dinteraction appendix}
    For all $T > 0$ large, we have the following expansions. 
    \begin{align*}
        \int_{\mathbb{R}} \left( W^{\prime \prime}(\mathbb{H}(t)) - 1 \right) \left(\mathbb{H}(- t - T) + 1 \right) \mathbb{H}^\prime(t) dt = - 2A_{-1}^2 e^{-T} + O(e^{-4T/3}), 
    \end{align*}
    and 
    \begin{equation*}
        \int_{\mathbb{R}} \left( W^{\prime \prime}(\mathbb{H}(t)) - 1 \right) \left(\mathbb{H}(T - t) - 1 \right) \mathbb{H}^\prime(t) dt = 2A_{1}^2 e^{-T} + O(e^{-4T/3}), 
    \end{equation*}
    as well as 
    \begin{align*}
        \int_{\mathbb{R}} \left( W^{\prime \prime}(\mathbb{H}(t) + \mathbb{H}(- t - T) + 1) - W^{\prime \prime}(\mathbb{H}(- t - T)) \right) \mathbb{H}^\prime(- t - T) \mathbb{H}^\prime(t) dt = - 2A_{-1}^2 e^{-T} + O(e^{-7T/6}), 
    \end{align*}
    together with
    \begin{align*}
        \int_{\mathbb{R}} \left( W^{\prime \prime}(\mathbb{H}(t) + \mathbb{H}(T - t) - 1) - W^{\prime \prime}(\mathbb{H}(T - t)) \right) \mathbb{H}^\prime(T - t) \mathbb{H}^\prime(t) dt = - 2A_{1}^2 e^{-T} + O(e^{-7T/6}). 
    \end{align*}
\end{lemma}
\begin{proof}
    The proof of the first two estimates can be found in \cite[Lemma A.1]{wang-wei}, while for the last two we refer the reader to \cite[Lemma A.2]{wang-weiadv}. 
\end{proof}
We conclude by discussing the spectrum of the linearized operator at $\mathbb{H}$, i.e. 
\begin{equation*}
    \mathcal{L} = - \frac{d^2}{dt^2} + W^{\prime \prime}(\mathbb{H}(t)). 
\end{equation*}
A direct computation shows that $\mathbb{H}^\prime$ is an eigenfunction of $\mathcal{L}$ corresponding to the 0 eigenvalue. One can prove that this is the lowest eigenvalue. On the other hand, we have the following result, which can be proven via a contradiction argument. 
\begin{theorem}\label{t.coercive.a}
    There exists a constant $\mu > 0$ such that for any $\varphi \in H^1(\mathbb{R})$ satisfying the orthogonality $\int_\mathbb{R} \varphi(t) \mathbb{H}^\prime(t) \, dt = 0$, we have 
    \begin{equation*}
        \int_{\mathbb{R}} \left( (\varphi^\prime(t))^2 + W^{\prime \prime}(\mathbb{H}(t)) \varphi^2(t) \right) dt \geq \mu \int_{\mathbb{R}} \left( (\varphi^\prime(t))^2 + \varphi^2(t) \right) dt. 
    \end{equation*}
\end{theorem}

\section{Proofs of exponential interaction decay estimates in Lemmas \ref{l.interaction.int} and  Lemma \ref{l.interaction.int.2}} \label{s.proofs.integrals}

In this appendix we prove Lemmas \ref{l.interaction.int} and \ref{l.interaction.int.2}.

\begin{proof}[Proof of Lemma \ref{l.interaction.int}]
On the interval of integration, Proposition \ref{p.conseq.almost.parallel} gives
\begin{equation}\label{e.tbeta.affine.for.integrals}
 |t_\beta(y,z)-t_\beta(y,0)-z|\leq C\varepsilon^{1/3}.
\end{equation}
The harmless error in \eqref{e.tbeta.affine.for.integrals} may be absorbed in the implicit
constant.  The elementary one-dimensional estimate
\[
 \int_{\mathbb R}(1+|z|)^k
 e^{-\sigma|T+z|-\varsigma|z|}\,dz
 \leq C(1+|T|)^{k+1}e^{-\min\{\sigma,\varsigma\}|T|}
\]
holds when $\sigma,\varsigma>0$. Since
$|t_\beta(y,z)|\leq |t_\beta(y,0)|+|z|+C\varepsilon^{1/3}$, this estimate proves both
\eqref{equation 3 in l.interaction.int} and \eqref{equation 1 in l.interaction.int}.
\end{proof}

\begin{proof}[Proof of Lemma \ref{l.interaction.int.2}]
Assume first that $\beta>\alpha$; the other case follows by reflection.  Then
$t_\beta(y,0)<0$.  The truncated heteroclinic satisfies
\[
 |\bH(s)+1|\leq C\min\{1,e^s\},\qquad |\bH'(z)|\leq Ce^{-|z|}.
\]
Using \eqref{e.tbeta.affine.for.integrals} and splitting at
$z=0$ and $z=|t_\beta(y,0)|$, we obtain
\begin{align*}
 &\int_{\mathbb R}|\bH(t_\beta(y,z))+1|\,|\bH'(z)|\,dz \leq C\int_{\mathbb R}
 \min\{1,e^{-|t_\beta(y,0)|+z}\}e^{-|z|}\,dz
 \leq C(1+|t_\beta(y,0)|)e^{-|t_\beta(y,0)|}.
\end{align*}
This proves \eqref{equation 2 in l.interaction.int}.  The estimate
$(1+s)e^{-s}\leq C_\sigma e^{-\sigma s}$ for $s\geq0$ gives
\eqref{e.interaction.sigma.correct}.
\end{proof}

\section{Some classical elliptic regularity estimates}\label{s.ell.est}
In this section, we state the Schauder estimate for the linear elliptic equation
\begin{equation}\label{e.ell.equ}
L\sfu=a^{ij}D_{ij}\sfu+b^iD_i\sfu+c\sfu=\snf+\operatorname{div}\snG,\quad \lambda \md{\xi}^2 \leq a^{ij}\xi_i\xi_j \leq \Lambda \vert \xi \vert^2.
\end{equation}
Let $\alpha\in(0,1)$ and assume the $C^{0,\alpha}$ norms of the coefficients $a^{ij},\;b^i,\;c$ are bounded by $\La$.

\begin{proposition}
[Schauder estimate]\label{t.sch} Then the following estimates hold, with $C=C(n,\alpha,\lambda,\Lambda)$.

\noindent(a) Interior estimate \cite[Theorems 8.32, 6.2]{gilbarg1977elliptic}, \cite[Theorem 1.2]{vita2022boundary}. \sps\ $\sfu$ satisfies \eqref{e.ell.equ} on $B_2=B_2(0)$. Then on $B_1=B_1(0)$
\begin{align}
&\|\sfu\|_{C^{1,\al}(B_1)}\leq C\left(\|\sfu\|_{C^0(B_2)}+\|\snf\|_{C^0(B_2)}+\|\snG\|_{C^{0,\al}(B_2)}\right)\label{e.sch.int1},\\
&\|\sfu\|_{C^{2,\al}(B_1)}\leq C\left(\|\sfu\|_{C^0(B_2)}+\|\snf\|_{C^{0,\al}(B_2)}+\|\snG\|_{C^{1,\al}(B_2)}\right)\label{e.sch.int2}.
\end{align}
\\
\noindent(b) Boundary estimate for the Dirichlet condition \cite[Corollary 6.7]{gilbarg1977elliptic}, \cite[Theorem 1.2]{vita2022boundary}. \sps\ $\sfu$ satisfies \eqref{e.ell.equ} on $B_2^+=B_2^+(0)$ and $\sfu=\ph$ on $\del_0B_2^+(0)$. Then on $B_1^+=B_1^+(0)$
\begin{align}
&\|\sfu\|_{C^{1,\al}(B_1^+)}\leq C\left(\|\sfu\|_{C^0(B_2^+)}+\|\snf\|_{C^0(B_2^+)}+\|\snG\|_{C^{0,\al}(B_2)} +\|\ph\|_{C^{1, \alpha}(B_2^+)}\right)\label{e.sch.dir1},\\
&\|\sfu\|_{C^{2,\al}(B_1^+)}\leq C\left(\|\sfu|_{C^0(B_2^+)}+\|\snf\|_{C^{0,\al}(B_2^+)}+\|\snG\|_{C^{1,\al}(B_2)} +\|\ph\|_{C^{2,\al}(B_2^+)}\right)\label{e.sch.dir2}.
\end{align}
\\
\noindent(c) Boundary estimate for the Neumann condition \cite[Lemma 6.29]{gilbarg1977elliptic}, \cite[Theorem 1.2]{vita2022boundary}. \sps\ $u$ satisfies \eqref{e.ell.equ} on $B_2^+=B_2^+(0)$ and $\del_1\sfu=\ph$ on $\del_0B_2^+(0)$. Then on $B_1^+=B_1^+(0)$
\begin{align}
&\|\sfu\|_{C^{1,\al}(B_1^+)}\leq C\left(\|\sfu\|_{C^0(B_2^+)}+\|\snf\|_{C^0(B_2^+)}+\|\snG\|_{C^{0,\al}(B_2)} +\|\ph\|_{C^{0, \alpha}(\del_0B_2^+)}\right)\label{e.sch.neu1},\\
&\|\sfu\|_{C^{2,\al}(B_1^+)}\leq C\left(\|\sfu\|_{C^0(B_2^+)}+\|\snf\|_{C^{0,\al}(B_2^+)}+\|\snG\|_{C^{1,\al}(B_2)}+\|\ph\|_{C^{1,\al}(\del_0B_2^+)}\right)\label{e.sch.neu2}.
\end{align}
\end{proposition}

\begin{remark}
In the estimates stated in \cite[Corollary 6.7]{gilbarg1977elliptic} and \cite[Lemma 6.29]{gilbarg1977elliptic}, there are $\|\sfu\|_{C^{2,\al}}$ on the left-hand side and $\|\snf\|_{C^{0,\al}}$ on the right-hand side. However, from the proofs of these two results, it follows that to estimate $\|\sfu\|_{C^{1, \alpha}}$, it is enough to have $\|\snf\|_{C^0}$ on the right hand side (see also equation 4.46 and Theorem 8.33 in \cite{gilbarg1977elliptic}); therefore we have the estimates \eqref{e.sch.dir1} and \eqref{e.sch.neu1}. 

The $C^{1,\theta}$ estimates with a scalar source merely in $C^0$ and a divergence source in $C^\theta$ follow from the corresponding first-order Campanato estimates, or from the boundary estimates in \cite[Theorem 1.2]{vita2022boundary}.  Equivalently, one first takes an exponent
$p>(n+1)/(1-\theta)$ in the local $W^{1,p}$ estimate for the divergence source and then applies Morrey's inequality. 
\end{remark}

\section{Liouville theorems and a priori estimates in gluing procedure}\label{s.gluing.a.apriori}
\newcommand{\tph}{\tilde{\vp}}\newcommand{\tE}{\tilde{E}}\newcommand{\tF}{\tilde{F}}\newcommand{\tps}{\tilde{\ps}}\newcommand{\B}{\bbb}\newcommand{\hph}{\hat{\vp}}
\newcommand{\ep}{\epsilon}
\begin{lemma}\label{l.entire.lin.ac}
	Let $L\vp=\De_y \vp + \del_{ss}\vp - W''(\Hb)\vp$, where $(y,s)\in \bbr^n\times \bbr$. 
	\begin{itemize}
		\item If \(\vp\) is a bounded, smooth solution of \(L\vp=0\) on \(\bbr^n \times \bbr \), then \(\vp(y,s)=c\Hb'(s)\) for some \(c\in \bbr \).
		\item If \(\vp\) is a bounded, smooth solution of \(L\vp=0\) on \(\bbr^n_+ \times \bbr \) and \(\vp=0\) on \(\del \bbr^n_+ \times \bbr\), then \(\vp(y,s)=c\Hb'(s)\) for some \(c\in \bbr \).
		\item If \(\vp\) is a bounded, smooth solution of \(L\vp=0\) on \(\bbr^n_+ \times \bbr \) and \(\del_{\nu} \vp=0 \) on \(\del \bbr^n_+ \times \bbr\), then \(\vp(y,s)=c\Hb'(s)\) for some \(c\in \bbr \).
	\end{itemize}
\end{lemma}
We omit the proof of this lemma since it is similar to \cite[Proof of Lemma 4.1]{del2010interface}.

\begin{lemma}\label{l.entire.lin.ac'}
	Given an open set $\Om\subset \bbr^{n+1}$, $a^i\in L^{\infty}(\Om)$, $b:\Om \to \bbr$ \w\ $\inf_{\Om}b = b_0>0$, define \(L\vp = \De \vp + a^i \del_i\vp - b\vp\).
	\begin{itemize}
		\item If \(\vp\) is a bounded, $C^2$ solution of \(L\vp=0\) on \(\bbr^n \times \bbr \), then \(\vp \equiv 0\).
		\item If \(\vp\) is a bounded, $C^2$ solution of \(L\vp=0\) on \(\bbr^n_+ \times \bbr \) and \(\vp=0\) on \(\del \bbr^n_+ \times \bbr\), then \(\vp\equiv 0\).
		\item If \(\vp\) is a bounded, $C^2$ solution of \(L\vp=0\) on \(\bbr^n_+ \times \bbr \) and \(\del_{\nu} \vp=0 \) on \(\del \bbr^n_+ \times \bbr\), then \(\vp\equiv 0\).
	\end{itemize}
	
\end{lemma}

\begin{proof}
	Let \(\vp_{\pm} = \exp(\pm \la |x|)\), where \(\la\) is to be specified later. Let \(\max_i \|a^i\|_{L^{\infty}} = a_0\). Then, for \(|x|\neq 0\),
	\begin{equation}
		L\vp_{\pm}\leq \big(\la^2 \pm \frac{n \la}{|x|} + a_0\la - b_0 \big) \vp_{\pm}.
	\end{equation}
We choose \(\la >0 \) \st\ \(\la^2 + a_0\la \leq \frac{b_0 }{4}\) and define \(r = \frac{4n\la}{b_0}\). Then for \(|x|\geq r\), \(L\vp_{\pm}<0\). For \(\ep > 0\), let
\[\vp_{\ep} = \|\vp\|_{L^{\infty}}e^{\la (r - |x|)} + \ep e^{\la |x|}. \]
By the maximum principle, \(|\vp(x)| \leq \vp_{\ep}(x)\) on $\{|x|\geq r\}$; letting \(\ep \to 0\), we obtain $|\vp(x)| \leq \|\vp\|_{L^{\infty}}e^{\la (r - |x|)}$
Again applying the maximum principle, we conclude that \(\vp \equiv 0.\)
\end{proof}

Let us introduce some notation, which will be used in Proposition \ref{p.lin.ac}. Suppose $r>0$, $\Si \subset \bar{\B}^+_r(0)\times \bbr$  is a hypersurface equipped with a Riemannian \mt\ and $\Si$ is the graph of a smooth function defined on $\bar{\B}^+_r(0)$. Then $\del_0\Si$ (resp. $\del_+\Si$) denotes the portion of $\Si$ lying over $\del_0\B^+_r(0)$ (resp. $\del_+\B^+_r(0)$), $\del_{*}\Si = \overline{\del_0\Si} \cap \overline{\del_+\Si}$, $\hat{\Si}= \Si\setminus \del_{*}\Si$. For $\Si'\subset \Si$, $\del_0\Si':= \del \Si' \cap \del_0\Si $. For $\de\in (0,1)$, 
\[\Si_{(-\de r)}=\Si\setminus \{p\in \Si: d(p,\del_0\Si),\;d(p,\del_+\Si)<\de r\}.\]
For $\snf\in L^1(\Si\times \bbr)$, $c_{\snf}(y)=\frac{1}{\sigma_0}\int_{-\infty}^{\infty}\snf(y,s)\Hb'(s)ds$.

\sps\ $\Om\subset \bbr^l$ (for some integer $l$) is the closure of an open set with smooth boundary, $m\geq 0$, $\tht, \si\in (0,1)$ and $\sfu\in C^{m,\tht}(\Om \times \bbr)$. Then
\begin{equation}
	\nm{\sfu}_{C^{m,\tht}_{\si}(\Om)}:=\sup\left\{e^{\si |s|}\sum_{i=0}^{m}(|\sfu|_i+|\sfu|_{i,\tht})(y,s):(y,s)\in \Om\times\bbr\right\}.
\end{equation}

\begin{proposition}\label{p.lin.ac}
	Given $\et>0$ sufficiently small and $\si,\de\in (0,1)$, \te\ two constants $\hat{\ve},C>0$ \st\ the following holds. Suppose
	\begin{enumerate}[(i)]	
		\item $0<\ve\leq\hat{\ve}$, $R=\ve^{-1}$, $\mathsf{h}\in C^{\infty}(\bar{\B}^+_R(0),(-R,R))$, $\Si \subset \bar{\B}^+_R(0)\times (-R,R)$ is the graph of $\mathsf{h}$, $|\mathsf{h}|_{1}\leq \et$, $|\mathsf{h}|_{i} = O(\ve)$ for $i \geq 2$; $\Si$ is parametrized by $(y_1,\dots,y_n)\mapsto (y_1,\dots,y_n,\mathsf{h}(y_1,\dots,y_n))$;
		\item $\hat{\mathtt{g}}$ is a Riemannian \mt\ on $\bar{\B}^+_1(0)\times (-1,1)$, $\|\hat{\mathtt{g}}_{ij}-\de_{ij}\|_{C^3}\leq \et$;
		\item the Riemannian \mt\ $\mathtt{g}$ on $\bar{\B}^+_R(0)\times (-R,R)$ is defined by $\mathtt{g}(x)=\hat{\mathtt{g}}(\ve x)$ and we restrict $\mathtt{g}$ on $\Si$.
	\end{enumerate}
	
	Then, for all $\snf\in C_c^{m-2,\tht}(\hat{\Si}\times \bbr)$ and $\ps\in C_c^{m-1,\tht}(\del_0\Si\times \bbr)$ ($m \geq 2$), \tes\ a unique solution $\varphi \in C^{0,\tht'}(\Si\times \bbr) \cap  C^{m,\tht}(\hat{\Si}\times \bbr)$ (for some $\tht'\in (0,1)$) to the following mixed boundary value problem:
	\begin{equation}\label{e.lin.AC}
\left\{
\begin{aligned}
 -\Delta_\Sigma\phi-\partial_{ss}\phi
 +W''(\mathbb H(s))\phi
 &=\snf-c_\snf(y)\mathbb H'(s)
 &&\text{in }\operatorname{int}(\Sigma)\times\mathbb R,\\
 \partial_{\nu_\Sigma}\phi&=\psi
 &&\text{on }\partial_0\Sigma\times\mathbb R,\\
 \phi&=0
 &&\text{on }\partial_+\Sigma\times\mathbb R,\\
 \int_{\mathbb R}\phi(y,s)\mathbb H'(s)\,ds&=0
 &&\text{for every }y\in\Sigma.
\end{aligned}\right.
\end{equation}
	Moreover, $\|\varphi\|_{C^{0}_{\si}\left(\Si\times \bbr\right)}<\infty$ and
	\begin{equation}\label{e.lin.ac.est}
		\|\varphi\|_{C^{m,\tht}_{\si}\left(\Si_{(-3\de R)}\times \bbr\right)} \leq C \left(\|\ps\|_{C^{m-1,\tht}_{\si} \left(\del_0\Si_{(-\de R)} \times \bbr\right)} + \|\snf\|_{C^{m-2 ,\tht}_{\si} \left(\Si_{(-\de R)}\times \bbr\right)}\right).
	\end{equation}
\end{proposition}
\begin{proof}
	The proof is similar to \cite[Proof of Proposition 4.1]{del2010interface}. Let $T_0:W^{1,2}(\Si \times \bbr)\to L^2(\del_0\Si \times \bbr)$ and $T_+:W^{1,2}(\Si \times \bbr)\to L^2(\del_+\Si \times \bbr)$ denote the trace maps. \(\mathcal{H}_0:=\{\sfu \in W^{1,2}(\Si \times \bbr): T_+\sfu = 0\}.\) Then $\mathcal{H}_0$ is the completion of \(C^1_c\left((\text{int}(\Si)\cup \del_0 \Si)\times \bbr\right)\) \wrt\ the $W^{1,2}$ norm. \(\mathcal{H}:=\left\{\sfu \in \mathcal{H}_0: \int_{-\infty}^{\infty}\sfu(y,s)\Hb'(s)ds = 0\;\forall \; y\in \Si\right\}\). Let the bilinear form \(\mathbf{a}\) on \(\mathcal{H}\) be defined by 
	\[\mathbf{a}(\vp_1,\vp_2)=\int_{\Si}\int_{\bbr}\left(\del_s\vp_1 \del_s\vp_2 + \langle \nabla_y \vp_1, \nabla_y \phi_2\rangle +W''(\Hb)\vp_1\vp_2\right)ds\: d\mathcal{H}^n(y).\]
	Then $\vp$ is a weak solution of \eqref{e.lin.AC} iff 
	\[\mathbf{a}(\vp,\vp')=\int_{\Si \times \bbr} \snf \vp'ds\:d\mathcal{H}^n(y)+\int_{\del_0\Si \times \bbr}\ps T_0(\vp') ds \: d\mathcal{H}^n(y)\quad \forall \; \vp'\in \mathcal{H}.\]
Theorem \ref{t.coercive.a} implies $\mathbf{a}$ is coercive; hence by the Lax-Milgram theorem, \tes\ a unique weak solution $\vp$ of \eqref{e.lin.AC}. By standard elliptic regularity \cite[Theorem 7.1]{Stampacchia65},  \(\vp\in C^{0,\tht'}(\Si\times \bbr) \cap  C^{m,\tht}(\hat{\Si}\times \bbr)\) (for some $\tht'\in (0,1)$).

Next, we show that the solution \(\vp\) satisfies $\|\varphi\|_{C^{0}_{\si}\left(\Si\times \bbr\right)}<\infty$. Since $\snf, \ps$ are compactly supported, \tes\ \(s_0\) \st\ (using the notation $I=\{s\in \bbr: |s|>s_0\}$)
\begin{align}
	 -&\De_y \vp - \del_{ss}\vp +W''(\Hb)\vp = c_{\snf}(y)\Hb'(s) \text{ on } \Si \times I,\\
	& \del_{y_1}\vp = 0 \text{ on } \del_0\Si \times I,\quad \vp=0 \text{ on } \del_+\Si \times I.
\end{align}
Using the barrier function
\[\vp_{\ep}=\|\vp\|_{L^{\infty}}\left(e^{-\si(|s|-s_0)} + \ep \cosh\si s +\ep \sum_{i=1}^{n} \cosh \si y_i \right),\]
one can show that $|\vp|\leq \|\vp\|_{L^{\infty}}e^{-\si(|s|-s_0)}$ on $\Si \times \bbr$, hence $\|\varphi\|_{C^{0}_{\si}\left(\Si\times \bbr\right)} <\infty$.

Lastly, we show that if $\vp\in C^{m,\tht}(\hat{\Si}\times \bbr)$ is a solution of \eqref{e.lin.AC} and $\|\varphi\|_{C^{0}_{\si}
 \left(\Si\times \bbr\right)} <\infty$ then \eqref{e.lin.ac.est} holds. To prove this, without loss of generality we can assume that $c_{\snf}=0$ and (because of Theorem \ref{t.sch}) it is enough to establish the following estimate
\begin{equation}\label{e.lin.ac.est'}
	\|\varphi\|_{C^{0}_{\si}\left(\Si_{(-2\de R)}\times \bbr\right)} \leq C \left(\|\ps\|_{C^{1,\tht}_{\si}\left(\del_0\Si_{(-\de R)}\times \bbr\right)} + \|\snf\|_{C^{0,\tht}_{\si}\left(\Si_{(-\de R)}\times \bbr\right)}\right).
\end{equation}
Let us assume by contradiction that the estimate in \eqref{e.lin.ac.est'} does not hold. So, \te\ sequences $\ve_k\to 0$, $R_k=\ve_k^{-1}$, $\mathsf{h}_k,\Si^k, \hat{\mathtt{g}}^k , \mathtt{g}^k, \vp_k, \snf_k, \ps_k$ \st\ the assumptions (i)--(iii) in Proposition \ref{p.lin.ac} are satisfied for \((\ve, R, \mathsf{h}, \Si, \hat{\mathtt{g}}, \mathtt{g} ) = ( \ve_k, R_k, \mathsf{h}_k, \Si^k, \hat{\mathtt{g}}^k, \mathtt{g}^k )\), \eqref{e.lin.AC} is satisfied for \((\vp, \snf, \ps ) = (\vp_k, \snf_k, \ps_k )\), and 
\begin{equation}\label{e.norm.1,0}
	\|\varphi\|_{C^{0}_{\si}\left(\Si^k_{(-2\de R)}\times \bbr\right)} =1, \left(\|\ps\|_{C^{1,\tht}_{\si}\left(\del_0\Si_{(-\de R)}\times \bbr\right)} + \|\snf\|_{C^{0,\tht}_{\si}\left(\Si_{(-\de R)}\times \bbr\right)}\right) \to 0.
\end{equation} 
\hn\ \tes\ $(p_k,s_k)\in \Si^k_{(-2\de R)}\times \bbr$ and $c>0$ \st\ 
\begin{equation}\label{e.(p_k,s_k)}
	e^{\si s_k}|\vp(p_k,s_k)|\geq c.
\end{equation}

If $q\in \text{int}(\Si^k_{(-2\de R)})$, using a normal coordinate system (cf. Lemmas \ref{l.exp.map}, \ref{l.tilde.t.Pi}) of \(\Si^k\) centreed at \(q\), we can find $Y_q:\B_r(0)\to \text{int}(\Si_k)$, where $r=\min\{R_k^{1/2}, d(q,\del\Si^k), \de R_k\}$, \st\ $Y_q$ is a diffeomorphism onto its image, $Y_q(0)=q$, $\|(Y_q^*\mathtt{g}^k)_{ij}-\de_{ij}\|_{C^2} = o(1)$ uniformly as $k\to \infty$. Similarly, if $q\in \del \Si^k_{(-2\de R)}$, using a Fermi coordinate system (cf. Lemmas \ref{l.Fermi.coord}, \ref{l.exp.map}, \ref{l.tilde.t.Pi}) of \(\del \Si^k\) centreed at \(q\), we can find $Y_q:\B^+_r(0)\to \Si_k$, where $r=\min\{R_k^{1/2}, \de R_k\}$, \st\ $Y_q$ is a diffeomorphism onto its image, $Y_q(0)=q$, $\|(Y_q^*\mathtt{g}^k)_{ij}-\de_{ij}\|_{C^2} = o(1)$ uniformly as $k\to \infty$.

Let $(p_k,s_k)$ be as in \eqref{e.(p_k,s_k)}. There are three possibilities:
\begin{equation}\label{e.p_k.del.Si^k}
	\text{(i)}\;\limsup_{k\to \infty} d(p_k,\del \Si^k) = \infty,\;\text{(ii)}\; \limsup_{k\to \infty} d(p_k,\del_0 \Si^k) < \infty,\; \text{(iii)}\; \limsup_{k\to \infty} d(p_k,\del_+ \Si^k) < \infty.
\end{equation}
If (i) happens, we set $q_k=p_k$; if (ii) (resp. (iii)) happens, we set $q_k$ to be the nearest point in \(\del_0 \Si^k\) (resp. \(\del_+ \Si^k\)) to $p_k$. \sps\ $\limsup_{k\to \infty} |s_k| <\infty$. Then in all the three cases (in \eqref{e.p_k.del.Si^k}), we define $Y_k = Y_{q_k}$, $\tilde{\vp}_k(\mathtt{y},s)= \vp_k(Y_k(\mathtt{y}), s)$, \(\tilde{\snf}_k(\mathtt{y},s)= \snf_k(Y_k(\mathtt{y}), s)\), and in case (ii) $\tilde{\ps}_k(\mathtt{y},s)= \ps_k(Y_k(\mathtt{y}), s)$. Then \(\tilde{\vp}_k\) satisfies the following equation.

\begin{equation}\label{e.tilde.vp.k.eq}
	-a^{ij}_k\del_{\mathtt{y}_i\mathtt{y}_j} \tilde{\vp}_k - b^i_k \del_{\mathtt{y}_i} \tilde{\vp}_k - \del_{ss} \tilde{\vp}_k + W''(\Hb)\tilde{\vp}_k = \tilde{\snf}_k,
\end{equation}
where \(\|a^{ij}_k-\de_{ij}\|_{C^{0,\tht}} + \|b^i_k\|_{C^{0,\tht}}=o(1)\) as \(k\to \infty\). In case (i), \eqref{e.tilde.vp.k.eq} is satisfied on \(\B_{r_k}(0) \times \bbr \), where \(r_k=\min\{R_k^{1/2}, d(q_k,\del\Si^k), \de R_k\}\). \tf\ by \eqref{e.norm.1,0}, \tes\ a bounded \fn\ \(\tilde{\vp}_{\infty}\in C^2(\bbr^n \times \bbr)\) \st\ up to a subsequence, \(\tilde{\vp}_k \to \tilde{\vp}_{\infty}\) in \(C^2_{\text{loc}}(\bbr^n \times \bbr)\),
\[-\De_{\mathtt{y}} \tilde{\vp}_{\infty} - \del_{ss}\tilde{\vp}_{\infty} + W''(\Hb)\tilde{\vp}_{\infty} =0 \text{ on }\bbr^n \times \bbr,\]
and \(\int_{-\infty}^{\infty}\tilde{\vp}_{\infty}(\mathtt{y},s)\Hb'(s)ds=0\) \fa\ \(\mathtt{y}\in \bbr^n\). \hn\ by Lemma \ref{l.entire.lin.ac}, \(\tilde{\vp}_{\infty}\equiv 0\). However, this is a contradiction since \eqref{e.(p_k,s_k)} implies \tes\ $s_{\infty}\in \bbr$ \st\ \(\tilde{\vp}_{\infty}(0,s_{\infty}) \neq 0\). In case (ii) (resp. (iii)), \eqref{e.tilde.vp.k.eq} is satisfied on \(\B^+_{r_k}(0) \times \bbr \), where \(r_k=\min\{R_k^{1/2}, \de R_k\}\) and $\del_{\nu} \tilde{\vp}_k=\tilde{\ps}_k$ (resp. \(\tilde{\vp}_k = 0 \)) on \(\del_0\B^+_{r_k}(0) \times \bbr \). \tf\ by \eqref{e.norm.1,0}, \tes\ a bounded \fn\ \(\tilde{\vp}_{\infty}\in C^2(\bbr^n_+ \times \bbr)\) \st\ up to a subsequence, \(\tilde{\vp}_k \to \tilde{\vp}_{\infty}\) in \(C^2_{\text{loc}}(\bbr^n_+ \times \bbr)\),
\[-\De_{\mathtt{y}} \tilde{\vp}_{\infty} - \del_{ss}\tilde{\vp}_{\infty} + W''(\Hb)\tilde{\vp}_{\infty} =0 \text{ on }\bbr^n_+ \times \bbr,\]
\(\del_{\nu}\tilde{\vp}_{\infty} = 0 \) (resp. $\tilde{\vp}_{\infty} = 0$) on \(\del \bbr^n_+ \times \bbr \)  and \(\int_{-\infty}^{\infty}\tilde{\vp}_{\infty}(\mathtt{y},s)\Hb'(s)ds=0\) \fa\ \(\mathtt{y}\in \bbr^n_+\). \hn\ by Lemma \ref{l.entire.lin.ac}, \(\tilde{\vp}_{\infty}\equiv 0\). However, this is a contradiction since \eqref{e.(p_k,s_k)} implies \tes\ $s_{\infty}\in \bbr$ \st\ \(\tilde{\vp}_{\infty} (0,s_{\infty}) \neq 0\). 

If $\limsup_{k\to \infty} |s_k| =\infty$, without loss of generality we can assume that \(s_k \to \infty\) as \(k\to \infty\). Let us consider the three possibilities stated in \eqref{e.p_k.del.Si^k} and define \(q_k\) as above. In all the three cases (in \eqref{e.p_k.del.Si^k}), we define $Y_k = Y_{q_k}$, $\tilde{\vp}_k(\mathtt{y},s)= e^{\si (s_k + s)}\vp_k(Y_k(\mathtt{y}), s)$, \(\tilde{\snf}_k(\mathtt{y},s)= e^{\si (s_k + s)}\snf_k(Y_k(\mathtt{y}), s)\), and in case (ii) $\tilde{\ps}_k(\mathtt{y},s)= e^{\si (s_k + s)} \ps_k(Y_k(\mathtt{y}), s)$. Then \(\tilde{\vp}_k\) satisfies the following equation.

\begin{equation}\label{e.tilde.vp.k.eq'}
	-a^{ij}_k\del_{\mathtt{y}_i\mathtt{y}_j} \tilde{\vp}_k - b^i_k \del_{\mathtt{y}_i} \tilde{\vp}_k - \del_{ss} \tilde{\vp}_k + 2\si \del_s \tilde{\vp}_k + (W''(\Hb(s_k+s)) - \si^2)\tilde{\vp}_k = \tilde{\snf}_k,
\end{equation}
where \(\|a^{ij}_k-\de_{ij}\|_{C^{0,\tht}} + \|b^i_k\|_{C^{0,\tht}}=o(1)\) as \(k\to \infty\). Arguing as in the previous case and using Lemma \ref{l.entire.lin.ac'}, we again arrive at a contradiction.
\end{proof}

\begin{lemma}\label{global gluing estimates 1}
	Let us fix $\de_i>0$, $i=0,1,2,3$, satisfying $\de_0+\de_1+\de_2<\de_3$ and $a,b>0$. Let $r\in (0, \de_1R)$, $\rh\in(1,\de_2R)$, $p\in B^+_{\de_0R}(0)$, $B^{+}_{s}(p)=B_s(p)\cap B^+_{2R}(0)$, $U\in C^{2,\tht}(B^+_{\de_3R}(0))$ with the uniform estimates
\begin{equation}\label{e.U.bound}
	U\geq a>0, \quad U+|\nabla U|\leq b \text{ on }B^+_{\de_3R}(0).
\end{equation}
Then \tes\ a constant $C>0$ independent of $r,\: \rho,\: p,\: U,\: \ve$ \w\ the following significance.
\begin{enumerate}[(i)]
	\item If $\ve>0$ is sufficiently small\footnote{The \mt\ $g$ in \eqref{e.De-1} depends on $\ve$.} and $\vp\in C^2(B^{+}_{r+\rh}(p))$ satisfies
	\begin{equation}\label{e.De-1}
		\begin{cases}
			\De_g\varphi-U\vp=E+\operatorname{div}(F)\text{ on } B^{+}_{r+\rh}(p),\\
			\partial_\nu \varphi=\psi \text{ on }\partial_0 B^{+}_{r+\rh}(p),\\
		\end{cases}
	\end{equation}
	there holds
	\begin{equation}\label{e.est1.De-1}
		\md{\vp}_{C^{1,\tht}(B^{+}_{r}(p))}\leq C\left(\md{E}_{C^{0}(B^{+}_{r+\rh}(p))}+\md{F}_{C^\theta(B^{+}_{r+\rh}(p))} +\md{\ps}_{C^{0, \tht}(\del_0B^{+}_{r+\rh}(p))}+\rh^{-1}\md{\vp}_{C^{1,\tht}(B^{+}_{r+\rh}(p))}\right).
	\end{equation}	

\item If $\ve>0$ is sufficiently small and $\vp\in C^2(B^{+}_{r+\rh}(p))$ satisfies
\begin{equation}\label{e.De-1'}
	\begin{cases}
		\De_g\varphi-U\vp=E+\operatorname{div}(F)\text{ on } B^{+}_{r+\rh}(p),\\
		\varphi=\psi \text{ on }\partial_0 B^{+}_{r+\rh}(p),\\
	\end{cases}
\end{equation}
there holds
\begin{equation}\label{e.est1.De-1'}
	\md{\vp}_{C^{1,\tht}(B^{+}_{r}(p))}\leq C\left(\md{E}_{C^{0}(B^{+}_{r+\rh}(p))}+\md{F}_{C^\theta(B^{+}_{r+\rh}(p))} +\md{\ps}_{C^{1, \tht}(\del_0B^{+}_{r+\rh}(p))}+\rh^{-1}\md{\vp}_{C^{1,\tht}(B^{+}_{r+\rh}(p))}\right).
\end{equation}
\end{enumerate}
\end{lemma}

\begin{proof}
	The proof is similar to \cite[proof of Lemma 8.1]{del2016introduction}. Since the proofs of (i) and (ii) are very similar, we will prove only (i). Let us first prove the following claim.
	\begin{claim}\label{outer claim1}
	\Tes\ a constant $C>0$ independent of $r,\: \rho,\: p,\: U,\: \ve$ \st\ if	$\ve>0$ is sufficiently small, $\tph,\tE,\tF,\tps$ are compactly supported in $B^{+}_{r+\rh}(p)$ and \eqref{e.De-1} is satisfied for $(\vp,E, F, \ps) =(\tph,\tE,\tF,\tps)$, there holds 
		\begin{equation}\label{e.est2.De-1}
			|\tph|_{C^{0}(B^{+}_{r+\rh}(p))}\leq C\left(|\tE|_{C^{0}(B^{+}_{r+\rh}(p))}+|\tF|_{C^\theta(B^{+}_{r+\rh}(p))}+|\tps|_{C^{0, \tht}(\del_0B^{+}_{r+\rh}(p))}\right).
		\end{equation}
	\end{claim}
 To prove this claim, we assume by contradiction that the estimate \eqref{e.est2.De-1} does not hold. Then \te\ sequences $\tilde{U}_i, \tph_i,\tE_i,\tF_i,\tps_i$, $r_i,\: \rho_i,\: p_i,\: \ve_i$ \st\  $\ve_i\ra 0$, $\tph_i,\tE_i,\tF_i,\tps_i$ are compactly supported in $B^{+}_{r_i+\rh_i}(p_i)$, 
		\begin{equation}\label{e.est.tph_n}
			|\tph_i|_{C^{0}(B^{+}_{r_i+\rh_i}(p_i))}=1,\quad \left(|\tE_i|_{C^{0}(B^{+}_{r_i+\rh_i}(p_i))}+|\tF_i|_{C^\theta (B^{+}_{r_i+\rh_i}(p_i))} +|\tps_i|_{C^{0, \tht}(\del_0B^{+}_{r_i+\rh_i}(p_i))}\right)\ra 0
		\end{equation} 
		and \eqref{e.De-1} is satisfied for $\ve=\ve_i$ and 
		$$(U,\vp,E, F, \ps,r,\rh,p)=(\tilde{U}_i, \tph_i,\tE_i, \tF_i, \tps_i,r_i, \rho_i, p_i).$$
Then
	\begin{align}\label{phi_n-equation out}
			& \hat{g}^{jk}(\ve_i\cdot)\del_{jk}\tph_i-\ve_i\hat{g}^{jk}(\ve_i\cdot)\hat{\Ga}_{jk}^l(\ve_i\cdot)\del_{l}\tph_i-\tilde{U}_i\tph_i=\tE_i+\operatorname{div}\tF_i \text{ on } B^{+}_{r_i+\rh_i}(p_i);\\
			&\partial_{\nu}\tph_i=\tps_i\text{ on }\partial_0 B^{+}_{r_i+\rh_i}(p_i).
	\end{align}
Since $\tph_i,\tE_i,\tF_i,\tps_i$ are compactly supported, we can assume that \eqref{phi_n-equation out} holds on $B^+_{\de_3 R}(0)$. Let $\de'_3=\de_0+\de_1+ \de_2$. \Tes\ $\tilde{x}_i\in B^+_{\de'_3 R}(0)$ \st\ $|\tph(\tilde{x}_i)|\geq 1/2$. 

Suppose $\limsup_{i\to\infty} \tilde{x}_{i,1} =\infty$. We define $\bar{\vp}_i(x)=\tph_i(\tilde{x}_i+x)$. Then $|\bar{\vp}_i(0)|\geq 1/2$ and (by Theorem \ref{t.sch}) the $C^{1,\tht}$ norm of $\bar{\vp}_i$ is uniformly bounded on $B(0,s_i)$ where $$s_i=\frac{1}{2}\min\left\{\tilde{x}_{i,1},(\de_3-\de'_3)R_i\right\}.$$
\tf\ \te\ $\vp_{\infty},U_{\infty}:\bbr^{n+1}\to \bbr$ \st\ up to a subsequence, $\bar{\vp}_i \to \vp_{\infty}$ in $C^1_{\text{loc}}(\bbr^{n+1})$,  
\[\De_{g_E}\vp_{\infty}-U_{\infty}\vp_{\infty}=0 \text{ on }\bbr^{n+1}\]
and $|\vp_{\infty}(0)|\geq 1/2$. However, this contradicts Lemma \ref{l.entire.lin.ac'}(i).

On the other hand, if $\limsup_{i\to\infty} \tilde{x}_{i,1} <\infty$, we define $\bar{x}_i=(0,\tilde{x}_{i,2},\dots, \tilde{x}_{i,n+1})$ and $\bar{\vp}_i(x) =\tph_i(\bar{x}_i+x)$. The $C^{1,\tht}$ norm of $\bar{\vp}_i$ is uniformly bounded on $B(0,\frac{1}{2}(\de_3-\de'_3)R_i)$. \tf\ \te\ $\vp_{\infty},U_{\infty}:\bbr^{n+1}_+\to \bbr$ satisfying
\begin{align}
	& \De_{g_E}\vp_{\infty}-U_{\infty}\vp_{\infty}=0 \text{ on }\bbr^{n+1}_+,\\
	& \del_{\nu}\vp_{\infty}=0 \text{ on }\del\bbr^{n+1}_+.
\end{align}
Moreover, $|\vp_{\infty}(x_{\infty})|\geq 1/2$ for some $x_{\infty}\in \bbr^{n+1}_+$, which contradicts Lemma \ref{l.entire.lin.ac'}(iii). This finishes the proof of the claim. 
	
	To prove the lemma, we choose a cut-off \fn $0\leq\ch\leq1$ \st\ $\ch=1$ on $B^+_r(p)$, $\text{spt}(\ch)\subset B^{+}_{r+\rh}(p)$, $\md{\na \ch}+\md{\na^2 \ch}\lesssim \rh^{-1}$ and apply the above claim to $\tph=\ch\vp$. Indeed, $		\De_{g}\tph -U\tph =\tE+\textrm{div}(\tilde{F}) \text{ on } B^{+}_{r+\rh}(p)$, where $\tE=\ch E+2\langle\na \vp, \na \ch\rangle +\vp\De\ch-\langle F, \na \ch\rangle $, $\tilde{F}=\chi F$ and $\partial_{\nu} \tph=\tps\quad \text{on }\partial_0 B^{+}_{r+\rh}(p)$ where $\tps=\ch\ps+\vp\del_1\ch$. \tf, the lemma follows from above claim and  elliptic estimate
	\begin{align}
		\md{\vp}_{C^{1,\tht}(B^{+}_{r}(p))}&=\md{\tph}_{C^{1,\tht}(B^{+}_{r}(p))}
        \lesssim \left(|\tE|_{C^{0}(B^{+}_{r+\rh}(p))}+|\tilde{F}|_{C^\theta(B^{+}_{r+\rh}(p))}+|\tps|_{C^{0, \tht}(\del_0B^{+}_{r+\rh}(p))}+\md{\tph}_{C^{0}(B^{+}_{r+\rh}(p))}\right).    
	\end{align}	
\end{proof}
\end{appendix}

\bibliography{biblio.bib}
\bibliographystyle{alpha}
\end{document}